\documentclass[11pt, a4paper, reqno]{amsart}
\usepackage{upgreek}
\usepackage{amssymb}

\usepackage[euler-digits]{eulervm}

\usepackage{fouridx}
\usepackage{mathpazo,mathabx}
\usepackage{amsfonts, amsbsy}
\usepackage{dsfont}
\usepackage{mathabx}
\usepackage{xcolor}
\usepackage{enumitem}

\setlist[enumerate]{label=(\roman*)}
\usepackage{shadethm}
\usepackage{cite} 

\usepackage{colortbl}
\usepackage[colorlinks,citecolor=blue,urlcolor=blue]{hyperref}

\usepackage{amsmath}
\usepackage{framed}
\usepackage{mathrsfs,mathtools}
\usepackage{esint}
\usepackage{todonotes}

\newcommand\newone[1]{{{#1}_{1}}}
\newcommand\newonep[1]{{{#1}^{\prime}_{1}}}
\newcommand\newzero[1]{{{#1}_{1}}}

\newcommand{\old}[1]{\todo[inline,color=red!40]{Old: #1}}
\newcommand{\new}[1]{\todo[inline,color=teal!40]{New: #1}}

\newcommand{\newdela}[1]{}
\newcommand{\olddela}[1]{}
\renewcommand{\P}{\boldsymbol{\Pi}}

\newcommand\comadel[1]{{\color{red}}}
\newcommand\addadel[1]{{\color{blue}}}

\newcommand{\mfI}{{\mathfrak{I}\/}}
\newcommand\linspan{\mathrm{linspan}}
\newcommand\Leb{\mathrm{Leb}}
\newcommand\loc{\mathrm{loc}}
\newcommand\rc{\mathrm{c}}
\newcommand\sol{\mathrm{sol}}
\newcommand\mc{\mathrm{m}}
\newcommand\rK{\mathrm{K}}
\newcommand{\rx}{x}                         
\newcommand\rU{U}
\newcommand{\Vtwo}{D(A_1)}
\newcommand{\embed}{\hookrightarrow}
\newcommand{\supp}{\mathrm{supp}\,}

\newcommand{\domO}{\mathcal{O}}
\newcommand{\mO}{E}
\newcommand{\nU}{{U}}
\newcommand{\bis}{{\prime\prime}}
\newcommand{\nHB}{{\mathbb{H}}}

\newcommand{\1}{$\mathds{1}$}

\newcommand{\bSi}{\bm{\Sigma}}

\newcommand{\newG}{{\mathfrak{G}}}

\newcommand{\tildeW}{\widetilde{W}}

\newcommand{\newK}{\kappa}

\newcommand{\Hsol}[4]{\fourIdx{#1}{#2}{#3}{#4}{\mathbb{H}}}

\newcommand{\ilsc}[3]{\fourIdx{}{}{}{#3}{\big( #1, #2 \big)}}

\newcommand{\duality}[4]{\fourIdx{}{#4}{}{#3}{\big\langle #1, #2 \big\rangle}}

\newcommand{\dualitybig}[5]{\fourIdx{}{#4}{#5}{#3}{\big\langle #1, #2 \big\rangle}}

\newcommand{\ipsmall}[3]{\fourIdx{}{}{}{#3}{\langle #1, #2 \rangle}}

\newcommand{\normm}[1]{\|#1\|}
\newcommand{\ip}[2]{\langle #1,\, #2 \rangle}

\newcommand{\HS}{{\mathscr{T}_2}}

\hypersetup{urlcolor=midnightblue, citecolor=red}

\theoremstyle{remark}

\newcommand{\bm}{\boldsymbol}
\newcommand{\toup}{\nearrow}

\newcommand{\eps}{\varepsilon}

\newcommand{\be}{\mbox{\boldmath{$e$}}}
\newcommand{\bn}{\mbox{\boldmath{$n$}}}
\newcommand{\bv}{\mbox{\boldmath{$v$}}}
\newcommand{\bu}{\mbox{\boldmath{$u$}}}

\newcommand{\bg}{\mbox{\boldmath{$g$}}}

\newcommand{\fb}{\mbox{\boldmath{$f$}}}
\newcommand{\bb}{\mbox{\boldmath{$b$}}}

\newcommand{\bX}{\ensuremath{\mbox{\(\scriptstyle\boldsymbol{X}\)}}} 

\newcommand{\bZ}{\mbox{\boldmath{$Z$}}}

\newcommand{\Aone}{A_1}
\newcommand{\Athree}{A_3}
\newcommand{\Atwo}{A_2}

\newcommand{\Stokes}{\mbox{\boldmath{$A_0$}}}

\newcommand{\StokesH}{{H_0}}

\newcommand{\StokesV}{{V_0}}
\newcommand{\StokesVp}{{V^{\prime}_0}}

\newcommand{\bB}{\mbox{\boldmath{$B$}}}

\newcommand{\bF}{\mbox{\boldmath{$F$}}}
\newcommand{\bG}{\mbox{\boldmath{$G$}}}

\newcommand{\bR}{\mbox{\boldmath{$R$}}}

\newcommand{\bw}{\mbox{\boldmath{$\mathrm{w}$}}}

\newcommand\zero[2]{\fourIdx{}{0}{#2}{}{#1}}

\newcommand{\cp}{\underbar{$\mu$}}
\newcommand{\avg}[1]{\overline{#1}}

\def\d{\mathrm{d}}

\DeclareMathOperator{\diver}{div}

\DeclareMathOperator*{\esssup}{ess\,sup}

\newcommand\dela[1]{}

\selectfont

\theoremstyle{plain}
\newtheorem{theorem}{\textbf{Theorem}}[section]
\newtheorem{lemma}[theorem]{\textbf{Lemma}}
\newtheorem{properties}[theorem]{\textbf{Properties}}
\newtheorem{proposition}[theorem]{\textbf{Proposition}}

\newtheorem{corollary}[theorem]{\textbf{Corollary}}
\theoremstyle{definition}
\newtheorem{notation}[theorem]{\textbf{Notation}}
\newtheorem{remark}[theorem]{\textbf{Remark}}

\newtheorem{definition}[theorem]{\textbf{Definition}}
\newtheorem{assumption}[theorem]{\textbf{Assumption}}
\newtheorem{example}[theorem]{\textbf{Example}}

\newshadetheorem{Theorem}[theorem]{Theorem}

\newshadetheorem{Proposition}[theorem]{Proposition}

\newtheorem{claim}{Claim}[section]

\numberwithin{equation}{section}
\numberwithin{figure}{section}
\usepackage{xcolor}
\makeatletter
\protected\def\old#1{\ifmmode{\color{red}#1}\else\textcolor{red}{#1}\fi}
\protected\def\new#1{\ifmmode{\color{blue}#1}\else\textcolor{blue}{[#1]}\fi}
\makeatother
\providecommand{\dela}[1]{}
\renewcommand{\dela}[1]{}
\AtBeginDocument{\renewcommand{\1}{\mathds{1}}}
\begin{document}
\title[Weak solutions to the 2D or 3D stochastic NSCHEs]{Weak martingale solutions to stochastic Navier-Stokes-Cahn-Hilliard system with transport noise}

\author{ Z. Brze\'zniak and A. Ndongmo Ngana}
\dedicatory{\vspace{-10pt}\normalsize{Department of Mathematics, University of York
Heslington, York, YO105DD, UK}}



\begin{abstract}
We consider a diffuse interface model for the mixture of two incompressible fluids driven by transport noise.
Under suitable abstract assumptions, we prove the existence of a global weak martingale solution as well as the pathwise uniqueness of a global strong solution in the two dimensional case. This is the first result addressing a unified framework derived from the works by Brze{\'z}niak \& Motyl and Mikulevicius \& Rozovskii, on the study of stochastic partial differential equations.

\end{abstract}

\maketitle


\section{Introduction}
This paper is concerned with the study of the Navier-Stokes-Cahn-Hilliard equations (NSCHEs) driven by transport noise. 
The NSCHEs are a phase field model where the interface separating two incompressible fluids is replaced by a diffuse one, i.e. it has a small but non-zero thickness, by introducing the so-called phase function $\phi$ (or order parameter), whose dynamics interact with the fluid velocity.
This Diffuse Interface model, also called Model H after the seminal work on dynamic critical phenomena \cite{Hohenberg+Halperin_1977}, describes incompressible binary mixtures with constant total density.
The model, which is basically a coupling of the Navier-Stokes equations (NSEs) and the convective Cahn-Hilliard equation (CHE), was derived in \cite{Gurtin+Polignone+Vinals_1996} through the framework of continuum mechanics and in \cite{Liu+Shen_2003} via an energetic variational framework.

\noindent
Consider the flow of an isothermal, homogeneous, incompressible, viscous binary fluid mixture occupying a domain $\domO$, under the action of a capillary force proportional to 
\[
\nabla \left(\frac{\eps}{2} \lvert \nabla \phi \rvert^2 + \alpha \psi(\phi) \right) - \nabla \cdot (\nabla \phi \otimes \nabla \phi),
\]
which arises from the surface tension due to the mixing of the two fluids. Near the interface, there is a thin transition layer where the two components coexist; it is in this sense that the fluid is a "mixture": the term refers to this thin interfacial region, not to a global homogeneous mixing of the two fluids. The positive constants $\alpha$ and $\eps$ are related to the model and are fixed throughout the paper.
For example, the parameter $\eps$ is related to the thickness of the interface separating the two fluids. The function $\psi$ is the homogeneous free energy density for the mixture, also called the potential.
We assume that $\domO$ is a bounded domain in  $\mathbb{R}^d$, $d=2,3$ with  $\bn$ denoting the outward, i.e. external, normal vector field  to $\partial \domO$. We further assume that 
$(\Omega, \mathscr{F}, \mathbb{P}, \mathbb{F})$ is a filtered probability space and $\widetilde{W}=(\tilde{w}_k)_{k \in \mathbb{N} \setminus\{0\}}$, where $ (\tilde{w}_k(t))_{t \in [0,\infty)}$, is a  sequence of i.i.d. standard Brownian motions on that probability space.

\noindent
We are interested in global solutions but for the simplicity we will work in finite time horizon and we choose and fix $T>0$. 
We consider the following stochastic governing equations for this flow:
\begin{equation}\label{eqn-stochastic-CHNSEs}
\left\{\begin{aligned}
&\d \bu + [\diver(\bu \otimes \bu) - \nu \Delta \bu + \nabla p + \newK \diver(\nabla \phi \otimes \nabla \phi)]\,\d t \\
&\hspace{0.5 truecm} = \sum_{k=1}^{\infty}  \left[(\sigma_k(t,x) \cdot \nabla) \bu -\nabla \bar{p}_k + \bg_k(t,x,u) \right]\d \tilde{w}_k &\mbox{ in } (0,T) \times \domO,
\\
&\diver \bu= 0 &\mbox{ in } (0,T) \times \domO,
  \\
&\d \phi + \diver(\phi \bu) \,\d t= \varrho_0 \,\Delta \left(- \eps \Delta \phi(t) + \alpha \psi^\prime(\phi(t)) \right)\d t &\mbox{ in } (0,T) \times \domO, 
\end{aligned}
\right.
\end{equation}
where the unknown functions are the velocity field
   \begin{equation*}
     \bu: [0,T] \times \domO \times   \Omega \to \mathbb{R}^d,
  \end{equation*}
the difference of the  fluid concentrations, called the phase field or the order parameter, 
     \begin{equation*}
       \phi: [0,T] \times \domO \times   \Omega \to \mathbb{R}
     \end{equation*}
and, the mean and turbulent pressures
    \begin{equation*}
      p,\, \bar{p}_k: [0,T] \times \domO \times   \Omega \to \mathbb{R}, \;\; k \in \mathbb{N} \setminus\{0\}.
   \end{equation*}
Given a solution $(\bu,\phi)$ of the system \eqref{eqn-stochastic-CHNSEs}, one often introduces a function  $\mu$ defined by 
   \begin{equation}\label{eqn-mu}
     \mu: [0,T] \ni  t  \mapsto  - \eps \Delta \phi(t) + \alpha \psi^\prime(\phi(t)),
   \end{equation}
which is called the chemical potential. Note that this chemical potential $\mu$ is obtained under an appropriate choice of boundary conditions as a variational derivative of the free energy functional
   \begin{equation}
     \mathscr{F}(\phi)= \int_{\domO} \left(\frac{\eps}{2}\lvert \nabla \phi \rvert^2 + \alpha \psi(\phi) \right)\d x.
   \end{equation}
In this way, the third equation in \eqref{eqn-stochastic-CHNSEs} takes the following compact form 
\begin{equation}\label{stochastic-CHNSEs-2}
\d \phi + \diver(\phi \bu)\,\d t= \varrho_0 \,\Delta\mu(t)\,\d t \mbox{ for } t \in (0,T).
\end{equation}
In fact, the last equation contains no Wiener process and can be regarded, for almost every fixed $\omega$, as a deterministic PDE with random coefficients, i.e.
   \begin{equation}\label{stochastic-CHNSEs-3}
     \frac{\partial  \phi}{\partial t} + \diver(\phi \bu) = \varrho_0 \,\Delta   \mu (t)  \mbox{ for } t \in (0,T).
   \end{equation}
The function $\psi^\prime$ in \eqref{eqn-mu} is the derivative of the potential $\psi$ which may be a polynomial or contain a logarithmic term.
In certain cases, the logarithmic or singular potential can be approximated by a polynomial of even degree with a strictly positive dominant coefficient.
A classical example is the Landau potential \cite{Landau+Lifshitz_1968}
              \begin{equation}\label{eqn-regular-potential-psi_0}
                  \psi=\psi_0: \mathbb{R} \ni s \mapsto  \frac{1}{4}(1-s^2)^2\in \mathbb{R},
              \end{equation}
which is an approximation of the Flory-Huggins potential
\begin{equation}\label{eqn-logarithmic-potential-psi_00}
\psi=\psi_{00}:      \mathbb{R} \ni s \mapsto             \frac{\theta}{2} [(1+s) \ln(1+s) + (1-s)\ln(1-s)] - \frac{\theta_0}{2} s^2 \in \mathbb{R}, 
\end{equation}
up to an additive constant in the limit $\theta/\theta_0 \to 1$ with $0<\theta<\theta_0$.
\newline
The positive constants $\newK$ and $\varrho_0$ are related to the model and are fixed throughout the paper. More precisely, $\newK$ corresponds to the surface tension \cite{Liu+Walkington_2001}, and 
$\varrho_0$ is the elastic relaxation time of the system. Without loss of generality, we can simplify our notation by putting:
    \begin{equation}\label{eqn-special parameters}
      \eps= \varrho_0= \alpha= 1.
    \end{equation}
The system \eqref{eqn-stochastic-CHNSEs} contains also the following  functions, 
\begin{align}\label{eqn-sigma_k}
\sigma_k: [0,T] \times \domO \ni (t,x) \mapsto \sigma_k(t,x) \in \mathbb{R}^d, \; k \in \mathbb{N} \setminus\{0\}, 
\\ \label{eqn-g_k}
 \bg_k: [0,T] \times \domO \times \mathbb{R}^d \ni (t,x,z) \mapsto \bg_k(t,x,z) \in \mathbb{R}^d, \;k \in \mathbb{N} \setminus\{0\}.
\end{align}
We also define the following first order differential operators $\sigma_k \cdot \nabla$ acting on vector  fields
   \begin{equation*}
     [(\sigma_k(t) \cdot \nabla)\bu](x) \coloneqq \left(\sum_{j=1}^d \sigma_k^j(t,x) \partial_j \bu^i(x) \right)_{i=1}^d, \;\; x \in \domO,
   \end{equation*}
and the tensor product of two vector fields  $\bu,\, \bv: \domO \to \mathbb{R}^d$ as 
\begin{equation*}
    [\bu \otimes \bv](x) \coloneqq (\bu(x) \otimes \bv(x))= \left[\bu^i(x) \bv^j(x)\right]_{i,j=1}^d, \;\; x \in \domO. 
 \end{equation*}
In particular, 
  \begin{equation*}
    [\nabla \phi \otimes \nabla \phi](x) \coloneqq (\nabla \phi(x) \otimes \nabla \phi(x))= \left[\partial_i \phi(x) \partial_j \phi(x)\right]_{i,j=1}^d, \;\; x \in \domO. 
 \end{equation*}
\medskip 

We supplement  the system \eqref{eqn-stochastic-CHNSEs} with the Dirichlet boundary conditions for the velocity $\bu$ and  the natural no-flux conditions for the phase field $\phi$, cf. \cite{Ciprian+Grasselli_2010}, i.e.
\begin{equation}\label{eqn-boundary-conditions-NSCHs}
\begin{aligned}
\bu&= 0,  
\\
\frac{\partial \phi}{\partial \bn} &= \frac{\partial  \Delta \phi}{\partial \bn}  = 0 \quad \text{on} \quad (0,T) \times \partial \domO,
\end{aligned}
\end{equation}
Finally, we consider  the following initial conditions
    \begin{equation}\label{eqn-initial conditions-NSCHs}
         \bu(0,\cdot)= \bu_0, \quad \phi(0,\cdot)= \phi_0   \mbox{ in } \domO.
    \end{equation}
The second part of the boundary condition \eqref{eqn-boundary-conditions-NSCHs} is similar to condition (4.3) in \cite{Brz+Masl+S_2005} and condition (1.5) in \cite{Da Prato+Debussche_1996}, which are often encountered in the study of the beam equation and the Cahn-Hilliard equation, respectively. Let us point out that this second part of the boundary condition \eqref{eqn-boundary-conditions-NSCHs}
implies the mass conservation of the phase field $\phi$, i.e. 
  \begin{equation}\label{eq-mass-conservation}
    \langle \phi(t) \rangle= \langle \phi(0) \rangle=: M_0, \quad \forall t\in (0,T).
  \end{equation}
In fact, \eqref{eqn-boundary-conditions-NSCHs} implies
    \begin{equation}
        \frac{\partial \mu}{\partial \bn} =0 \quad \text{on} \quad (0,T) \times \partial \domO,
    \end{equation}
which in turn implies the conservation of the following quantity:
   \begin{equation}
     \langle \phi(t) \rangle\coloneqq \fint_{\domO} \phi(t,x) \, \d x, \;\; t\in (0,T).
   \end{equation}
Let us present now a very rough version of Theorem \ref{First-main-result} which is our main result.
\begin{theorem}\label{thm-main-result}
Assume that $\domO$ is a bounded and sufficiently smooth domain in $\mathbb{R}^d$ with $d \in \{2,3\}$. 
Assume that $\nu$ and $\kappa$ are positive constants.  Assume that the functions $\sigma_k$ and $\bg_k$ are measurable and sufficiently regular w.r.t. the 3rd variable. Assume that $\psi: \mathbb{R}    \to \mathbb{R}$ is the Landau potential. Assume that  $(\Omega,\mathscr{F},\mathbb{P})$ is a probability space with a right-continuous filtration $\mathbb{F}=(\mathscr{F}_t)_{t \in [0,T]}$ 
 such that on this probability space two i.i.d. copies of  $\ell^2$-cylindrical $\mathbb{F}$-adapted  Wiener processes are defined.
Then, there exist an $\ell^2$-cylindrical $\mathbb{F}$-adapted Wiener process $W=(w_k)_{k\in\mathbb{N} \setminus\{0\}}$, progressively measurable processes $\bu$, $\phi$, $p$, and a sequence $(p_k)_{k\in\mathbb{N} \setminus\{0\}}$, whose trajectories belong to suitable functional spaces, such that the tuple $(\bu,p,\phi,(p_k)_{k\in\mathbb{N} \setminus\{0\}})$ is a weak solution of the following problem:
\begin{equation}\label{eqn-stochastic-CHNSEs-intro}
\left\{\begin{aligned}
&\d \bu + [\diver(\bu \otimes \bu) - \nu \Delta \bu + \nabla p + \newK \diver(\nabla \phi \otimes \nabla \phi)]\,\d t \\
&\hspace{0.5 truecm} = \sum_{k=1}^{\infty}  \left[(\sigma_k(t,x) \cdot \nabla) \bu -\nabla p_k + \bg_k(t,x,u) \right]\d w_k &\mbox{ in } (0,T) \times \domO,
\\
&\diver \bu= 0 &\mbox{ in } (0,T) \times \domO,
  \\
&\d \phi + \diver(\phi \bu) \,\d t= 
\,\Delta \left(-  \Delta \phi(t) + \,\psi^\prime(\phi(t)) \right)\d t &\mbox{ in } (0,T) \times \domO.
\end{aligned}
\right.
\end{equation}
\end{theorem}
A rough version of Theorem \ref{Thm-uniqueness-solution} about the uniqueness of solutions is as follows. 
\begin{theorem}\label{thm-main-result-second}
Assume that $d=2$ and the coefficients  $\bg_k$ are Lipschitz w.r.t. the 3rd variable. Assume that the progressively measurable processes 
$(\bu^1,p^1,\phi^1), (p^1_k)_{k\in \mathbb{N} \setminus\{0\}}$ and $(\bu^2,p^2,\phi^2),(p^2_k)_{k\in \mathbb{N} \setminus\{0\}}$, defined on the same filtered probability space 
$(\Omega, \mathscr{F}, \mathbb{P}, \mathbb{F})$ are weak solutions of the  problem
\eqref{eqn-stochastic-CHNSEs-intro} with the same i.i.d. standard Brownian motions  $W=({w}_k)_{k \in \mathbb{N} \setminus\{0\}}$. Then, in a suitable sense, almost surely, 
$\bu^1=\bu^2$, $\phi^1=\phi^2$, $\nabla(p^1-p^2)=0$ and $\nabla(p^1_k-p^2_k)=0$.
\end{theorem}
The mathematical analysis of NSCHEs has been addressed in several papers, covering deterministic and stochastic cases, constant and non-constant densities, regular and singular potentials. Some of the main achievements in this topic are summarized below.

\noindent
\textit{State of the art for deterministic NSCHEs.}
Analogously to the classical theory of NSEs, definitive results are available in the deterministic setting; see, e.g., \cite{Abels_2009,Abels+Depner+Garcke_2013,Ciprian+Grasselli_2010,Gal+Grasselli+Miranville_2016,Giorgini+Miranville+Temam_2019,Giorgini+Temam_2020,Zhao_2019}.
In \cite{Abels_2009}, the author proved the existence of global weak solutions for a class of physical relevant and singular free energy densities. In particular, unique strong solutions exist globally in time in $2D$ and locally in time in $3D$ for "nice" initial data. 
Furthermore, in the absence of external forces, weak solutions converge to stationary states as $t \to \infty$.
In \cite{Giorgini+Miranville+Temam_2019}, the existence, uniqueness, and regularity of weak and strong solutions were proved for the singular potential \eqref{eqn-logarithmic-potential-psi_00}. Analogous results for the Landau and Flory-Huggins potentials were established in \cite{Giorgini+Temam_2020}, assuming the density is bounded from above and bounded away from zero, and that the viscosity depends on the fluid concentration $\phi$. The uniqueness of the solution in $2D$ was subsequently proved in \cite{Giorgini+Ngana+Tachim+Temam_2023}.
We also mention \cite{Gal+Grasselli+Miranville_2016}, where the existence of global weak solutions is proved for boundary conditions accounting for moving contact lines, and \cite{Liu+Shen_2003}, where a Fourier-spectral method for numerical approximation for the homogeneous NSCHEs was proposed. 
Finally, \cite{Tachim+Tone+Tone_2021} investigated the Pontryagin maximum principle for optimal control problems associated with the NSCH model. 

\noindent
\textit{State of the art for stochastic NSCHEs.}
To the best of our knowledge, the first stochastic analysis of the NSCHEs on a bounded $2D$ domain was presented in \cite{Tachim_2017}, proving the existence and uniqueness of probabilistic strong solutions. This was followed by several works \cite{Deugoue+Ngana+Medjo_2021,Deugoue+Ngana+Medjo_2023,Deugoue+Medjo_2018,Deugoue+TT_2018}.
In \cite{Deugoue+Ngana+Medjo_2021}, a local maximal solution in the PDE sense was established and this maximal solutions turn to be a global one in the $2D$ case. In \cite{Deugoue+Ngana+Medjo_2023}, convergence results were established. 
Furthermore, the existence of unique strong solutions in the probability and PDEs senses, along with the stability results, were achieved in \cite{Deugoue+Medjo_2018} and \cite{Deugoue+TT_2018} for 3D globally modified NSCHEs.
However, we are not aware of any results dealing with the analysis of the NSCHEs with transport-type noise, as presented in system \eqref{eqn-stochastic-CHNSEs}. Finally, let us mention a paper \cite{Di Primio+Scarpa+Zanella_2025} about stochastic Allen-Cahn-Navier-Stokes Equations. This and the present  papers are similar in that in both we consider a coupled system consisting the NSEs and another one. But the second part of this system is different.

\medskip

\noindent
The purpose of this article is to provide a deep analysis of the Stochastic  NSCHEs perturbed by gradient, or transport, type noise, see the right-hand side of the first equation in \eqref{eqn-stochastic-CHNSEs},which constitutes a more general class of diffuse-interface models that is widely regarded as more physically relevant.
The main contribution of this work is the development of a unified framework that combines two complementary approaches, 
building on the results \cite{Mik+Roz_2005} and \cite{Brz+Motyl_2013}. This combination allows us to establish the existence of solutions for NSCHEs perturbed by gradient-type noise, a setting that cannot be treated by either method alone.

\medskip
\noindent
Let us make  two comments about our approach.  Firstly, as in those two papers, we consider a sequence of  approximating problems and show the tightness of the laws of the corresponding solutions on a certain non-metric topological space $\bZ_T$. Unlike \cite{Brz+Motyl_2013}, we do not use the Jakubowski-Skorokhod Theorem, see \cite[Theorem 2]{Jakubowski_1998}, to construct a new probability space and a sequence of random elements on this new probability space that converge point-wise to a another random element, which at the end is proved to be a solution. Instead, we use a generalization of the Prohorov Theorem, see
Proposition \ref{prop-Prohorov-general}, which is based on a discussion on \cite[p. 174]{Jakubowski_1998}, 
to prove that the laws above have a sequence that are weakly convergent, in a strong sense described in Proposition \ref{prop-Prohorov-general},  to a probability measure  $\mathbb{P}_T$ 
on the above mentioned  space $\bZ_T$. This part of the argument, to the best of our understanding, was not made explicit in \cite{Mik+Roz_2005}. 
Following the steps of \cite{Mik+Roz_2005}, we prove that the canonical process 
on the space $\bZ_T$ is a martingale solution of our problem with respect to the measure $\mathbb{P}_T$.  
Secondly, the approach of \cite{Brz+Motyl_2013} by the first named author and Motyl  does not seem to apply here: owing to the coupled nature of our system, we are unable to derive a priori estimates for moments of the velocity field higher than the second order.

\medskip
\noindent

The structure of the paper is as follows. Section \ref{Sect-Functional analysis} introduces the notation and preliminaries on frequently used function spaces, operators, and potentials. In Section \ref{sec-stochastic preliminaries} we present some useful preliminary results on stochastic processes, define the Hilbert-Schmidt operators, and state the assumptions on the coefficients related to system \eqref{eqn-stochastic-CHNSEs}. Section \ref{sec-Reformulation of the problem} introduces the notion of weak solutions for the problem under consideration, while the main results are stated in Theorems \ref{First-main-result} and \ref{First-main-result-uniqueness} under the abstract assumptions given in Section \ref{Ass-Abstract formulation}. Sections \ref{subsec-Galerkin} and \ref{Sec-Tightness} are devoted to the existence and properties of global strong solutions of the Galerkin approximation problem. Further properties of some stochastic processes defined on the canonical sample space $\bZ_T$ are collected in Section \ref{Limit-meaures}. In this same section, we also state and prove Theorem \ref{thm-5.3}, which forms the central core of the proof of Theorem \ref{First-main-result}. Section \ref{sec-proof-main-result} is devoted to the proof of our main result, namely Theorem \ref{First-main-result}. In Section \ref{Sect-approximation of g} we investigate the application of the abstract results from Section \ref{Ass-Abstract formulation} to the stochastic Navier-Stokes-Cahn-Hilliard equations, while in Section \ref{eqn-uniqueness-weak-solution-2-dim} we prove the pathwise uniqueness of weak solutions in dimension 2. Finally, the appendices gather some auxiliary results and technical lemmas, propositions, and corollaries used throughout the paper.

 \section{Mathematical setting}\label{Sect-Functional analysis}
\subsection{Notation and Preliminaries}
Here we collect the  necessary notation and information for the mathematical analysis of the problem \eqref{eqn-stochastic-CHNSEs}-\eqref{eqn-initial conditions-NSCHs}. 
Normed vector spaces will be denoted by capital letters, e.g. $X$ and  $Y$. 
The norm in a vector space $X$ will be usually denoted  by $\Vert \cdot \Vert_X$.
Hilbert spaces will be denoted also by capital letters, e.g. $H$ and $K$. The inner product in a Hilbert space $H$ will be usually denoted  by $(\cdot, \cdot)_H$, or simply by $(\cdot,\cdot)$. 
\newline
For two normed spaces $X$ and $Y$, $\mathscr{L}(X,Y)$ denotes the spaces of bounded linear operators from $X$ to $Y$ and its operator norm is denoted by $\Vert \cdot \Vert_{\mathscr{L}(X,Y)}$, whereas by $X^\prime$, we denote the dual space of $X$, i.e. $X^\prime=\mathscr{L}(X,\mathbb{R})$, and by $\duality{\cdot}{\cdot}{X}{X^\prime}$ (or simply by $\langle \cdot,\cdot \rangle$ if there is not confusion) we will denote the duality pairing between $X^\prime$ and $X$. Similar notation is used when $X$ is a topological normed vector space. 
If $X$ is a topological space then by $\mathscr{B}(X)$ we denote the Borel $\sigma$-field on $X$. 
\newline
If $K$ and $H$ are two separable Hilbert spaces, we use the notation $\mathscr{T}_2(K,H)$ for the space of Hilbert-Schmidt operators from $K$ to $H$. By $\mathcal{C}_b(\domO,\mathbb{R}^d)$, we denote the space of continuous and bounded $\mathbb{R}^d$-valued functions on $\domO$.
If $p \in [1,\infty]$, then by $L^p(\domO)=L^p(\domO,\mathbb{R})$, resp. $\mathbb{L}^p(\domO) \coloneqq L^p(\domO,\mathbb{R}^d)$, we denote the Lebesgue space of all (equivalence classes) of Lebesgue-measurable $\mathbb{R}$, resp. $\mathbb{R}^d$-valued functions defined on $\domO$  endowed with the classical norm. When $p=2$, these Banach spaces are Hilbert spaces. If additionally $k \in \mathbb{R}_+$, then we will denote by
$W^{k,p}(\domO)$, resp. $\mathbb{W}^{k,p}(\domO)\coloneqq W^{k,p}(\domO,\mathbb{R}^d)$, the Sobolev space (of order $k$), of $\mathbb{R}$, resp. $\mathbb{R}^d$-valued  functions on $\domO$. 
By $W^{k,p}_0(\domO)$, resp. $\mathbb{W}^{k,p}_0(\domO)$ we denote the closure of the space $\mathcal{C}_0^\infty(\domO)$, resp. $\mathcal{C}_0^\infty(\domO,\mathbb{R}^d)$, in the space 
 $W^{k,p}(\domO)$, resp. $\mathbb{W}^{k,p}(\domO)$, with the norm inherited from the larger space. 
 When $p=2$, we use notation  $H^k(\domO)\coloneqq W^{k,2}(\domO)$,
  $\mathbb{H}^k(\domO)\coloneqq \mathbb{W}^{k,2}(\domO)$, 
 $H^k_0(\domO)\coloneqq W^{k,2}_0(\domO)$, and  $\mathbb{H}^k_0(\domO)\coloneqq \mathbb{W}^{k,2}_0(\domO)$. By $|\cdot|_{L^2}$ and $(\cdot,\cdot)$ we will denoted the norm and the inner product in $L^2(\domO)$.
\begin{notation}\label{not-D_i}
The weak partial derivative in the direction $i$ will be denoted by $\frac{\partial }{\partial x_i}$, $\partial_i$ or simply by $D_i$. By $\nabla $ we will denote the weak gradient. 
\\
In this paper, we use the convention that the set $\mathbb{N}$ of natural numbers is the one whose  smallest element is $1$. 
\\
If $Y$ is a Hilbert space, then by  $\ell^2(Y)=\ell^2(\mathbb{N},Y)$, we denote the Hilbert  space of all $Y$-valued  sequences $y=(y_n)_{n\in\mathbb{N}}$ such that $\sum\limits_{n=1}^\infty \Vert y_n \Vert_Y^2<\infty$,  i.e. 
\begin{equation}\label{eqn-l^2Y}
\ell^2(Y) \coloneqq \left\{ y=(y_n)_{n\in\mathbb{N}}: \sum\limits_{n=1}^\infty \Vert y_n \Vert_Y^2<\infty  \right\},
\end{equation}
endowed with the natural  inner product 
\[
(y,y^\prime)_{\ell^2}=\sum\limits_{n=1}^{\infty} (y_n, y_n^\prime)_Y, \;\; y=(y_n),\, y^\prime=(y_n^\prime)\in \ell^2(\mathbb{N},Y).
\]
In the special case $Y=\mathbb{R}$, the space $\ell^2(\mathbb{N},\mathbb{R})$ is denoted by  $\ell^2$. 
\end{notation}

\subsection{Some useful inequalities}

Let us recall the classical Sobolev or Gagliardo-Nirenberg interpolation inequality, which will be useful in the sequel; see, e.g., \cite{Brezis_2011} and references therein. There exists a constant $C=C(\domO,p)>0$ such that 
\begin{equation}\label{Gagliardo-Nirenberg-inequality}
\begin{aligned}
\Vert \phi \Vert_{L^p} \leq C \lvert \phi \rvert_{L^2}^{\frac2p} \Vert \phi \Vert_{H^1}^{1 - \frac2p},\;\;\; \forall \phi \in H^1(\domO), &\mbox{ if } p \geq 2 \; \mbox{ and } \;  d=2,
\\
\Vert \phi \Vert_{L^p} \leq C \lvert \phi \rvert_{L^2}^\frac{6-p}{2p} \Vert \phi \Vert_{H^1}^\frac{3p - 6}{2p},\;\;\; \forall \phi \in H^1(\domO), &\mbox{ if } 2 \leq p \leq 6 \; \mbox{ and } \;  d=3.
\end{aligned}
\end{equation}
We also recall the following Agmon inequalities, see, e.g., \cite[Chapter II, (1.40)]{Temam_1997},
\begin{equation}\label{eq-Agmon's-inequalities}
\begin{aligned}
\Vert \phi \Vert_{L^\infty}\leq C \vert \phi \rvert_{L^2}^{1/2} \Vert \phi \Vert_{H^2}^{1/2},\;\; \forall \phi \in H^2(\domO), \mbox{ if } d=2,
\\
\Vert \phi \Vert_{L^\infty}\leq C \Vert \phi \Vert_{H^1}^{1/2} \Vert \phi \Vert_{H^2}^{1/2},\;\; \forall \phi \in H^2(\domO), \mbox{ if } d=3.
\end{aligned}
\end{equation}

\subsection{Notation related to the Navier-Stokes part}\label{subsec-notation related to the Navier-Stokes part}

Next, we introduce the function spaces that will be used throughout the paper, beginning with the notation for the Navier-Stokes component of the system \eqref{eqn-stochastic-CHNSEs}.
\begin{equation}\label{eqn-spaces-NSEs}
\begin{aligned}
\mathcal{V} &\coloneqq \{\bv \in \mathcal{C}_0^\infty(\domO,\mathbb{R}^d): ~ \diver \bv= 0\},
\\
\StokesH&\coloneqq \text{the closure of} ~ \mathcal{V} ~ \mbox{ in } ~ \mathbb{L}^2(\domO),
   \\
\StokesV&\coloneqq \text{the closure of} ~ \mathcal{V} ~ \mbox{ in } ~ \mathbb{H}^1(\domO).
\end{aligned}
\end{equation}
By $\lvert \cdot \rvert_{\StokesH}$ and $(\cdot,\cdot)_{\StokesH}\coloneqq (\cdot,\cdot)$ we will denote the norm and the inner product in $\StokesH$. 
 We endow the space $\StokesV$ with the norm and inner product inherited from $\mathbb{H}^1(\domO)$, i.e.
\begin{equation*}
(\bu,\bv)_{\StokesV}\coloneqq (\bu,\bv) + (\nabla \bu,\nabla \bv), \quad \bu,\,\bv \in \StokesV,
\end{equation*}
and the norm induced by the inner product $(\cdot,\cdot)_{\StokesV}$ in $\StokesV$ is then given by
\[
\Vert \bv \Vert_{\StokesV}^2= \lvert \bv \rvert_{\StokesH}^2 + \lvert \nabla \bv \rvert_{\mathbb{L}^2}^2, \quad \bv \in \StokesV.
\]
Notice that the spaces $\StokesH$ and $\StokesV$ can also be characterized in a different way, see \cite[Theorems 1.1/1.4 and 1.1/1.6]{Temam_2001}:
   \begin{align*}
      \StokesH&= \{\bv \in \mathbb{L}^2(\domO): \diver \bv=0, \; \bu \cdot \bn=0 \mbox{ on } \partial  \domO \}, \\
      \StokesV&= \{\bv \in \mathbb{H}_0^1(\domO): \diver \bv=0\}=\mathbb{H}_0^1(\domO) \cap \StokesH.
    \end{align*}
For every $k \in[0,\infty)$, we introduce the following standard scale of Hilbert spaces
  \begin{equation}\label{eqn-Vk0-space}
    \Hsol{}{}{k}{\sol}= \mbox{ the closure of } \mathcal{V} \mbox{ in }  \mathbb{H}^{k}(\domO).
  \end{equation}
Arguing as in the proof of \cite[Theorem 1.1]{Temam_2001} one can show that
\[
\Hsol{}{}{k}{\sol}= \mathbb{H}^{k}(\domO) \cap \StokesH, \quad k \in[0,\infty).
\]
In what follows, we will denote the dual of the space $\Hsol{}{}{k}{\sol}$ by $\Hsol{}{}{-k}{\sol}$ and we choose and fix a real number $k_0>0$ such that
\begin{equation}\label{eqn-k_0}
    k_0> \frac d2 + 1
\end{equation}
and put 
\begin{equation}\label{eqn-U_0}
\rU_0 \coloneqq \Hsol{}{}{k_0}{\sol}.
\end{equation}
Note that $\rU_0 \embed \StokesH$ and by identifying $\StokesH$ with its dual, we have the following Gelfand triple:
\begin{equation}\label{eqn-U_0-H_0}
\rU_0  \embed \StokesH \cong \StokesH^\prime \embed  \rU_0^\prime.
\end{equation}
Moreover, we have the following  generalized Gelfand triple: 
  \begin{equation}\label{eqn-Gelfand triple}
    \rU_0 \embed  \StokesV \embed  \StokesH \cong \StokesH^\prime \embed  \StokesVp \embed \rU_0^\prime
  \end{equation}
and the following identity holds: for all $\bv \in \rU_0$ and $\bu \in \StokesVp$,
     \begin{equation}\label{eqn-duality-identity}
         \duality{\bu}{\bv}{\StokesV}{\StokesVp}=\duality{\bu}{\bv}{\rU_0}{\rU_0^\prime}.
     \end{equation}
Next, in view of \eqref{eqn-k_0}, we infer that  the following Sobolev embeddings also hold
   \begin{equation}\label{eq-H^s-1-to-C_b}
      \mathbb{H}^{k_0 - 1}(\domO) \embed  \mathcal{C}_b(\domO,\mathbb{R}^d) \embed  \mathbb{L}^\infty(\domO).
   \end{equation}
Let us denote by 
\begin{equation}\label{eqn-LHP}
\P: \mathbb{L}^2(\domO) \to \StokesH,
\end{equation}
the orthogonal projection usually called the Helmholtz-Leray projection.
It is known, see \cite[Remark I.1.6]{Temam_2001}, that $\P$ maps continuously the Hilbert space $H^1(\domO)$ into itself.\\
From now on, we denote by $\Stokes$ the Stokes operator defined by 
\begin{equation}\label{eqn-Stokes operator}
\left\{
\begin{aligned}
D(\Stokes) &\coloneqq \mathbb{H}^2(\domO) \cap \StokesV,
\\
\Stokes u &\coloneqq -\P \Delta u, \;\; u \in  D(\Stokes).
\end{aligned}
\right.
\end{equation}
It is well known, see \cite[Chapter I, Section 2.6]{Temam_1979}, that $\Stokes$ is a nonnegative self-adjoint operator in the Hilbert space $\StokesH$. 

\subsection{Notation related to the Cahn-Hilliard part}\label{subsec-notation related to the Cahn-Hilliard part}

Now, we discuss  the notation related to the Cahn-Hilliard part of the system \eqref{eqn-stochastic-CHNSEs}.  
According to the  generalized Poincar\'{e} inequality, see \cite[Chapter II, Section 1.4]{Temam_1997}, the function 
  \begin{equation}\label{eqn-H^1-norm-special}
  H^1(\domO) \ni    \phi  \mapsto \left(\lvert \nabla \phi \rvert_{\mathbb{L}^2}^2 + \left \lvert  \langle \phi \rangle\right \rvert^2 \right)^{1/2} \in [0,\infty)
  \end{equation}
is a norm on $H^1(\domO)$, equivalent to the standard $H^1(\domO)$ norm. \\
Furthermore, the function 
 \begin{equation}\label{eqn-H^2-norm-special}
     H^2(\domO) \ni \phi  \mapsto \left( \lvert -\Delta \phi \rvert_{L^2}^2 + \left \lvert  \langle \phi \rangle \right \rvert^2\right)^{1/2} \in [0,\infty)
\end{equation}
is a norm on $H^2(\domO)$, equivalent to the standard graph norm. Here, $-\Delta$ is the minus Laplacian with the Neumann boundary conditions. \\
Now let us recall the mass conservation of the phase field $\phi$, see, \eqref{eq-mass-conservation}, i.e.
  \begin{equation*}
    \langle \phi(t) \rangle= \langle \phi(0) \rangle, \quad \forall t>0.
  \end{equation*}
For the remainder of the paper, we will assume that
  \begin{equation*}
   \langle \phi(0) \rangle=0,
  \end{equation*}
since, up to a constant shift of $\phi$, the initial mean can be taken to be zero, and then by the mass conservation \eqref{eq-mass-conservation} it remains zero for all positive times.\\
Next, we also introduce the function spaces that will be used throughout the paper.
\begin{equation}
\begin{aligned}
\zero{L}{2} (\domO)&\coloneqq \{\phi \in L^2(\domO): \langle \phi \rangle=0\},
 \\
\zero{H}{s}(\domO)&\coloneqq \{\phi \in H^s(\domO): \langle \phi \rangle=0\},\;\; s\in\{1,2,3\}.
\end{aligned}
\end{equation}
We denote by $\Aone$ the minus Neumann Laplacian on mean-zero functions defined by 
\begin{equation}\label{eqn-Aone}
\begin{cases}
D(\Aone)\coloneqq \{\phi \in H^2(\domO) \cap \zero{L}{2}(\domO): \; \partial_{\bn} \phi= 0 \mbox{ on }  \partial \domO\},
\\
\Aone \phi\coloneqq - \Delta \phi, \;\; \phi \in D(\Aone).
\end{cases}
\end{equation}
Observe that 
     \begin{equation}\label{eqn-\Aone-integral=0}
       \fint_{\domO} \Aone \phi(x)\,\d x =0, \;\; \phi \in D(\Aone)
     \end{equation}
and  that $\Aone$ is a definite positive, unbounded, and self-adjoint operator on $\zero{L}{2}(\domO)$ with the inverse $\Aone^{-1}$ being compact.
In particular, there exists an orthonormal basis (ONB) $\{\hat{e}_j\}_{j=1}^\infty \subset \Vtwo$ of $\zero{L}{2}(\domO)$  and an increasing sequence $\{\alpha_j\}_{j=1}^\infty$ of the eigenvalues of $\Aone$, with $\alpha_1>0$ and $\alpha_j \nearrow \infty$ as $j \to \infty$ such that $\Aone \hat{e}_j= \alpha_j \hat{e}_j$, $j=1,2,\ldots$ \\
Moreover, the operator $\Aone^2$ defined by
\begin{equation}\label{eqn-Aone^2}
\begin{cases}
D(\Aone^2 + I)= D(\Aone^2)\coloneqq \{\phi \in \Vtwo: \; \Aone \phi \in \Vtwo\},
\\
\Aone^2 \phi\coloneqq \Delta^2 \phi, \;\; \phi \in D(\Aone^2),
\end{cases}
\end{equation}
is also self-adjoint on $\zero{L}{2}(\domO)$, see \cite[Theorem 13.13]{Rudin_1991}. In particular, we see that
\begin{equation}
\begin{aligned}
D(\Aone^2)
&=\{\phi \in \Vtwo: \; \Aone \phi \in \Vtwo\}
\\
&= \{\phi \in H^2(\domO) \cap \zero{L}{2}(\domO):\; \Delta \phi \in H^2(\domO) \cap \zero{L}{2}(\domO),\; \partial_{\bn} \phi= \partial_{\bn} \Delta \phi= 0 \mbox{ on } \partial \domO \}
\\
&= \{\phi \in H^4(\domO) \cap \zero{L}{2}(\domO): \partial_{\bn} \phi= \partial_{\bn} \Delta \phi= 0 \mbox{ on } \partial \domO \}.
\end{aligned}
\end{equation}
Following   \cite[Remark 4.6 in Chapter 4, p.253]{Lions+Magenes_1972_vol-1}, we introduce the following sets:
\begin{align}\label{eqn-H-LM-R.4.6}
\newone{H} &\coloneqq \zero{H}{1}(\domO)= H^1(\domO) \cap \zero{L}{2}(\domO),
\\
\label{eqn-V-LM-R.4.6}
\newone{V} &\coloneqq \{\phi \in \newone{H}:  \Delta \phi \in H^1(\domO), \; \partial_{\bn} \phi=  0 \mbox{ on } \partial \domO\}.
\end{align}
We endow the space $\newone{H}$ with the inner product $(\cdot,\cdot)_{\newone{H}}$ and norm $\lvert \cdot \rvert_{\newone{H}}$ defined, resp., by 
\begin{align*}
(\phi,\psi)_{\newone{H}}\coloneqq (\nabla \phi,\nabla \psi),\;\;\quad \lvert \phi \rvert_{\newone{H}}\coloneqq \lvert \nabla \phi \rvert_{\mathbb{L}^2},\;\; \forall \phi,\,\psi \in \newone{H}.
\end{align*}
We endow the space $\newone{V}$ with the following inner product
\begin{equation*}
(\phi,\psi)_{\newone{V}}\coloneqq (\phi,\psi)_{\newone{H}} + (\Delta \phi,\Delta \psi)_{\newone{H}},\;\; \phi,\,\psi \in \newone{V}.
\end{equation*}
We will denote the norm induced by the above inner product by $\Vert \cdot \Vert_{\newone{V}}$, i.e.
\[
\Vert \phi \Vert_{\newone{V}}\coloneqq (\lvert \phi \rvert_{\newone{H}}^2 + \lvert \Delta \phi \rvert_{\newone{H}}^2)^{1/2},\;\;\phi \in \newone{V}.
\]
Recall, see \cite{Temam_2001}, that $D(\Aone^{1/2})=\newone{H}=H^1(\domO) \cap \zero{L}{2}(\domO)$. For future reference, let us recall, see \cite{Lions+Magenes_1972_vol-1} that the dual of the space $H^1(\domO)$ is not a subset of the space of distributions on  $\domO$. The reason for this fact is that the space $C_0^\infty(\domO)$ is not dense in $H^1(\domO)$. Furthermore, notice that
$\newone{V}= D(\Aone^{3/2})$.\\
Identifying the space $\newone{H}$ with its dual, as above, we construct the following Gelfand triple 
\begin{equation}\label{eqn-Gelfand triple-abstract-2}
\newone{V} \embed \newone{H} \cong \newonep{H} \embed \newonep{V},    
\end{equation}
so that 
\[
\duality{\phi}{\psi}{\newone{V}}{\newonep{V}}= (\phi, \psi)_{\newone{H}}\;\; \mbox{ if } \phi \in \newone{H} \mbox{ and } \psi \in \newone{V}.
\]
Define next a bilinear form 
\[
\tilde{a}_1: \newone{V} \times \newone{V} \ni (\phi, \psi) \mapsto (\phi,\psi)_{\newone{H}} + (\Delta \phi,\Delta  \psi)_{\newone{H}} \in \mathbb{R}.
\]
Since this form  is continuous and $\newone{V}$-coercive, by the Lax-Milgram Theorem, the 
linear map 
\[
\newone{\tilde{\mathscr{A}}}: \newone{V} \to \newonep{V}
\]
defined, for all $\phi \in \newone{V}$ and $\zeta \in  \newonep{V}$,  by 
\begin{equation}\label{eqn-A_1-script}
\newone{\tilde{\mathscr{A}}}\phi= \zeta \;\; \iff 
\tilde{a}_1(\phi,\psi)= \duality{\zeta}{\psi}{\newone{V}}{\newonep{V}}\;\; \mbox{ for every } \psi\in \newone{V},
\end{equation}
is an isomorphism. Moreover, following \cite{Lions+Magenes_1972_vol-1}, one can prove that 
\[
\newone{\tilde{\mathscr{A}}}\phi= \Delta^2\phi + \phi, \;\; \phi \in \newone{V}.
\]
We finally put 
\[
\newone{\mathscr{A}}: \newone{V} \ni  \phi \mapsto \newone{\tilde{\mathscr{A}}}\phi -\phi= \Delta^2 \phi \in \newonep{V}
\]
and observe also that  $\newone{\mathscr{A}}$  can be decomposed as
\[
\newone{\mathscr{A}}= \Athree \circ \Atwo, 
\]
where
\begin{align*}
    \Atwo: \newone{V}  \ni \phi \mapsto -\Delta  \phi \in \newone{H} \mbox{ and }
\Athree: \newone{H}  \ni \phi \mapsto -\Delta  \phi \in \newonep{V}.   
\end{align*}
We conclude this subsection by also introducing the following Cartesian product spaces, which will be used throughout this work:
\begin{align}\label{eqn-spaces full}
\nU= \rU_0 \times \newone{V}, \;\;\mathbb{H}
= \StokesH \times \newone{H}, \mbox{ and }
\mathbb{V}= \StokesV \times \newone{V}.
\end{align}
These spaces also form the following Gelfand triple:
\begin{equation}
\label{eqn-Gelfand-full}  \nU \embed \nHB\cong \mathbb{H}^\prime \embed \nU^\prime.
\end{equation}
\subsection{The trilinear forms \(b_0\), \(b_1\) and \(r_0\)}
Let us introduce some trilinear maps and their associated bilinear operators, which play a crucial role in the analysis of system \eqref{eqn-stochastic-CHNSEs}. 
\begin{definition}
We define maps  $\bb_0$ and $r_0$ by
    \begin{equation}\label{eqn-b form}
      \bb_0: \StokesV \times \StokesV \times \StokesV\ni(\bu,\bv,\bw)\mapsto \sum_{i,j= 1}^{d} \int_{\domO} \bu^i(x)\, \partial_i \bv^j(x)\,\bw^j(x) \in \mathbb{R}
    \end{equation}
and
   \begin{align}\label{r-trilinear-form}
       r_0: W^{1,4}(\domO) \times W^{1,4}(\domO) \times \StokesV \ni  (\phi,\psi, \bu)&\mapsto  \sum_{i,j=1}^d \int_{\domO} \partial_i \phi(x)\, \partial_j \psi(x)\, \partial_j \bu^i(x)\,\d x \in \mathbb{R}.
      \end{align}
By the H\"older inequality, there exists a generic constant $C>0$ such that 
   \begin{equation}\label{eqn-b_0-trilinear-estimate-weak}
      \lvert \bb_0(\bu,\bv,\bw) \rvert \leq C \Vert \bu \Vert_{\mathbb{L}^4} \lvert \nabla\bw \rvert_{\mathbb{L}^2} \Vert \bv \Vert_{\mathbb{L}^4}, \;\;\forall \bu,\,\bv,\,\bw \in \StokesV. 
   \end{equation}
   and
   \begin{equation}\label{eqn-r_0-trilinear-estimate}
      \lvert r_0(\phi,\psi,\bu) \rvert \leq C \Vert \bu \Vert_{\StokesV} \Vert \nabla \phi \Vert_{\mathbb{L}^4} \Vert \nabla \psi  \Vert_{\mathbb{L}^4} , \;\;\forall \bu \in \StokesV,\, \phi,\, \psi \in W^{1,4}(\domO).
   \end{equation}
In view of the Gagliardo-Nirenberg inequality \eqref{Gagliardo-Nirenberg-inequality}, employing a standard density argument, we deduce the following inequality and Lemma
   \begin{equation}\label{eqn-b_0-trilinear-estimate}
      \lvert \bb_0(\bu,\bv,\bw) \rvert  \leq  C \lvert \bu \rvert_{\StokesH}^{1 -\frac d4} \lvert \nabla \bu \rvert_{\mathbb{L}^2}^{\frac d4} \lvert \bv \rvert_{\StokesH}^{1 - \frac d4} \lvert \nabla \bv \rvert_{\mathbb{L}^2}^{\frac d4} \lvert \nabla \bw \rvert_{\mathbb{L}^2}, \;\;\;\forall \bu,\,\bv,\,\bw \in \StokesV,
   \end{equation}
\end{definition}
\begin{lemma}\label{lem-r form}
There exists $C=C(\domO)>0$ such that 
for all  $\phi,\, \psi \in \zero{H}{2}(\domO)$, and $\bu \in \StokesV$, 
   \begin{equation}\label{r-estimate}
      \lvert r_0(\phi,\psi,\bu) \rvert \leq C\lvert  \phi \rvert_{\newone{H}}^{1 - \frac d4} \Vert \phi \Vert_{\zero{H}{2}}^{\frac d4} \lvert  \psi \rvert_{\newone{H}}^{1 - \frac d4} \Vert \psi \Vert_{\zero{H}{2}}^{\frac d4} \lvert \nabla \bu \rvert_{\mathbb{L}^2}.
  \end{equation}
\end{lemma}
%
Recall that $\bb_0$ is the well-known trilinear form used in the mathematical analysis of the NSEs, see e.g. \cite{Temam_1979}.\\
From the above inequalities  \eqref{eqn-b_0-trilinear-estimate-weak} and \eqref{r-estimate}, we infer the following results, see e.g. \cite[Section II.1.2]{Temam_1979} for the first part.
\begin{corollary}\label{cor-B_0 and R} 
There exist  bounded bilinear maps
   \begin{equation}\label{eqn-B_0}
      \bB_0: \StokesV \times \StokesV \to \StokesVp
    \end{equation}
\begin{equation}\label{eqn-R_0}
 \bR_0:   \zero{H}{2}(\domO) \times \zero{H}{2}(\domO) \to \StokesVp,
\end{equation}
such that, respectively, 
     \begin{align}\label{properties-B_0}
         \duality{\bB_0(\bu,\bv)}{\bw}{\StokesV}{\StokesVp}&= \bb_0(\bu, \bv, \bw),\;\;\forall \bu,\,\bv,\,\bw \in \StokesV, 
     \end{align}
and  
    \begin{equation}\label{eq-2.12}
      \duality{\bR_0(\phi,\psi)}{\bu}{\StokesV}{\StokesVp}= r_0(\phi,\psi,\bu), \;\;\;\forall \phi,\,\psi \in \zero{H}{2}(\domO),\,\bu \in \StokesV.
    \end{equation}
Moreover, the operator $\bR_0$ can be uniquely extended to a bounded bilinear operator (denoted by the same letter),
   \begin{equation*}
     \bR_0: \newone{H} \times  \newone{H} \to \rU_0^\prime.
   \end{equation*}
In particular, there exists a constant $C$ depending only on $\domO$ such that
     \begin{equation}\label{eqn-R_0-k_0}
         \Vert \bR_0(\phi,\psi) \Vert_{\rU_0^\prime} \leq C \lvert \phi \rvert_{\newone{H}} \lvert \psi \rvert_{\newone{H}}, \; \phi,\,\psi \in \newone{H}.
     \end{equation}
\end{corollary}
\begin{proof}[Proof of Corollary \ref{cor-B_0 and R}]
Assertion \eqref{eqn-R_0} is a direct consequence of the fact that for a fixed $\bv \in \StokesV$, the mapping $r_0(\cdot,\cdot,\bv)$ defined on $ \zero{H}{2}(\domO) \times  \zero{H}{2}(\domO)$ with values in $\mathbb{R}$ is continuous.
\newline
By the definition \eqref{r-trilinear-form} and Sobolev embeddings \eqref{eq-H^s-1-to-C_b}, we infer that there exists $C>0$ such that for all $\phi,\,\psi \in \newone{H}$ and $\bv \in \rU_0$,
    \begin{equation*}
       \lvert r_0(\phi,\psi,\bv) \rvert
       \leq \lvert \phi \rvert_{\newone{H}} \lvert \psi \rvert_{\newone{H}} \Vert \nabla \bv \Vert_{\mathbb{L}^\infty}
       \leq C \lvert \phi \rvert_{\newone{H}} \lvert \psi \rvert_{\newone{H}} \Vert \bv \Vert_{\rU_0}.
    \end{equation*}
Therefore, for an arbitrary $\bv \in \rU_0$, the mapping 
\[
r_0(\cdot,\cdot,\bv): \newone{H} \times  \newone{H} \to \mathbb{R}
\]
 is bilinear and continuous what  implies the existence of a continuous bilinear map $\bR_0: \newone{H} \times  \newone{H} \to \rU_0^\prime$ such that $\duality{\bR_0(\phi,\psi)}{\bv}{\rU_0}{\rU_0^\prime}= r_0(\phi,\psi,\bv)$ for all $\phi,\,\psi \in \newone{H}$, satisfying property \eqref{eqn-R_0-k_0}. This completes the proof of Corollary \ref{cor-B_0 and R}.
\end{proof}
It is well known, see, e.g. \cite{Temam_2001},  that 
\begin{align}\label{properties-B_0-2}
\begin{aligned}
\bb_0(\bu, \bv, \bw)&= - \bb_0(\bu, \bw,\bv),\;\;\forall \bu,\,\bv,\,\bw \in \StokesV,
 \\
 \duality{\bB_0(\bu,\bv)}{\bv}{\StokesV}{\StokesVp}&= 0,\;\;\forall \bu,\,\bv\in \StokesV.
\end{aligned}
\end{align}
Moreover, see  \cite{Brz+Motyl_2013}, in view of assumption \eqref{eqn-k_0},  $\bB_0$ can be uniquely extended to a bounded bilinear operator (still) denoted by $\bB_0$,
\[
\bB_0: \StokesH \times \StokesH \to \rU_0^\prime \mbox{ and } \bB_0: \StokesH \times \rU_0 \to \StokesH, 
\] 
which hence  satisfy for some constant $C=C(\domO)$,
\begin{equation}\label{eq-B_0-k_0}
\begin{aligned}
\Vert \bB_0(\bu,\bv) \Vert_{\rU_0^\prime} 
&\leq C \lvert \bu \rvert_{\StokesH} \lvert \bv \rvert_{\StokesH},\;\;\forall \bu,\,\bv \in \StokesH,
\\
\lvert \bB_0(\bu,\bv) \rvert_{\StokesH} 
&\leq C \lvert \bu \rvert_{\StokesH} \Vert \bv \Vert_{\rU_0},\;\;\forall \bu \in \StokesH,\; \bv \in \rU_0.
\end{aligned}
\end{equation}

\begin{remark}\label{rem-R_0}
If $\phi,\, \psi \in \zero{H}{2}(\domO)$, then clearly $\diver(\nabla \phi \otimes \nabla \psi) \in \mathbb{L}^2(\domO)$.
Thus, the restriction of the map $\bR_0$ to the space $\zero{H}{2}(\domO) \times \zero{H}{2}(\domO)$  has the following representation:
    \begin{equation}\label{eqn-R_0-1}
      \bR_0(\phi,\psi)= \P[\diver(\nabla \phi \otimes \nabla \psi)], \;\;\forall \phi,\, \psi \in \zero{H}{2}(\domO).
    \end{equation}
\end{remark}
However, we have another equivalent form for the operator $\bR_0$.
\begin{proposition}\label{prop-R_0} 
If $\phi \in \zero{H}{2}(\domO)$ and 
   \begin{equation}\label{eqn-mu-2}
       \cp(\phi) \coloneqq - \varepsilon \Delta \phi + \alpha(\psi^\prime(\phi) - \avg{\psi^\prime(\phi)}), \;\; \; \phi \in \zero{H}{2}(\domO),
   \end{equation}
then $\cp(\phi)\nabla \phi \in \mathbb{L}^2(\domO)$ and 
  \begin{equation}\label{eqn-R_0-2}
     \bR_0(\phi,\phi)=  -\frac{1}{\eps} \P (\cp(\phi)\nabla \phi).
   \end{equation} 
\end{proposition}

\begin{proof}[Proof of Proposition \ref{prop-R_0}]
Identity \eqref{eqn-R_0-2}  follows from the  equality \eqref{eqn-R_0-1}, the fact that
\begin{align*}
&-\diver(\nabla \phi \otimes \nabla \phi)
= - \Delta \phi \nabla \phi - \nabla \left(\frac{1}{2} \lvert \nabla \phi \rvert^2 \right)
\\
&= \frac{1}{\eps} [\cp(\phi)\nabla \phi -  \alpha \nabla (\psi(\phi) - \avg{\psi^\prime(\phi)} \phi)] - \nabla \left(\frac{1}{2} \lvert \nabla \phi \rvert^2 \right)
\\
&= \frac{1}{\eps} (\cp(\phi)\nabla \phi)  - \nabla \left(\frac{\alpha }{\eps} [\psi(\phi) - \avg{\psi^\prime(\phi)} \phi] + \frac{1}{2} \lvert \nabla \phi \rvert^2 \right),
\end{align*}
by applying  the Leray-Helmholtz decomposition, i.e. Theorem I.1.4 from \cite{Temam_2001}.
\end{proof}  
Note that the above identities are not self-contradictory, as the chemical potential $\cp$ in \eqref{eqn-mu-2} depends on the parameter $\eps$.
\begin{definition}
Consider the trilinear form $b_1$ defined by
   \begin{equation}\label{eqn-b_1-trilinear-form-1}
b_1: \StokesV \times \zero{H}{2}(\domO) \times W^{1,4}(\domO) \ni 
   (\bu,\phi,\psi)\mapsto  \sum_{i=1}^d \int_{\domO} \nabla (\bu_i(x) \partial_i \phi(x)) \cdot \nabla \psi(x)\,\d x \in \mathbb{R}.
  \end{equation}
Since $H^1 \embed L^4$ by the Sobolev embedding Theorem, the Lebesgue integral on the RHS of 
\eqref{eqn-b_1-trilinear-form-1} exists. Then, by the H\"older inequality, we deduce that there exists a  positive constant $C$ such that the following estimates hold for all $\bu \in \StokesV,\,\phi \in H^2(\domO)$, and $\psi \in W^{1,4}(\domO)$:
\begin{equation}
\begin{aligned}
&\lvert b_1(\bu,\phi,\psi) \rvert
\leq \lvert \nabla \bu \rvert_{\mathbb{L}^2} \Vert \nabla \phi \Vert_{\mathbb{L}^4} \Vert \nabla \psi \Vert_{\mathbb{L}^4} + \Vert \bu \Vert_{\mathbb{L}^4} \left(\sum_{i,j=1}^d \lvert \partial_i(\partial_j \phi) \rvert_{L^2}^2 \right)^{1/2} \Vert \nabla \psi \Vert_{\mathbb{L}^4}
\\
&\leq \lvert \nabla \bu \rvert_{\mathbb{L}^2} \Vert \nabla \phi \Vert_{\mathbb{L}^4} \Vert \nabla \psi \Vert_{\mathbb{L}^4} + \Vert \bu \Vert_{\mathbb{L}^4} \Vert \phi \Vert_{H^2} \Vert \nabla \psi \Vert_{\mathbb{L}^4}
\leq C \Vert \bu \Vert_{\StokesV} \Vert \phi \Vert_{H^2} \Vert \psi \Vert_{W^{1,4}}. 
\end{aligned}
\end{equation}
\end{definition}
We also observe that by \eqref{eqn-b_1-trilinear-form-1}, if $\bu \in \StokesV$ and  $\phi,\, \psi \in \zero{H}{2}(\domO)$, then 
 \begin{equation}\label{eqn-b_1-trilinear-form-2}
b_1(\bu,\phi,\psi)= \sum_{i= 1}^{d} \int_{\domO} \bu^i(x) \, \partial_i \phi(x) \, (-\Delta \psi)(x) \,\d x,
\end{equation}
and the integral on the RHS of \eqref{eqn-b_1-trilinear-form-2} is still well defined, since $\zero{H}{2}(\domO) \subset W^{1,4}(\domO)$.\\
Moreover, there exists a constant $C>0$ such that for all $\bu \in \StokesV,\; \phi \in \zero{H}{2}(\domO),\,\psi \in \zero{H}{3}(\domO)$, 
\begin{equation}\label{eqn-b_1-trilinear-form-3}
\lvert b_1(\bu,\phi,\psi) \rvert
\leq \Vert \bu \Vert_{\mathbb{L}^4} \lvert \phi \rvert_{\newone{H}} \Vert -\Delta \psi \Vert_{L^4}
\leq C \Vert \bu \Vert_{\mathbb{L}^4} \lvert \phi \rvert_{\newone{H}} \Vert \psi \Vert_{\zero{H}{3}}.
\end{equation}
Thus, the form $b_1$ can be extended in a unique way to a bounded trilinear form (denoted by the same letter)
\[ 
b_1: \StokesV \times \newone{H} \times \newone{V} \to \mathbb{R}, 
\] 
satisfying \eqref{eqn-b_1-trilinear-form-3}.
In particular, with the same constant the new form $b_1$ satisfies 
\begin{align}\label{eqn-b_1-trilinear-form-ineq-full}
\vert b_{1}(\bu,\phi, \psi) \vert & \leq C  \Vert \bu \Vert_{\StokesV}  \Vert  \phi \Vert_{\newone{H}} \Vert  \psi \Vert_{\newone{V}},\;\; \forall \bu \in \StokesV,\,\phi \in \newone{H},\mbox{ and } \psi \in \newone{V}. 
\end{align}
Hence, we infer that there exists a unique bilinear map
   \begin{equation}\label{eqn-B_1}
     \tilde{B}_1: \StokesV \times \newone{H} \to \newonep{V},
    \end{equation}
such that 
\begin{align*}
      \duality{\tilde{B}_1(u,\phi)}{\psi}{\newone{V}}{\newonep{V}}&= b_1(\bu, \phi, \psi),\;\; \forall \bu \in \StokesV,\, \phi \in \newone{H},\mbox{ and } \psi \in \newone{V}. 
\end{align*}
We can also prove that the maps $b_1$ and $\tilde{B}_1$ satisfy the following stronger inequalities:
\begin{align}\label{eq-B1}
\begin{aligned}
\vert b_1(\bu,\phi,\psi)\vert &\leq C \Vert \bu \Vert_{\mathbb{L}^4} \Vert \phi \Vert_{L^4}  \Vert \psi \Vert_{\newone{V}},\;\; \forall \bu \in \StokesV,\,\phi \in \newone{H}, \mbox{ and } \psi \in  \newone{V}, 
\\
\Vert \tilde{B}_1(\bu, \phi) \Vert_{\newonep{V}} &\leq C \Vert \bu \Vert_{\mathbb{L}^4} \Vert \phi \Vert_{L^4},\;\;\;\;\;\;\;\;\;\; \forall \bu \in \StokesV,\,\phi \in \newone{H}, \mbox{ and } \psi \in  \newone{V}. 
\end{aligned}
\end{align}
Moreover, by the definition \eqref{eqn-b_1-trilinear-form-1} of $b_1$ and the Sobolev embedding $H^2(\domO) \embed L^\infty(\domO)$, we infer that there exists a constant $C>0$ such that for all $\bu \in \StokesV,\, \phi \in \zero{H}{3}(\domO)$, and $\psi \in W^{1,4}(\domO)$,
\begin{equation}\label{eqn-b_1-trilinear-form-4}
\begin{aligned}
&\lvert b_1(\bu,\phi,\psi) \rvert
\leq \lvert \nabla \bu \rvert_{\mathbb{L}^2} \Vert \nabla \phi \Vert_{\mathbb{L}^\infty} \lvert \psi \rvert_{\newone{H}} + \sum_{i,j=1}^d \lvert \bu_i \rvert_{L^4} \lvert \partial_i(\partial_j \phi) \rvert_{L^4} \lvert D_j \psi \rvert_{L^2}
 \\
&\leq \lvert \nabla \bu \rvert_{\mathbb{L}^2} \Vert \nabla \phi \Vert_{\mathbb{L}^\infty} \lvert \psi \rvert_{\newone{H}} +  \Vert \bu \Vert_{\mathbb{L}^4} \left(\sum_{i,j=1}^d \lvert \partial_i(\partial_j \phi) \rvert_{L^4}^2 \right)^{\frac12} \lvert \psi \rvert_{\newone{H}}
   \\
&\leq \lvert \nabla \bu \rvert_{\mathbb{L}^2} \Vert \nabla \phi \Vert_{\mathbb{L}^\infty} \lvert \psi \rvert_{\newone{H}} +  C \Vert \bu \Vert_{\mathbb{L}^4} \Vert \phi \Vert_{H^3} \lvert \psi \rvert_{\newone{H}}
\leq C \Vert \bu \Vert_{\StokesV} \Vert \phi \Vert_{H^3} \lvert \psi \rvert_{\newone{H}}.
\end{aligned}
\end{equation}
 Hence, the trilinear form $b_1$ in \eqref{eqn-b_1-trilinear-form-1} can also be uniquely extended to a trilinear form, denoted by the same letter,
    \begin{equation}
      b_1: \StokesV \times \newone{V} \times \newone{H} \to \mathbb{R}, 
   \end{equation}
satisfying \eqref{eqn-b_1-trilinear-form-4}; and  there exists a unique bounded bilinear operator
   \begin{equation}\label{eqn-titlde-B_1}
     B_1: \StokesV \times \newone{V} \to \newone{H},
    \end{equation}
i.e., satisfying the following estimate;
      \begin{equation}
         \lvert B_1(\bu,\phi) \rvert_{\newone{H}} \leq C \Vert \bu \Vert_{\StokesV} \Vert \phi \Vert_{ \newone{V}},\;\;\forall \bu \in \StokesV,\, \phi \in  \newone{V}.  
       \end{equation}
The following result is fundamental, as it links the two trilinear forms $r_0$ and $\bb_0$.
\begin{proposition}\label{prop-b_1-r_0}
Assume that $\bu \in \StokesV$ and $\phi \in \newone{V} $. Then 
        \begin{equation}\label{eqn-b_1-r_0}
           b_1(\bu,\phi,\phi)=r_0(\phi,\phi,\bu).
         \end{equation}
\end{proposition}
\begin{corollary}\label{cor-b_1-r_0}
Assume that $\bu \in \StokesV$ and $\phi \in \newone{V} $. Then  
        \begin{equation}\label{eq2.7}
          (B_1(\bu,\phi),\phi)_{\newone{H}}= \duality{\bR_0(\phi,\phi)}{\bu}{\StokesV}{\StokesVp}.
        \end{equation}
\end{corollary}
The following fundamental result connects the forms \eqref{eqn-b_1-trilinear-form-1}, \eqref{eqn-B_1}, and \eqref{eqn-titlde-B_1}.
\begin{proposition}\label{prop-B_1-b_1}
Assume that $\bu \in \StokesV$ and $\phi,\,\psi \in \newone{V}$. Then, the following equalities hold:  
   \begin{equation}
      (B_1(\bu,\phi),\psi)_{\newone{H}}= \duality{\tilde{B}_1(u,\phi)}{\psi}{\newone{V}}{\newonep{V}}.
    \end{equation}
Moreover, if $\bu \in \StokesV$, $\phi \in \newone{V}$, and $\psi \in \zero{H}{2}(\domO)$, then
      \begin{equation}
         (B_1(\bu,\phi),\psi)_{\newone{H}}= b_1(\bu,\phi,\psi).
      \end{equation}
\end{proposition}
\begin{corollary}
Assume that $\phi \in \zero{H}{2}(\domO)$ and $\bu \in \StokesV$. Then
\begin{equation}
 \left(B_1(\bu, \phi), \Aone^{-1} f(\phi) \right)_{\newone{H}}
 =0,
\end{equation}
where $f(\phi)\coloneqq \psi^\prime(\phi) - \avg{\psi^\prime(\phi)}\coloneqq \psi^\prime(\phi) - \fint_{\domO} \psi^\prime(\phi(x))\,\d x$.
\end{corollary}
\begin{proof}
Assume that $\phi \in \zero{H}{2}(\domO)$ and set $\zeta(\phi)= \Aone^{-1} f(\phi)$.
In other words, $\zeta(\phi)$ solves the following Neumann boundary value problem:
\begin{equation*}
\begin{cases}
-\Delta \zeta(\phi)= f(\phi),
\\
 \frac{\partial \zeta(\phi)}{\partial \bn}=0,
 \\
 \fint_{\domO} \zeta(\phi(x))\,\d x=0.
\end{cases}
\end{equation*}
Since $\phi \in \zero{H}{2}(\domO)$, we infer that
$ -\Delta \zeta(\phi)= f(\phi) \in H^2(\domO)$.
By the smoothness of $\domO$ the standard elliptic theory for second order operators, see \cite[Theorem 8.8]{Gilbarg+Trudinger_1983}, there exists  a constant $C_1>0$ such that
\begin{align*}
&\Vert \zeta(\phi) \Vert_{H^2}
\leq C_1 (\lvert -\Delta \zeta(\phi) \rvert_{L^2} + \lvert \zeta(\phi) \rvert_{L^2})
= C_1 (\lvert f(\phi) \rvert_{L^2} + \lvert \zeta(\phi) \rvert_{L^2}). 
\end{align*}
Moreover, using the Poincar\'e Wirtinger and Cauchy Schwarz inequalities, we infer that there exists a generic constant $C_2=C_2(\domO)>0$ such that
\begin{align*}
&\lvert \zeta(\phi) \rvert_{L^2}
\leq C_2 \lvert \zeta(\phi) \rvert_{\newone{H}}
= C_2 [(\nabla \zeta(\phi), \nabla \zeta(\phi))_{L^2}]^{1/2}
\\
&= C_2 [(\zeta(\phi), -\Delta \zeta(\phi))_{L^2}]^{1/2}
= C_2 [(\zeta(\phi), f(\phi))_{L^2}]^{1/2}
\leq C_2 \lvert \zeta(\phi) \rvert_{L^2}^{1/2}  \lvert f(\phi) \rvert_{L^2}^{1/2}.
\end{align*}
This implies $\lvert \zeta(\phi) \rvert_{L^2}^{1/2}\leq C_2 \lvert f(\phi) \rvert_{L^2}^{1/2}$. Consequently,
\[
\Vert \zeta(\phi) \Vert_{H^2}
\leq C_1(1 + C_2^2) \lvert f(\phi) \rvert_{L^2}.
\]
Now, since $\zeta(\phi) \in \zero{H}{2}(\domO)$ and $\diver\bu=0$, then by integration by parts, we deduce that  
\begin{align*}
&\left(B_1(\bu, \phi), \zeta(\phi)\right)_{\newone{H}}
=b_1(\bu,\phi,\zeta(\phi))
=\sum_{i= 1}^{d} \int_{\domO} \bu^i \, \partial_i \phi \, (-\Delta \zeta(\phi)) \,\d x
\\
&
= \sum_{i= 1}^{d} \int_{\domO} \bu^i \, \partial_i \phi \, f(\phi) \,\d x
= \sum_{i= 1}^{d} \int_{\domO} \bu^i \, \partial_i \phi \, \psi^\prime(\phi) \,\d x
- \avg{\psi^\prime(\phi)} \sum_{i= 1}^{d} \int_{\domO} \bu^i \, \partial_i \phi\,\d x
\\
&= \sum_{i= 1}^{d} \int_{\domO} \bu^i\, \partial_i\psi(\phi) \,\d x
- \avg{\psi^\prime(\phi)} \sum_{i= 1}^{d} \int_{\domO} \bu^i \, \partial_i \phi\,\d x
=0.
\end{align*}
This completes the proof of the corollary. 
\end{proof}

\subsection{Regular potential}\label{subsec-regularpotential}

In our paper we concentrate on the Landau potential, i.e.
              \begin{equation}\label{eqn-regular-potential}
                  \psi(s)= \frac{1}{4}(1-s^2)^2, \;  s \in \mathbb{R}.
              \end{equation}
However, see Remark \ref{rem-potential general}, our results are valid for a much larger class of potentials. 
\begin{notation}\label{not-Psi}
The Nemytskii map associated with $\psi$ will be denoted by $\Psi$. In other words, if $X$ is a functional space of real valued functions, and for every 
$u \in X$, the composition map $\psi \circ u$ belongs to $X$ as well, then the map $X \ni u \mapsto \psi \circ u \in X$ is denoted by $\Psi$.
\end{notation}
In the whole paper, we will use the above form of the potential. We establish its properties in the following basic but important lemma. 
\begin{lemma}\label{eqn-Lemma-Psi'} The function $\psi$ defined in \eqref{eqn-regular-potential}  is of $C^2$ class and 
\begin{equation}\label{eqn-Psi'}
\psi^\prime(s)= s(s^2-1) \mbox{ and } \psi^{\prime\prime}(s)= 3s^2-1, \;  s \in \mathbb{R}.
\end{equation}
Moreover, $\Psi^\prime$ maps $L^6(\domO)$ into $L^2(\domO)$, and 
     \begin{equation}\label{eqn-Psi'-L^6}
       \vert \psi^\prime(\phi)\vert_{L^2} \leq \Vert  \phi \Vert_{L^6}^3 + \vert \phi \vert_{L^2}, \;\; \phi \in L^6(\domO),
     \end{equation}
and,
    \begin{equation}\label{eqn-Psi''}
      -\int_{\domO} \psi^{\prime \prime}(\phi(x))\lvert \nabla \phi(x) \rvert^2\,\d x 
       \leq  \lvert \nabla \phi \rvert_{\mathbb{L}^2}^2, \;\; \phi \in H^1(\domO) \cap L^\infty(\domO).
    \end{equation}
\end{lemma}
\begin{remark}\label{rem-eqn-Psi''}
In our case, inequality can be given a more precise form as follows 
    \begin{equation}\label{eqn-Psi''-precise}
      -\int_{\domO} \psi^{\prime \prime}(\phi(x))\lvert \nabla \phi(x) \rvert^2\,\d x 
      = -\frac34 \lvert \nabla (\phi^2) \rvert_{\mathbb{L}^2}^2+ \lvert \nabla \phi \rvert_{\mathbb{L}^2}^2, \;\; \phi \in H^1(\domO) \cap L^\infty(\domO).
    \end{equation}
\end{remark}
\begin{proof}[Proof of equality \eqref{eqn-Psi''-precise}]
Let us choose and fix $\phi \in H^1(\mathcal{O}) \cap L^\infty(\domO)$.\\
Then $\phi^2 \in H^1(\mathcal{O}) \cap L^\infty(\domO)$ and, by the  weak product rule, see e.g. \cite[Proposition 9.4]{Brezis_2011}, we have
\[
D_i(\phi^2)= 2 \phi D_i \phi , \;\; i=1,\ldots, d.  \]
This, together with \eqref{eqn-Psi'}, yields that
\begin{align*}
&\int_{\domO} \frac34 \sum_{i=1}^d \left \lvert D_i(\phi^2(x))\right \rvert^2\d x - \int_{\domO} \lvert \nabla \phi(x) \rvert^2\,\d x
= \int_{\domO} 3 \phi^2(x) \sum_{i=1}^d \left \lvert D_i\phi(x) \right \rvert^2\d x - \int_{\domO} \lvert \nabla \phi(x) \rvert^2\,\d x \\
&= \int_{\domO} (3 \phi^2(x) - 1) \lvert \nabla \phi(x) \rvert^2\,\d x
=\int_{\domO} \psi^{\prime \prime}(\phi(x)) \lvert \nabla \phi(x) \rvert^2\,\d x.
\end{align*}
\end{proof}
\begin{remark}\label{rem-potential general} Let us notice that Lemma \ref{eqn-Lemma-Psi'} holds when 
a  $C^2$-class potential $\psi:\mathbb{R} \to \mathbb{R}$ satisfies the following growth conditions.  There exist constants $\kappa_1>0,\,\kappa_2>0$, and $\kappa_3>0$ such that for every $s \in \mathbb{R}$,
\begin{equation}\label{General-smooth-potential}
\begin{aligned}
 \psi(s) &\geq  - \kappa_1,  & \\
- \psi^{\prime\prime}(s)  &\leq  \kappa_2,  \\
\vert \psi^{\prime}(s) \vert &\leq  \kappa_3(  \vert s \vert  + \vert s\vert ^3).
\end{aligned}
\end{equation}
\end{remark}

\section{Stochastic preliminaries}\label{sec-stochastic preliminaries}
\subsection{Cylindrical Wiener process}\label{subsec-cylindrical Wiener process}
Due to the presence of the stochastic integral, all quantities in \eqref{eqn-stochastic-CHNSEs} must be interpreted as random variables defined on a probability space satisfying the assumption below. 
\begin{assumption}\label{ass-usual}
We assume that     
\begin{equation}\label{eqn-probability space}
(\Omega, \mathscr{F}, \mathbb{F},\mathbb{P})
\end{equation}
is a filtered complete probability space 
 with filtration $\mathbb{F}= \{\mathscr{F}_t\}_{t\in[0,T]}$,  satisfying the so called usual assumptions, i.e.     
\begin{trivlist}
\item[(i)] $\mathbb{F}$ is right-continuous, that is, $\mathscr{F}_t= \bigcap_{s>t} \mathscr{F}_s$, for all $t \in [0,T]$; 
\item[(ii)] $\mathscr{F}_0$ contains all null sets, that is, $\mathbb{P}$-negligible subset of $\mathscr{F}$ belongs to $\mathscr{F}_0$.
\end{trivlist}
\end{assumption}
\begin{definition}\label{def-cylindrical Wiener process}
Assume that $\rK$ is a separable infinite dimensional Hilbert space. A $\rK$-cylindrical Wiener process $W=(W(t): t \in [0,T])$ is a family of formal series
   \begin{equation}\label{eqn-cylindrical Wiener process-abstract}
      W(t)\coloneqq  \sum_{k=1}^\infty w_k(t) e_k, \;\; t \in [0,T], 
    \end{equation} 
    where  $\left(e_k: \; k \in \mathbb{N} \right)$ is an ONB of $\rK$ and $\left(w_k: \; k \in \mathbb{N} \right)$ is a sequence of iid standard Brownian motions on some fixed probability space. \\
    The space $\rK$ is often called the Reproducing kernel Hilbert space, or Cameron-Martin space, of the 
    cylindrical Wiener process $W$.
\end{definition}
It is well known that the above definition is equivalent to the following one, see Definition 4.1 in \cite{Brz+Peszat_2001}.
\begin{definition}\label{rem-def-cylindrical Wiener process}
Let $(\Omega ,\mathscr{F}, \mathbb{F}, \mathbb{P})$ be a filtered probability space with the filtration 
$\mathbb{F}=(\mathscr{F}_t)_{t\in [0,T]}$
and $\rK$ be a real separable Hilbert space. By an  $\mathbb{F}$-adapted cylindrical Wiener process on $\rK$, we understand a family $W(t)$, $t\ge 0$, of bounded linear
operators from ${\rK}$ into  $L^2(\Omega,{\mathcal F},\mathbb{P})$ such that:
\begin{trivlist}
\item[(i)] for all $t \in [0,T]$, and $\psi,\varphi \in {\rK}$,
\[
\mathbb{E} \left[W(t)\psi W(t)\varphi \right]=t \langle\psi,\varphi \rangle _{\rK}\;;
\]
\item[(ii)] 
for each $\psi\in \rK$, $W(t)\psi$, $t \in [0,T]$ is a real valued
$\mathbb{F}$-adapted Wiener process.
\end{trivlist}
\end{definition}
\begin{remark}\label{rem-cylindrical Wiener process}
It follows from this definition that if the family of formal series' \eqref{eqn-cylindrical Wiener process-abstract} is a $\rK$-cylindrical Wiener process and $(\tilde{e}_k: \; k \in \mathbb{N})$ is an ONB  of $\ell^2$, then 
the following family of formal series'
   \begin{equation}\label{eqn-cylindrical Wiener process-l^2}
      \tildeW(t)\coloneqq  \sum_{k=1}^\infty \tilde{w}_k(t) \tilde{e}_k, \;\; t \in [0,T], 
   \end{equation}
is an $\ell^2$-cylindrical Wiener process.
\end{remark}
As in \cite{Brz+Peszat_2001}, one can define the stochastic It\^o  integral $\int_0^t \upxi(s) \,\d W(s)$, as an $X$-valued continuous martingale,  provided 
$\upxi$ is a progressively measurable process taking values in the space $\gamma(\rK,X)$ of $\gamma$-radonifying linear operators from $\rK$ to a martingale type $2$ Banach space $X$
and satisfying the following natural condition:
     \begin{equation}\label{eqn-xi-assumption}
          \mathbb{E} \int_0^t \Vert \upxi (s)\Vert^2_{\gamma(\rK,X)}\, ds <\infty, \;\; \mbox{ for every } t \in [0,T].     
    \end{equation}
Let us point out that if  two processes $\upxi_1$ and $\upxi_2$ satisfying the conditions listed earlier, are modifications of each other, 
then the $X$-valued  continuous martingales $\int_0^t \upxi_1(s) \,\d W(s)$  and $\int_0^t \upxi_2(s) \,\d W(s)$ are indistinguishable. 
If $\phi \in X^\prime$  and $\upxi$ is a  process satisfying the conditions listed earlier, 
then  $\mathbb{R}$-valued process $\phi \circ \upxi$ defined by 
\[
\phi \circ \upxi: [0,\infty) \times \Omega \ni (t,\omega) \mapsto \phi(\upxi(t,\omega)) \in \gamma(\rK,\mathbb{R})
\]
also satisfies the conditions listed earlier and the two  $\mathbb{R}$-valued 
continuous martingales $\int_0^t \phi \circ\upxi(s) \,\d W(s)$ and  $ \phi\left(\int_0^t \upxi(s) \,\d W(s)\right)$ are indistinguishable. \\
If $\{\phi_j\}$ is a countable linearly dense sequence of elements of  $X^\prime$ and $\upxi_1$, $\upxi_2$ are  two $\gamma(\rK,X)$-valued processes satisfying the conditions listed earlier such that
for every $j$, the two  $\mathbb{R}$-valued 
continuous martingales $\int_0^t \phi_j \circ\upxi_1(s) \,\d W(s)$ and $\int_0^t \phi_j \circ\upxi_2(s) \,\d W(s)$, are indistinguishable, then 
the two  $X$-valued 
continuous martingales $\int_0^t  \upxi_1(s) \,\d W(s)$ and $\int_0^t \upxi_2(s) \,\d W(s)$, are indistinguishable as well. \\
Notice that if $X$ is a Hilbert space, then 
$\gamma(\rK,X)$ is equal to the space $\mathscr{T}_2(\rK,X)$ of Hilbert Schmidt operators 
from $\rK$ to  $X$, and if $X=\mathbb{R}$, we infer that  the space $\mathscr{T}_2(\rK,\mathbb{R})$
is isometrically isomorphic, via the Riesz Lemma,  to $\rK$. 

\noindent
Therefore, if  $X=\mathbb{R}$ and  $\upxi$ is  a $\gamma(\rK,\mathbb{R})$-valued process satisfying the conditions listed earlier, 
then the   It\^o integral $\int_0^t \upxi(s)\,\d W(s)$, which is an $\mathbb{R}$-valued continuous martingale,  can be, and will be, written in the following form: 
\begin{equation}\label{eqn-Ito integral-inner product}
\int_0^t \upxi(s)\,\d W(s)=\int_0^t \ipsmall{\upxi(s)}{\d W(s)}{\rK}, 
\end{equation}
where in the 2nd and 3rd integrals, by $\upxi$ we understand the corresponding $\rK$-valued process.

\noindent
In particular, if $\phi \in X^\prime$ and $\upxi$ is a $\gamma(\rK,X)$-valued   process satisfying the conditions listed earlier, then 
\[
\int_0^t \phi \circ\upxi(s) \,\d W(s)=\int_0^t \langle \phi \circ \upxi(s),\d W(s) \rangle_{\rK}, 
\]
where again in the integral on the RHS, by $\phi \circ \upxi$ we understand the corresponding $\rK$-valued process.

\noindent
For convenience, we fix the standard ONB $(\tilde{e}_k)_{k \in \mathbb{N}}$ of $\ell^2$. Let us state the following assumption on the process $\tildeW= (\tildeW(t): \; t \in [0,T])$.
\begin{assumption}\label{ass-cylindrical Wiener process}
We assume that  
           \begin{equation}\label{wiener-process-representation}
             \tildeW(t)\coloneqq  \sum_{k=1}^\infty \tilde{w}_k(t) \tilde{e}_k, \;\; t \in [0,T], 
           \end{equation}
is a $\ell^2$-cylindrical Wiener process on the filtered complete probability space $(\Omega, \mathscr{F}, \mathbb{F},\mathbb{P})$. 
\end{assumption}

\subsection{A special case}\label{subsec-Ito integral}

In this subsection, we consider a special case. It is important from the point of view of understanding Lemma 3.2 from \cite{Mik+Roz_1998}.\\
Assume that $(\mO, \mathcal{E}, \tilde{\kappa})$ is  a $\sigma$-finite measure space 
such that the Hilbert  space 
\begin{equation}\label{eqn-rK}
\rK\coloneqq L^2(\mO, \tilde{\kappa})=L^2(\mO, \tilde{\kappa};\mathbb{R})
\end{equation}
is separable. Assume also that 
$H$ is a separable real Hilbert space with inner product
$\ip{\cdot}{\cdot}_H$, and let $A : \rK \to H$ be a bounded
linear operator. Let us recall the following folk theorem. 

\begin{theorem}\label{thm:HS}
Suppose that $A :\rK \to H$ is a bounded linear operator. Then, $A$ is Hilbert--Schmidt if and only if there exists a (strongly) $\tilde{\kappa}$-measurable function
  $\Phi : \mO \to H$ with
\begin{equation}\label{eq:L2H}
\int_{\mO} \normm{\Phi(x)}_H^2\,\tilde{\kappa}(\d x)< \infty,
\end{equation}
such that
\begin{equation}\label{eq:Af}
Af = \int_{\mO} f(x)\, \Phi(x)\, \tilde{\kappa}(\d x),
\;\; \forall\, f \in \rK,
\end{equation}
where the integral is a Bochner integral in $H$. In this case,
\begin{equation}\label{eqn-HS-norm of A}
\normm{A}_{\HS(\rK, H)}^2
= \int_{\mO} \normm{\Phi(x)}_H^2\, \tilde{\kappa}(\d x),
\end{equation}
where $\normm{A}_{\HS(\rK,H)}^2$ is the Hilbert-Schmidt norm defined by  
  \begin{equation}\label{eqn-def-HS-norm}
    \normm{A}_{\HS(\rK,H)}^2
    = \sum_{n=1}^\infty \normm{A \tilde{e}_n}_H^2
  \end{equation}
  for any complete orthonormal system $(\tilde{e}_n)_{n=1}^\infty$ in $\rK=L^2(\mO, \tilde{\kappa})$.
\end{theorem}
We have the following immediate corollary. 
\begin{corollary}
\label{cor:HS}    
The space $ \HS(L^2(\mO, \tilde{\kappa});H)$ of all Hilbert-Schmidt operators from $L^2(\mO, \tilde{\kappa})$ to $ H$, endowed with the standard norm \eqref{eqn-def-HS-norm},  is isometrically isomorphic to the Hilbert space $ L^2(\mO,\tilde{\kappa}; H)$ of all (equivalence classes) of all 
(strongly) $\tilde{\kappa}$-measurable functions  $\Phi : \mO \to H$ such that \eqref{eq:L2H} equipped with the classical norm 
  \begin{equation}\label{eqn-norm-L2H}
    \normm{\Phi}_{L^2(\mO,\tilde{\kappa};H)}^2 \coloneqq \int_{\mO} \normm{\Phi(x)}_H^2\, \tilde{\kappa}(\d x). 
  \end{equation}
\end{corollary}
\begin{example}\label{example:HS}    
Assume that $\mO=\mathbb{N}$, $\mathscr{E}=2^\mO$, and $\tilde{\kappa}$ is the counting measure. Then, obviously the space  $L^2(\mO, \tilde{\kappa})$ is equal to the space 
$\ell^2=\ell^2(\mathbb{N},\mathbb{R})$; and $L^2(\mO,\tilde{\kappa};H)$ is equal to the space $\ell^2(H)$. Thus, according to the Corollary \ref{cor:HS}, the spaces 
\[
\HS(\ell^2;H)\mbox{ and } \ell^2(H) \] are isometrically isomorphic.
\end{example}

\begin{proof}[Sketch of the proof of Corollary \ref{cor:HS}]
Define a map 
\begin{equation}\label{eqn-Lambda}
\Lambda: L^2(\mO,\tilde{\kappa}; H) \ni \Phi \mapsto A \in  \HS(\rK, H), 
\end{equation}
where $A$ is defined by \eqref{eq:Af}. By Theorem \ref{thm:HS}, the map $\Lambda$ is a well defined isometric map. Obviously, it is linear; hence, it suffices to prove that it is surjective. To show this, we choose and fix 
$ A \in  \HS(\rK,H)$. Then the adjoint $A^\ast : H \to \rK$ is also a Hilbert--Schmidt operator, see \cite{Kato_1995}. 
Let $(h_k)_{k \geq 1}$ be an ONB 
  in $H$.  Define
  \begin{equation}\label{eq:Phi-series}
    \Phi: \mO \ni \rx \mapsto \sum_{k=1}^\infty A^\ast h_k(\rx)\, h_k \in H,  
  \end{equation}
where the sum is taken in $H$. This series converges in $L^2(\mO,\tilde{\kappa};H)$ since
\begin{equation}\label{eq:Af-3}
\int_\mO \normm{\Phi(x)}_H^2\,\tilde{\kappa}(\d x)
= \int_\mO \sum_{k=1}^\infty \lvert A^\ast h_k(x) \rvert^2\, \tilde{\kappa}(\d x)
= \sum_{k=1}^\infty \normm{A^\ast h_k}_{\rK}^2 
= \normm{A^\ast}_{\HS(H,\rK)}^2 < \infty,
\end{equation}
where we used the orthonormality of $(h_k)$ in $H$ and the Parseval identity. \\
The function $\Phi : \mO \to H$ defined by \eqref{eq:Phi-series} is strongly $\tilde{\kappa}$-measurable as an $L^2(\mO,\tilde{\kappa};H)$-limit of strongly measurable functions.
Moreover, if $f \in \rK$ and $h \in H$, then  
\begin{equation}\label{eq:Af-2}
\begin{aligned}
&\langle  \int_{\mO} f(x)\, \Phi(x)\, \tilde{\kappa}(\d x), h \rangle_H 
= \int_{\mO} f(x)\,  \left(  \sum_{k=1}^\infty A^\ast h_k(x)\, \langle h_k,h \rangle_H \right)\, \tilde{\kappa}(\d x)
\\
& = \sum_{k=1}^\infty \langle h_k,h \rangle_H  \int_{\mO} f(x)\,   A^\ast h_k(x)\, \, \tilde{\kappa}(\d x) 
  = \sum_{k=1}^\infty \langle h_k,h \rangle_H\, \langle f, A^\ast h_k \rangle_{\rK}  
\\
&= \sum_{k=1}^\infty \langle h_k,h \rangle_H \langle A f, h_k \rangle_{H}= \langle A f, h \rangle_{H},
\end{aligned}
\end{equation}
where the last equality is a consequence of the Parseval identity. \\
Thus, \eqref{eq:Af} follows, completing the proof.
\end{proof}
\begin{assumption}\label{ass-ONB of rK}
Now, we assume that  $\left( e_k\right)_{k=1}^\infty$ is an ONB of $\rK=L^2(\mO, \tilde{\kappa})$ and 
           \begin{equation}\label{rk-wiener-process-representation}
             W(t)\coloneqq  \sum_{k=1}^\infty w_k(t) e_k, \;\; t \in [0,T], 
           \end{equation}
is a $\rK$-cylindrical Wiener process on the filtered complete probability space $(\Omega, \mathscr{F}, \mathbb{F},\mathbb{P})$. 
\end{assumption}
We also have 
\begin{assumption}\label{ass-function sigma}
Suppose that the function 
  \begin{equation}\label{eqn-first-def-sigma}
  \sigma : [0,\infty) \times \Omega \times \mO \to U^\prime
  \end{equation}
is a $\mathcal{P}(\mathbb{F}) \otimes \mathcal{E}$-measurable function, so that for each $(s, \omega) \in [0,\infty) \times \Omega $, the map $\rx \mapsto \sigma_s(\rx)$ is an $\mathcal{E}$-measurable $U^\prime$-valued function on $\mO$. We assume that   
    \begin{equation}\label{eq:integ-cond-sigma}
    \mathbb{P}\!\left(\int_0^t \int_{\mO} \normm{\sigma_s(\rx)}_{U^\prime}^2\,
    \tilde{\kappa}(\d\rx)\, \d s < \infty\right) = 1 \quad \forall\, t > 0,
  \end{equation}
or, respectively,
\begin{equation}\label{eq:integ-cond-ito-sigma}
  \mathbb{E} \int_0^t \!\left[\int_{\mO} \normm{\sigma_s(\rx)}_{U^\prime}^2\,
    \tilde{\kappa}(\d\rx)\right]\d s< \infty \quad \forall\, t > 0.
\end{equation}
\end{assumption}
Note that the above implies that  
for almost every  $(s, \omega) \in [0,\infty) \times \Omega$, the map 
\[
\sigma_s(\rx) \coloneqq  \left\{ \mO \ni \rx \mapsto \sigma_s(\rx) \in U^\prime\right\}
\]
belongs to  $L^2(\mO,\tilde{\kappa};U^\prime)$. Hence, in view of Corollary \ref{cor:HS}, we can define the following  process 
\begin{equation}\label{eqn-xi}
  \xi : [0,\infty) \times \Omega \ni (s,\omega)   \mapsto \xi_s(\omega) \coloneqq  \Lambda \left( \sigma(s,\omega, \cdot)\right) \in  \HS(\rK,U^\prime).
  \end{equation}
This new process satisfies one of the following two conditions:   
\begin{equation}\label{eq:integ-cond}
\mathbb{P}\!\left(\int_0^t  \normm{\xi_s}^2_{\HS(\rK,U^\prime)}\,
     \d s < \infty\right) = 1 \qquad \forall\, t > 0,
\end{equation}
or, respectively,
\begin{equation}\label{eq:integ-cond-ito}
  \mathbb{E} \left[ \int_0^t \, \normm{\xi_s}^2_{\HS(\rK,U^\prime)}
    \, \d s \right] < \infty \qquad \forall\, t > 0.
\end{equation}
Hence, the following $U^\prime$-valued continuous local martingale, respectively, martingale, is defined 
\begin{equation}\label{Eqn-M_t=int_0^t xi_s  d W_s}
    M_t \coloneqq \int_0^t \xi_s \, \d W_s, \;\; t \in [0,T].
\end{equation}
The stochastic integral on the RHS above can  be, maybe not  naturally,   denoted by 
\begin{equation}\label{eqn-M_t=int_0^t xi_s  d W_s}
    \int_0^t  \sigma_s \, \d W_s \coloneqq \int_0^t \xi_s \, \d W_s, \;\; t \in [0,T].
\end{equation}
A more natural notation, in view of definitions \eqref{eqn-xi} and \eqref{eqn-Lambda} could be  the following one
\[
\int_0^t \int_{\mO} \sigma_s(\rx) \, \d W_s(\rx) \; \tilde{\kappa}(\d\rx), \;\; t \in [0,T].
\]
The notation introduced above in \eqref{eqn-M_t=int_0^t xi_s  d W_s} was used in \cite{Mik+Roz_1998} and it will thus be used by us. 

\subsection{Assumptions on the coefficients}
\label{subsec-HS}
Given the functions $\sigma_k$ and $\bg_k$ as defined in \eqref{eqn-sigma_k} and \eqref{eqn-g_k}, we propose the following definition. 
\begin{definition}\label{def-HS} Assume that  $k \in \mathbb{N}$. 
We define  maps $\Sigma_k$ and $\bG_k$  by 
\begin{equation}\label{eqn-Sigma_k}
\bSi_k: [0,T] \times \StokesV \ni (t,\bu) \mapsto \bSi_k(t)\bu \coloneqq   \P \left[(\sigma_k(t,\cdot )  \cdot\nabla) \bu(\cdot)\right] \in \StokesH,
\end{equation}
and 
\begin{equation}\label{eqn-G_k}
\bG_k: [0,T] \times \StokesH \ni (t,\bu) \mapsto   \P \left[
\bg_k(t,\cdot,\bu(\cdot)) \right] \in \StokesH.
\end{equation}
Next, we define  maps $\bG$ and $\bSi$ by
\begin{equation}\label{eqn-G}
\bG: [0,T] \times \StokesH \ni (t,\bu) \mapsto \left(  \bG_k(t,\bu)\right)_{k=1}^\infty
 \in \ell^2(\StokesH)
\end{equation}
and 
   \begin{equation}\label{Eqn-def-Sigma}
      \bSi: [0,T] \times \StokesV \ni (t,\bu) \mapsto  \bSi(t)\bu \coloneqq \left( \bSi_k(t)\bu \right)_{k=1}^\infty
     \in  \ell^2(\StokesH),
   \end{equation}
where the Hilbert space  $ \ell^2(\StokesH)$ has been defined in \eqref{eqn-l^2Y}.
\end{definition}

\begin{remark}
\begin{trivlist}
\item[(i)] 
Notice that  for all $(t,\bu)$ and $k$,
\[
[\bG(t,\bu)](\tilde{e}_k)=\bG_k(t,\bu) \mbox{ and } [\bSi(t)\bu](\tilde{e}_k)=\bSi_k(t)\bu.
\]
\item[(ii)] By the above definitions, the It\^o terms in  \eqref{eqn-stochastic-CHNSEs} can be written in the following forms:
\begin{align*}
\bG(t,\bu)\, \d \tilde{W}(t)= \sum_{k=1}^\infty \bG_k(t,\bu)\, \d \tilde{w}_k(t) \mbox{ and }
[\bSi(t)\bu]\, \d \tilde{W}(t)= \sum_{k=1}^\infty [\bSi_k(t)\bu]\, \d \tilde{w}_k(t).
\end{align*}
\item[(iii)] 
By definition, the  norm in the space $\ell^2(\StokesH)$ is the following 
    \begin{align*}
      \Vert \bG \Vert_{\ell^2(\StokesH)}^2 \coloneqq \sum_{k=1}^\infty \lvert \bG_k \rvert_{\StokesH}^2,\;\; \bG=(\bG_k)_{k=1}^\infty.
    \end{align*}
\end{trivlist}
\end{remark}
\begin{assumption}\label{ass-G+Sigma}
We assume that \\
for all $(t,u) \in [0,T] \times \StokesV$,   $\bSi(t)\bu \in \ell^2(\StokesH)$, and \\
for all $(t,u) \in [0,T] \times \StokesH$,  $\bG(t,\bu) \in \ell^2(\StokesH)$.
\end{assumption}
\begin{remark}
By our assumptions, for every $t \in [0,T]$, the map 
\[
\bSi(t) \colon \StokesV \ni \bu \mapsto  \bSi(t)\bu  \in \ell^2(\StokesH) \mbox{ is linear.}
\]
\end{remark}
Let us now list the assumptions about the coefficients. We follow here the paper \cite{Mik+Roz_2005}. We begin with introducing a useful notation:
      \begin{equation*}
            \mathrm{Y}\coloneqq \ell^2(\mathbb{R}^d).    
      \end{equation*}
\begin{assumption}[Boundedness]\label{ass-bg+sigma}
The following functions
\begin{align}
\sigma: [0,T]\times \mathbb{R}^d \ni (t,x) &\mapsto \left(\sigma_k(t,x)\right)_{k=1}^\infty \in \mathrm{Y},
\\
 \diver \sigma: [0,T]\times \mathbb{R}^d \ni (t,x) &\mapsto \left(\diver \sigma_k(t,x)\right)_{k=1}^\infty \in \ell^2(\mathbb{R}),
\end{align}
are measurable and bounded, i.e. there exists $C_1>0$ and $C_2>0$  such that 
\begin{align}
\label{eqn-A10}
\sum_{k=1}^\infty \vert \sigma_k(t,x) \vert^2 \leq C_1 \mbox{ and }  \sum_{k=1}^\infty \vert \diver \sigma_k(t,x) \vert^2 \leq C_2, \;\; (t,x) \in  [0,T]\times \domO.
\end{align}
\end{assumption}

\begin{assumption}[Coercivity]\label{ass-coercivity}
There exists $\delta_0 >0$ such that for all  $(t,x)\in [0,T] \times \domO$ and $\upxi \in \mathbb{R}^d$,
    \begin{equation}\label{eqn-coercivity}
      \sum_{i,j=1}^d \left[\nu \delta_{ij} - \frac 12
      \sum_{k=1}^\infty \sigma_k^i (t, x) \sigma_k^j (t, x)\right] \upxi_i\upxi_j \geq \delta_0 \lvert \upxi \rvert^2.
    \end{equation}
\end{assumption}

\begin{assumption}\label{ass-A3}
The function
         \begin{equation}
            \bg: [0,T] \times \domO \times \mathbb{R}^d \ni (t,x,z) \mapsto (\bg_k(t,x,z))_{k=1}^\infty \in \mathrm{Y} \mbox{ is measurable.}
        \end{equation}
Moreover, for each $(t,x) \in [0,T] \times \domO$, 
the function 
\[ \mathbb{R}^d \ni z \mapsto \bg(t,x,z) \in \mathrm{Y} \] 
is continuous  and, there exists a constant $C_g>0$ and a   measurable function
   \begin{equation}
     \tilde{H}: [0,T] \times \mathcal {O}\to  [0,\infty),
   \end{equation}
such that 
    \begin{equation}\label{eqn-H-2}
      \Vert \tilde{H} \Vert_{L^2([0,T] \times \domO)} <\infty,
    \end{equation}
and 
    \begin{equation}\label{eqn-H-1}
      \Vert \bg(t,x,z) \Vert_{\mathrm{Y}} \leq C_g \lvert z \rvert + \tilde{H}(t,x), \; \; (t,x,z) \in [0,T] \times \domO \times \mathbb{R}^d.
    \end{equation}
\end{assumption}
\begin{remark}
The coercivity assumption \eqref{eqn-coercivity} implies \eqref{eqn-A10}.
This can be checked by taking $\upxi= \upxi_i \bar{e}_i \in \mathbb{R}^d \setminus\{0\}$, where $\bar{e}_i$ denotes the $i$-th canonical basis vector of $\mathbb{R}^d$. In contrast, the condition $\left(\sigma_k(t,x)\right)_{k=1}^\infty \in \mathrm{Y}$ for $(t,x) \in  [0,T]\times \domO$ implies that for each $\upxi \in \mathbb{R}^d$, the series 
$\sum_{k=1}^\infty \left( \sum_{i,j=1}^d \sigma_k^i (t, x) \sigma_k^j (t, x) \upxi_i\upxi_j \right),\;\; (t,x) \in [0,T]\times \domO$,
is convergent, since
\begin{align*}
\sum_{k=1}^\infty \left \lvert \sum_{i,j=1}^d \sigma_k^i (t, x) \sigma_k^j (t, x) \upxi_i\upxi_j \right \rvert
\leq \sum_{k=1}^\infty \vert \sigma_k(t,x) \vert^2 \lvert \upxi \rvert^2<\infty.
\end{align*}
\end{remark}

\section{Reformulation of the problem}\label{sec-Reformulation of the problem}
Since we seek a divergence-free solution $\bu$, to have a divergence free  stochastic part, we require that 
\begin{equation*}
\nabla \bar{p}_k(t)=  (I - \P) \left[(\sigma_k(t) \cdot \nabla) \bu(t)  + \bg_k(t,\bu(t)) \right], \quad k \in \mathbb{N},
\end{equation*}
where $I$ is the identity operator. Thus, we consider the following equivalent problem, i.e. the problem with $\P$ projection:
\begin{equation}\label{eqn-Modified-stochastic-CHNSEs-2}
\begin{cases}
\d \bu= - [ \nu \Stokes \bu + \bB_0(\bu,\bu) + \newK \bR_0(\phi,\phi)] \,\d t + [\bSi(t)\bu + \bG(t,\bu)]\,\d \tildeW,
\\
\d \phi=- B_1(\bu,\phi) \, \d t - \Athree \cp\, \d t,
   \\
\cp(t)= \Atwo \phi(t) + \psi^\prime(\phi(t)) - \avg{\psi^\prime(\phi(t))}, \;\; t \in [0,T],
      \\
(\bu(0),\phi(0))= (\bu_0,\phi_0),
\end{cases}
\end{equation}     
where we assume that the initial data $(\bu_0,\phi_0)$ satisfies
   \begin{equation}\label{eq-hypo-initial-data}
     (\bu_0,\phi_0) \in \nHB.
    \end{equation}
Furthermore, we define the following maps
\begin{equation}\label{eqn-F and G}
\begin{aligned}
\bF_1 &: \mathbb{V} \ni
\bX=(\bu,\phi) \mapsto
\Bigl(- \nu \Stokes \bu, - B_1(\bu,\phi)  - \Athree \cp\Bigr) \in  \mathbb{V}^\prime,
\\
\bF_2&: \mathbb{V} \ni \bX=(\bu,\phi) \mapsto
\Bigl(- \bB_{0}(\bu,\bu) - \newK \bR_{0}(\phi,\phi),0\Bigr) \in  \mathbb{V}^\prime,
\\
\newG_0&: [0,T] \times \mathbb{V} \ni (t,\bX)=(t,(\bu,\phi)) \mapsto \left(\bSi(t)\bu + \bG(t,\bu), 0\right) \in 
\ell^2(\mathbb{H}),
\\
\bF&\coloneqq \bF_1 + \bF_2,
\end{aligned}
\end{equation}
where, 
     \begin{equation}\label{eqn-mu from pgi}
       \cp = \Atwo \phi + \psi^\prime(\phi) - \avg{\psi^\prime(\phi)}, \mbox{ for }  \phi \in \newone{V}.
    \end{equation}
Then, we reformulate problem \eqref{eqn-Modified-stochastic-CHNSEs-2} in abstract form
\begin{equation}\label{eqn-compact-modified-stochastic-CHNSEs-3}
\begin{cases}
\d \bX= \bF(\bX)\,\d t + [ \newG_0(t,\bX)]\,\d \tildeW(t),
\\
\cp(t)= \Atwo \phi(t) + \psi^\prime(\phi(t)) - \avg{\psi^\prime(\phi(t))}, \;\; t \in[0,T], \\
\bX(0)= (\bu_{0},\phi_{0}).
\end{cases}
\end{equation}
\begin{definition}\label{def-compact-modified-stochastic-CHNSEs-3}
Assume that $(\Omega, \mathscr{F}, \mathbb{F},\mathbb{P})$
is a filtered complete probability space with filtration $\mathbb{F}= \{\mathscr{F}_t\}_{t\in[0,T]}$ satisfying the usual assumptions \ref{ass-usual}.     
Assume that $\tildeW$ is an $\ell^2$-cylindrical Wiener process on that probability space.
A $\mathbb{V}$-valued progressively measurable process $\bX=(\bu,\phi)$ is a weak solution, in the PDE sense, to Problem \eqref{eqn-compact-modified-stochastic-CHNSEs-3} if and only if 
\begin{trivlist}
\item[(i)] for $\mathbb{P}$-almost all $\omega \in \Omega$, the trajectory $\cp(\cdot,\omega)$ of the process  $\cp=(\cp(t): t\in [0,T])$ defined by
              \begin{equation}\label{Eq-identition-mu-0}
                 \cp(t)= \Atwo \phi(t) + \psi^\prime(\phi(t)) - \avg{\psi^\prime(\phi(t))}, \;\; t\in [0,T]    
              \end{equation}
belongs to the space $L^{2}(0,T;\newone{H})$ and  
\begin{equation}\label{eqn-condition on mu(s)}
\mathbb{E} \int_0^T \lvert \cp(t)\rvert_{\newone{H}}^2\,\d t <\infty, 
\end{equation}
and 
\item[(ii)] 
the process $\left(\bF(\bX(t)): t \in [0,T]\right)$ is a $\mathbb{V}^\prime$-valued progressively measurable process such that 
\begin{equation}\label{eqn-condition on F(X(t))}
    \mathbb{E} \int_0^T \Vert \bF(\bX(s)) \Vert_{\mathbb{V}^\prime}\,\d s <\infty,
\end{equation}
the process $\left(\newG_0(t,\bX(t)): t \in [0,T]\right)$ is an $\ell^2(\mathbb{H})$-valued progressively measurable process such that 
\begin{equation}\label{eqn-condition on G(s,X(s))}
    \mathbb{E} \int_0^T \Vert \newG_0(s,\bX(s)) \Vert^2_{\ell^2(\mathbb{H})}\,\d s <\infty,
\end{equation}
and, $\mathbb{P}$-almost surely, 
\begin{equation}\label{eqn-main}
\bX(t)=\bX_0+ \int_0^t  \bF(\bX(s))\,\d s + \int_0^t [ \newG_0(s,\bX(s))]\,\d \tildeW(s) \mbox{ in }  \nU^\prime, \mbox{ for every } t \in [0,T],
\end{equation}
where by $\int_0^t [ \newG_0(s,\bX(s))]\,\d \tildeW(s)$, we mean the It\^o integral understood as an $\mathbb{H}$-valued continuous martingale defined earlier,
and the maps $\bF$ and $\newG_0$ are defined in \eqref{eqn-F and G}.
\end{trivlist}
\end{definition}
\begin{remark}\label{rem-def-compact-modified-stochastic-CHNSEs-3} Assume that a $\mathbb{V}$-valued  sequence  $\left( \tilde{\be}_k: k \in \mathbb{N}\right)$  is an ONB of $\mathbb{H}$. If a process satisfies conditions \eqref{eqn-condition on F(X(t))} and \eqref{eqn-condition on G(s,X(s))}, then condition \eqref{eqn-main} holds $\mathbb{P}$-almost surely if and only if the following one holds:
\begin{trivlist}
\item[(ii')]For every $k \in \mathbb{N}$, 
$\mathbb{P}$-almost surely, 
\begin{equation}\label{eqn-main-1}
\begin{aligned}
    \duality{\bX(t)}{\tilde{\be}_k}{\mathbb{V}}{\mathbb{V}^\prime}=& \duality{\bX_0}{\tilde{\be}_k}{\mathbb{V}}{\mathbb{V}^\prime} + \int_0^t  \duality{\bF(\bX(s))}{\tilde{\be}_k}{\mathbb{V}}{\mathbb{V}^\prime}\,\d s 
    \\
    &+
\int_0^t \langle (\newG_0(s,\bX(s)),\tilde{\be}_k)_{\mathbb{H}},\,\d\tildeW(s)  \rangle \mbox{ in } \mathbb{R}, \;\; \mbox{for every $t \in [0,T]$}.  
\end{aligned}
\end{equation}
\end{trivlist}
\end{remark}
\begin{remark}\label{rem-def-compact-change of Wiener process} 
Suppose  that $(\Omega, \mathscr{F}, \mathbb{F},\mathbb{P})$
is a filtered complete probability space with filtration $\mathbb{F}= \{\mathscr{F}_t\}_{t\in[0,T]}$ satisfying the usual assumptions \ref{ass-usual}.     
Assume that $\tildeW$ is an $\ell^2$-cylindrical Wiener process on that probability space as in Assumption \ref{ass-cylindrical Wiener process}, i.e. it has representation \eqref{wiener-process-representation}. 
Assume that the coefficients $\bF$ and $\newG_0$ are as in \eqref{eqn-F and G}. 
Assume that  a $\mathbb{V}$-valued progressively measurable process $\bX=(\bu,\phi)$ is a weak solution, in the PDEs sense, to Problem \eqref{eqn-compact-modified-stochastic-CHNSEs-3} in the sense of the above Definition.  Suppose next  that $\rK$ is a separable infinite dimensional Hilbert space 
and that $\mfI: \rK \to \ell^2$ is an isometric isomorphism. Suppose further that $W$ is an $\rK$-cylindrical Wiener process on the same probability space as in Definition 
\ref{def-cylindrical Wiener process}, i.e. 
it has representation \eqref{eqn-cylindrical Wiener process-abstract} with the same sequence $\left( w_k\right)_{k=1}^\infty$ of iid standard Brownian motions and 
$e_k \coloneqq \mfI^{-1} \tilde{e}_k$, $k \in \mathbb{N}$. Define a map 
\[
\widetilde{{\newG_0}}
\colon [0,T] \times \mathbb{V} \ni (t,\bX)=(t,(\bu,\phi)) \mapsto  \newG_0(t,\bX) \circ \mfI  \in 
\mathscr{T}_2(\rK,\mathbb{H}).
\]
Then the ($\mathbb{V}$-valued progressively measurable) process $\bX=(\bu,\phi)$ is a weak solution, in the PDEs sense, to the following Problem
\begin{equation}\label{eqn-compact-modified-stochastic-CHNSEs-3--change of Wiener process}
\begin{cases}
\d \bX= \bF(\bX)\,\d t + [ \widetilde{\newG_0}(t,\bX)]\,\d W(t),
\\
\cp(t)= \Atwo \phi(t) + \psi^\prime(\phi(t)) - \avg{\psi^\prime(\phi(t))}, \;\; t \in[0,T], \\
\bX(0)= (\bu_{0},\phi_{0}).
\end{cases}
\end{equation}
The converse  assertion is obviously also true.
\end{remark}
\section{Abstract formulation}\label{Ass-Abstract formulation} 
In what follows, we will study the existence of solutions to problem \eqref{eqn-compact-modified-stochastic-CHNSEs-3}. 
For the sake of clarity, we first formulate our main assumptions divided into a few pieces.
\begin{assumption}\label{assumption-abstract-1}
\begin{trivlist}
\item[(i)] $(\StokesH, (\cdot, \cdot)_{\StokesH})$ and
$(\StokesV,  \langle \cdot, \cdot \rangle_{\StokesV})$  are  separable Hilbert spaces such that $\StokesV$  is a dense subspace of  $\StokesH$ and the natural embedding $i_0 \colon \StokesV \embed \StokesH $ is continuous.  
\item[(ii)] Let also $\rU_0$ be another Hilbert space such that the embedding  $\rU_0 \embed \StokesV$ is dense and compact.
\end{trivlist}
\end{assumption}
In view of the Riesz Lemma, the space $\StokesH$ can be  identified with its dual $\StokesH^\prime$; consequently, 
\begin{equation}\label{eqn-Gelfand triple-abstract}
\StokesV \embed \StokesH \cong \StokesH^\prime \embed \StokesVp,    
\end{equation}
where the embedding $i_0^\prime$ is the adjoint of $i_0$. In other words, \eqref{eqn-Gelfand triple-abstract} forms a Gelfand
triple. In particular, the following identity holds 
\[
\duality{\bu}{\bv}{\StokesV}{\StokesVp}= (\bu, \bv)_{\StokesH} \;\; \mbox{ if } \bu \in \StokesH \mbox{ and } \bv \in \StokesV.
\]
Let $\Stokes: \StokesV \to \StokesVp$ be the linear isomorphism  defined via the Lax-Milgram Theorem such that 
\[
\duality{\Stokes \bu}{\bv}{\StokesV}{\StokesVp}=  \langle \bu, \bv \rangle_{\StokesV} \;\; \mbox{ if } \bu \in \StokesV \mbox{ and } \bv \in \StokesV.
\]
 By Remark 3.3 in \cite{Mik+Roz_1998} we can find an ONB $(\be_j: j \in \mathbb{N})$ of $\StokesH$, which is orthogonal in 
$\rU_0$. Let us denote $\lambda_j^{k_0} \coloneqq \Vert \be_j \Vert_{\rU_0}^{2}$.\\ 
Note that 
\[
(\bu,\bv)_{\rU_0}= \sum_{j=1}^\infty\, \lambda_j^{k_0}\, (\bu,\be_j)_{\StokesH}\, (\bv,\be_j)_{\StokesH},\; \bu,\,\bv \in \rU_0.
\]
\begin{definition}\label{def-pi_0n}
If $n \in \mathbb{N}$, then we define  $H_{0,n}\coloneqq \linspan \{\be_1,\ldots,\be_n\}$ and  \[\pi_{0,n}: \StokesH \to H_{0,n} \] to be the orthogonal projection.
\end{definition}
Let us note that $H_{0,n} \subset D(\Stokes)$ and all norms on $H_{0,n}$ are equivalent.
Let us fix $n \in \mathbb{N}$ and let $\pi_{0,n}: \rU_0^\prime \to H_{0,n}$ be the operator defined by
   \begin{equation}\label{eqn-projection-pi_{0,n} bu-ast}
     \pi_{0,n}\bu^\ast= \sum_{j=1}^n \duality{\bu^\ast}{\be_j}{\rU_0}{\rU_0^\prime}\be_j,\; \bu^\ast \in \rU_0^\prime.
  \end{equation}
Note that since $\rU_0\embed \StokesH \cong \StokesH^\prime \embed \rU_0^\prime$ is a  Gelfand triple, the restriction $\pi_{0,n}$ to $\StokesH$ satisfies 
   \begin{equation}\label{eqn-projection-pi_{0,n} bu}
       \pi_{0,n}\bu = \sum_{j=1}^n (\bu,\be_j)_{\StokesH}\be_j,\; \bu \in \StokesH.
   \end{equation}
Therefore, the restriction $\pi_{0,n}$ to $\StokesH$  is  the $\StokesH$-orthogonal projection from $\StokesH$ onto $H_{0,n}$.\newline
Let us now state and prove some of the properties of the operator $\pi_{0,n}$.
\begin{properties}\label{eqn-properties-pi_{0,n}} 
Let $n \in \mathbb{N}$ and let $\tilde{\pi}_{0,n}:= \pi_{0,n}\lvert_{\rU_0}: \rU_0 \to H_{0,n}$.
\begin{trivlist}
\item[(o)] For all $\bu^\ast \in \rU_0^\prime$ and $\bv \in \rU_0$,
\[
\duality{\pi_{0,n} \bu^\ast}{\bv}{\rU_0}{\rU_0^\prime}= \duality{\bu^\ast}{\tilde{\pi}_{0,n} \bv}{\rU_0}{\rU_0^\prime}.
\]
\item[(i)] $\tilde{\pi}_{0,n}$ is the $\rU_0$-orthogonal projection onto  $H_{0,n}$.
\item[(ii)] $\Vert \pi_{0,n} \Vert_{\mathcal{L}(\rU_0^\prime)} \leq 1$.
\item[(iii)] For all $\bu \in \rU_0$, we have
\[
\lim_{n \to \infty} \Vert \tilde{\pi}_{0,n} \bu - \bu \Vert_{\rU_0}=0.
\]
\end{trivlist}
\end{properties}

\begin{proof}Let us choose and 
fix $\bu^\ast \in \rU_0^\prime$ and $\bv \in \rU_0$. By the definition \eqref{eqn-projection-pi_{0,n} bu-ast} and since $\rU_0\embed \StokesH \cong \StokesH^\prime \embed \rU_0^\prime$, we have
\begin{align*}
&(\pi_{0,n}\bu^\ast,\bv)_{\StokesH}
= \sum_{j=1}^n \duality{\bu^\ast}{\be_j}{\rU_0}{\rU_0^\prime} (\be_j,\bv)_{\StokesH}
= \sum_{j=1}^n \duality{\bu^\ast}{\be_j}{\rU_0}{\rU_0^\prime}\,\duality{\be_j}{\bv}{\rU_0}{\rU_0^\prime}\\
&= \duality{\sum_{j=1}^n \duality{\bu^\ast}{\be_j}{\rU_0}{\rU_0^\prime}\be_j}{\bv}{\rU_0}{\rU_0^\prime}
=\duality{\pi_{0,n}\bu^\ast}{\bv}{\rU_0}{\rU_0^\prime}
\end{align*}
and
\begin{align*}
(\pi_{0,n}\bu^\ast,\bv)_{\StokesH}
&= \sum_{j=1}^n \duality{\bu^\ast}{\be_j}{\rU_0}{\rU_0^\prime}\,\duality{\be_j}{\bv}{\rU_0}{\rU_0^\prime}
=  \duality{\bu^\ast}{\sum_{j=1}^n \duality{\be_j}{\bv}{\rU_0}{\rU_0^\prime} \be_j}{\rU_0}{\rU_0^\prime}\\
&= \duality{\bu^\ast}{\sum_{j=1}^n \duality{\bv}{\be_j}{\rU_0}{\rU_0^\prime} \be_j}{\rU_0}{\rU_0^\prime}
=\duality{\bu^\ast}{\tilde{\pi}_{0,n} \bv}{\rU_0}{\rU_0^\prime}.
\end{align*}
Therefore, the proof of Part $(o)$ follows.

\noindent
Let us consider the proof of Part $(i)$. Put $\tilde{\be_i}\coloneqq \frac{\be_i}{\Vert \be_i \Vert_{\rU_0}},\, i \in \mathbb{N}$.
We have
\begin{equation}\label{eqn-(be_i,be_j)_{rU_0}}
(\be_i,\be_j)_{\rU_0}
= \sum_{k=1}^\infty \lambda_k^{k_0} (\be_i,\be_k)_{\StokesH} (\be_j,\be_k)_{\StokesH}
=\lambda_i^{k_0} \delta_{ij},\;\; i,\,j \in \mathbb{N},
\end{equation}
from which we infer that
     \begin{equation}\label{eqn-lambda_i-equal-norm-e_i}
        \lambda_i^{k_0}= \Vert \be_i \Vert_{\rU_0}^2,\; i \in \mathbb{N}
     \end{equation}
and    
\begin{equation}
(\tilde{\be}_i,\tilde{\be}_j)_{\rU_0}
= \left(\frac{\be_i}{\Vert \be_i \Vert_{\rU_0}}, \frac{\be_j}{\Vert \be_j \Vert_{\rU_0}} \right)_{\rU_0}
= \frac{1}{\Vert \be_i \Vert_{\rU_0} \Vert \be_j \Vert_{\rU_0}} (\be_i,\be_j)_{\rU_0}
= \frac{\Vert \be_i \Vert_{\rU_0}^2}{\Vert \be_i \Vert_{\rU_0} \Vert \be_j \Vert_{\rU_0}}\delta_{ij},\;\; i,\,j \in \mathbb{N}.
\end{equation}
Thus, the system $(\tilde{\be_i})_{i \in \mathbb{N}}$ is orthonormal in $\rU_0$.
Let us check that this system forms a basis in $\rU_0$. For this purpose, let $\bv \in \rU_0$ be an arbitrary vector $\rU_0$-orthogonal to each $\be_i,\,i\in \mathbb{N}$.\newline
We have
\begin{align*}
0= (\bv,\be_i)_{\rU_0}
= \sum_{j=1}^\infty \lambda_j^{k_0} (\bv,\be_j)_{\StokesH} (\be_i,\be_j)_{\StokesH}
= \sum_{j=1}^\infty \lambda_j^{k_0} (\bv,\be_j)_{\StokesH} \delta_{ij}
= \lambda_i^{k_0} (\bv,\be_i)_{\StokesH},\; i \in \mathbb{N}.
\end{align*}
Now, since $\lambda_i^{k_0}>0$, we infer that $(\bv,\be_i)_{\StokesH}=0,\; i \in \mathbb{N}$ and because $(\be_i)_{i\in \mathbb{N}}$ is the $\StokesH$-ONB, we then deduce that $\bv=0$. Therefore, the system $(\tilde{\be_i})_{i \in \mathbb{N}}$ is an ONB in $\rU_0$.

\noindent
To complete the proof of Part $(i)$, we observe that if $\bu \in \rU_0$, then
\begin{align*}
(\bu,\tilde{\be}_i)_{\rU_0}
&= \sum_{j=1}^\infty \lambda_j^{k_0} (\bu,\be_j)_{\StokesH} (\tilde{\be}_i,\be_j)_{\StokesH}
= \sum_{j=1}^\infty \lambda_j^{k_0} \times \frac{1}{\Vert \be_i \Vert_{\rU_0}} (\bu,\be_j)_{\StokesH} (\be_i,\be_j)_{\StokesH}\\
&= \lambda_i^{k_0} \times \frac{1}{\Vert \be_i \Vert_{\rU_0}} (\bu,\be_i)_{\StokesH}
= \left(\bu, \frac{\lambda_i^{k_0} \be_i}{\Vert \be_i \Vert_{\rU_0}} \right)_{\StokesH},
\end{align*}
which implies that
\begin{align*}
(\bu,\be_i)_{\StokesH}\be_i
= \left(\bu, \frac{\lambda_i^{k_0} \be_i} {\Vert \be_i \Vert_{\rU_0}} \right)_{\StokesH} \frac{\be_i} {\Vert \be_i \Vert_{\rU_0}}
= (\bu,\tilde{\be}_i)_{\rU_0} \frac{\be_i}{\Vert \be_i \Vert_{\rU_0}}
= (\bu,\tilde{\be}_i)_{\rU_0} \tilde{\be}_i.
\end{align*}
This, jointly with \eqref{eqn-projection-pi_{0,n} bu} yields that
   \begin{equation}\label{eqn-tilde-pi_{0,n}-bu}
       \tilde{\pi}_{0,n}\bu = \sum_{j=1}^n (\bu,\tilde{\be}_j)_{\rU_0} \tilde{\be}_j,\; \bu \in \rU_0.
   \end{equation}
The latter equality implies that $\tilde{\pi}_{0,n}$ is an $\rU_0$-orthogonal projection onto $H_{0,n}$ and, in particular $\Vert \tilde{\pi}_{0,n} \Vert_{\mathcal{L}(\rU_0)}=1$. This completes the proof of Part $(i)$. \newline
The proof of Part $(ii)$ is a \textit{verbatim} reproduction of similar result in \cite[Lemma 2.6]{BKMR_2025}, while the proof of Part $(iii)$ is a direct consequence of \eqref{eqn-tilde-pi_{0,n}-bu}.
\end{proof}

\medskip
\addtocounter{theorem}{-2}

\begin{assumption}\label{assumption-abstract-2}
\begin{trivlist}
\item[(iii)] Assume that $\newone{V}$ and $\newone{H}$ are Hilbert spaces defined earlier in \eqref{eqn-V-LM-R.4.6} and \eqref{eqn-H-LM-R.4.6}.
\end{trivlist}
\end{assumption}

\addtocounter{theorem}{+1}

Let us recall that $\{\hat{e}_k\}_{k \in \mathbb{N}}$ is  an ONB in  $\zero{L}{2}(\domO)$ which is orthogonal in $\newone{V}$. Put
\[
\bar{e}_k\coloneqq \frac{ \hat{e}_k}{\lvert \hat{e}_k \rvert_{\newone{H}}},\; k \in \mathbb{N}.
\]
Then $\{\bar{e}_k\}_{k \in \mathbb{N}}$ is an ONB in $\newone{H}$ which is orthogonal in $\newone{V}$. \newline
Define $H_{1,n}\coloneqq \linspan \{\bar{e}_1,\ldots,\bar{e}_n \}$ and  note that $H_{1,n} \subset \newone{V}$ and all norms on $H_{1,n}$ are equivalent.
\begin{definition}\label{def-pi_1n}
If $n \in \mathbb{N}$, then  $\pi_{1,n}: \newonep{V} \to H_{1,n}$ is  the operator defined by
   \begin{equation}\label{eqn-projection-pi_{1,n} bu-ast}
     \pi_{1,n}\phi^\ast= \sum_{j=1}^n \duality{\phi^\ast}{\bar{e}_j}{\newone{V}}{\newonep{V}}\bar{e}_j,\; \;\; \phi^\ast \in \newonep{V}.
  \end{equation}
  \end{definition}
  
Let us also define  a map  \[\tilde{\pi}_{1,n}:= \pi_{1,n}\lvert_{\newone{V}}: \newone{V} \to H_{1,n}.\]
  These maps have the following properties similar to the previous case. 

\begin{properties}\label{eqn-properties-pi_{1,n}} 
\begin{trivlist}
\item[(o)] The restriction of the map  $\pi_{1,n}$ to the space $\newone{H}$  is  the $\newone{H}$-orthogonal projection from $\newone{H}$ onto $H_{1,n}$ and 
   \begin{equation}\label{eqn-projection-pi_{1,n} phi}
       \pi_{1,n}\phi = \sum_{j=1}^n (\phi,\bar{e}_j)_{\newone{H}}\bar{e}_j,\; \phi \in \newone{H}.
   \end{equation}
\item[(i)] For all $\phi^\ast \in \newonep{V}$ and $\phi \in \newone{V}$,
\[
\duality{\pi_{1,n} \phi^\ast}{\phi}{\newone{V}}{\newonep{V}}= \duality{\phi^\ast}{\tilde{\pi}_{1,n} \phi}{\newone{V}}{\newonep{V}}.
\]
\item[(ii)] $\tilde{\pi}_{1,n}$ is the $\newone{V}$-orthogonal projection onto  $H_{1,n}$.
\item[(iii)] $\Vert \pi_{1,n} \Vert_{\mathcal{L}(\newonep{V})} \leq 1$.
\item[(iv)] For all $\phi \in \newone{V}$, we have
\[
\lim_{n \to \infty} \Vert \tilde{\pi}_{1,n} \phi - \phi \Vert_{\newone{V}}=0.
\]
\end{trivlist}
\end{properties}

\medskip
\addtocounter{theorem}{-3}
We continue with our assumptions. 

\begin{assumption}\label{assumption-abstract-3}
\begin{trivlist}
\item[(iv)] Assume that 
\[
\bG_0:[0,T] \times \StokesH \ni (s,\bu) \mapsto ( \bG_{0}^k(s,\bu))_{k=1}^\infty \in \ell^2(\StokesH),
\]
is a Carath\'eodory function  satisfying the following linear growth condition. 
There exists a function $\tilde{h} \in L^2(0,T)$ and constants $C_3,\,\tilde{C}_3>0$ 
such that for all $(t, \bu) \in [0,T] \times \StokesH$,
       \begin{equation}\label{eqn-linear growth}
         \Vert \bG_0(t,\bu)\Vert_{\ell^2(\StokesH)}^2 \leq \tilde{C}_3  \vert \tilde{h}(t)\vert^2 + C_3 \vert \bu \vert_{\StokesH}^2.
       \end{equation}
\item[(v)] 
            \begin{itemize}
                  \item[(v)-a)]  There exists a sequence of Carath\'eodory functions 
                  \[\bG_{0,n}:[0,T] \times H_{0,n} \to \ell^2(H_{0,n}) \subset \ell^2(\StokesH) \mbox{ satisfying } \eqref{eqn-linear growth} \mbox{ uniformly in } n.\] 
                  Furthermore, each $\bG_{0,n}$ is bounded and Lipschitz from $H_{0,n}$ to $\ell^2(H_{0,n})$, uniformly with respect to $t \in [0,T]$. 
                  \item[(v)-b)] There exists $\Delta_T \subset [0,T]$ of full measure such that for fixed $s \in \Delta_T$, 
                               \begin{equation*}
                                  \bG_{0,n}(s,\bu_n) \to \bG_0(s,\bu) \mbox{ in } \ell^2(\StokesH) \mbox{ whenever } \bu_n \to \bu  \mbox{ in } \StokesH.
                               \end{equation*}
            \end{itemize}
\medskip
\item[(vi)] A map 
\[
\bSi_0:[0,T] \times \StokesV 
\ni (s,\bu) \mapsto ( \bSi_{0}^k(s,\bu))_{k=1}^\infty \in  \ell^2(\StokesH) \] is a measurable function that is linear in its second variable, and it satisfies one of the following equivalent versions of the coercivity  condition:
there exist $\delta_0 \in (0,\nu)$ and $\tilde{C}_1>0$ such that
    \begin{equation}\label{Eqn-coercivity-2}
      \nu \,\duality{\Stokes\bu}{\bu}{\StokesV}{\StokesVp}  - \frac12 \Vert \bSi_0(t)\bu \Vert_{\ell^2(\StokesH)}^2 
       \geq  \delta_0 \Vert \bu \Vert_{\StokesV}^2, \;\; \bu \in \StokesV,
    \end{equation}
and
    \begin{equation}\label{Eqn-coercivity-3}
      \Vert \bSi_0(t)\bu \Vert_{\ell^2(\StokesH)}^2
      \leq \tilde{C}_1 \Vert \bu \Vert_{\StokesV}^2, \;\; \bu \in \StokesV.
    \end{equation}

\item[(vii)] We define a sequence $(\bSi_0^n)_{n \in \mathbb{N}}$ of maps from $[0,T] \times H_{0,n}$ to $\mathscr{L}(H_{0,n},\ell^2(H_{0,n}))$ by

\begin{equation}\label{eqn-Sigma^n}
\bSi_0^n: [0,T] \times H_{0,n} \ni (t,\bu) \mapsto 
 [\pi_{0,n} \circ  (\bSi_0(t)\bu)] \in \ell^2(H_{0,n}).
\end{equation}

\end{trivlist}
\end{assumption}

\addtocounter{theorem}{+2}

\begin{remark}
\begin{trivlist}
\item[(o)] A specific example of $\bG_0$ is the map $\bG$ defined in \eqref{eqn-G}.
\item[(i)] In the case of the map $\bG$ from \eqref{eqn-G}, we show that the above assumption (v) is satisfied in Section \ref{Sect-approximation of g}.
\item[(ii)] As an example of the map $\bSi_0$, we take the map $\bSi$ defined in formula \eqref{Eqn-def-Sigma}.
\end{trivlist}
\end{remark}
In the following lemma, we show that the coercivity condition \eqref{Eqn-coercivity-2} and inequality \eqref{Eqn-coercivity-3} 
hold when $\StokesH$ and $\Stokes$ are the Hilbert space and Stokes operator defined in \eqref{eqn-spaces-NSEs} and \eqref{eqn-Stokes operator}, resp.
Moreover, we note that the constant $\delta_0$ is from inequality \eqref{eqn-coercivity} from Assumption \ref{ass-coercivity}.
\begin{lemma}\label{Lem1}
Let Assumptions \ref{ass-G+Sigma}-\ref{ass-bg+sigma} be satisfied. Then, for every $t \in[0,T]$, we have
    \begin{equation}\label{Eqn-Sigma^n-estimate}
        \nu \,\duality{\Stokes\bu}{\bu}{\StokesV}{\StokesVp} - \frac12 \Vert \bSi(t)\bu \Vert_{\ell^2(\StokesH)}^2 
       \geq  \delta_0 \Vert \bu \Vert_{\StokesV}^2, \;\; \bu \in \StokesV,
    \end{equation}
and
        \begin{equation}\label{Eqn-Sigma^n-estimate-1}
            \Vert \bSi(t)\bu \Vert_{\ell^2(\StokesH)}^2
             \leq %
             2(\nu-\delta_0) \Vert \bu \Vert_{\StokesV}^2, \;\; \bu \in \StokesV.
        \end{equation}
Furthermore, there exists a constant $C_4>0$ such that for every $t \in[0,T]$,
   \begin{equation}\label{eq-3.20aa}
       \Vert \bSi(t)\bu \Vert_{\ell^2(\StokesVp)}^2
      \leq C_4 \lvert \bu \rvert_{\StokesH}^2, \;\; \bu \in \StokesH.
   \end{equation}
\end{lemma}
\begin{proof}
Let us fix $\bu \in \StokesV$. From Assumptions \ref{ass-G+Sigma} and \ref{ass-coercivity}, we infer that for every $t \in [0,T]$,
\begin{equation}\label{eq-3.20}
\begin{aligned}
&\frac12 \Vert \bSi(t)\bu \Vert_{\ell^2(\StokesH)}^2
= \frac12 \sum_{k=1}^{\infty} \lvert \Pi (\sigma_k(t) \cdot \nabla) \bu \rvert_{\StokesH}^2
\\
&\leq \frac12 \sum_{k=1}^{\infty} \lvert (\sigma_k(t) \cdot \nabla) \bu \rvert_{\mathbb{L}^2}^2
= \frac12 \sum_{i,j,\ell=1}^d \int_{\domO} \sum_{k=1}^\infty \sigma_k^i (t, x) \sigma_k^j (t, x) \partial_i \bu^\ell(x) \partial_j \bu^\ell(x)\, \d x 
 \\
&\leq \nu \sum_{i,j,\ell=1}^d \delta_{ij} \int_{\domO} \partial_i \bu^\ell(x) \partial_j \bu^\ell(x)\, \d x - \delta_0 \sum_{i,\ell=1}^d \int_{\domO} \lvert \partial_i \bu^\ell(x) \rvert^2\, \d x 
\\
&= (\nu - \delta_0) \sum_{i,\ell=1}^d \int_{\domO} \lvert \partial_i \bu^\ell(x) \rvert^2\, \d x
= (\nu - \delta_0) \lvert \nabla \bu \rvert_{\mathbb{L}^2}^2
= (\nu  - \delta_0)  \,\duality{\Stokes\bu}{\bu}{\StokesV}{\StokesVp} 
 \\
&= \nu \,\duality{\Stokes\bu}{\bu}{\StokesV}{\StokesVp}   - \delta_0 \,\duality{\Stokes\bu}{\bu}{\StokesV}{\StokesVp} 
= \nu \,\duality{\Stokes\bu}{\bu}{\StokesV}{\StokesVp}   - \delta_0 \Vert \bu \Vert_{\StokesV}^2,
\end{aligned}
\end{equation}
from which we deduce \eqref{Eqn-Sigma^n-estimate}. Consequently, it also follows that
\begin{equation}\label{eq-3.200}
\begin{aligned}
&\frac12 \Vert \bSi(t)\bu \Vert_{\ell^2(\StokesH)}^2
\leq  (\nu  - \delta_0)  \,\duality{\Stokes\bu}{\bu}{\StokesV}{\StokesVp} 
= (\nu  - \delta_0)  \Vert \bu \Vert_{\StokesV}^2,
\end{aligned}
\end{equation}
which implies \eqref{Eqn-Sigma^n-estimate-1}. \newline
It remains to prove \eqref{eq-3.20aa}. To this end, we fix \(t \in  [0,T]\), $k\in \mathbb{N}$, and define the bilinear form
   \begin{equation*}
     \hat{b}_t^{k}: \mathcal{V} \times \mathcal{V} \subset \StokesH \times \StokesV \ni (\bu,\bv) \mapsto \int_{\domO} (\sigma_k(t,x) \cdot \nabla) \bu(x) \cdot \bv(x)\,\d x \in \mathbb{R}.
   \end{equation*}
By integration by parts and the H\"older inequality, we have for all $\bu,\, \bv \in \mathcal{V}$:
\begin{align*}
\hat{b}_t^{k}(\bu,\bv)
& = - \int_{\domO} [\diver \sigma_k (t,x) \bu(x) \cdot \bv(x) + (\sigma_k(t,x) \cdot \nabla) \bv(x) \cdot \bu(x)] \, \d x
        \\
& \leq \Vert \diver \sigma_k \Vert_{L^\infty([0,T] \times \domO)} \lvert \bu \rvert_{\StokesH} \Vert \bv \Vert_{\StokesV} + \Vert \sigma \Vert_{L^\infty([0,T] \times \domO)} \Vert \bv \Vert_{\StokesV} \lvert \bu \rvert_{\StokesH}.
\end{align*}
Note that the above inequality extends to all $\bu,\,\bv \in \StokesV$ by the density of $\mathcal{V}$ in $\StokesV$.
Hence, we can define a linear map 
\[
\hat{B}_t^{k}: \StokesH \ni \bu  \mapsto \hat{b}_t^{k}(\bu,\cdot) \in \StokesVp,
\]
 and  deduce that 
\begin{equation*}
	\Vert \hat{B}_t^{k}(\bu) \Vert_{\StokesVp} \leq (\Vert \diver \sigma_k \Vert_{L^\infty([0,T] \times \domO)} +\Vert \sigma_k \Vert_{L^\infty([0,T] \times \domO)}) \lvert \bu \rvert_{\StokesH}, \; \; \bu \in \StokesH,
\end{equation*}
or, equivalently,
   \begin{equation*}
	 \Vert (\sigma_k(t)\cdot \nabla) \bu \Vert_{\StokesVp} \leq (\Vert \diver \sigma_k \Vert_{L^\infty([0,T] \times \domO)} + \Vert \sigma_k \Vert_{L^\infty([0,T] \times 
       \domO)}) \lvert \bu \rvert_{\StokesH}, \;\; \bu \in \StokesH.
  \end{equation*}
From the latter inequality, we infer that
\begin{align*}
&\Vert \bSi(t)\bu \Vert_{\ell^2(\StokesVp)}^2
= \sum_{k=1}^\infty \Vert \bSi_{k}(t)\bu \Vert_{\StokesVp}^2 
= \sum_{k=1}^\infty \Vert \P [(\sigma_k(t) \cdot \nabla) \bu] \Vert_{\StokesVp}^2 
 \\
&\leq \sum_{k=1}^\infty \Vert (\sigma_k(t) \cdot \nabla) \bu \Vert_{\StokesVp}^2 
\leq 2 \sum_{k=1}^\infty (\Vert \diver \sigma_k\Vert_{L^\infty([0,T] \times \domO)}^2 + \Vert \sigma_k \Vert_{L^\infty([0,T] \times \domO)}^2) \lvert \bu \rvert_{\StokesH}^2.
\end{align*}
This, combined with Assumption \ref{ass-bg+sigma}, yields that for every $t \in [0,T]$,
   \begin{equation*}
      \Vert \bSi(t)\bu \Vert_{\ell^2(\StokesVp)}^2
     \leq 2(C_1 + C_2) \lvert \bu \rvert_{\StokesH}^2, \; \bu \in \StokesH.
   \end{equation*}
This completes the proof of Lemma \ref{Lem1}.  
\end{proof}
The maps $\bSi_0^n$ defined in \eqref{eqn-Sigma^n} satisfy the following "uniform" coercivity condition.
\begin{lemma}
\label{lem-coercivity on H_{0,n}}
Under the assumptions stated in Section \ref{Ass-Abstract formulation}, for all $\bu \in H_{0,n}$ and $t \in [0,T]$,
         \begin{equation}\label{eqn-coercivity on H_{0,n}}
            \nu\,\duality{\Stokes_{,n}\bu}{\bu}{\StokesV}{\StokesVp} - \frac12 \Vert \bSi^n_0(t) \bu\Vert^2_{\ell^2(H_{0,n})} \geq \delta_0 \Vert \bu \Vert^2_{\StokesV},     
         \end{equation}
where $\Stokes_{,n}$ is the restriction and co-restriction of $\Stokes$ to $H_{0,n}$, defined by
\[
\Stokes_{,n} \bu = \pi_{0,n} (\Stokes \bu)= \Stokes \bu \in H_{0,n}, \;\; \bu \in H_{0,n}.
\]
\end{lemma}
\begin{proof}[Proof of Lemma \ref{lem-coercivity on H_{0,n}}]
Since $\Vert \pi_{0,n} \Vert_{\mathscr{L}(\mathbb{L}^2(\domO))}\leq 1$ and $\Stokes_{,n} \bu = \Stokes \bu$ for $\bu \in H_{0,n}$, it follows from the assumption \eqref{Eqn-coercivity-2} that for all $\bu \in H_{0,n}$ and $t \in [0,T]$,
\begin{align*}
&\duality{\nu \Stokes_{,n}\bu}{\bu}{\StokesV}{\StokesVp} - \frac12 \Vert \bSi^n_0(t) \bu\Vert^2_{\ell^2(H_{0,n})}
=\duality{\nu \Stokes_{,n}\bu}{\bu}{\StokesV}{\StokesVp} - \frac12 \sum_{k=1}^\infty \Vert \bSi^n_0(t) \bu[\tilde{e}_k] \Vert^2_{H_{0,n}}
\\
&= \duality{\nu \Stokes \bu}{\bu}{\StokesV}{\StokesVp} - \frac12 \sum_{k=1}^\infty \Vert \bSi^n_0(t) \bu[\tilde{e}_k] \Vert^2_{H_{0,n}}
\geq \duality{\nu \Stokes\bu}{\bu}{\StokesV}{\StokesVp}  - \frac12 \Vert \bSi_0(t)\bu \Vert_{\ell^2(\StokesH)}^2 \geq  \delta_0 \Vert \bu \Vert_{\StokesV}^2,
\end{align*}
where $(\tilde{e}_k)_{k=1}^\infty$ is the canonical ONB of $\ell^2$.
From the above inequality, we deduce \eqref{eqn-coercivity on H_{0,n}}.
\end{proof}
Regarding the maps $\bG_{0,n}$, we have the following auxiliary result.
\begin{lemma}\label{Lem-approximation-c}
Assume the assumptions stated in Section \ref{Ass-Abstract formulation}. If $\bu_n \to \bu$ in $L^2(0,T;\StokesH)$, then
     \begin{equation}
        \lim_{n \to \infty} \int_0^T \Vert \bG_{0,n}(s,\bu_n(s)) - \bG_0(s,\bu(s)) \Vert_{\ell^2(\StokesH)}^2\,\d s=0.
     \end{equation}
\end{lemma}
\begin{proof}[Proof of Lemma \ref{Lem-approximation-c}]
Note that if $\bu_n \to \bu$ in $L^2(0,T;\StokesH)$, then, up to a subsequence, $\bu_n(s) \to \bu(s)$ in $\StokesH$ for Leb-a.a. $s \in [0,T]$.
Hence, by property $(v)\text{-b})$ of $\bG_{0,n}$, we have
    \begin{equation*}
       \bG_{0,n}(s,\bu_n(s)) \to \bG_0(s,\bu(s)) \mbox{ in } \ell^2(\StokesH) \mbox{ a.e. }
    \end{equation*}
Next, by property $(v)\text{-a})$ of $\bG_{0,n}$, we infer that 
\begin{align*}
\int_0^T \Vert \bG_{0,n}(s,\bu_n(s))\Vert_{\ell^2(\StokesH)}^2\,\d s
\leq \tilde{C}_3  \int_0^T \vert \tilde{h}(s) \vert^2\,\d s + C_3 \int_0^T \vert \bu_n(s) \vert_{\StokesH}^2\,\d s.
\end{align*}
Since the RHS of the above inequality is uniformly integrable, so is the LHS by comparison. 
Therefore, by applying the Vitali Convergence Theorem, see \cite[Theorem C.4]{Oksendal_2003}, we complete the proof of Lemma \ref{Lem-approximation-c}.
\end{proof}
\subsection{Weak martingale solutions}\label{subsec-weak solutions}
Let
\begin{equation}
\beta=
\begin{cases}
3/2, &\mbox{ if } \;\; d=3,
\\
2, &\mbox{ if } d=2.
\end{cases}
\end{equation}
The main result of this paper, which establishes the existence of weak martingale solutions to problem \eqref{eqn-compact-modified-stochastic-CHNSEs-3} with Landau potential \eqref{eqn-regular-potential}, is stated in the following theorem.
\begin{Theorem}\label{First-main-result}
Let $\domO \subset \mathbb{R}^d$, $d=2$ or $d=3$, be a bounded domain of  $\mathcal{C}^3$-class  and let $T$ be a fixed positive time.
Assume that the assumptions of Section \ref{Ass-Abstract formulation} hold and that $\bX_0= (\bu_0,\phi_0) \in \mathbb{H}$.   We assume that 
 $(\Omega,\mathscr{F},\mathbb{P})$ is a probability space with a right-continuous filtration $\mathbb{F}=(\mathscr{F}_t)_{t \in [0,T]}$ such that on this probability space two i.i.d. copies of  $\ell^2$-cylindrical $\mathbb{F}$-adapted  Wiener processes are defined.  Then, there exists 
\begin{trivlist}
\item[(i)] 
an $\ell^2$-cylindrical $\mathbb{F}$-adapted  Wiener process $W=(W(t): t \in [0,T])$, an  $\mathbb{F}$-progressively measurable process $(\bu,\phi)=\left(
(\bu(t),\phi(t)): \; t \in [0,T]\right)$, and a process $\cp=(\cp(t): \; t \in [0,T])$ satisfying the following conditions. 
       \begin{itemize}
         \item[(i-i)] The process $\bX \coloneqq (\bu,\phi)$ is $\mathbb{H}$-valued weakly continuous and 
                \begin{equation*}
                  \mathbb{E} \sup_{s \in [0,T]} \Vert \bX(s) \Vert_{\mathbb{H}}^2 + \mathbb{E} \int_0^T \Vert \bX(s) \Vert_{\StokesV \times \zero{H}{2}}^2\, \d s< \infty,
                \end{equation*}
       \item[(i-ii)] 
                \[\mathbb{E} \int_0^T \Vert  \phi(s) \Vert_{\newone{V}}^{\beta}\,\d s< \infty,\]
\[\cp(t)= \Atwo \phi(t) + \psi^\prime(\phi(t)) - \avg{\psi^\prime(\phi(t))},\;\; \mathbb{P}\mbox{-a.s.,}\]
\end{itemize}
\item[(ii)] the tuple $(W,\bu,\phi,\cp)$ is a solution to Problem \eqref{eqn-compact-modified-stochastic-CHNSEs-3} in the sense of Definition \ref{def-compact-modified-stochastic-CHNSEs-3}; that is, \eqref{eqn-compact-modified-stochastic-CHNSEs-3} holds $\mathbb{P}$-a.s.\ for every $t \in [0,T]$.
\end{trivlist}
\end{Theorem}
\begin{Theorem}\label{First-main-result-uniqueness} 
Under the assumptions of Theorem \ref{First-main-result}, if $d=2$, then the $\nHB$-valued process $\bX$ is (strongly) continuous w.r.t. $t$. Moreover, the pathwise uniqueness holds. 
\end{Theorem}
The proof of Theorem \ref{First-main-result} is given in Section \ref{sec-proof-main-result}. The proof of the paths continuity part of  Theorem \ref{First-main-result-uniqueness} is 
presented together with the proof of the third part of Theorem \ref{thm-5.3}. The uniqueness part  of  Theorem \ref{First-main-result-uniqueness} is stated as the last part of Theorem \ref{Thm-uniqueness-solution} and so proven therein. 
\begin{remark}\label{rem-energy initial data}
Our assumption on the initial data, i.e. that  $\bX_0= (\bu_0,\phi_0) \in \mathbb{H}$, and the assumption on the potential $\phi$, imply that the energy of  $\bX_0$ is finite, i.e. 
\begin{equation}\label{eqn-E-energy}
\mathscr{E}(\bX_0) \coloneqq \frac12 \lvert \bu_0 \rvert_{\StokesH}^2  + \newK \left(  \frac 12\lvert \phi_0 \rvert_{\newone{H}}^2  +\int_{\domO} \psi (\phi_0(x))\, \d x\right) < \infty.    
\end{equation}
\end{remark}
\begin{remark}\label{rem-existence of solutions} Of course there exist  probability spaces that satisfy  Assumption (i) of our main existence Theorem.
\end{remark}


\section{Solving the Galerkin Approximation  Problem and a priori estimates}
\label{subsec-Galerkin}
We consider the sequence of non-random variables $(\bu_{0,n}, \phi_{0,n})$ defined by 
\[ 
\bu_{0,n}=\pi_{0,n} \bu_0 \mbox{ and }\phi_{0,n}=\pi_{1,n} \phi_0, \;\; n \in \mathbb{N}.
\]
where $\pi_{0,n}$ has been defined in \eqref{def-pi_0n} and $\pi_{1,n}$ has been defined in \eqref{def-pi_1n}.
It follows that, as $n \to \infty$,
     \begin{equation}\label{convergence-for-initial-data}
        \bu_{0,n} \to \bu_0 \mbox{ in } \StokesH \mbox{ and } \phi_{0,n} \to \phi_0 \mbox{ in }  \newone{H}.
      \end{equation}
We also define the following  Lipschitz on balls maps 
\begin{equation}
\begin{aligned}
  \bB_{0,n} &: H_{0,n} \times H_{0,n} \ni (\bu,\bv)\mapsto   \pi_{0,n} \bB_0(\bu,\bv) \in H_{0,n}, 
\\
              B_{1,n}&: H_{0,n} \times H_{1,n} \ni (\bu,\phi)\mapsto \pi_{1,n} B_1(\bu,\phi)\in H_{1,n}.
\end{aligned}
\end{equation}
Next, we define a nonlinear map that is Lipschitz on bounded balls as follows:
\begin{equation}
\bar{R}_{0,n}: H_{1,n} \times H_{1,n} \ni (\phi,\psi) \mapsto  \pi_{0,n} ([\Atwo \psi + \pi_{1,n} (\psi^\prime(\phi) - \avg{\psi^\prime(\phi)})] \nabla \phi) \in H_{0,n}.
\end{equation}
We now consider the following problem, named Problem $(\mathscr{P}^n)$, 
\begin{equation}\label{eqn-Galerkin-Modified-stochastic-CHNSEs-n}
\begin{cases}
\d \bu_n= - [ \nu \Stokes_{,n} \bu_n + \bB_{0,n}(\bu_n,\bu_n) - \newK \bar{R}_{0,n}(\phi_n,\phi_n)]\,\d t +   
[\bG_{0,n}(t,\bu_n) + \bSi_0^n(t)\bu_n]\,\d \tildeW,
\\
\d \phi_n=- B_{1,n}(\bu_n,\phi_n)\,\d t -  \Athree \cp_n\,\d t,
   \\
\cp_n(t)=  \Atwo \phi_n(t) + \pi_{1,n} (\psi^\prime(\phi_n(t)) - \avg{\psi^\prime(\phi_n(t))}),\;\; t \in[0,T], \\
(\bu_n(0),\phi_n (0))= (\bu_{0,n},\phi_{0,n}).
\end{cases}
\end{equation}
Put
  \begin{equation*}
    \mathbb{H}_{n} \coloneqq H_{0,n} \times H_{1,n}.
   \end{equation*}
Before we continue with our proof, let us reflect a bit about the above equations.
First of all, the space $\mathbb{H}_{n}$ is finite dimensional and therefore all norms on it are equivalent. In calculations below we will use the norm which is the norm inherited from the space
$\nHB$. Thus, we will use the It\^o Lemma for the Lyapunov functional
\[
\Phi_n: \mathbb{H}_{n} \ni \bX=(\bu,\phi) \mapsto \frac12 \lvert \bu \rvert_{\StokesH}^2  + \newK \left( \frac12 \lvert \phi \rvert_{\newone{H}}^2 
 +\int_{\domO} \psi (\phi(x))\, \d x \right)  \in [0,\infty).
\]
Our unknown is a process
\[
\bX_n\coloneqq (\bu_n, \phi_n): [0,T] \times \Omega \to \mathbb{H}_{n} 
.\]
The drift in that system is a function
    \begin{equation}
       \begin{aligned}
          \bF_n: \mathbb{H}_{n} \ni
           \bX_n \mapsto [\bF_{0,n}(\bX_n) + \bF_{2,n}(\bX_n)] \in \mathbb{H}_{n},
        \end{aligned}
     \end{equation}
where
\begin{equation}
\begin{aligned}
\bF_{0,n}&:  \mathbb{H}_{n} \ni \bX_n \mapsto
\Bigl(- \nu  \Stokes_{,n}\bu_n, - B_{1,n}(\bu_n,\phi_n) - \Athree \left(\Atwo \phi_n + \pi_{1,n} (\psi^\prime(\phi_n) - \avg{\psi^\prime(\phi_n)})\right) \Bigr) \in  \mathbb{H}_{n},
\\
\bF_{2,n}&: \mathbb{H}_{n} \ni
\bX_n \mapsto \Bigl(- \bB_{0,n}(\bu_n,\bu_n) + \newK \bar{R}_{0,n}(\phi_n,\phi_n), 0\Bigr) \in  \mathbb{H}_{n}.
\end{aligned}
\end{equation}     
The diffusion coefficient is
   \begin{equation}\label{Eqn-bG_n}
      \newG_{0,n}: [0,T] \times \mathbb{H}_{n} \ni
     (t,\bX_n) \mapsto \left(\bG_{0,n}(t,\bu_n) + \bSi_0^n(t)\bu_n,0\right) \in \ell^2(\mathbb{H}_{n}).
  \end{equation}
We may now reformulate the system \eqref{eqn-Galerkin-Modified-stochastic-CHNSEs-n} in abstract form
   \begin{equation}\label{eqn-Compact-Galerkin-Modified-stochastic-CHNSEs-n}
     \d \bX_n= \bF_n(\bX_n)\, \d t + \newG_{0,n}(t,\bX_n)\,\d W.
   \end{equation}
The equation \eqref{eqn-Compact-Galerkin-Modified-stochastic-CHNSEs-n} satisfies the conditions of existence from \cite[Theorem 38, p.303]{Protter_2004}, which do require the Lipschitz and 
locally Lipschitz conditions on the coefficients. Hence, there exists a local maximal solution, i.e. an adapted process $\bX_n= (\bu_n,\phi_n)$ in $\mathbb{H}_{n}$, an $[0,T]$-valued stopping time $\tau_n$, and an increasing sequence of $[0,T]$-valued stopping times $(T_{n,R})_{R \in \mathbb{N}}$ such that
\begin{trivlist}
\item[(i)] $T_{n,R} \nearrow \tau_n$ as $R\to \infty$  a.s.,
\item[(ii)] $\bX_n(\cdot \wedge T_{n,R}) \in  C([0,T];\mathbb{H}_{n})$,
\item[(iii)]  with probability $1$, the stopped  process
$\bX_n(t\wedge T_{n,R})$, $t \in[0,T]$, solves \eqref{eqn-Compact-Galerkin-Modified-stochastic-CHNSEs-n}.
\end{trivlist}
Because the coefficients are Lipschitz on balls on the finite dimensional state space $\mathbb{H}_{n}$, it is known that the following holds $\mathbb{P}$-almost surely, 
   \begin{equation}\label{eqn-tau_n-01}
     \mbox{ if } \tau_n(\omega)<T \mbox{ then } \lim_{t \toup \tau_n(\omega)} \Vert \bX_n(t,\omega)\Vert_{\mathbb{H}_{n}}=\infty.    
   \end{equation}
The following Proposition implies that the solution is, in fact, global.
\begin{Proposition}\label{prop-First-propo} 
Suppose that the assumptions in Section \ref{Ass-Abstract formulation} hold. Then $\tau_n=T $ a.s., and  there exists a positive constant $C$ such that for every $n \in \mathbb{N}$, 
\begin{equation}\label{eq-3.23}
\begin{aligned}
&\mathbb{E} \sup_{s \in [0,T]} \vert \bu_n(s) \vert_{\StokesH}^2 +  \delta_0\, \mathbb{E}\Vert \bu_n \Vert_{L^2(0,T;\StokesV)}^2
\\
& \hspace{0.2 truecm} + \newK \mathbb{E} \left[ \sup_{s \in [0,T]} \lvert \phi_n(s) \rvert_{\newone{H}}^2 + \sup_{s \in [0,T]} \Vert \psi(\phi_n(s)) \Vert_{L^1}  + \Vert \cp_n \Vert_{L^2(0,T;\newone{H})}^2 \right]
\leq C.
\end{aligned}
\end{equation}
\end{Proposition}
\begin{remark}\label{rem-prop-First-propo}
Let us emphasize also that the equation for $\phi_n$ does not contain any Wiener process, we are only able to find estimates for $ \lvert \phi_n \rvert_{\newone{H}}^2$ with the expectation, as in inequality \eqref{eq-3.23}.
But below in Proposition \ref{prop-phi-n-estimates} we have found some pathwise estimates for $\phi_n$.
\end{remark}
Let us now present the proof of Proposition \ref{prop-First-propo}.
\begin{proof}[Part 1]
Let us choose and fix a natural number $n$. We begin with a proof that 
        \begin{equation}\label{eqn-tau_n=infty}
          \tau_n=  T \mbox{  a.s. }
         \end{equation}
Suppose by contradiction that \eqref{eqn-tau_n=infty} is not satisfied, i.e. $\tau_n<  T $ with positive probability. Then, without loss of generality, as it is done in many papers, we can assume that $\tau_n<  T $ almost surely.  Let $(\tau_{n,m})_{m\geq 1}$ be an increasing sequence of stopping times defined by
  \begin{equation}\label{eqn-def-tau_{n,m}}
    \tau_{n,m}= \inf \left\{t \in [0,\tau_n): \lvert \bu_n(t) \rvert_{\StokesH}^2 + \lvert \phi_n(t) \rvert_{\newone{H}}^2 + \Vert \bu_n \Vert_{L^2(0,t;\StokesV)}^2 \geq m^2 \right\}.
   \end{equation}
It follows from \eqref{eqn-tau_n-01} that  almost surely,  $\tau_{n,m}< \tau_n$ and $\tau_{n,m} \toup  \tau_n$ as $m \to \infty$.
\newline   
Next, we fix $t \in [0,T]$. In the sequel, the duality pairing $\duality{\cdot}{\cdot}{\StokesV}{\StokesVp}$ will be denoted by $\duality{\cdot}{\cdot}{}{}$. 
Applying the It\^o formula in the classical version of \cite[Theorem 4.32]{Prato+Zabczyk_2014} to the functional $\Phi_{0,n}:H_{0,n} \ni \bu \mapsto \frac12 \lvert \bu \rvert_{\StokesH}^2$, using the fact that
\begin{equation*}
\duality{\bar{R}_{0,n}(\phi_n,\phi_n)}{\bu_n}{}{} 
= (\pi_{0,n} (\cp_n \nabla \phi_n), \bu_n) 
= (\bu_n \cdot \nabla \phi_n, \cp_n)
\end{equation*}
and \(\duality{\bB_{0,n}(\bu_n,\bu_n)}{\bu_n}{}{}= 0\), we obtain 
\begin{align}\label{Ito-for-u_n}
&\frac12 \lvert \bu_n(t \wedge \tau_{n,m}) \rvert_{\StokesH}^2 - \newK \int_0^{t \wedge \tau_{n,m}} (\bu_n(s) \cdot \nabla \phi_n(s), \cp_n(s))\,\d s \notag
\\
& + \nu \int_0^{t \wedge \tau_{n,m}} \duality{\Stokes_{,n}\bu_n(s)}{\bu_n(s)}{}{}\, \d s 
- \frac12 \int_0^{t \wedge \tau_{n,m}} \Vert \bSi_0^n(s)\bu_n(s) + \bG_{0,n}(s,\bu_n(s))\Vert_{\ell^2(H_{0,n})}^2\,\d s   
 \\
 \nonumber
&\leq \frac{1}{2} \lvert \bu_{0,n} \rvert_{\StokesH}^2 +  \int_0^{t \wedge \tau_{n,m}} ([\bG_{0,n}(s,\bu_n(s)) + \bSi_0^n(s)\bu_n(s)]\d W(s), \bu_n(s)). 
\end{align}
Owing to the Cauchy-Schwarz and Young inequalities, we have for every $\eta>0$,
\begin{equation}\label{Eqn-Hilbert-norm-Sigma^n+G^n}
\begin{aligned}
\frac12 \Vert \bSi_0^n(t)\bu_n + \bG_{0,n}(t,\bu_n)\Vert_{\ell^2(H_{0,n})}^2
\leq \left(\frac12 + \eta \right) \Vert \bSi_0^n(t)\bu_n \Vert_{\ell^2(H_{0,n})}^2 + \left(\frac12 + \frac{1}{4 \eta} \right) \Vert \bG_{0,n}(t,\bu_n)\Vert_{\ell^2(H_{0,n})}^2.
\end{aligned}
\end{equation}
By Assumption $(vi)$, the definition of $\bSi_0^n$, and inequality \eqref{Eqn-coercivity-3} from Section \ref{Ass-Abstract formulation}, we infer that
\begin{equation}
\begin{aligned}
\Vert \bSi_0^n(t)\bu_n \Vert_{\ell^2(H_{0,n})}^2
\leq \Vert \bSi_0(t)\bu_n \Vert_{\ell^2(\StokesH)}^2
\leq \tilde{C}_1 \Vert \bu_n \Vert_{\StokesV}^2.
\end{aligned}
\end{equation}
Hence, the linear growth assumption \eqref{eqn-linear growth} ensures that 
\begin{equation}\label{Eqn-3001}
\begin{aligned}
&\frac12 \Vert \bSi_0^n(t)\bu_n + \bG_{0,n}(t,\bu_n)\Vert_{\ell^2(H_{0,n})}^2
\\
& \leq  \frac12 \Vert \bSi_0^n(t)\bu_n \Vert_{\ell^2(H_{0,n})}^2 + \tilde{C}_1 \eta \Vert \bu_n \Vert_{\StokesV}^2
+ \left(\frac12 + \frac{1}{4 \eta} \right) \left( \tilde{C}_3 \vert \tilde{h}(t)\vert^2 + C_3 \vert \bu_n \vert_{\StokesH}^2\right). 
\end{aligned}
\end{equation}
Now, at a crucial moment, we choose $\eta$ (and fix for the remainder of the paper) small enough so that 
\[
\tilde{C}_1 \eta \leq \frac{\delta_0}{2}.
\]
We also introduce other constants $C_5$ and $\tilde{C}_5$ by 
\[
\frac12 C_5=C_3 \left(\frac12 + \frac{1}{4 \eta} \right) \mbox{ and }
\frac12 \tilde{C}_5=\tilde{C}_3 \left(\frac12 + \frac{1}{4 \eta} \right).
\]
Therefore, by the coercivity assumption \eqref{eqn-coercivity on H_{0,n}} in Lemma \ref{lem-coercivity on H_{0,n}},  we find that 
\begin{align*}
&\nu  \duality{\Stokes_{,n}\bu_n(t)}{\bu_n(t)}{}{} -\frac12 \Vert \bSi_0^n(t)\bu_n(t) + \bG_{0,n}(t,\bu_n(t))\Vert_{\ell^2(H_{0,n})}^2 
\\
&\geq   \nu  \duality{\Stokes_{,n}\bu_n(t)}{\bu_n(t)}{}{} -
\frac12 \Vert \bSi_0^n(t)\bu_n(t) \Vert_{\ell^2(H_{0,n})}^2 - \tilde{C}_1 \eta \Vert \bu_n(t) \Vert_{\StokesV}^2
- \frac12 \left( \tilde{C}_5 \vert \tilde{h}(t)\vert^2 + C_5 \vert \bu_n(t) \vert_{\StokesH}^2\right)  
\\
&\geq \frac{\delta_0}{2} \Vert \bu_n(t) \Vert_{\StokesV}^2 - \frac12 \left( \tilde{C}_5 \vert \tilde{h}(t) \vert^2 + C_5 \vert \bu_n(t) \vert_{\StokesH}^2\right). 
\end{align*}
Combining the above with \eqref{Ito-for-u_n}, we deduce the following inequality:
\begin{equation}\label{Ito-for-u_n-2}
\begin{aligned}
& \lvert \bu_n(t \wedge \tau_{n,m}) \rvert_{\StokesH}^2 + \delta_0\, \Vert \bu_n \Vert^2_{L^2(0,t \wedge \tau_{n,m};\StokesV)} 
- 2\newK\, \int_0^{t \wedge \tau_{n,m}} (\bu_n(s) \cdot \nabla \phi_n(s), \cp_n(s))\,\d s 
\\
&\leq \lvert \bu_{0,n} \rvert_{\StokesH}^2 +  2 \int_0^{t \wedge \tau_{n,m}} ([\bG_{0,n}(s,\bu_n(s)) + \bSi_0^n(s)\bu_n(s)]\d W(s), \bu_n(s))
  \\
&\qquad + C_5\, \Vert \bu_n \Vert_{L^2(0,t \wedge \tau_{n,m};\StokesH)}^2+  \tilde{C}_5\, \Vert \tilde{h} \Vert_{L^2(0,t \wedge \tau_{n,m};\mathbb{R})}^2.  
\end{aligned}
\end{equation}
Next, from equation \eqref{eqn-Compact-Galerkin-Modified-stochastic-CHNSEs-n}, we infer that $\phi_n$ solves the following equation:
   \begin{equation}\label{eqn-Galerkin-Modified-stochastic-CHNSEs-n-1}
     \d \phi_n + \newone{\mathscr{A}} \phi_n\,\d t= - B_{1,n}(\bu_n,\phi_n)\,\d t  - \Athree (\pi_{1,n}(\psi^\prime(\phi_n) - \avg{\psi^\prime(\phi_n)}))\,\d t. 
   \end{equation}
Consider the following maps, which will be useful for the subsequent analysis. 
\[
\newzero{a}: \newone{V} \times \newone{V} \ni (\phi,\psi) \mapsto ((-\Delta \phi),(-\Delta \psi))_{\newone{H}}
\]
and
\[
\newzero{b}: H^1(\domO) \times H^1(\domO) \ni (\phi,\psi) \mapsto (\nabla \phi,\nabla \psi)_{L^2}.
\]
Put, for the sake of brevity,
\begin{equation}\label{eqn-def-g_n}
f(\phi_n)\coloneqq \psi^\prime(\phi_n) -  \avg{\psi^\prime(\phi_n)} \mbox{ and }
g_n\coloneqq \phi_n +  \pi_{1,n}(\Aone^{-1}f(\phi_n)). 
\end{equation}
Observe also that
\begin{align}\label{eqn-Delta g_n}
\Aone g_n
= \Aone \phi_n + \pi_{1,n}f(\phi_n)
= \Atwo \phi_n + \pi_{1,n} f(\phi_n)
= \cp_n. 
\end{align}
From equation \eqref{eqn-Galerkin-Modified-stochastic-CHNSEs-n-1} and Proposition 4.5 in \cite{Lions+Magenes_1972_vol-1}, 
we infer that a.e. in $[0,T]$,
\begin{equation}\label{eqn-Galerkin-Modified-stochastic-CHNSEs-n-3-a}
\begin{aligned}
&\duality{\d \phi_n}{g_n}{\newone{V}}{\newonep{V}} + \newzero{a}(\phi_n,g_n)\,\d t
= -\duality{B_{1,n}(\bu_n,\phi_n)}{g_n}{\newone{V}}{\newonep{V}}\,\d t - \duality{\Athree (\pi_{1,n}f(\phi_n))}{g_n}{\newone{V}}{\newonep{V}}\,\d t 
  \\
&= - (B_{1,n}(\bu_n,\phi_n),g_n)_{\newone{H}}\,\d t - (\Athree (\pi_{1,n}f(\phi_n)),g_n)_{\newone{H}}\,\d t,
\end{aligned}
\end{equation}
where we omitted $(t)$. By Proposition \ref{prop-B_1-b_1} and \eqref{eqn-Delta g_n},  we deduce that
\begin{align*}
&(B_{1,n}(\bu_n,\phi_n),g_n)_{\newone{H}}
=(B_{1}(\bu_n,\phi_n), \pi_{1,n} g_n)_{\newone{H}}
=(B_{1}(\bu_n,\phi_n), g_n)_{\newone{H}}
\\
&=b_1(\bu_n,\phi_n,g_n)
= \int_{\domO} \bu_n(x)\cdot \nabla \phi_n(x)\, \Aone g_n(x)\,\d x
\\
&= \int_{\domO} \bu_n(x)\cdot \nabla \phi_n(x)\,\cp_n(x)\,\d x 
=(\bu_n\cdot \nabla \phi_n,\cp_n)_{L^2} \coloneqq (\bu_n\cdot \nabla \phi_n,\cp_n).
\end{align*}
Using integration by parts, we obtain the following auxiliary result.
\begin{lemma}\label{lem-IP} If $\phi_{1,n},\, \psi_{1,n} \in H_{1,n}$ then 
\begin{equation}
(\Athree \phi_{1,n}, \psi_{1,n}) 
=(\nabla\phi_{1,n},\nabla \psi_{1,n}).
\end{equation}
\end{lemma}
This follows from the Stokes Theorem, because  every element $\phi$ of $H_{1,n}$ satisfies the following  boundary conditions: 
$\partial_{\bn} \phi=0$.
\newline
Now, by identity \eqref{eqn-Delta g_n} together with Lemma \ref{lem-IP}, because e.g. $g_n \in H_{1,n}$,  we infer that,
\begin{align*}
&(\Athree (\pi_{1,n}f(\phi_n)),g_n)_{\newone{H}}
=(\Athree (\pi_{1,n}f(\phi_n)),\Aone g_n)
\\
&=(\Athree (\pi_{1,n}f(\phi_n)),\cp_n)
= (\nabla(\pi_{1,n}f(\phi_n)),\nabla \cp_n).
\end{align*}
Following the steps of Remark 4.6 from Section 4.7.5 in \cite{Lions+Magenes_1972_vol-1}, using \eqref{eqn-Delta g_n} and since the mean of the process $\cp_n$ is zero, i.e. $\langle\cp_n \rangle=0$, we find that
\begin{align*}
&\newzero{a}(\phi_n,g_n)
= \newzero{b}(-\Delta \phi_n,\cp_n)
=\int_{\domO} \nabla (-\Delta \phi_n) \cdot \nabla \cp_n\,\d x
\\
&= \lvert \nabla \cp_n \rvert_{\mathbb{L}^2}^2 - (\nabla(\pi_{1,n} f(\phi_n)),\nabla \cp_n)
= \lvert \cp_n \rvert_{\newone{H}}^2 - (\nabla(\pi_{1,n} f(\phi_n)),\nabla \cp_n).
\end{align*}
Regarding the first term on the LHS of equation \eqref{eqn-Galerkin-Modified-stochastic-CHNSEs-n-3-a}, we have 
\begin{align*}
&\duality{\d \phi_n}{g_n}{\newone{V}}{\newonep{V}}
=(\d \phi_n,g_n)_{\newone{H}}
=(\d \phi_n,\Aone g_n)
=(\d \phi_n,\cp_n)
\\
&
=(\d \phi_n,\Atwo \phi_n) + (\d  \phi_n, \pi_{1,n}f(\phi_n)) 
= (\d \phi_n,\Atwo \phi_n) + (\d  \phi_n, \psi^\prime(\phi_n) -  \avg{\psi^\prime(\phi_n)}) 
\\
&= \frac12 \d \lvert \phi_n \rvert_{\newone{H}}^2 + \d \int_{\domO} \psi(\phi_n(x))\,\d x - \int_{\domO}\psi^\prime(\phi_n(x))\,\d x\cdot \d \langle \phi_n \rangle
\\
&= \frac12 \d \lvert \phi_n \rvert_{\newone{H}}^2 + \d \int_{\domO} \psi(\phi_n(x))\,\d x
= \frac12 \d \lvert \phi_n \rvert_{\newone{H}}^2 + \d \Vert  \psi(\phi_n) \Vert_{L^1},
\end{align*}
where we used the fact that the mean of $\phi_n$ is also zero.
Henceforth, we can rewrite equation \eqref{eqn-Galerkin-Modified-stochastic-CHNSEs-n-3-a} as follows:
\begin{equation}\label{eqn-Galerkin-Modified-stochastic-CHNSEs-n-4}
\frac12\d \lvert \phi_n \rvert_{\newone{H}}^2 +  \d \Vert  \psi(\phi_n) \Vert_{L^1} + \lvert \cp_n \rvert_{\newone{H}}^2\,\d t 
= - (\bu_n\cdot \nabla \phi_n,\cp_n)\,\d t,
\end{equation}
which, in turn, yields that
\begin{equation}\label{Ito-for-phi_n-new}
\begin{aligned}
& \lvert \phi_n(t \wedge \tau_{n,m}) \rvert_{\newone{H}}^2 + 2 \Vert \psi(\phi_n(t \wedge \tau_{n,m})) \Vert_{L^1} + 2 \Vert \cp_n \Vert_{L^2(0,t \wedge \tau_{n,m};\newone{H})}^2 
\\
&= \lvert \phi_{0,n} \rvert_{\newone{H}}^2 + 2 \Vert \psi(\phi_{0,n}) \Vert_{L^1}  - 2 \int_0^{t \wedge \tau_{n,m}} (\bu_n(s)\cdot \nabla \phi_n(s),\cp_n(s))\,\d s.
\end{aligned}
\end{equation}
From the definitions of \(\bu_{0,n}\) and $\phi_{0,n}$, see \eqref{convergence-for-initial-data}, and since \(\bu_0 \in \StokesH\) and \(\phi_0 \in \newone{H}\), we get
  \begin{equation*}
    \lvert \bu_{0,n} \rvert_{\StokesH} \leq \lvert \bu_0 \rvert_{\StokesH}, \quad \lvert \phi_{0,n} \rvert_{\newone{H}} \leq \lvert \phi_0 \rvert_{\newone{H}}.
  \end{equation*}
Moreover, since \(\newone{H} \embed L^4(\domO)\), we infer that for every $n \in \mathbb{N}$,
\begin{align*}
\Vert \psi(\phi_{0,n}) \Vert_{L^1}
\leq \frac{1}{4} ( 1 + \Vert \phi_{0,n} \Vert_{L^4}^4)
\leq C(\domO) (1 + \lvert \phi_{0,n} \rvert_{\newone{H}}^4) 
\leq  C(\domO) (1 + \lvert \phi_0 \rvert_{\newone{H}}^4).
\end{align*}
Now, by summing \eqref{Ito-for-u_n-2} and  \eqref{Ito-for-phi_n-new} (the latter multiplied by $\newK$), we deduce that $\forall m,\,n \in \mathbb{N}$,
\begin{equation}\label{Galerkin-energy-equality}
\begin{aligned}
& \lvert \bu_n(t \wedge \tau_{n,m}) \rvert_{\StokesH}^2 + \newK \lvert \phi_n(t \wedge \tau_{n,m}) \rvert_{\newone{H}}^2 + 2 \newK \Vert \psi(\phi_n(t \wedge \tau_{n,m})) \Vert_{L^1} 
\\
& + \delta_0 \Vert \bu_n \Vert^2_{L^2(0,t \wedge \tau_{n,m};\StokesV)} + 2 \newK \Vert \cp_n \Vert_{L^2(0,t \wedge \tau_{n,m};\newone{H})}^2 
  \\
&\leq C(1 + \lvert \bu_0 \rvert_{\StokesH}^2 + \lvert \phi_0 \rvert_{\newone{H}}^4) +  \tilde{C}_5\, \Vert \tilde{h}\Vert_{L^2(0,t \wedge \tau_{n,m};\mathbb{R})}^2 + C_5\, \Vert \bu_n \Vert_{L^2(0,t \wedge \tau_{n,m};\StokesH)}^2
    \\
&\qquad +  2 \int_0^{t \wedge \tau_{n,m}} (\bu_n(s),[\bG_{0,n}(s,\bu_n(s)) + \bSi_0^n(s)\bu_n(s)]\d W(s)) .  
\end{aligned}
\end{equation}
Hence,
\begin{align*}
& \mathbb{E} \lvert \bu_n(t \wedge \tau_{n,m}) \rvert_{\StokesH}^2 + \newK \mathbb{E} \lvert \phi_n(t \wedge \tau_{n,m}) \rvert_{\newone{H}}^2 + 2 \newK \mathbb{E} \Vert \psi(\phi_n(t \wedge \tau_{n,m})) \Vert_{L^1} 
\\
& +  \delta_0\, \mathbb{E} \Vert \bu_n \Vert^2_{L^2(0,t \wedge \tau_{n,m};\StokesV)} + 2 \newK\, \mathbb{E} \Vert \cp_n \Vert_{L^2(0,t \wedge \tau_{n,m};\newone{H})}^2
  \\
&\leq C(1 + \lvert \bu_0 \rvert_{\StokesH}^2 + \lvert \phi_0 \rvert_{\newone{H}}^4) + \tilde{C}_5\, \mathbb{E} \Vert \tilde{h}\Vert_{L^2(0,t \wedge \tau_{n,m};\mathbb{R})}^2 + C_5\, \mathbb{E} \Vert \bu_n \Vert_{L^2(0,t \wedge \tau_{n,m};\StokesH)}^2.  
\end{align*}
Next, by the Gronwall inequality applied to the function  $t \mapsto \mathbb{E} \lvert \bu_n(t \wedge \tau_{n,m}) \rvert_{\StokesH}^2$,
\begin{equation}\label{eq3.25}
\begin{aligned}
& \mathbb{E} [\lvert \bu_n(t \wedge \tau_{n,m}) \rvert_{\StokesH}^2 + \lvert \phi_n(t \wedge \tau_{n,m}) \rvert_{\newone{H}}^2 +  \Vert \psi(\phi_n(t \wedge \tau_{n,m})) \Vert_{L^1} \\ 
& + \Vert \bu_n \Vert^2_{L^2(0,t \wedge \tau_{n,m};\StokesV)} + \Vert \cp_n \Vert_{L^2(0,t \wedge \tau_{n,m};\newone{H})}^2] 
\leq C_0,
\end{aligned}
\end{equation}
with the constant $C_0$ being independent of $n$ and $\tau_{n,m}$. \newline
Hereafter, let 
   \begin{equation*}
     X_{n}(t)= \lvert \bu_n(t) \rvert_{\StokesH}^2 + \lvert \phi_n(t) \rvert_{\newone{H}}^2 + \Vert \bu_n \Vert_{L^2(0,t;\StokesV)}^2\; \mbox{ for } 0 \leq t < \tau_n.
   \end{equation*}
Note that by definition of $\tau_{n,m}$, see \eqref{eqn-def-tau_{n,m}}, it follows that $\mathbb{E} X_{n}(\tau_{n,m}) \geq m^2$ and 
 $\mathbb{E}1_{\{\tau_{n,m}< t\}}= \mathbb{P}(\tau_{n,m}< t)$ for every $t \in [0,T]$. 
Next, let us temporarily choose and fix $t \in [0,T]$. So,
\begin{align*}
&\mathbb{E} X_{n}(t \wedge \tau_{n,m})
=\mathbb{E} [X_{n}(t \wedge \tau_{n,m})\, 1_{\{\tau_{n,m}< t\}}] + \mathbb{E} [X_{n}(t \wedge \tau_{n,m})\,1_{\{\tau_{n,m}\geq t\}}]
\\
&=\mathbb{E} [X_{n}(\tau_{n,m})\, 1_{\{\tau_{n,m}< t\}}] + \mathbb{E} [X_{n}(t)\,1_{\{\tau_{n,m}\geq t\}}] 
\\
&> \mathbb{E}[X_{n}(\tau_{n,m})\, 1_{\{\tau_{n,m}< t\}}] 
> m^2 \mathbb{E} 1_{\{\tau_{n,m}<t\}} 
= m^2 \mathbb{P}(\tau_{n,m}< t),
\end{align*}
which, in turn, using \eqref{eq3.25}, leads us to
  \begin{equation*}
    \mathbb{P}(\tau_{n,m}< t)< m^{-2} \mathbb{E} X_{n}(t \wedge \tau_{n,m}) \leq m^{-2} C_0.
   \end{equation*}
Consequently,
\begin{equation}
\lim_{m \to \infty} \mathbb{P}(\tau_{n,m}< t)=0.
\end{equation}
This means $\tau_{n,m} \to T$ in probability as $m \to \infty$. 
Since, as we have already proven, $\tau_{n,m} \to \tau_n$ almost surely as $m \to \infty$, we deduce that  
$\tau_n=T$ almost surely. This contradiction concludes the proof  of \eqref{eqn-tau_n=infty}.
\end{proof}
\begin{proof}[Part 2]
It follows that inequality \eqref{eq3.25} can be written without the stopping times as follows:
\begin{equation}\label{eq3.25-1}
\begin{aligned}
\mathbb{E} [\lvert \bu_n(t) \rvert_{\StokesH}^2 + \lvert \phi_n(t) \rvert_{\newone{H}}^2 +  \Vert \psi(\phi_n(t )) \Vert_{L^1} 
+  \Vert \bu_n \Vert^2_{L^2(0,t;\StokesV)} + \Vert \cp_n \Vert_{L^2(0,t;\newone{H})}^2] \leq C_0, \;\; t \in [0,T].
\end{aligned}
\end{equation}
Now, we take $\mathbb{E} \left[\sup_{s \in[0,t ]} |\cdot| \right]$ in \eqref{Galerkin-energy-equality}. 
By the Burkholder-Davis-Gundy inequality and Example \ref{example:HS}, we deduce that for every $n \in \mathbb{N}$,
\begin{align*}
&\mathbb{E} \sup_{s \in[0,t]} \left\lvert \int_0^{s} (\bSi_0^n(r)\bu_n(r))\d W(r), \bu_n(r)) \right\rvert
\leq C \mathbb{E} \left( \int_0^t \Vert \bSi_0^n(s)\bu_n(s) \Vert_{\ell^2(H_{0,n})}^2 \lvert \bu_n(s) \rvert_{\StokesH}^2\,\d s \right)^{\frac12}
\\
&\leq C \mathbb{E}\left[ \left(\sup_{s \in[0,t]} \lvert \bu_n(s) \rvert_{\StokesH}^2 \right)^{\frac12} \left(\int_0^t \Vert \bSi_0^n(s)\bu_n(s) \Vert_{\ell^2(H_{0,n})}^2\,\d s \right)^{\frac12} \right]
  \\
&\leq \frac14 \mathbb{E} \sup_{s \in[0,t]} \lvert \bu_n(s) \rvert_{\StokesH}^2 + C \mathbb{E} \int_0^t \Vert \bSi_0^n(s)\bu_n(s) \Vert_{\ell^2(H_{0,n})}^2\,\d s 
\leq \frac14 \mathbb{E} \sup_{s \in[0,t]} \lvert \bu_n(s) \rvert_{\StokesH}^2 + C \mathbb{E} \Vert \bu_n \Vert_{L^2(0,t;\StokesV)}^2.
\end{align*}
Analogously, by Assumption $\mbox{(v)-a)}$ from Section \ref{Ass-Abstract formulation}, we have for every $n \in \mathbb{N}$,
\begin{align*}
&\mathbb{E} \sup_{s \in[0,t]} \left\lvert \int_0^{s} (\bG_{0,n}(r,\bu_n(r))\d W(r), \bu_n(r)) \right\rvert
\leq C \mathbb{E} \left( \int_0^t \Vert \bG_{0,n}(s,\bu_n(s)) \Vert_{\ell^2(H_{0,n})}^2 \lvert \bu_n(s) \rvert_{\StokesH}^2\,\d s \right)^{\frac12}
  \\
&\leq \frac14 \mathbb{E} \sup_{s \in[0,t]} \lvert \bu_n(s) \rvert_{\StokesH}^2 + C \mathbb{E} \int_0^t \Vert \bG_{0,n}(s,\bu_n(s)) \Vert_{\ell^2(H_{0,n})}^2\,\d s 
    \\
&\leq \frac14 \mathbb{E} \sup_{s \in[0,t]} \lvert \bu_n(s) \rvert_{\StokesH}^2 + C \Vert \tilde{h} \Vert_{L^2(0,t;\mathbb{R})}^2 + C \mathbb{E} \Vert \bu_n \Vert_{L^2(0,t;\StokesH)}^2.
\end{align*}
Taking $\mathbb{E} \left[\sup_{s \in[0,t]} \lvert \cdot \rvert \right]$ in \eqref{Galerkin-energy-equality}, using the above estimates, 
we deduce that for every $n \in \mathbb{N}$,
\begin{align}\label{Galerkin-energy-equality-a}
&\frac12 \mathbb{E} \sup_{s \in [0,t]} \lvert \bu_n(s) \rvert_{\StokesH}^2  + \newK\, \mathbb{E} \sup_{s \in [0,t]} \lvert \phi_n(s) \rvert_{\newone{H}}^2 + 2 \newK\, \mathbb{E} \sup_{s \in [0,t]} \Vert \psi(\phi_n(s)) \Vert_{L^1} \notag
\\
& + 2 \newK\, \mathbb{E} \Vert \cp_n \Vert_{L^2(0,t;\newone{H})}^2 + \delta_0 \mathbb{E} \Vert \bu_n \Vert_{L^2(0,t;\StokesV)}^2
\\
&\leq C(1 + \lvert \bu_{0} \rvert_{\StokesH}^2 + \lvert \phi_{0} \rvert_{\newone{H}}^4) + C [\Vert \tilde{h} \Vert_{L^2(0,t;\mathbb{R})}^2
+  \mathbb{E} \Vert \bu_n \Vert_{L^2(0,t;\StokesV)}^2 + \mathbb{E} \vert \bu_n \vert_{L^2(0,t;\StokesH)}^2], \notag
\end{align}
where the constant $C$ may depend on $\domO,\,C_3,\, \delta_0,\,\tilde{C}_5,\, C_5$, and $\newK$.
From the estimate \eqref{eq3.25-1}, it follows that for every $n \in \mathbb{N}$,
\begin{align*}
&\frac12 \mathbb{E} \sup_{s \in [0,t]} \lvert \bu_n(s) \rvert_{\StokesH}^2  + \newK\, \mathbb{E} \sup_{s \in [0,t]} \lvert \phi_n(s) \rvert_{\newone{H}}^2 + 2 \newK \mathbb{E} \sup_{s \in [0,t]} \Vert \psi(\phi_n(s)) \Vert_{L^1} 
\\
&\leq C [1 + \lvert \bu_{0} \rvert_{\StokesH}^2 + \lvert \phi_0 \rvert_{\newone{H}}^4] + C \Vert \tilde{h} \Vert_{L^2(0,t;\mathbb{R})}^2 + C \int_0^t  \mathbb{E} \sup_{0 \leq s \leq r} \lvert \bu_n(s) \rvert_{\StokesH}^2\, \d r.
\end{align*}
By the Gronwall inequality, we infer that there exists a positive constant 
$C_{1,0}= C_{1,0}(\domO,C_3,\delta_0,\tilde{C}_5,C_5,\newK,t)$, independent of $n$ such that 
\begin{equation}\label{eq3.35}
\mathbb{E} \left[\sup_{s \in [0,t]} \lvert \bu_n(s) \lvert_{\StokesH}^2 + \sup_{s \in [0,t]} \newK \lvert \phi_n(s) \rvert_{\newone{H}}^2 + \newK \sup_{s \in [0,t]} \Vert \psi(\phi_n(s)) \Vert_{L^1} \right]
\leq C_{1,0}.
\end{equation}
This completes the proof of Proposition \ref{prop-First-propo}.
\end{proof}
Now, we state and prove the following crucial estimate.
\begin{Proposition}\label{prop-phi-n-estimates}
There exists a positive constant $C= C(\domO)>0$ such that  for every $n \in \mathbb{N}$ and every  $p \geq 2$,  the 
 solution $\bX_n$  to \eqref{eqn-Compact-Galerkin-Modified-stochastic-CHNSEs-n}, satisfies, 
 $\mathbb{P}$-a.s. $\omega \in \Omega$, the following inequalities
\begin{equation}\label{eq-estimate-for-phi-c}
\lvert \phi_n(t) \rvert_{L^2}^p  
+ \frac p2 \int_0^t \lvert \phi_n(s) \rvert_{L^2}^{p-2} \lvert \Atwo \phi_n(s) \rvert_{L^2}^2\,\d s 
\leq C\, e^{\frac{p}{2} t} \lvert \phi_0 \rvert_{\newone{H}}^p, \;\; t \in [0,T],
\end{equation}
and
   \begin{equation}\label{eq-estimate-for-phi-d-a}
      \int_0^t \lvert \phi_n(s)\rvert_{\newone{H}}^4 \, \d s
      \leq C\, e^{2 t} \lvert \phi_0 \rvert_{\newone{H}}^4, \;\; t \in [0,T]. 
   \end{equation}
\end{Proposition}
\begin{proof}[Proof of Proposition \ref{prop-phi-n-estimates}]
Let $p \geq 2$ be fixed. Arguing as in \eqref{eqn-Galerkin-Modified-stochastic-CHNSEs-n-3-a}, we infer that
    \begin{align*}
     (\d \phi_n,\Aone g_n) + \newzero{b}(-\Delta \phi_n,\Aone  g_n)\,\d t
     =  - (B_{1,n}(\bu_n,\phi_n),g_n)_{\newone{H}}\,\d t - (\Athree (\pi_{1,n}f(\phi_n)),g_n)_{\newone{H}}\,\d t,
   \end{align*}
where $g_n= \Aone^{-1} [\lvert \phi_n \rvert_{L^2}^{p-2} \phi_n]$, $n\in \mathbb{N}$. Now, since $\Aone g_n= \lvert \phi_n \rvert_{L^2}^{p-2} \phi_n$, we infer that
\begin{align*}
&\newzero{b}(-\Delta \phi_n,\lvert \phi_n \rvert_{L^2}^{p-2} \phi_n)
= \lvert \phi_n \rvert_{L^2}^{p-2} \newzero{b}(-\Delta \phi_n,\phi_n)
=  \lvert \phi_n \rvert_{L^2}^{p-2} \int_{\domO} \nabla (-\Delta \phi_n(x))\cdot \nabla \phi_n(x)\,\d x
\\
&= \lvert \phi_n \rvert_{L^2}^{p-2} \int_{\domO} (-\Delta \phi_n(x))\cdot (-\Delta \phi_n(x))\,\d x
= \lvert \phi_n \rvert_{L^2}^{p-2} \lvert \Atwo  \phi_n \rvert_{L^2}^2,
\end{align*}

\begin{align*}
&(\Athree (\pi_{1,n}f(\phi_n)),g_n)_{\newone{H}}
=(\Athree (\pi_{1,n}f(\phi_n)), \Aone g_n)
\\
&= \lvert \phi_n \rvert_{L^2}^{p-2} (\Athree (\pi_{1,n}f(\phi_n)),\phi_n)
=\lvert \phi_n \rvert_{L^2}^{p-2} (\Athree f(\phi_n),\phi_n)
\\
&= \lvert \phi_n \rvert_{L^2}^{p-2} ((-\Delta) \psi^\prime(\phi_n),\phi_n)
=\lvert \phi_n \rvert_{L^2}^{p-2} \int_{\domO} \psi^{\prime \prime}(\phi_n(x)) \lvert \nabla \phi_n(x) \rvert^2\,\d x,
\end{align*}
and
\begin{align*}
&(B_{1,n}(\bu_n,\phi_n),g_n)_{\newone{H}}
=b_1(\bu_n,\phi_n,g_n)
= \int_{\domO} \bu_n(x)\cdot \nabla \phi_n(x) (\Aone g_n)(x)\,\d x
\\
&= \lvert \phi_n \rvert_{L^2}^{p-2} \int_{\domO} \bu_n(x)\cdot \nabla \phi_n(x)\,\phi_n(x)\,\d x
=\frac12 \lvert \phi_n \rvert_{L^2}^{p-2} \int_{\domO} \bu_n(x)\cdot \nabla (\phi_n^2)(x)\,\d x
=0.
\end{align*}
Combining the above equalities together with the equality \eqref{eqn-Psi''}, we deduce that
\begin{equation}\label{eq-estimate-for-phi-a}
\begin{aligned}
&\frac1p \d \lvert \phi_n \rvert_{L^2}^p + \lvert \phi_n \rvert_{L^2}^{p-2} \lvert \Atwo \phi_n \rvert_{L^2}^2\,\d t
= - \lvert \phi_n \rvert_{L^2}^{p-2} \int_{\domO} \psi^{\prime \prime}(\phi_n(x)) \lvert \nabla \phi_n(x) \rvert^2\,\d x\,\d t
 \\
&\leq  \lvert \phi_n \rvert_{L^2}^{p-2} \lvert \phi_n \rvert_{\newone{H}}^2\,\d t
= \lvert \phi_n \rvert_{L^2}^{p-2}(\Atwo  \phi_n,\phi_n)\,\d t
\leq \lvert \phi_n \rvert_{L^2}^{p-2} \lvert \Atwo  \phi_n\rvert_{L^2} \lvert \phi_n \rvert_{L^2}\,\d t
   \\
&\leq \frac12 \lvert \phi_n \rvert_{L^2}^{p-2} \lvert \Atwo \phi_n \rvert_{L^2}^2\,\d t + \frac12 \lvert \phi_n \rvert_{L^2}^{p}\,\d t.
\end{aligned}
\end{equation}
This implies that
\begin{equation}\label{eq-estimate-for-phi-aa}
\frac{\d}{\d t} \lvert \phi_n(t) \rvert_{L^2}^p \leq \frac p2 \lvert \phi_n(t) \rvert_{L^2}^p,\;\; t \in [0,T].
\end{equation}
An application of the Gronwall Lemma yields that for every $n \in \mathbb{N}$,
 \begin{equation}\label{eq-estimate-for-phi-b}
   \lvert \phi_n(t) \rvert_{L^2}^p \leq e^{\frac{p}{2}t} \lvert \phi_{0,n} \rvert_{L^2}^p 
   \leq C\,e^{\frac{p}{2}t} \lvert \phi_0 \rvert_{\newone{H}}^p,\;\; t \in [0,T],
 \end{equation}
 where $C=C(\domO)$.
Substituting \eqref{eq-estimate-for-phi-b} into the integrated equation \eqref{eq-estimate-for-phi-a}, we deduce \eqref{eq-estimate-for-phi-c}.
\newline
Let us consider the proof of \eqref{eq-estimate-for-phi-d-a}.
By \eqref{eq-estimate-for-phi-c}, we infer that for every $n \in \mathbb{N}$,
\begin{equation}
\int_0^t \lvert \phi_n(s) \rvert_{\newone{H}}^4\, \d s
\leq \int_0^t \lvert \phi_n(s) \rvert_{L^2}^2 \lvert \Atwo \phi_n(s) \rvert_{L^2}^2\, \d s
\leq C\, e^{2 t} \lvert \phi_0 \rvert_{\newone{H}}^4,\;\; t \in [0,T].
\end{equation}
From the latter inequality, we deduce \eqref{eq-estimate-for-phi-d-a}, completing the proof.
\end{proof}
\begin{remark}
Suppose that $\phi \in L^{5/3}(0,T;H^2(\domO)) \cap L^{\infty}(0,T;H^1(\domO))$. Then, there exists a generic constant $C=C(T,\domO)>0$ such that 
  \begin{equation}\label{Eq-Psi-second-gradient-phi-b}
     \Vert \psi^{\bis}(\phi) \nabla \phi \Vert_{L^{4/3}(0,T;\mathbb{L}^2(\domO))} 
      \leq C (\Vert \phi \Vert_{L^{4/3}(0,T;H^1(\domO))} + \Vert \phi \Vert^{7/4}_{L^{\infty}(0,T;H^1(\domO))} \Vert \phi \Vert^{5/4}_{L^{5/3}(0,T;H^2(\domO))}).
    \end{equation}
Indeed, by the Agmon inequality \eqref{eq-Agmon's-inequalities},
\begin{align*}
&\vert \psi^{\bis}(\phi) \nabla \phi \rvert_{\mathbb{L}^2}
\leq \vert \nabla \phi\rvert_{\mathbb{L}^2} + 3 \vert \phi^2 \nabla \phi \rvert_{\mathbb{L}^2} 
\leq \vert \nabla \phi \rvert_{\mathbb{L}^2} + \frac{3}{2} \Vert \phi\Vert_{L^\infty} \vert \nabla (\phi^2) \rvert_{\mathbb{L}^2} 
\\
&\leq \Vert \phi\Vert_{H^1} + C \Vert \phi \Vert_{H^1}^{1/2} \Vert \phi \Vert_{H^2}^{1/2} \vert \nabla (\phi^2)\rvert_{\mathbb{L}^2} 
\leq \Vert \phi \Vert_{H^1} + C \Vert \phi \Vert_{H^1}^{7/4} \Vert \phi \Vert_{H^2}^{5/4},\;\; \phi \in H^2(\domO),
\end{align*}
for some constant $C=C(\domO)$. From the above inequality, we deduce the estimate \eqref{Eq-Psi-second-gradient-phi-b}.
\end{remark}
Using the previous results, we can deduce that
\begin{corollary}\label{cor-Propositions 5.1 and 5.3}
Let the assumptions of Proposition \ref{prop-First-propo} be satisfied. 
There exists a positive constant $C= C(\domO)>0$ such that  for every $n \in \mathbb{N}$ and every  $p \geq 2$,  the 
 solution $\bX_n$  to \eqref{eqn-Compact-Galerkin-Modified-stochastic-CHNSEs-n}, satisfies, 
 $\mathbb{P}$-a.s. $\omega \in \Omega$, the following inequalities: 
\begin{align}\label{eq-estimate-for-phi}
\Vert \phi_n \Vert_{C([0,T];L^2(\domO))}^p &\leq C\, e^{\frac{p}{2} T} \lvert \phi_0 \rvert_{\newone{H}}^p,
\\
\label{eq-A1-estimate-for-phi}
\int_0^T \lvert \phi_n(s) \rvert_{L^2}^{p-2} \lvert \Atwo \phi_n(s) \rvert_{L^2}^2\,\d s 
&\leq C\, e^{\frac{p}{2} T} \lvert \phi_0 \rvert_{\newone{H}}^p,
\\
\label{eq-estimate-for-phi-e}
\int_0^T \lvert \phi_n \rvert_{L^2}^{p-2} \Vert \phi_n \Vert_{H^2}^2 \, \d s 
&\leq C\, e^{p T} \lvert \phi_0 \rvert_{\newone{H}}^p.
\end{align}
\end{corollary}
\begin{proof}
Estimates \eqref{eq-estimate-for-phi} and \eqref{eq-A1-estimate-for-phi} are direct consequences of \eqref{eq-estimate-for-phi-c} and \eqref{eq-estimate-for-phi-d-a}.
\newline
By the smoothness of $\domO$ the standard elliptic theory for second order operators, see \cite[Theorem 8.8]{Gilbarg+Trudinger_1983}, there exists $M_1>0$ such that for every $n \in \mathbb{N}$,
   \begin{equation}\label{eqn-phi_n-H-2-norm}
      2 \Vert \phi_n \Vert_{H^2}^2 \leq M_1 \left(\lvert \Atwo \phi_n \rvert_{L^2}^2 + \lvert \phi_n \rvert_{L^2}^2 \right).
    \end{equation}
Therefore, from \eqref{eq-estimate-for-phi-b} and \eqref{eq-estimate-for-phi-c}, we obtain for all $t \in [0,T]$,
\begin{equation}\label{eq-3.63}
\int_0^t \lvert \phi_n \rvert_{L^2}^{p-2} \Vert \phi_n \Vert_{H^2}^2 \, \d s 
\leq \frac{M_1}{2} \int_0^t \left(\lvert \phi_n \rvert_{L^2}^{p-2} \lvert \Atwo \phi_n \rvert_{L^2}^2 + \lvert \phi_n \rvert_{L^2}^p \right)\d s
\leq \frac{C\,M_1 e^{\frac{p}{2} t}}{p} \lvert \phi_0 \rvert_{\newone{H}}^p.
\end{equation}
Thus, \eqref{eq-estimate-for-phi-e} follows, completing the proof of Corollary \ref{cor-Propositions 5.1 and 5.3}.
\end{proof}
Before we proceed further, we state and prove the following auxiliary result.
\begin{lemma}\label{lem-phi_n to mu_n}
Define a nonlinear map
\begin{equation}
\label{eqn-mu_n}   
H_{1,n} \ni \phi_n \mapsto \cp_n\coloneqq \pi_{1,n}( \Atwo \phi_n + \psi^\prime(\phi_n) - \avg{\psi^\prime(\phi_n)})  \in H_{1,n}.
\end{equation}
Then, there exists a constant $C>0$ such that, for every $n$ and all $\phi_n \in H_{1,n}$, 
\begin{align}
\label{eq-A1-estimate-for-phi-a}
\lvert \Atwo \phi_n \rvert_{L^2}^3 &\leq  C (\lvert \cp_n \rvert_{\newone{H}}^{2} \lvert \phi_n \rvert_{L^2} + \lvert \phi_n \rvert_{L^2}^3),
\\
\label{eq-A1-3-over-2-estimate-for-phi}
\lvert \Atwo^{3/2} \phi_n \rvert_{\mathbb{L}^2}^{3/2} &\leq C (\lvert \cp_n \rvert_{\newone{H}}^{3/2} + \lvert \phi_n \rvert_{L^2}^{3/4} \lvert \Atwo \phi_n \rvert_{L^2}^{3/4} +  \lvert \phi_n \rvert_{L^2}^{3/2} \Vert \phi_n \Vert_{H^2}^{3}),\; \mbox{ for d=2 or d=3},
\\
\label{eq-A1-3-over-2-estimate-for-phi-1}
\lvert \Atwo^{3/2} \phi_n \rvert_{\mathbb{L}^2}^2
&\leq C(\lvert \cp_n \rvert_{\newone{H}}^2 + \lvert \phi_n \rvert_{\newone{H}}^2 +  \lvert \phi_n \rvert_{L^2}^2 \lvert \phi_n \rvert_{\newone{H}}^2 \Vert \phi_n \Vert_{H^2}^2), \mbox{ for d=2}.
\end{align}
\end{lemma}  
\begin{proof}[Proof of Lemma \ref{lem-phi_n to mu_n}]
In this proof, the generic constant $C>0$ is independent of $n$. \\
Notice that if $\phi_n \in H_{1,n}$, then $\cp_n \in H_{1,n}$. 
Now, we multiply \eqref{eqn-mu_n}  by $\Atwo \phi_n$ and integrate the resulting equality over $\domO$. Applying integration by parts and using the boundary condition for $\phi_n$  in conjunction with the inequality \eqref{eqn-Psi''} in Lemma \ref{eqn-Lemma-Psi'}, a series of calculations yields the following inequality:
\begin{align*}
&\lvert \Atwo \phi_n \rvert_{L^2}^2
\leq (\nabla \cp_n, \nabla \phi_n) + \lvert \phi_n \rvert_{L^2} \lvert \Atwo \phi_n \rvert_{L^2}
\\ 
&\leq \lvert \cp_n \rvert_{\newone{H}} \lvert \phi_n \rvert_{\newone{H}}  + \lvert \phi_n \rvert_{L^2} \lvert \Atwo \phi_n \rvert_{L^2} 
\leq \lvert \cp_n \rvert_{\newone{H}} \lvert \phi_n \rvert_{L^2}^{1/2} \lvert \Atwo \phi_n \rvert_{L^2}^{1/2} + \lvert \phi_n \rvert_{L^2} \lvert \Atwo \phi_n \rvert_{L^2}
\end{align*}
The latter implies \eqref{eq-A1-estimate-for-phi-a}.
\newline
It remains to prove \eqref{eq-A1-3-over-2-estimate-for-phi}.
We use \eqref{Gagliardo-Nirenberg-inequality}, which gives for every $\phi_n \in H_{1,n}$,
\begin{equation}\label{Ladyzhenskaya inequality-a}
\begin{aligned}
\Vert \nabla \phi_n \Vert_{\mathbb{L}^4} &\leq C \lvert \phi_n \rvert_{\newone{H}} ^{1/2} \Vert \phi_n \Vert_{H^2}^{1/2},\; \; \mbox{ if } \; \; d=2, 
\\
\Vert \nabla \phi_n \Vert_{\mathbb{L}^4} &\leq C \lvert \phi_n \rvert_{\newone{H}} ^{1/4} \Vert \phi_n \Vert_{H^2}^{3/4},\; \; \mbox{ if } \; \; d=3,
\end{aligned}
\end{equation}
and from \eqref{eqn-Psi'}  and \eqref{Ladyzhenskaya inequality-a}, we obtain in the case $d=3$ that
\begin{equation}\label{eq-Psi-second-gradient-phi-n}
\begin{aligned}
&\lvert \psi^{\prime\prime}(\phi_n) \nabla \phi_n \rvert_{\mathbb{L}^2}
\leq \lvert \phi_n \rvert_{\newone{H}}  + 3 \lvert \phi_n^2 \nabla \phi_n \rvert_{\mathbb{L}^2}
\leq \lvert \phi_n \rvert_{\newone{H}}  + \frac{3}{2} \Vert \phi_n \Vert_{L^8}^2 \Vert \nabla \phi_n \Vert_{\mathbb{L}^4} 
\\
&\leq \lvert \phi_n \rvert_{\newone{H}}  +  C \lvert \phi_n \rvert_{L^2}^{7/8} \Vert \phi_n \Vert_{H^1}^{1/4} \Vert \phi_n \Vert_{H^2}^{15/8} 
\leq \lvert \phi_n \rvert_{\newone{H}}  +  C \lvert \phi_n \rvert_{L^2} \Vert \phi_n \Vert_{H^2}^{2}.
\end{aligned}
\end{equation}
An analogous argument shows that the same inequality holds for $d=2$. However, notice that if $d=2$, we can also use the Agmon inequality \eqref{eq-Agmon's-inequalities}  and estimate \eqref{Ladyzhenskaya inequality-a} to deduce that 
\begin{equation}\label{eq-Psi-second-gradient-phi-n-1}
\begin{aligned}
&\lvert \psi^{\prime\prime}(\phi_n) \nabla \phi_n \rvert_{\mathbb{L}^2}
\leq \lvert \phi_n \rvert_{\newone{H}}  + \frac{3}{2} \Vert \phi_n \Vert_{L^\infty}^2 \lvert \phi_n \rvert_{\newone{H}}  
\leq \lvert \phi_n \rvert_{\newone{H}}  +  C \lvert \phi_n \rvert_{L^2} \lvert \phi_n \rvert_{\newone{H}}  \Vert \phi_n \Vert_{H^2}.
\end{aligned}
\end{equation}
Next, by definition of $\cp_n$, see \eqref{eqn-mu_n}, we infer after some calculations that
\begin{equation}\label{eq-3.65}
\begin{aligned}
\lvert \Atwo^{3/2} \phi_n \rvert_{\mathbb{L}^2}
&\leq  \lvert \Atwo^{1/2} \cp_n \rvert_{\mathbb{L}^2} + \lvert \Aone^{1/2} \psi^\prime(\phi_n) \rvert_{\mathbb{L}^2}
= \lvert \nabla \cp_n \rvert_{\mathbb{L}^2} + \lvert \nabla \psi^\prime(\phi_n) \rvert_{\mathbb{L}^2}
\\
&=\lvert \cp_n \rvert_{\newone{H}} +  \lvert \psi^{\prime \prime}(\phi_n) \nabla \phi_n \rvert_{\mathbb{L}^2}.
\end{aligned}
\end{equation}
From the latter inequality and \eqref{eq-Psi-second-gradient-phi-n}, we obtain 
\begin{align*}
&\lvert \Atwo^{3/2} \phi_n \rvert_{\mathbb{L}^2}^{\frac32} 
\leq C (\lvert \cp_n \rvert_{\newone{H}}^{\frac32} +  \lvert \psi^{\bis}(\phi_n) \nabla \phi_n \rvert_{\mathbb{L}^2}^{\frac32}) 
\leq C (\lvert \cp_n \rvert_{\newone{H}}^{\frac32} +  \lvert \phi_n \rvert_{\newone{H}} ^{\frac32} +  \lvert \phi_n \rvert_{L^2}^{\frac32} \Vert \phi_n \Vert_{H^2}^{3})  \\
&\leq C \lvert \cp_n \rvert_{\newone{H}}^{\frac32}  + \lvert \phi_n \rvert_{L^2}^{\frac34} \lvert \Atwo \phi_n \rvert_{L^2}^{\frac34} +  \lvert \phi_n \rvert_{L^2}^{\frac32} \Vert \phi_n \Vert_{H^2}^{3},
\end{align*}
which holds for $d=2$ or $d=3$, whereas, by \eqref{eq-3.65} and  \eqref{eq-Psi-second-gradient-phi-n-1}, we find that for $d=2$,
\begin{align*}
\lvert \Atwo^{3/2} \phi_n \rvert_{\mathbb{L}^2}^2
\leq C(\lvert \cp_n \rvert_{\newone{H}}^2 + \lvert \phi_n \rvert_{\newone{H}} ^2 +  \lvert \phi_n \rvert_{L^2}^2 \lvert \phi_n \rvert_{\newone{H}} ^2 \Vert \phi_n \Vert_{H^2}^2).
\end{align*}

which, in turn, yields the estimates \eqref{eq-A1-3-over-2-estimate-for-phi} and \eqref{eq-A1-3-over-2-estimate-for-phi-1}.
\end{proof}
We now assume that $\bX_n=(\bu_n,\phi_n)$ is a solution to the problem \eqref{eqn-Compact-Galerkin-Modified-stochastic-CHNSEs-n}, and recall the definition of the process $\cp_n$ in \eqref{eqn-mu_n}. 
By applying Lemma \ref{lem-phi_n to mu_n}, we can derive uniform estimates for the sequence $(\cp_n)_{n \in \mathbb{N}}$ in 
$H^1(\domO)$ and for $(\phi_n)_{n \in \mathbb{N}}$ in
$H^3(\domO)$. 
\begin{proposition}\label{prop-2nd proposition}
Let the assumptions of Proposition \ref{prop-First-propo} be satisfied. Then, for every $T>0$, there exists a constant $C_T$ such that for every $n \in \mathbb{N}$, 
\begin{align}
\label{eqn-improved-estimates} 
\mathbb{E} \int_0^{T} \Vert \cp_n(s) \Vert_{H^1}^{2}\, \d s &\leq C_T,
\\
\label{eqn-improved-H2-estimates}
\mathbb{E} \int_0^{T} \Vert \phi_n(s) \Vert_{H^2}^{3}\, \d s &\leq C_T,
\\
\label{H3-phi-n-estimate}
\mathbb{E} \int_0^T \Vert  \phi_n(s) \Vert_{H^3}^{3/2}\, \d s &\leq  C_T,\; \mbox{ for d=3},
\\
\label{H3-phi-n-estimate-1}
\mathbb{E} \int_0^T \Vert \phi_n(s) \Vert_{H^3}^2\, \d s &\leq  C_T,\; \mbox{ for d=2}.
\end{align} 
\end{proposition}
\begin{proof}
 In what follows, all the constants $C$ and $C_T$ are independent of $n$. Since $\langle \cp_n \rangle=0$, then estimate \eqref{eqn-improved-estimates} follows by using the  generalized Poincar\'{e} inequality \eqref{eqn-H^1-norm-special} and  the estimate \eqref{eq-3.23}. \newline
Recall that $2 \Vert \phi_n \Vert_{H^2}^2 \leq M_1 \left(\lvert \Atwo \phi_n \rvert_{L^2}^2 + \lvert \phi_n \rvert_{L^2}^2 \right)$. 
Therefore, by \eqref{eq-estimate-for-phi-c} and \eqref{eq-A1-estimate-for-phi-a}, we infer that
\begin{align*}
&\int_0^T \Vert \phi_n(s) \Vert_{H^2}^3\, \d s
\leq C \int_0^T (\lvert \bar  \mu_n(s) \rvert_{\newone{H}}^{2} \lvert \phi_n(s) \rvert_{L^2} + \lvert \phi_n(s) \rvert_{L^2}^3)\, \d s
\leq C_T (\lvert \phi_0 \rvert_{L^2}^3 + \lvert \phi_0 \rvert_{L^2} \Vert \cp_n \Vert_{L^2(0,T;\newone{H})}^2).
\end{align*}
From the latter inequality and \eqref{eq-3.23}, we complete the proof of \eqref{eqn-improved-H2-estimates}.
\newline
By the estimates \eqref{eq-A1-3-over-2-estimate-for-phi} and \eqref{eq-estimate-for-phi-c}, we deduce that
\begin{align*}
&\int_0^T \lvert \Atwo^{3/2} \phi_n(s) \rvert_{\mathbb{L}^2}^{\frac32}\, \d s
\leq C \int_0^T (\lvert \cp_n(s) \rvert_{\newone{H}}^{\frac32} + \lvert \phi_n(s) \rvert_{L^2}^{\frac34} \lvert \Atwo \phi_n(s) \rvert_{L^2}^{\frac34} + \lvert \phi_n(s) \rvert_{L^2}^{\frac32} \Vert \phi_n(s) \Vert_{H^2}^{3})\, \d s \\
&\leq C_T \mathbb{E}  (\Vert \cp_n \Vert_{L^{3/2}(0,T;\newone{H})}^{3/2} 
 + \lvert \phi_0 \rvert_{L^2}^{3/2} + \lvert \phi_0 \rvert_{L^2}^{3/2} \Vert \phi_n \Vert_{L^3(0,T;H^2(\domO))}^3).
\end{align*}
This, combined with estimates \eqref{eq-3.23} and \eqref{eqn-improved-H2-estimates}, leads us to \eqref{H3-phi-n-estimate}. \\
Finally, \eqref{eq-A1-3-over-2-estimate-for-phi-1} follows from \eqref{eq-3.23} and the first two estimates in Corollary \ref{cor-Propositions 5.1 and 5.3}.
\end{proof}
\begin{remark}
Let us fix $T>0$ and $n \in \mathbb{N}$. Assume $\phi \in L^{4/3}(0,T;H^3(\domO)) \cap L^\infty(0,T;H^1(\domO))$ and put
\[
\cp= \pi_{1,n} \Atwo \phi + \pi_{1,n} (\psi^\prime(\phi) - \avg{\psi^\prime(\phi)})\in H_{1,n}.
\]
Then, there exists a positive constant $C=C(T,\domO)$ independent of $n$ such that
\begin{equation}\label{Eq-nabla-mu}
\Vert \cp\Vert_{ L^{4/3}(0,T;\newone{H})}
\leq C (\Vert \phi \Vert_{ L^{4/3}(0,T, H^3(\domO))} + \Vert \phi \Vert_{L^{4/3}(0,T;H^1(\domO))} + \Vert \phi \Vert^{7/4}_{L^{\infty}(0,T;H^1(\domO))}
\Vert \phi \Vert^{5/4}_{L^{5/3}(0,T;H^2(\domO))}).
\end{equation}
Indeed, if $\phi \in H^3(\domO)$, then for every $n \in \mathbb{N}$,
\begin{align*}
&\vert \cp \vert_{\newone{H}} 
\leq \lvert \pi_{1,n} \Atwo \phi \rvert_{\newone{H}} + \lvert \pi_{1,n} (\psi^\prime(\phi) - \avg{\psi^\prime(\phi)}) \rvert_{\newone{H}}
\leq \Vert \pi_{1,n} \Vert_{\mathcal{L}(\newone{H})} (\lvert \Atwo \phi \rvert_{\newone{H}} + \lvert \psi^\prime(\phi) - \avg{\psi^\prime(\phi)} \rvert_{\newone{H}})
\\
&\leq \lvert \Atwo \phi \rvert_{\newone{H}} + \lvert \psi^\prime(\phi) - \avg{\psi^\prime(\phi)} \rvert_{\newone{H}}
= \lvert \nabla \Delta \phi \rvert_{\mathbb{L}^2} + \lvert \nabla \psi^\prime(\phi) \rvert_{\mathbb{L}^2}
\leq \Vert\phi\Vert_{H^3} + \lvert \psi^{\bis}(\phi) \nabla \phi \rvert_{\mathbb{L}^2}.
\end{align*}
Moreover, since $[H^1(\domO),H^3(\domO)]_{1/2}= H^2(\domO)$, then if $\phi \in L^{4/3}(0,T;H^3(\domO)) \cap L^\infty(0,T;H^1(\domO))$, we infer $\phi \in L^2(0,T;H^2(\domO))$. From this latter observation and \eqref{Eq-Psi-second-gradient-phi-b}, we deduce \eqref{Eq-nabla-mu}.
\end{remark}
Henceforth, we denote by $\StokesH_{,w}$ the Hilbert space $\StokesH$ endowed with its weak topology.
Similarly, $\newone{H}_{,w}$ denotes the Sobolev space $\newone{H}$ equipped with the weak topology. 
Now we introduce the spaces in which the solution will live.
\begin{definition}\label{def-notation}
\begin{trivlist}
\item[(o)] $C([0,T];\StokesH_{,w}) \coloneqq$ the space of continuous functions $\bv: [0,T] \to \StokesH_{,w}$ endowed with the weakest topology $\mathscr{Z}_0$ such that for all $\bv \in \StokesH$, the mappings
     \begin{equation*}
        C([0,T];\StokesH_{,w}) \ni \bu \mapsto (\bu(\cdot),\bv) \in C([0,T];\mathbb{R}) \; \; \text{are continuous}.
     \end{equation*}
In particular, $\bv_n \to \bv$ in $C([0,T];\StokesH_{,w})$ if and only if 
   \begin{equation*}
      \lim_{n \to \infty} \sup_{t \in[0,T]} \lvert (\bv_n(t) - \bv(t),\bw) \rvert=0, \;\; \mbox{for all $\bw \in \StokesH$.}
    \end{equation*}
    Since the space $\StokesH_{,w}$ is not metrizable and the space $C([0,T];\StokesH_{,w})$ contains $\StokesH_{,w}$ as a closed subspace via the constant maps, 
we infer that the  topology   $\mathscr{Z}_0$ is not metrizable. 
\item[(i)] $C([0,T];\rU_0^\prime)\coloneqq$ the space of continuous functions
$\bv: [0,T] \to \rU_0^\prime$ 
endowed with  the topology $\mathscr{Z}_1$ induced by the $\sup$ norm, i.e. the topology of  the uniform convergence on $[0,T]$, 

\item[(ii)] $L_w^2(0,T;\StokesV)\coloneqq$ the Hilbert space $L^2(0,T;\StokesV)$ endowed with the weak topology $\mathscr{Z}_2$. In particular,  $\bv_n \to \bv$ in $L_w^2(0,T;\StokesV)$ if and only if  for all $\bw \in L^2(0,T;\StokesVp)$,
     \begin{equation*}
       \lim_{n \to \infty} \int_0^T \duality{\bw(s)}{\bv_n(s) - \bv(s)}{\StokesV}{\StokesVp} \, \d s=0.
    \end{equation*}    
We can verify as above that the topology $\mathscr{Z}_2$ is not metrizable either.

\item[(iii)] $C([0,T];\newone{H}_{,w})\coloneqq$ the space of  continuous functions $\bv: [0,T] \to \newone{H}_{,w}$ endowed with the topology $\mathscr{Z}_3$ defined as in item (o).

\item[(iv)] $C([0,T];\newonep{V})\coloneqq$  the space of $\newonep{V}$-valued trajectories with the topology $\mathscr{Z}_4$ of the uniform convergence on $[0,T]$. $C([0,T];\newonep{V})$  is a separable Banach space.

\item[(v)] $L_w^{\beta}(0,T;\newone{V})\coloneqq$ the space $L^\beta(0,T;\newone{V})$ with the weak topology $\mathscr{Z}_5$.

\item[(vi)] The separable Hilbert space $L^2(0,T;\StokesH)$ with the topology $\mathscr{Z}_6$  induced by the norm. 

\item[(vii)] The separable Hilbert space $L^\beta(0,T;\newone{H})$ with the topology $\mathscr{Z}_7$  induced by the norm.
\end{trivlist}
\end{definition}
We now introduce the following topological space 
   \begin{equation}\label{eqn-Z_T}
      \bZ_T \coloneqq \bZ_{T,1} \times \bZ_{T,2},
    \end{equation}
where
\begin{align}
\label{eqn-Z_T1}
&\bZ_{T,1}= C([0,T];\rU_0^\prime) \cap C([0,T];\StokesH_{,w}) \cap L_w^2(0,T;\StokesV) \cap L^2(0,T;\StokesH),
\\
\label{eqn-Z_T2}
&\bZ_{T,2}= 
C([0,T];\newonep{V}) \cap C([0,T];\newone{H}_{,w}) \cap L_w^{\beta}(0,T;\newone{V}) \cap L^{\beta}(0,T;\newone{H}). 
\end{align}
Let $\mathscr{Z}^1$ be the supremum of the following four topologies: $\mathscr{Z}_1, \, \mathscr{Z}_0,\, \mathscr{Z}_2,\,\mathscr{Z}_6$. 
We denote by $\mathscr{Z}^2$ the supremum of  the following topologies: $\mathscr{Z}_{4},\,\mathscr{Z}_3,\,\mathscr{Z}_5,\,\mathscr{Z}_7$. 
\newline
From now on let us summarize the main result of this section. 
\begin{theorem}\label{thm-main-Galerkin}
Assume that $(\bu_0,\phi_0) \in\mathbb{H}$.
Then, for every $n \in \mathbb{N}$, there exists a unique strong global solution $\bX_n=(\bu_n,\phi_n)$ of the problem \eqref{eqn-Galerkin-Modified-stochastic-CHNSEs-n}. Moreover, for every $T>0$, 
\begin{equation}\label{eq-3.23-global}
\begin{aligned}
&\sup_{n \in \mathbb{N}} \mathbb{E}  \Bigl[\sup_{s \in [0,T ]} (\lvert \bu_n(s) \rvert_{\StokesH}^2 + \newK \vert \phi_n(s) \rvert_{\newone{H}}^2 +  \newK \Vert \psi(\phi_n(s)) \Vert_{L^1})
\\
&\hspace{1.5cm} + \delta_0 \Vert \bu_n \Vert_{L^2(0,T;\StokesV)}^2 +  \newK \Vert \cp_n \Vert_{L^2(0,T;\newone{H})}^{2}
+ \Vert \phi_n \Vert_{L^\beta(0,T;\newone{V})}^{\beta}  
+  \Vert \phi_n \Vert_{L^4(0,T;\newone{H})}^{4}\Bigr]
< \infty. 
\end{aligned}
\end{equation}
\end{theorem}
It follows from the above result the following assertion (and a definition). 
\begin{corollary}\label{cor-P_n}
For every $T>0$, almost surely, the trajectories of the process $\bX_n$ belong to the space $\bZ_T$ and the map 
\[
\Omega \ni \omega \mapsto \bX_n(\cdot,\omega) \in \bZ_T \quad \text{is} \quad \mathscr{F}/\mathscr{Z} \text{-measurable}.
\]
Let us denote by $\mathbb{P}_n$ the law of the process $\bX_n$ on the measure space 
$(\bZ_T,\mathscr{Z})$. 
\end{corollary}
The following auxiliary result will be useful for the relative compactness result.
\begin{corollary}\label{2nd-corollary}
Assume that $k_0$ satisfies condition \eqref{eqn-k_0}, i.e.  $k_0> \frac d2 +1$. Then there  exists a constant $C$ independent of $n$ such that $\d t \otimes \d \mathbb{P}$-a.e.,
\begin{equation}\label{eq-4.1}
\begin{aligned}
\Vert \bB_{0,n}(\bu_n,\bu_n) \Vert_{\rU_0^\prime} &\leq C \lvert \bu_n \rvert_{\StokesH}^2,
\\
\newK  \Vert \bar{R}_{0,n}(\phi_n,\phi_n) \Vert_{\rU_0^\prime}
   &\leq C \lvert \cp_n \rvert_{\newone{H}} \lvert \phi_n \rvert_{\newone{H}},
\\
\Vert B_{1,n}(\bu_n,\phi_n) \Vert_{\newonep{V}} &\leq C \lvert \bu_n \rvert_{\StokesH}^{1-\frac d4} \lvert \nabla \bu_n \rvert_{\mathbb{L}^2}^{\frac d4} \lvert \phi_n \rvert_{\newone{H}},
\\
\Vert \newG_{0,n}(t,\bX_n) \Vert_{\ell^2(\nU^\prime)}^2 &\leq C (\lvert \bu_n \rvert_{\StokesH}^2 + \lvert \tilde{h}(t) \rvert^2).
\end{aligned}
\end{equation}
\end{corollary}
\begin{proof}
The estimates in \eqref{eq-4.1} are a direct consequence of \eqref{eq-B1}, \eqref{eq-B_0-k_0}, and \eqref{eqn-linear growth}. \newline
By integration by parts and Sobolev embedding $H^1 \embed L^4$, we have for all $\bv \in \StokesV$ and every $n \in \mathbb{N}$,
\begin{align*}
&\lvert (\bar{R}_{0,n}(\phi_n,\phi_n),\bv) \rvert 
= \lvert (\cp_n \nabla \phi_n, \pi_{0,n} \bv) \rvert 
= \lvert - ( \phi_n \nabla \cp_n, \pi_{0,n} \bv) \rvert \\
&\leq \lvert \nabla \cp_n \rvert_{\mathbb{L}^2} \Vert \phi_n \Vert_{L^4}  \Vert \pi_{0,n} \bv \Vert_{\mathbb{L}^4} 
\leq C \lvert \cp_n \rvert_{\newone{H}} \lvert \phi_n \rvert_{\newone{H}}  \Vert \bv \Vert_{\StokesV},
\end{align*}
and then $  \Vert \bar{R}_{0,n}(\phi_n,\phi_n) \Vert_{\StokesVp}
            \leq C \lvert \cp_n \rvert_{\newone{H}} \lvert \phi_n \rvert_{\newone{H}}$.
From this latter inequality and the embedding $\StokesVp \embed \rU_0^\prime$, we infer that for every $n \in \mathbb{N}$,
\begin{equation*}
  \Vert \bar{R}_{0,n}(\phi_n,\phi_n) \Vert_{\rU_0^\prime}
   \leq C \lvert \cp_n \rvert_{\newone{H}} \lvert \phi_n \rvert_{\newone{H}}.
\end{equation*}
\end{proof}

\section{Tightness of the laws of the approximating sequences}\label{Sec-Tightness}
\begin{definition}\label{def-P_n}
Assume that $n \in \mathbb{N}$. 
By $\mathbb{P}^1_n$ we denote the Borel probability measure on the space $\bZ_{T,1}$, which is the law of the process $\bu_n$.
\\
By $\mathbb{P}_n^2$, we denote the Borel probability measure on the space $\bZ_{T,2}$ which is 
the law of the process $\phi_n$.
 \\
Finally, we denote by $\mathbb{P}_n$ the Borel probability measure on the space $\bZ_{T}$, which is 
the law of the process $\bX_n=(\bu_n,\phi_n)$, i.e.  $\mathbb{P}_n=\mathbb{P}^1_{n} \otimes \mathbb{P}^2_{n}$.
\end{definition}
We will prove that the set of measures $\{\mathbb{P}_n,\; n\geq 1\}$ is tight on $(\bZ_T,\mathscr{Z})$.
Before stating our relative weak compactness result for the set of measures $\{\mathbb{P}_n,\; n\geq 1\}$, we will need the following important compactness criterion adapted to our situation as in \cite{Brz+Motyl_2013,Mik+Roz_2005}. 
Let us emphasize that in our proof we  use the compactness of the  embedding $ \newone{V} \embed \newone{H}$.
\begin{definition}
Assume that $E$ is a normed vector space, e.g. $E=\nU^\prime$,  and 
 $\bu \in C([0,T];E)$. The modulus of continuity of  $\bu$  on $[0,T]$ is defined by:
  \begin{equation*}
    \mc_{E}(\bu,\delta)= \sup_{s,\,t \in[0,T],\; \lvert t-s \rvert \leq \delta} \Vert \bu(t) - \bu(s) \Vert_{E}, \quad 0< \delta \leq 1.
  \end{equation*}
\end{definition}
\begin{lemma}\label{Lem-first-compactness-result-for-phi}
Let 
\[
\tilde{\bZ}_{T,2} \coloneqq C([0,T];\newonep{V}) \cap L_w^{\beta}(0,T;\newone{V}) \cap L^{\beta}(0,T;\newone{H}).
\]
and let $\tilde{\mathscr{Z}}_2$ be the supremum of the corresponding topologies. Then a set $\mathcal{K} \subset \tilde{\bZ}_{T,2}$ is $\tilde{\mathscr{Z}}_2$-relatively compact if the following two conditions hold:
\begin{trivlist}
\item[(i)] 
$\sup_{\phi \in \mathcal{K}} \left[ \Vert \phi \Vert_{L^\beta(0,T;\newone{V})}^\beta<\infty \right]$, i.e. $\mathcal{K}$ is bounded in $L^\beta(0,T;\newone{V})$,
\item[(ii)] $\lim_{\delta \to 0} \sup_{\phi \in \mathcal{K}} \mc_{\newonep{V}}(\phi,\delta)=0$.
\end{trivlist}
\end{lemma}
\begin{proof}[Proof of Lemma \ref{Lem-first-compactness-result-for-phi}]
\textbf{Step 1:} Without loss of generality, we assume that $\mathcal{K}$ is a closed subset of $\tilde{\bZ}_{T,2}$. 
By Assumption $(i)$, the restriction to $\mathcal{K}$ of the weak topology in $L_w^{\beta}(0,T;\newone{V})$ is metrizable. Moreover, the space $ C([0,T];\newonep{V})$ is metrizable. Therefore, the compactness of a subset of 
$\tilde{\bZ}_{T,2}$ is equivalent to its sequential compactness. 
Let $(\phi_n)_{n \in\mathbb{N}}$ be a sequence in $\mathcal{K}$. Thanks to Banach-Alaoglu's Theorem, Assumption $(i)$ implies that 
$\mathcal{K}$ is compact in $L_w^{\beta}(0,T;\newone{V}) $. Thus, we may assume, up to a subsequence, that
$\phi_n \to \phi$ in $L_w^{\beta}(0,T;\newone{V}) $ as $n \to \infty$.
\newline 
\textbf{Step 2:} Now, we will prove that there exists a subsequence $(\phi_{n_k})_k \subset (\phi_n)_n$ convergent in $C([0,T];\newonep{V})$. For the proof, we proceed analogously to the proof of the classical Arzel\`a-Ascoli Theorem. To this end, we introduce the following set
   \begin{equation}\label{Eqn-I-infty-set}
    I_\infty \coloneqq \left\{t \in [0,T]: \lim_{n \to \infty} \Vert \phi_n(t) \Vert_{\newone{V}}=\infty \right\}.
  \end{equation}
The set $I_\infty$ is Lebesgue measurable since
\[
I_\infty
=\left\{t \in [0,T]: \lim_{n \to \infty} \Vert \phi_n(t) \Vert_{\newone{V}}^{\beta}=\infty \right\}
= \bigcap_{n=1}^\infty \bigcup_{k=n}^\infty \bigcap_{i=k}^\infty \left\{ t \in [0,T]: \Vert \phi_i(t) \Vert_{\newone{V}}^{\beta} \geq n \right\}.
\]
Moreover, its measure is equal to zero, i.e. $\mbox{Leb}(I_\infty)=0$. Indeed, let us note that otherwise
\[
\int_0^T \Vert \phi_n(t) \Vert_{\newone{V}}^{\beta} \, \d t \geq \int_{I_\infty}  \Vert \phi_n(t) \Vert_{\newone{V}}^{\beta} \, \d t \geq n \mbox{Leb}(I_\infty) \to \infty \mbox{ as } n \to \infty,
\]
which in turn contradicts the assumption $(i)$. 
\newline
From now on, by \eqref{Eqn-I-infty-set}, for every $t \in [0,T] \, \setminus \, I_\infty$, the sequence $(\phi_n(t))_{n \in \mathbb{N}}$ contains a subsequence bounded in $\newone{V} \subset \newone{H}$ and in $\newone{H}$ also. In addition, since the embedding $\newone{V} \stackrel{c}{\embed} \newone{H}$, this subsequence contains a subsequence convergent in $\newonep{V}$.
\newline
Let $\{t_\ell\}_{\ell \in \mathbb{N}} \subset  [0,T] \, \setminus \, I_\infty$ be a dense subset of $[0,T]$. Using the diagonal method, we can choose a subsequence, still denoted by $(\phi_n)_{n \in \mathbb{N}}$ such that for each $\ell \in \mathbb{N}$, the sequence
   \begin{equation}
    (\phi_n(t_\ell))_{n \in \mathbb{N}}  \mbox{ is convergent in } \newonep{V}.
   \end{equation}
Let us fix $\eps>0$. By the assumption $(ii)$, there exists $\delta_0>0$ such that for every $\delta\in (0,\delta_0]$,
   \begin{equation*}
     \sup_{\varphi \in \mathcal{K}} \sup_{s,\,t \in[0,T],\; \lvert t-s \rvert \leq \delta} \Vert \varphi(t) - \varphi(s) \Vert_{\newonep{V}}= \sup_{\varphi \in \mathcal{K}} \mc_{\newonep{V}}(\varphi,\delta)< \frac{\eps}{3}.
   \end{equation*}
Next, let us fix $t \in [0,T]$. There exists $\ell \in \mathbb{N}$ such that $\lvert t - t_\ell \rvert \leq \delta$,  $\delta\in (0,\delta_0]$. Then for sufficiently large $m,\,n \in \mathbb{N}$, we observe that 
\[
\Vert \phi_n(t) - \phi_m(t) \Vert_{\newonep{V}}
\leq \Vert \phi_n(t) - \phi_n(t_\ell) \Vert_{\newonep{V}} + \Vert \phi_n(t_\ell) - \phi_m(t_\ell) \Vert_{\newonep{V}} + \Vert \phi_m(t_\ell) - \phi_m(t) \Vert_{\newonep{V}}
\leq \epsilon.
\]
Hence, by arbitrariness of $t \in [0,T]$, we infer that
\[
\sup_{t \in [0,T]} \Vert \phi_n(t) - \phi_m(t) \Vert_{\newonep{V}} \leq \epsilon,
\]
then, $(\phi_n)_{n \in \mathbb{N}}$ is a Cauchy sequence in $C([0,T];\newonep{V})$. Therefore, there exists a subsequence  $(\phi_{n_k}) \subset (\phi_n)$ and $\phi \in C([0,T];\newonep{V}) \cap L^{\beta}(0,T;\newone{V})$ such that
\[
 \phi_{n_k} \to \phi \mbox{ in } C([0,T];\newonep{V}) \cap L_w^{\beta}(0,T;\newone{V})  \mbox{ as } k \to \infty.
\]
\textbf{Step 3:} By \cite[Lemma 2]{Brz+Hornung+Manna_2020} the latter implies that
   \begin{equation}\label{eqn-boundeness-phi_{n-k}-in V_1-prime}
     \sup_{k \in \mathbb{N}} \sup_{t \in [0,T]} \Vert \phi_{n_k}(t) \Vert_{\newonep{V}}<\infty. 
   \end{equation}
Hence, by the Lebesgue Dominated convergence Theorem, we infer that
\[
\phi_{n_k} \to \phi \mbox{ in } L^{\beta}(0,T;\newonep{V}) \mbox{ as } k \to \infty.
\]
We now prove that $\phi_{n_k} \to \phi$ strongly in $L^3(0,T;\newone{H})$ as $k \to \infty$. 
Let us choose and fix $\eta \in (0,1)$. Since the embedding $\newone{V} \embed \newone{H}$ is compact, then by \cite[Theorem 16.4]{Lions+Magenes_1972_vol-1}, we infer that there exists a constant $C(\eta)$ such that a.e. in $[0,T]$ and for every $k \in \mathbb{N}$,
\[
\lvert \phi_{n_k}(s) - \phi(s) \rvert_{\newone{H}}^\beta
\leq 2^{\beta -1} \eta \Vert \phi_{n_k}(s) - \phi(s) \Vert_{\newone{V}}^\beta + C(\eta) \Vert \phi_{n_k}(s) - \phi(s) \Vert_{\newonep{V}}^\beta.
\]
Updating the constant $C(\eta)$ accordingly, we infer from the above inequality that 
\begin{align*}
\Vert \phi_{n_k} - \phi \Vert_{L^\beta(0,T;\newone{H})}^\beta
\leq 2^{\beta-1}  \eta \Vert \phi_{n_k} - \phi \Vert_{L^\beta(0,T;\newone{V})}^\beta + C(\eta) \Vert \phi_{n_k} - \phi \Vert_{L^\beta(0,T;\newonep{V})}^\beta.
\end{align*}
Moreover, there exists a generic constant $C>0$ such that for every $k \in \mathbb{N}$,
\begin{align*}
2^{\frac{\beta -1}{\beta}} \Vert \phi_{n_k} - \phi \Vert_{L^\beta(0,T;\newone{V})}
\leq C (\Vert \phi_{n_k} \Vert_{L^\beta(0,T;\newone{V})} + \Vert \phi \Vert_{L^\beta(0,T;\newone{V})})
\leq C \sup_{\varphi \in \mathcal{K}} \Vert \varphi \Vert_{L^\beta(0,T;\newone{V})}\coloneqq C_1, 
\end{align*}
and so passing to the limit as $k \to \infty$, we deduce that
  \begin{equation}
    \lim_{k \to \infty} \Vert \phi_{n_k} - \phi \Vert_{L^\beta(0,T;\newone{H})}^\beta
    \leq 2^{\beta -1} C_1 \eta \Vert \phi_{n_k}(s) - \phi(s) \Vert_{\newone{V}}^\beta. 
 \end{equation}
By the arbitrariness of $\eta$, it follows that
 \begin{equation}\label{eqn-convergence-phi_n-k-exponent-beta}
  \lim_{k \to \infty} \Vert \phi_{n_k} - \phi \Vert_{L^\beta(0,T;\newone{H})}^\beta=0,
 \end{equation}
completing the proof of Lemma \ref{Lem-first-compactness-result-for-phi}.
\end{proof}
Hereafter, let us consider the ball 
\[
\mathbb{B}\coloneqq \{ \phi \in \newone{H}: \lvert \phi \rvert_{\newone{H}} \leq r\}.
\]
We denote by $\mathbb{B}_w$ the ball $\mathbb{B}$ endowed with the weak topology. Furthermore, it is well known that $\mathbb{B}_w$ is metrizable, see \cite{Brezis_1983}. Let 
\[
C([0,T];\mathbb{B}_w)
= \left \{\phi \in C([0,T];\newone{H}_{,w}): \sup_{s \in [0,T]} \lvert \phi(s) \rvert_{\newone{H}} \leq r \right\}.
\]
The space $C([0,T];\mathbb{B}_w)$ endowed with 
\[
\varrho_2(v,v_1)= \sup_{t \in [0,T]} \varrho_1(v(t),v_1(t)) \mbox{ is metrizable}.
\]
Here, $\varrho_1$ denotes a metric that is compatible with the weak topology on $\mathbb{B}$.
By the Banach–Alaoglu Theorem, the set $\mathbb{B}_w$ is compact in the weak topology, and hence $(C([0,T];\mathbb{B}_w),\varrho_2)$ is a complete metric space.
\newline
\begin{lemma}\label{Lem-2nd-compactness-result-for-phi}
Let $\phi_n: [0,T] \to \newone{H}$, $n \in \mathbb{N}$ be functions such that
\begin{trivlist}
\item[(i)] $\sup_{n \in \mathbb{N}} \, \sup_{s \in [0,T]} \lvert \phi_n(s) \rvert_{\newone{H}} \leq r$.
\item[(ii)] $\phi_n \to \phi$ in $C([0,T];\newonep{V})$ as $n \to \infty$.
\end{trivlist}
Then $\phi,\, \phi_n \in C([0,T];\mathbb{B}_w)$ and $\phi_n \to \phi$ in $C([0,T];\mathbb{B}_w)$ as $n \to \infty$.
\end{lemma}
\begin{proof}
To carry out the proof, we follow the approach of \cite[Lemma 2.1]{Brz+Motyl_2014}, with suitable modifications to adapt it to our setting. We begin by observing that by Lemma \ref{Lem C1},
   \begin{equation}
      L^\infty(0,T;\newone{H}) \cap C([0,T];\newone{V}_{,w}^\prime) \subset C([0,T];\newone{H}_{,w}),
   \end{equation}
from which we infer $\phi_n \in C([0,T];\newone{H}_{,w})$.
\newline
Next, we prove that
\[
\phi_n \to \phi \mbox{ in } C([0,T];\mathbb{B}_w) \mbox{ as } n \to \infty,
\]
i.e. for every $h \in \newone{H}$, 
   \begin{equation}
    \lim_{n \to \infty} \sup_{s \in [0,T]} \lvert (\phi_n(s) - \phi(s),h)_{\newone{H}} \rvert= 0.
   \end{equation}
Let $h \in \newone{V}$ be arbitrary. Since $\newone{V} \embed \newone{H}$, see \eqref{eqn-Gelfand triple-abstract-2}, we infer that a.e. in $[0,T]$ and for every $n \in \mathbb{N}$,
\begin{align*}
\lvert (\phi_n(s) - \phi(s),h)_{\newone{H}} \rvert
= \lvert \duality{\phi_n(s) - \phi(s)}{h}{\newone{V}}{\newonep{V}} \rvert
\leq \Vert \phi_n(s) - \phi(s) \Vert_{\newonep{V}} \Vert h \Vert_{V_1}.
\end{align*}
Therefore, by Assumption $(ii)$ in Lemma \ref{Lem-2nd-compactness-result-for-phi}, we deduce that
\begin{align}\label{eqnt-5.5}
\sup_{s \in [0,T]} \lvert (\phi_n(s) - \phi(s),h)_{\newone{H}} \rvert
\leq \sup_{s \in [0,T]} \Vert \phi_n(s) - \phi(s) \Vert_{\newonep{V}} \cdot \Vert h \Vert_{\newone{V}} \underset{n \to \infty}{\to}  0.
\end{align}
Let us now choose and fix $h \in \newone{H}$. Let $\epsilon>0$. Since $\newone{V}$ is dense in $\newone{H}$, there exists $h_\epsilon \in \newone{V}$ such that $\lvert h - h_\epsilon \rvert_{\newone{H}} \leq \epsilon$. By assumption $(i)$ in Lemma \ref{Lem-2nd-compactness-result-for-phi}, we infer that for a.e. $s \in [0,T]$,
\begin{align*}
&\lvert (\phi_n(s) - \phi(s),h)_{\newone{H}} \rvert
\leq \lvert (\phi_n(s) - \phi(s),h - h_\epsilon)_{\newone{H}} \rvert + \lvert (\phi_n(s) - \phi(s),h_\epsilon)_{\newone{H}} \rvert
\\
&\leq \lvert \phi_n(s) - \phi(s)\rvert_{\newone{H}} \lvert h - h_\epsilon\rvert_{\newone{H}} + \lvert (\phi_n(s) - \phi(s),h_\epsilon)_{\newone{H}} \rvert
\\
&\leq 2 \epsilon \sup_{n \in \mathbb{N}} \, \sup_{s \in [0,T]} \lvert \phi_n(s) \rvert_{\newone{H}} + \lvert (\phi_n(s) - \phi(s),h_\epsilon)_{\newone{H}} \rvert
\leq 2 \epsilon r + \lvert (\phi_n(s) - \phi(s),h_\epsilon)_{\newone{H}} \rvert.
\end{align*}
Therefore, passing to the upper limit as $n \to \infty$, using \eqref{eqnt-5.5}, we deduce that
\[
\limsup_{n \to \infty} \sup_{s \in [0,T]} \lvert (\phi_n(s) - \phi(s),h)_{\newone{H}} \rvert
\leq 2 r \epsilon.
\]
By the arbitrariness of $\epsilon$, 
\[
\lim_{n \to \infty} \sup_{s \in [0,T]} \lvert (\phi_n(s) - \phi(s),h)_{\newone{H}} \rvert= 0.
\]
Finally, since $C([0,T];\mathbb{B}_w)$ is a complete metric space, we infer $\phi \in C([0,T];\mathbb{B}_w)$. This completes the proof of Lemma \ref{Lem-2nd-compactness-result-for-phi}.
\end{proof}
Equipped with Lemmas \ref{Lem-first-compactness-result-for-phi} and \ref{Lem-2nd-compactness-result-for-phi}, we can argue as in \cite[Lemma 3.3]{Brz+Motyl_2013} to deduce the following key result.
\begin{lemma}[Deterministic compactness criterion for $\phi$]\label{phi-compactness-criterion}
Let us consider the topological space $(\bZ_{T,2},\mathscr{Z}^2)$ defined in \eqref{eqn-Z_T2}. A set $\mathcal{K}_2$ is $\mathscr{Z}^2$-relatively compact in the intersection space $\bZ_{T,2}$ if the following three conditions hold:
\begin{trivlist}
\item[(i)] $\sup_{\phi \in \mathcal{K}_2} \sup_{s \in [0,T]} \lvert \phi(s) \rvert_{\newone{H}}< \infty$,
\item[(ii)] $\sup_{\phi \in \mathcal{K}} \left[\int_0^T \Vert  \phi(s) \Vert_{\newone{V}}^{\beta}\,\d s \right]<\infty$, i.e. $\mathcal{K}_2$ is bounded in $L^\beta(0,T;\newone{V})$,

\item[(iii)] $\lim_{\delta \to 0} \sup_{\phi \in \mathcal{K}_2} \mc_{\newonep{V}}(\phi,\delta)=0$.
\end{trivlist}
\end{lemma}
\begin{corollary}\label{Coro-convergence-phi_n-in-H_2}
Assume $(\phi_n)_{n \in \mathbb{N}} \subset \bZ_{T,2}$ is such that $\phi_n \to \phi$ in $\bZ_{T,2}$ as $n \to \infty$. Then $\phi_n \to \phi$ in $L^3(0,T;\newone{H}) \cap L^2(0,T;\zero{H}{2}(\domO))$.
\end{corollary}

\begin{proof} Let us choose and fix a $\bZ_{T,2}$-valued  sequence $(\phi_n)$ such that $\phi_n \to \phi$ in $\bZ_{T,2}$.
By the definition of the topological space $\bZ_{T,2}$, cf. \eqref{eqn-Z_T2}, we infer that
$\phi_n \to \phi$ in $C([0,T];\newone{H}_{,w}) $ and $\phi_n \to \phi$ in $L^{\frac32}(0,T;\newone{H})$. Thus, by \cite[Lemma 2]{Brz+Hornung+Manna_2020}, we deduce that there exists a constant $C$ independent of $n$ such that
   \begin{equation*}
      \sup_{s \in [0,T]} \lvert \phi_n(s)\rvert_{\newone{H}}^{3/2}\leq C.
    \end{equation*}
Observe also that
\begin{align*}
&\Vert \phi_n - \phi \Vert_{L^3(0,T;\newone{H})}^3
\leq \sup_{s \in [0,T]} \lvert \phi_n(s) - \phi(s) \rvert_{\newone{H}}^{3/2}
\Vert \phi_n - \phi \Vert_{L^{3/2}(0,T;\newone{H})}^{3/2}
\\
&\leq \sqrt{2} \left[\sup_{s \in [0,T]} \lvert \phi_n(s)\rvert_{\newone{H}}^{3/2} + \sup_{s \in [0,T]} \lvert \phi(s) \rvert_{\newone{H}}^{3/2} \right]
\Vert \phi_n - \phi \Vert_{L^{3/2}(0,T;\newone{H})}^{3/2}
 \\
&\leq \sqrt{2} \left[C + \sup_{s \in [0,T]} \lvert \phi(s) \rvert_{\newone{H}}^{3/2} \right]
\Vert \phi_n - \phi \Vert_{L^{3/2}(0,T;\newone{H})}^{3/2}.
\end{align*}
This, jointly with the fact that $\phi_n \to \phi$ in $L^{3/2}(0,T;\newone{H})$ by assumption, yields
\[
\lim_{n \to \infty} \Vert \phi_n - \phi \Vert_{L^3(0,T;\newone{H})}^3=0.
\]
This completes the proof of the first part of the corollary.\\
Proof of the second part of Corollary \ref{Coro-convergence-phi_n-in-H_2}:
Since by interpolation, $[H^1(\domO),H^3(\domO)]_{\frac12}=H^2(\domO)$, we deduce that there exists $C=C(\domO)>0$ such that
   \begin{equation}\label{H^2-interpolation}
     \Vert \varphi \Vert_{H^2} \leq C \Vert \varphi \Vert_{H^1}^{1/2} \Vert \varphi \Vert_{H^3}^{1/2}, \; \varphi \in H^3(\domO).
   \end{equation}
In particular, for every $n \in \mathbb{N}$ and for almost $s \in [0,T]$,  
\[
\Vert \phi_n(s) - \phi(s) \Vert_{\zero{H}{2}} \leq C \lvert \phi_n(s) - \phi(s) \rvert_{\newone{H}}^{1/2} \Vert \phi_n(s) - \phi(s)  \Vert_{\newone{V}}^{1/2}.
\]
Therefore, by the H\"older inequality, we infer that for every $n \in \mathbb{N}$,
\begin{align*}
&\Vert \phi_n - \phi \Vert_{L^2(0,T;\zero{H}{2})}^2
\leq C \int_0^T \lvert \phi_n(s) - \phi(s) \rvert_{\newone{H}} \Vert \phi_n(s) - \phi(s) \Vert_{\newone{V}}\,\d s
\\
& \leq C \Vert \phi_n- \phi \Vert_{L^3(0,T;\newone{H})} \Vert \phi_n - \phi \Vert_{L^{\frac32}(0,T;\newone{V})}
\leq C \Vert \phi_n - \phi \Vert_{L^3(0,T;\newone{H})} [\Vert \phi_n \Vert_{L^{3/2}(0,T;\newone{V})} + \Vert \phi \Vert_{L^{3/2}(0,T;\newone{V})}].
\end{align*}
We recall that $\phi_n \to \phi$ in $L^{3}(0,T;\newone{H})$, cf. the first part of Corollary \ref{Coro-convergence-phi_n-in-H_2}.
\end{proof}
\begin{lemma}[Deterministic compactness criterion for $\bu$]\label{u-compactness-criterion}
Let us consider the topological space $(\bZ_{T,1},\mathscr{Z}^1)$ defined in \eqref{eqn-Z_T1}.
A set $\mathcal{K}_1$ is $\mathscr{Z}^1$-relatively compact in the intersection space $\bZ_{T,1}$ if the following three conditions hold:
\begin{trivlist}
\item[(i)] $\sup_{\bu \in \mathcal{K}_1} \sup_{s \in [0,T]} \lvert \bu(s) \rvert_{\StokesH} < \infty$,
\item[(ii)] $\sup_{\bu \in \mathcal{K}_1} \Vert \bu \Vert_{L^2(0,T;\StokesV)}^2< \infty$,
\item[(iii)] $\lim \limits_{\delta \to 0}  \sup_{\bu \in \mathcal{K}_1} \mc_{\rU_0^\prime}(\bu,\delta)= 0$.
\end{trivlist}
\end{lemma}
In view of Lemma \ref{u-compactness-criterion}, to show the law of $\bu_n$ is tight, in what follows, we formulate a result related to Corollary 3.9 in \cite{Brz+Motyl_2013}, i.e. the corollary \ref{cor-Corollary 3.9 in BM_2013} below.
We begin with the following definition. 
\begin{definition}\label{lem-Aldous-Rebolledo} Assume that $\mO$ is a separable Banach space and 
that $(\Omega, \mathscr{F}, \mathbb{F},\mathbb{P})$ is a filtered complete probability space satisfying the usual condition \ref{ass-usual}, and    
 $(\bu_n)_{n \in \mathbb{N}}$is a sequence of $\mathbb{F}$-adapted, $\mO$-valued continuous processes. We say that this sequence satisfies the \textbf{Aldous-Rebolledo} condition in the space $\mO$  if and only if 
for every  $\eta>0$, for  every sequence $(\tau_n)_{n=1}^\infty$ of $[0,T]$-valued  stopping times, for every $[0,1]$-valued  sequence $\vartheta_n$ such that  $\vartheta_n \to 0$ as $n \to \infty$, 
\begin{equation}\label{eqn-Aldous-Rebolledo}
    \lim_{n \to \infty} \mathbb{P} (\Vert \bu_n(\tau_n + \vartheta_n) - \bu_n(\tau_n) \Vert_{\mO} > \eta)= 0.
  \end{equation}
Equivalently, for every sequence $(\tau_n)_{n=1}^\infty$ of $[0,T]$-valued  stopping times, for every $[0,1]$-valued  sequence $\vartheta_n$ such that  $\vartheta_n \to 0$ as $n \to \infty$, for every  $\eta>0$, condition \eqref{eqn-Aldous-Rebolledo} holds. \\
 Equivalently, 
 for  every sequence $(\tau_n)_{n=1}^\infty$ of $[0,T]$-valued  stopping times, for every $[0,1]$-valued  sequence $\vartheta_n$ such that  $\vartheta_n \to 0$ as $n \to \infty$, 
    \begin{equation}\label{eqn-Aldous-Rebolledo-1}
    \lim_{n \to \infty} \Vert \bu_n(\tau_n + \vartheta_n) - \bu_n(\tau_n) \Vert_{\mO} = 0 \mbox{ in } \mathbb{P}. 
  \end{equation}
\end{definition}
\begin{remark}\label{rem-lem-Aldous-Rebolledo} 
As in the Aldous paper \cite{Aldous_1978}, it is convenient to regard each process $\bu_n$ to be defined on the time interval $[0,T+1]$, by putting $\bu_n(t)=\bu_n(T)$, if $t \in [T,T+1]$. 
\end{remark}
Hypothesis \eqref{eqn-Aldous-Rebolledo} is equivalent to the following assertion.

\begin{proposition}\label{prop-Aldous-Rebolledo}
 Assume that $(\Omega, \mathscr{F}, \mathbb{F},\mathbb{P})$ is a probability space satisfying the usual conditions and that    
$(\bu_n)_{n \in \mathbb{N}}$ is a sequence of $\mathbb{F}$-adapted, $\mO$-valued continuous processes.
This sequence  satisfies the  \textbf{Aldous-Rebolledo} condition in the space $\mO$ if and only if the following condition is satisfied. \\
\item[\textbf{[A]}]
For every $\eps>0$, $\eta>0$, there exists a $\zeta \in (0,1]$  and $n_0 \in \mathbb{N}$ such that for  every sequence $(\tau_n)_{n=1}^\infty$ of $[0,T]$-valued  stopping times,
          \begin{equation}\label{eqn-Aldous}
             \sup_{n \geq n_0} \sup_{ \vartheta \in [0,\zeta]} \mathbb{P} \left( \Vert \bu_n(\tau_n+\vartheta)-\bu_n(\tau_n) \Vert_{\mO} \geq \eta\right) \leq \eps.
         \end{equation} 
\end{proposition}
The following result is proved in \cite[Theorem 3.2]{Metivier_1988}.
\begin{lemma}\label{lem-Aldous-Rebolledo-1}
Assume the framework of Definition \ref{lem-Aldous-Rebolledo}. 
 Let us denote by $\mathbb{P}^1_n$ the law of the process $\bu_n$ on the space 
 $C([0,T];E)$. If the Aldous-Rebolledo condition is satisfied, then   
\item[\textbf{[T]}] for every $\eta >0$,
         \begin{equation}\label{eqn-Aldous-Rebolledo-2}
            \lim_{\zeta \to 0} \varlimsup_{n \to \infty} \mathbb{P}^1_n \{\omega \in C([0,T];\mO): \mc_{\mO}(\omega,\zeta)> \eta\}= 0,
         \end{equation}
or, equivalently, 
       \begin{equation}\label{eqn-Aldous-Rebolledo-2b}
         \lim_{\zeta \to 0} \varlimsup_{n \to \infty} \mathbb{P} \{\omega \in \Omega: \mc_{\mO}(\bu_n,\zeta)> \eta\}= 0,
       \end{equation}
and vice versa.
\end{lemma}
\begin{corollary}\label{cor-Corollary 3.9 in BM_2013} 
Assume that $T>0$ and the sequence $(\bu_n)_{n \in \mathbb{N}}$ of continuous $\mathbb{F}$-adapted, $\rU_0^\prime$-valued processes satisfies the Aldous-Rebolledo condition in $\rU_0^\prime$, as well as the following two conditions:
\begin{align}\label{eqn-a}
    &\sup_{n \in \mathbb{N}} \mathbb{E} \sup_{s \in [0,T]} \lvert \bu_n(s) \rvert_{\StokesH}^2 <\infty,  
      \\
    \label{eqn-b}
    & \sup_{n \in \mathbb{N}} \mathbb{E} \Vert \bu_n \Vert_{L^2(0,T;\StokesV)}^2<\infty.
\end{align}
 Then the family $\{\mathbb{P}^1_n: n \in \mathbb{N}\}$ is tight on $\bZ_{T,1}$.
\end{corollary}
\begin{proof}[Proof of Corollary \ref{cor-Corollary 3.9 in BM_2013}]
 Our objective is to prove that the family $\{\mathbb{P}^1_n: n \in \mathbb{N}\}$ is tight. \newline
Let us first observe that  according to  Lemma \ref{lem-Aldous-Rebolledo-1}, the following  assertion  follows from the \textbf{Aldous-Rebolledo} condition. 
For each $\eta>0$, the sequence $(\bu_n)_{n=1}^\infty$ satisfies
   \begin{equation}\label{Eqt-Aldous-condition}
      \lim_{\zeta \to 0} \varlimsup_{n \to \infty} \mathbb{P} \{\omega \in \Omega: \mc_{\rU_0^\prime}(\bu_n(\omega),\zeta)> \eta\}= 0.
  \end{equation}
In other words, Hypothesis \textbf{[T]} is satisfied. \newline 
Condition \eqref{Eqt-Aldous-condition} means that for every $\eta>0$, for every $\eps>0$, there exists $\zeta_0>0$ and there exists $n_0 \in \mathbb{N}$ such that for every $\zeta\in (0,\zeta_0]$ and every $n\geq n_0+1$,
\[
\mathbb{P} \{\omega \in \Omega: \mc_{\rU_0^\prime}(\bu_n(\omega),\zeta)> \eta\} \leq \eps.
\]
The law of each $\bu_n$ is tight on $C([0,T];\rU_0^\prime)$, so for every $n \in \{1,\cdots,n_0\}$, there exists $\zeta_n>0$ such that 
for every $\zeta\in (0,\zeta_n]$,
\[
\mathbb{P} \{\omega \in \Omega: \mc_{\rU_0^\prime}(\bu_n(\omega),\zeta)> \eta\} \leq \eps.
\]
Put $\zeta^\ast \coloneqq \min\{\zeta_k: k=0,\cdots n_0\}$. We infer that if $\zeta \in (0,\zeta^\ast]$, 
then $\zeta \leq \zeta_k$ for every $k=0,\cdots n_0$ and so for every $n=1,\cdots,n_0$, 
\[
\mathbb{P} \{\omega \in \Omega: \mc_{\rU_0^\prime}(\bu_n(\omega),\zeta)> \eta\} \leq \eps.
\]
Hence, we proved that for every $\eta>0$, for every $\eps>0$ there exists $\zeta_0>0$  such that for every $\zeta\in (0,\zeta_0]$ and every $n\in \mathbb{N}$, 
\[
\mathbb{P} \{\omega \in \Omega: \mc_{\rU_0^\prime}(\bu_n(\omega),\zeta)> \eta\} \leq \eps.
\]
In other words, we proved that 
for each $\eta>0$, the sequence $(\bu_n)_{n=1}^\infty$ satisfies
       \begin{equation}\label{Eqt-Aldous-condition-2}
           \lim_{\zeta \to 0} \sup_{n} \mathbb{P} \{\omega \in \Omega: \mc_{\rU_0^\prime}(\bu_n(\omega),\zeta)> \eta\}= 0.
       \end{equation}
Subsequently, let us choose and fix $\eps>0$. We will find a compact set $K_\eps \subset \bZ_{T,1}$ such that
\begin{equation}\label{eqn-K_eps}
\mathbb{P}^1_n(K_\eps) \geq 1 - \eps \mbox{ for every } n \in \mathbb{N}.
\end{equation}
\textbf{Step 1:} We will find a Borel set $\mathbb{A}_\eps \subset C([0,T];\rU_0^\prime)$ such that
     \begin{equation}\label{Eqn-measure-of-A-epsilon}
     \begin{aligned}
       \mathbb{P}^1_n(\mathbb{A}_\eps) &\geq 1- \eps/3,  \mbox{  for each $n$, and }
       \\
       &\lim_{\zeta \to 0} \sup_{\bu \in \mathbb{A}_\eps} \mc_{\rU_0^\prime}(\bu,\zeta)=0.
       \end{aligned}
     \end{equation}
First, we infer from \eqref{Eqt-Aldous-condition-2} that for each $k \in \mathbb{N}$, there exists $\zeta_k$ such that
   \begin{equation}\label{Eqtn-5.9-1}
     \sup_{n} \mathbb{P} \left\{\omega \in \Omega: \mc_{\rU_0^\prime}(\bu_n(\omega),\zeta_k)> \frac1k \right\} \leq \frac{\eps}{3 \cdot 2^{k+1}}. 
  \end{equation}
From now on, for each $k$, let
     \begin{align*}
       B_k \coloneqq \left\{\bu \in C([0,T];\rU_0^\prime): \mc_{\rU_0^\prime}(\bu,\zeta_k) \leq \frac1k \right\}   \mbox{ and }  \mathbb{A}_\eps \coloneqq \bigcap_{k=1}^\infty B_k.
     \end{align*}
Next, thanks to \eqref{Eqtn-5.9-1}, we have 
\begin{align*}
&\mathbb{P}^1_n(C([0,T];\rU_0^\prime) \setminus \mathbb{A}_\eps)
= \mathbb{P}^1_n\left(\bigcup_{k=1}^\infty (C([0,T];\rU_0^\prime)  \setminus  B_k)\right)
    \\
&\leq \sum_{k=1}^\infty \mathbb{P}^1_n\left(C([0,T];\rU_0^\prime)  \setminus  B_k\right)
=\sum_{k=1}^\infty \mathbb{P}^1_n\left(\bu \in C([0,T];\rU_0^\prime): \mc_{\rU_0^\prime}(\bu,\zeta_k)> \frac1k \right)
      \\
&= \sum_{k=1}^\infty \mathbb{P} \left(\omega \in \Omega: \mc_{\rU_0^\prime}(\bu_n(\omega),\zeta_k)> \frac1k \right)
\leq \sum_{k=1}^\infty \frac{\eps}{3 \cdot 2^{k+1}}
= \frac{\eps}{3}.
\end{align*}
Consequently, $\mathbb{P}^1_n(\mathbb{A}_\eps) \geq 1- \frac{\eps}{3}$. This proves the first part of \eqref{Eqn-measure-of-A-epsilon}.

\noindent
Let us now move to the proof of the second part of \eqref{Eqn-measure-of-A-epsilon}. 
To this end, we choose and fix $\bar{\eps}>0$. Obviously, from the definition of $\mathbb{A}_\eps$, we get $\sup_{\bu \in \mathbb{A}_\eps} \mc_{\rU_0^\prime}(\bu,\zeta_k) \leq \frac1k$, $k \in \mathbb{N}$. 
Next, let us choose and fix $\bar{k}_0 \in \mathbb{N}$ such that $1/\bar{k}_0 \leq \bar{\eps}$. Since  $\mathbb{A}_\eps \subset B_{\bar{k}_0}$,  we infer that  for every $\zeta \in (0,\zeta_{\bar{k}_0}]$, we have 
\[
\mc_{\rU_0^\prime}(\bu,\zeta) \leq \mc_{\rU_0^\prime}(\bu,\zeta_{\bar{k}_0}) \leq \bar{\eps}, \; \; \bu \in \mathbb{A}_\eps.
\]
This completes the proof of the second part of \eqref{Eqn-measure-of-A-epsilon}.

\noindent
\textbf{Step 2:} Let us set
\[
R \coloneqq \frac{3}{\eps} \sup_{n \in \mathbb{N}} \mathbb{E} \left[ \sup_{s \in [0,T]} \lvert \bu_n(s)\rvert_{\StokesH}^2 + \Vert \bu_n \Vert^2_{L^2(0,T;\StokesV)}  \right] <\infty.
\]
Then, the Chebyshev inequality, \eqref{eqn-a}, and \eqref{eqn-b} lead to
\begin{align*}
\mathbb{P} \left(\sup_{s \in [0,T]} \lvert \bu_n(s) \rvert_{\StokesH}^2> R \right)
&\leq \frac1R \mathbb{E} \sup_{s \in [0,T]} \lvert \bu_n(s) \rvert_{\StokesH}^2
\leq \frac{\eps}{3}, 
  \\
\mathbb{P} \left(\Vert \bu_n \Vert^2_{L^2(0,T;\StokesV)} > R \right)
&\leq \frac{1}{R} \mathbb{E} \Vert \bu_n \Vert^2_{L^2(0,T;\StokesV)}
\leq \frac{\eps}{3}.
\end{align*}
We now introduce the following sets:
      \begin{align}
          \mathbb{B}_1&\coloneqq \{\bv \in \bZ_{T,1}: \sup_{s \in [0,T]} \lvert \bv(s)\rvert_{\StokesH}^2 \leq R\} \mbox{ and }
            \\
          \mathbb{B}_2&\coloneqq \{\bv \in \bZ_{T,1}: \Vert \bv \Vert^2_{L^2(0,T;\StokesV)}\leq  R\},
      \end{align}
and we infer that $\mathbb{P}^1_n(\mathbb{B}_1) \geq 1 - \eps/3$ and $\mathbb{P}^1_n(\mathbb{B}_2)\geq 1 - \eps/3$. \newline
It then follows from the previous two \textbf{steps} that
\[
\mathbb{P}^1_n(\overline{\mathbb{B}_1 \cap \mathbb{B}_2 \cap \mathbb{A}_\eps})
\geq 1 - \eps/3 - \eps/3 - \eps/3= 1 - \eps.
\]
Let us put $K_\eps\coloneqq \overline{\mathbb{B}_1 \cap \mathbb{B}_2 \cap \mathbb{A}_\eps}$. The above implies that property 
\eqref{eqn-K_eps} holds. It remains to show that $K_\eps$ is a  compact subset of $\bZ_{T,1}$. But this is a consequence of Lemma \ref{u-compactness-criterion}. This completes the proof of Corollary \ref{cor-Corollary 3.9 in BM_2013}.
\end{proof}
In view of Lemma \ref{phi-compactness-criterion}, to show that the family of law of $\phi_n$ is tight, we need the following result.
\begin{corollary}\label{cor-Corollary 3.9 in BM_2013-a}
Assume that $T>0$ and the sequence $(\phi_n)_{n \in \mathbb{N}}$ of continuous $\mathbb{F}$-adapted, $\newonep{V}$-valued processes satisfies the Aldous-Rebolledo condition in $\newonep{V}$, as well as the following two conditions:
\begin{trivlist}
\item[(i)] $\sup_{n \in \mathbb{N}} \mathbb{E} \sup_{s \in [0,T]} \lvert \phi_n(s) \rvert_{\newone{H}}^2< \infty$,
\item[(ii)] $\sup_{n \in \mathbb{N}} \mathbb{E} \int_0^T \Vert \phi_n(s) \Vert_{\newone{V}}^\beta< \infty$.
\end{trivlist}
Then the laws $\{\mathbb{P}_n^2: n \in \mathbb{N}\}$ of $\phi_n$ are tight on 
$\bZ_{T,2}$. 
\end{corollary}
\begin{proof}[Proof of Corollary \ref{cor-Corollary 3.9 in BM_2013-a}]
Similar to the proof of Corollary \ref{cor-Corollary 3.9 in BM_2013}.
\end{proof}
The following result is relative to the tightness of laws of $(\bX_n)_{n=1}^\infty$, $\bX_n=(\bu_n,\phi_n)$.
\begin{lemma}\label{Lem-compactness-measure}
The set of measures $\{\mathbb{P}_n,\; n\geq 1\}$ is tight  on $(\bZ_T,\mathscr{Z})$.
\end{lemma}
\begin{proof}[Proof of Lemma \ref{Lem-compactness-measure}]
For the proof, we will use Corollaries \ref{cor-Corollary 3.9 in BM_2013} and \ref{cor-Corollary 3.9 in BM_2013-a}. 
Let $\bX_n$, $n\in \mathbb{N}$, be the solution to \eqref{eqn-Compact-Galerkin-Modified-stochastic-CHNSEs-n}. 
Before proceeding with the proof of Lemma \ref{Lem-compactness-measure}, we point out that, throughout the proof, $C$ denotes a generic positive constant independent of $n$, but possibly depending on $\domO,\,\delta_0,\,\lambda_1,\,\newK,\,C_3$, and $T$. 
\newline
We begin with the first sequence. From the estimate \eqref{eq-3.23-global}, the sequence $(\bu_n)_n$ satisfies both conditions of Corollary \ref{cor-Corollary 3.9 in BM_2013}, i.e. \eqref{eqn-a} and \eqref{eqn-b}. 
It therefore remains to verify the Aldous-Rebolledo condition for the process $\bu_n$, $n \in \mathbb{N}$.
To this end, we fix sequences of stopping times $\tau_n= \tau_n(\bX_n)$ and $\vartheta_n$ as in Definition \ref{lem-Aldous-Rebolledo}.
From \eqref{eqn-Compact-Galerkin-Modified-stochastic-CHNSEs-n}, we have
\begin{align*}
\bu_n(t)
&= \bu_{0,n} - \nu \int_0^t \Stokes_{,n} \bu_n(s) \, \d s - \int_0^t \bB_{0,n}(\bu_n(s),\bu_n(s)) \,\d s + \newK \int_0^t \bar{R}_{0,n}(\phi_n(s),\phi_n(s)) \,\d s \\
& + \int_0^t [\bG_{0,n}(s,\bu_n(s)) + \bSi_0^n(s)\bu_n(s)]\, \d W(s) 
\coloneqq I_1^n(t) + I_2^n(t) + I_3^n(t) + I_4^n(t) + I_5^n(t),\; t \in [0,T].
\end{align*}
We now estimate the terms on the RHS individually.
Recall that the number $k_0$ has been chosen and fixed so that it satisfies the condition \eqref{eqn-k_0}. 
\newline
Since 
$\Vert \Stokes_{,n} \bu_n \Vert_{\StokesVp}
= \Vert \Stokes \bu_n \Vert_{\StokesVp} 
\leq \Vert \bu_n \Vert_{\StokesV}$ and the embedding $\StokesVp \embed \rU_0^\prime$ is continuous, 
we infer from \eqref{eq-3.23-global} and the H\"older inequality that
\begin{align*}
&\mathbb{E} \Vert I_{2}^n(\tau_n + \vartheta_n) - I_{2}^n(\tau_n) \Vert_{\rU_0^\prime}
= \mathbb{E} \left\Vert \int_{\tau_n}^{\tau_n + \vartheta_n} \nu \Stokes_{,n} \bu_n(s)\, \d s \right\Vert_{\rU_0^\prime} 
\leq \nu \mathbb{E} \int_{\tau_n}^{\tau_n + \vartheta_n} \Vert \Stokes_{,n} \bu_n(s) \Vert_{\rU_0^\prime}\, \d s 
    \\
&= \nu \mathbb{E} \int_{\tau_n}^{\tau_n + \vartheta_n} \Vert \Stokes \bu_n(s) \Vert_{\rU_0^\prime}\, \d s 
\leq C \mathbb{E} \int_{\tau_n}^{\tau_n + \vartheta_n} \Vert \Stokes \bu_n(s) \Vert_{\StokesVp}\, \d s 
\\
&\leq C \mathbb{E} \int_{\tau_n}^{\tau_n + \vartheta_n} \Vert \bu_n(s) \Vert_{\StokesV}\, \d s
\leq C \cdot \left[\left(\mathbb{E} \int_0^T \Vert \bu_n(s) \Vert_{\StokesV}^2\, \d s \right)^{\frac12} \right] \cdot \vartheta_n^{1/2}
\leq C \vartheta_n^{1/2}.
\end{align*}
The latter implies
  \begin{equation}\label{eq-4.3}
    \lim_{n \to \infty} \mathbb{E} \Vert I_{2}^n(\tau_n + \vartheta_n) - I_{2}^n(\tau_n) \Vert_{\rU_0^\prime}= 0.
  \end{equation}
Analogously, by Proposition \ref{prop-First-propo} and Corollary \ref{2nd-corollary}, we deduce that for every $n \in \mathbb{N}$,
\begin{align*}
&\mathbb{E} \Vert I_{3}^n(\tau_n + \vartheta_n) - I_{3}^n(\tau_n) \Vert_{\rU_0^\prime} 
= \mathbb{E} \left \Vert \int_{\tau_n}^{\tau_n + \vartheta_n} \bB_{0,n}(\bu_n(s),\bu_n(s))\, \d s \right \Vert_{\rU_0^\prime} 
\\
&\leq \mathbb{E} \int_{\tau_n}^{\tau_n + \vartheta_n} \Vert \bB_{0,n}(\bu_n(s),\bu_n(s)) \Vert_{\rU_0^\prime}\, \d s 
\leq C \mathbb{E} \int_{\tau_n}^{\tau_n + \vartheta_n} \lvert \bu_n(s) \rvert_{\StokesH}^2\, \d s 
  \\
&\leq  C \,\vartheta_n \,\mathbb{E} \sup_{s \in [0,T]} \lvert \bu_n(s) \rvert_{\StokesH}^2
\leq  C \,\vartheta_n.
\end{align*}
Similarly, we have for every $n \in \mathbb{N}$, 
\begin{align*}
&\mathbb{E} \Vert I_{4}^n(\tau_n + \vartheta_n) - I_{4}^n(\tau_n) \Vert_{\rU_0^\prime} 
= \mathbb{E} \left \Vert \int_{\tau_n}^{\tau_n + \vartheta_n} \newK \bar{R}_{0,n}(\phi_n(s),\phi_n(s))\, \d s \right \Vert_{\rU_0^\prime}
\\
&\leq  \mathbb{E} \int_{\tau_n}^{\tau_n + \vartheta_n} \Vert \newK \bar{R}_{0,n}(\phi_n(s),\phi_n(s)) \Vert_{\rU_0^\prime}\, \d s 
\leq C \mathbb{E} \int_{\tau_n}^{\tau_n + \vartheta_n} \lvert \cp_n(s) \rvert_{\newone{H}} \lvert \phi_n(s) \rvert_{\newone{H}} \, \d s \\
&\leq  C \,\mathbb{E} \left[ \int_{\tau_n}^{\tau_n + \vartheta_n} \lvert \cp_n(s) \rvert_{\newone{H}}^2\, \d s\right]^\frac{1}{2} \left[\int_{\tau_n}^{\tau_n + \vartheta_n} \lvert \phi_n(s) \rvert_{\newone{H}}^2\, \d s\right]^\frac{1}{2}
  \\
&\leq  C \,\vartheta_n^{1/2} \left[ \mathbb{E} \sup_{s \in [0,T]} \lvert \phi_n(s) \rvert_{\newone{H}}^2 \right]^{\frac{1}{2}} \left[ \mathbb{E} \int_{\tau_n}^{\tau_n + \vartheta_n} \lvert \cp_n(s) \rvert_{\newone{H}}^2\, \d s\right]^\frac{1}{2}
\\
&\leq  C \,\vartheta_n^{1/2} \left[ \mathbb{E} \sup_{s \in [0,T]} \lvert \phi_n(s) \rvert_{\newone{H}}^2 \right]^{\frac{1}{2}} \left[ \mathbb{E} \int_0^T \lvert \cp_n(s) \rvert_{\newone{H}}^2\, \d s\right]^\frac{1}{2} 
\leq  C \,\vartheta_n^{1/2},
\end{align*}
and since the constant $C$ is independent of $n$, we infer that
\begin{equation}\label{eq-4.3-a}
\begin{aligned}
\lim_{n \to \infty} \mathbb{E} \Vert I_{3}^n(\tau_n + \vartheta_n) - I_{3}^n(\tau_n) \Vert_{\rU_0^\prime}= 0, 
  \\
\lim_{n \to \infty} \mathbb{E} \Vert I_{4}^n(\tau_n + \vartheta_n) - I_{4}^n(\tau_n) \Vert_{\rU_0^\prime}= 0.
\end{aligned}
\end{equation}
Let us consider the term $I_{5}^n$. Notice that
\begin{align*}
&\mathbb{E} \Vert I_{5}^n(\tau_n + \vartheta_n) - I_{5}^n(\tau_n) \Vert_{\rU_0^\prime}^2
= \mathbb{E}
\left \Vert \int_{\tau_n}^{\tau_n + \vartheta_n} [\bG_{0,n}(s,\bu_n(s)) + \bSi_0^n(s)\bu_n(s)]\, \d W(s) \right \Vert_{\rU_0^\prime}^2 \\
&= \mathbb{E} \int_{\tau_n}^{\tau_n + \vartheta_n} \Vert \bG_{0,n}(s,\bu_n(s)) + \bSi_0^n(s)\bu_n(s) \Vert_{\mathscr{T}_2(\ell^2,\rU_0^\prime)}^2 \,\d s.
\end{align*}
By Example \ref{example:HS}, the continuity of the injection $\StokesVp \embed \rU_0^\prime$, Proposition \ref{prop-First-propo}, and Corollary \ref{2nd-corollary}, we infer that for every $n \in \mathbb{N}$,
\begin{align*}
&\mathbb{E} \Vert I_{5}^n(\tau_n + \vartheta_n) - I_{5}^n(\tau_n) \Vert_{\rU_0^\prime}^2
\leq C \mathbb{E} \int_{\tau_n}^{\tau_n + \vartheta_n} \Vert \bG_{0,n}(s,\bu_n(s)) + \bSi_0^n(s)\bu_n(s) \Vert_{\ell^2(\StokesVp)}^2\d s
\\
&\leq  C \mathbb{E} \int_{\tau_n}^{\tau_n + \vartheta_n} \left(\lvert \bu_n(s) \rvert_{\StokesH}^2 + \lvert \tilde{h}(s) \rvert^2 \right)\d s 
\leq C \left[\vartheta_n \, \mathbb{E} \sup_{s \in [0,T]} \lvert \bu_n(s) \rvert_{\StokesH}^2 + \mathbb{E} \int_{\tau_n}^{\tau_n + \vartheta_n} \lvert \tilde{h}(s) \rvert^2 \, \d s \right]
\\
&\leq C \left[\vartheta_n + \mathbb{E} \int_{\tau_n}^{\tau_n + \vartheta_n} \lvert \tilde{h}(s) \rvert^2\,\d s \right].
\end{align*}
Thus,
      \begin{equation}\label{eq-4.3-b}
        \lim_{ n \to \infty} \mathbb{E} \Vert I_{5}^n(\tau_n + \vartheta_n) - I_{5}^n(\tau_n) \Vert_{\rU_0^\prime}^2= 0.
     \end{equation}
The convergences \eqref{eq-4.3}, \eqref{eq-4.3-a} and \eqref{eq-4.3-b} lead to
  \begin{equation}\label{convergence-4.5}
   \lim_{n \to \infty} \mathbb{E} \Vert \bu_n(\tau_n + \vartheta_n) - \bu_n(\tau_n) \Vert_{\rU_0^\prime}= 0,
 \end{equation}
which, together with the Chebyshev inequality, yields that the sequence $(\bu_n)_{n \in \mathbb{N}}$ satisfies the Aldous-Rebolledo condition \eqref{eqn-Aldous-Rebolledo} in the space $\rU_0^\prime$. 
Hence, we can apply Corollary \ref{cor-Corollary 3.9 in BM_2013} to deduce that the sequence of laws of $\bu_n$ in $\bZ_{T,1}$ is tight.

\noindent
By \eqref{eq-3.23-global}, the sequence $(\phi_n)_{n\in \mathbb{N}}$ satisfies both conditions of Corollary \ref{cor-Corollary 3.9 in BM_2013-a}. We now show that the sequence $(\phi_n)_{n\in \mathbb{N}}$ also satisfies the Aldous-Rebolledo condition in the space  $\newonep{V}$.
Indeed, for every $t \in [0,T]$, it follows from \eqref{eqn-Compact-Galerkin-Modified-stochastic-CHNSEs-n} that
\begin{align*}
\phi_n(t)
= \phi_{0,n} - \int_0^t  B_{1,n}(\bu_n(s),\phi_n(s))\,\d s - \int_0^t  \Athree \cp_n(s)\,\d s
=:  J_1^n(t) + J_2^n(t) + J_3^n(t).
\end{align*}
By the estimates in Corollary \ref{2nd-corollary} and Theorem \ref{thm-main-Galerkin}, we deduce that for every $n \in \mathbb{N}$,
\begin{align*}
&\mathbb{E} \Vert J_{2}^n(\tau_n + \vartheta_n) - J_{2}^n(\tau_n) \Vert_{\newonep{V}} 
\leq C \mathbb{E} \int_{\tau_n}^{\tau_n + \vartheta_n} \Vert \bu_n(s) \Vert_{\StokesV} \lvert \phi_n(s) \rvert_{\newone{H}}   \, \d s
\\
&\leq C \mathbb{E} \left[\sup_{s \in [0,T]} \vert \phi_n(s) \rvert_{\newone{H}} \left(\int_{\tau_n}^{\tau_n + \vartheta_n} \Vert \bu_n(s) \Vert_{\StokesV}^2\, \d s\right)^{\frac12} \right] \cdot \vartheta_n^{1/2}
\\
&\leq C \left[\mathbb{E }\sup_{s \in [0,T]} \lvert \phi_n(s) \rvert_{\newone{H}}^2 \right]^{\frac12} \left[\mathbb{E }\int_0^T \Vert \bu_n(s) \Vert_{\StokesV}^2\, \d s\right]^{\frac12} \cdot \vartheta_n^{1/2}
\leq C \vartheta_n^{1/2},
\end{align*}
which, in turn, yields that
\begin{equation*}
    \lim_{n \to \infty} \mathbb{E} \Vert J_{2}^n(\tau_n + \vartheta_n) - J_{2}^n(\tau_n) \Vert_{\newonep{V}}= 0.
\end{equation*}
Note that $\Vert \Athree \cp_n \Vert_{\newonep{V}} \leq C \lvert \cp_n \rvert_{\newone{H}}$ for some constant $C>0$ which is independent of $n$. Then, it follows from the estimate \eqref{eq-3.23-global} that for every $n \in \mathbb{N}$,
\begin{align*}
&\mathbb{E} \Vert J_{3}^n(\tau_n + \vartheta_n) - J_{3}^n(\tau_n) \Vert_{\newonep{V}}^2
\leq  \vartheta_n \cdot \mathbb{E} \int_{\tau_n}^{\tau_n + \vartheta_n} \Vert \Athree \cp_n(s) \Vert_{\newonep{V}}^2\,\d s 
\\
&\leq C \vartheta_n  \mathbb{E} \int_{\tau_n}^{\tau_n + \vartheta_n} \lvert \cp_n(s) \rvert_{\newone{H}}^2\,\d s 
\leq C \vartheta_n  \mathbb{E} \int_0^T \lvert \cp_n(s) \rvert_{\newone{H}}^2\,\d s 
\leq C \vartheta_n.
\end{align*}
Thus,
  \begin{equation*}
    \lim_{n \to \infty} \mathbb{E} \Vert J_{3}^n(\tau_n + \vartheta_n) - J_{3}^n(\tau_n) \Vert_{\newonep{V}}= 0.
  \end{equation*}
The convergence results obtained above imply the following.
\begin{equation*}
   \lim_{n \to \infty} \mathbb{E} \Vert \phi_n(\tau_n + \vartheta_n) - \phi_n(\tau_n) \Vert_{\newonep{V}}= 0,
 \end{equation*}
i.e. the sequence $(\phi_n)_{n}$  satisfies the Aldous-Rebolledo condition in the space $\newonep{V}$. Hence, we can apply Corollary \ref{cor-Corollary 3.9 in BM_2013-a} and deduce that the laws of $\phi_n$ in $\bZ_{T,2}$ are tight.  
This finishes the proof of Lemma \ref{Lem-compactness-measure}.
\end{proof}
For two Banach spaces $X_1$ and $X_2$, we understand that $(X_1 \times X_2)^\prime \coloneqq X_1^\prime \times X_2^\prime$.
Define
    \begin{align*}
      \duality{z^\ast}{z}{\nU}{\nU^\prime}= \duality{z_1^\ast}{z_1}{\rU_0}{\rU_0^\prime} + \duality{z_2^\ast}{z_2}{\newone{V}}{\newonep{V}},\;\; z\in \nU=\rU_0 \times \newone{V}, \;\; z^\ast \in \nU^\prime,
    \end{align*}
and the operator norm on $\nU^\prime$ is given by
   \begin{equation*}
      \Vert z^\ast \Vert_{\nU^\prime}\coloneqq \sup\left\{\lvert \duality{z^\ast}{z}{\nU}{\nU^\prime} \rvert: \; z \in \nU  \mbox{ with } \Vert z \Vert_{\nU}=1 \right\}, \; \; z^\ast \in \nU^\prime.
   \end{equation*}
Notice that for all $s \in [0,T]$ and $\bX=(\bu,\phi) \in \bZ_T$,
\begin{align*}
&\duality{\bF_{0,n}(\bX(s))}{z}{\nU}{\nU^\prime}
= \duality{- \nu \Stokes_{,n} \bu(s)}{z_1}{\rU_0}{\rU_0^\prime} 
 + \duality{-B_{1,n}(\bu(s),\phi(s))}{z_2}{\newone{V}}{\newonep{V}} 
 + (\pi_{1,n} \cp(s), -\Delta z_2)_{\newone{H}}
  \\
&= \duality{- \nu \Stokes \bu(s)}{\tilde{\pi}_{0,n} z_1}{\rU_0}{\rU_0^\prime} 
  + \duality{- B_1(\bu(s),\phi(s))}{\tilde{\pi}_{1,n} z_2}{\newone{V}}{\newonep{V}} 
  + (\cp(s), \tilde{\pi}_{1,n}(-\Delta z_2))_{\newone{H}},
\end{align*}
where $\cp(t)= \Atwo \phi(t) + \psi^\prime(\phi(t)) - \avg{\psi^\prime(\phi(t))}, \;\; t \in [0,T]$; 
and
\begin{align*}
&\duality{\bF_{2,n}(\bX(s))}{z}{\nU}{\nU^\prime}  
= \duality{- \bB_{0,n}(\bu(s),\bu(s)) + \newK \bar{R}_{0,n}(\phi(s),\phi(s))}{z_1}{\rU_0}{\rU_0^\prime}
   \\
&= \duality{ - \bB_0(\bu(s),\bu(s))  + \newK (\pi_{1,n} (\cp(s)) \nabla \phi(s)}{\tilde{\pi}_{0,n} z_1}{\rU_0}{\rU_0^\prime}.
\end{align*}

\section{Identification of a limit measure}\label{Limit-meaures}
Here, we establish several convergences results which will allow us to conclude that the limiting objects that we found are in fact a weak martingale solution to problem \eqref{eqn-compact-modified-stochastic-CHNSEs-3}.

\noindent
Firstly, from the tightness  result, cf. Lemma \ref{Lem-compactness-measure}, and the discussion about the Prohorov Theorem from the paper \cite[p. 174]{Jakubowski_1998} by Jakubowski, we deduce the following result. 
\begin{proposition}\label{prop-Prohorov-general}
There exists  a Borel probability measure $\mathbb{P}_T$ on $(\bZ_T,\mathscr{Z})$ 
and a subsequence of     the sequence $(\mathbb{P}_n)_{n=1}^\infty$, which we denote by the same symbol,  which converges weakly to $\mathbb{P}_T$. Moreover, for every bounded and measurable function $f:\bZ_T \to \mathbb{R}$, which is sequentially continuous $\mathbb{P}_T$-a.s., 
   \begin{equation}\label{eqn-P-measure}
     \int_{\bZ_T} f\,\d\mathbb{P}_n \to  \int_{\bZ_T} f\,\d\mathbb{P}_T.
   \end{equation}
\end{proposition}
Let us point out that if $\bZ_T$ was  a metric space and $f$ was the indicator function of a set $B \in \mathscr{Z}$, the above result would follow from  \cite[Theorem 11.1.1 (d)]{Dudley_2004}.
\begin{proof}[Proof of Proposition  \ref{prop-Prohorov-general}]
    Let us choose and fix a set $B \in \mathscr{Z}$ such that $\mathbb{P}_T(B^c)=0$ and a measurable and bounded   function $f:\bZ_T \to \mathbb{R}$ which is sequentially continuous at every element of the set $B$. By  \cite[Theorem 2]{Jakubowski_1998} there exists a sequence $(\xi_k)_{k=0}^\infty $ 
    which is the "Skorokhod representation" of a certain  subsequence $(\mathbb{P}_{n_k})_{k=1}^\infty$ and measure $\mathbb{P}_T$ defined on the standard probability measure space 
    $([0,1], \mathscr{B}([0,1]), \Leb)$. In particular, $\xi_k \to \xi_0$ pointwise on $[0,1]$ as $k \to \infty$. Hence, by the Change of Measure and the Lebesgue DC Theorems  we have
\begin{align}
\int_{\bZ_T} f\,\d\mathbb{P}_{n_k}=\int_{[0,1]} f \circ \xi_k\,  \d \Leb \to  \int_{[0,1]} f \circ \xi_0\,\d \Leb= \int_{\bZ_T} f\,\d\mathbb{P}_T,
\end{align}
    because $\Leb([0,1] \setminus \xi_0^{-1}(B))=0$ and  $f \circ \xi_k \to f \circ \xi_0 $ pointwise on $\xi_0^{-1}(B)$. 
    The proof is complete. 
\end{proof}

Recall that $\nU= \rU_0 \times \newone{V}$ and let $\bX$ be the canonical process on $C([0,T];\nU^\prime)$ defined by
\[
\bX: [0,T] \times  C([0,T];\nU^\prime) \ni (t,w) \mapsto \bX_t(w)= w(t) \in \nU^\prime.  
\]
We also define the following objects 
\begin{equation}\label{eqn-D_t-D_t+}
\begin{aligned}
\mathscr{D}_t&= \sigma(\bX_s,\,s \leq t), \\
\mathscr{D}_{t +}&= \bigcap_{s >t} \sigma \left( \mathscr{D}_s \cup \mathscr{N}\right),
\end{aligned}
\end{equation}
where $\mathscr{N}$ is the family of all $\mathbb{P}_T$-null sets. 
We finally define the filtration $\mathbb{D}$ by 
\begin{equation}\label{eqn-mb-D}
\mathbb{D}= \left(\mathscr{D}_{t +}:  t \in [0,T) \right).
\end{equation}
Let the process $\bX_n$ be the solution to Problem \eqref{eqn-Compact-Galerkin-Modified-stochastic-CHNSEs-n}, and recall that $\mathbb{P}_n$ is the law of the process $\bX_n$ on the path space $\bZ_T$.  
We denote by $\mathbb{E}^{\mathbb{P}_n}$ the expectation with respect to the measure $\mathbb{P}_n$.
The first part of the following result follows from \eqref{eqn-Compact-Galerkin-Modified-stochastic-CHNSEs-n}.
\begin{proposition}\label{prop-martingale-1}
The process  
\[
 \bX_n(t)-\bX_n(0)-\int_0^t \bF_n(\bX_n(s))\, \d s =\int_0^t \newG_{0,n}(s,\bX_n(s))\,\d W(s),\;\; t \in [0,T],
\]
is a square integrable $\mathbb{H}_n$-valued martingale on the original probability space. 
\newline
If $f: \mathbb{R} \to \mathbb{C}$ is a $C^2_b$-class function and   $z \in \nU$, then the  following process 
    \begin{equation}\label{eqn-h(X_n)}
      h(\bX_n(t)) - h(\bX_n(0))-\int_0^t \mathscr{L}_{n,s}(h)(\bX_n(s))\,\d s,\;\; t \in [0,T],
   \end{equation}
where   the function $h:\nU^\prime \to \mathbb{C}$ has the following form,
\begin{equation}\label{eqn-h}
    h(\bX)=f(\duality{\bX}{z}{}{}),\;\; \bX \in \nU^\prime
    \end{equation} 
 and $\mathscr{L}_{n,s}$ is given by
     \begin{equation}\label{eqn-L_n}
         \mathscr{L}_{n,s} (h)(\bX) \coloneqq D h(\bX) \left[\bF_n(\bX)\right] 
         + \frac12 \sum_{k=1}^\infty D^2 h(\bX) \Bigl[[\newG_{0,n}(s,\bX)](\tilde{e}_k), [\newG_{0,n}(s,\bX)](\tilde{e}_k)\Bigr],\;\; \bX \in \nU^\prime,  
     \end{equation}
is a martingale. \newdela{NOTE: as written the sentence reads ``the function $h$ \dots\ and $\mathscr{L}_{n,s}$ \dots\ is a martingale'', but the martingale is the \emph{process} \eqref{eqn-h(X_n)}; $h$ and $\mathscr{L}_{n,s}$ are just its ingredients. Recommend rephrasing: ``\dots\ the process \eqref{eqn-h(X_n)}, where $h$ is given by \eqref{eqn-h} and $\mathscr{L}_{n,s}$ by \eqref{eqn-L_n}, is a martingale.''}
\end{proposition}
\begin{remark}\label{rem-prop-martingale-1}
If $h$ is of the form \eqref{eqn-h}, then 
\begin{equation}\label{eqn-L_n-2}
\begin{aligned}
\mathscr{L}_{n,s}(h)(\bX)
&= 
f^\prime(\duality{\bX}{z}{}{}) \duality{\bF_n(\bX)}{z}{}{} + \frac12 f^{\prime\prime} (\duality{\bX}{z}{}{}) \sum_{k=1}^\infty \dualitybig{[\newG_{0,n}(s,\bX)](\tilde{e}_k)}{z}{}{}{}{}{^2}
\\
&=: \widetilde{\mathscr{L}}_{n,s}(f)(\bX),\;\; z \in \nU,\; \; \bX \in \nU^\prime.
\end{aligned}
\end{equation}
\end{remark}
Note that $\widetilde{\mathscr{L}}_{n,s}$ is a linear operator that transforms functions of $C^2_b$-class on $\mathbb{R}$ into $C_b$-functions on $\nU^\prime$. In particular, if  $\theta \in \mathbb{R}$, $\tilde{e}_\theta:\mathbb{R} \ni x \mapsto e^{i \theta x} \in \mathbb{C}$,
and $h$ of the form \eqref{eqn-h} with $f$ replaced by $\tilde{e}_\theta$, then we have 
\begin{equation}
\label{eqn-L_n-3}
\begin{aligned}
\mathscr{L}_{n,s}(h)(\bX)
&= i\theta \tilde{e}_\theta (\duality{\bX}{z}{}{}) \duality{\bF_n(\bX)}{z}{}{} 
  - \frac{\theta^2}{2}  \tilde{e}_\theta (\duality{\bX}{z}{}{}) \sum_{k=1}^\infty 
  \dualitybig{[\newG_{0,n}(s,\bX)](\tilde{e}_k)}{z}{}{}{}{}{^2}
\\&
=: \widetilde{\mathscr{L}}_{n,s}(\tilde{e}_\theta)(\bX),\;\; z \in \nU,\; \; \bX \in \nU^\prime.
\end{aligned}
\end{equation}
\begin{corollary}\label{cor-martingale-1}
If $f: \mathbb{R} \to \mathbb{C}$ is a $C^2_b$-class function, then for every $z \in \nU$, 
the processes
   \begin{equation}\label{eqn-f(X_n)}
      f(\duality{\bX_n(t)}{z}{}{})  - \int_0^t \widetilde{\mathscr{L}}_{n,s}(f)(\bX_n(s))\,\d s,\;\; t \in [0,T],
   \end{equation}
and 
   \begin{equation}\label{eqn-e_theta(X_n)}
      \tilde{e}_\theta(\duality{\bX_n(t)}{z}{}{})-\int_0^t \widetilde{\mathscr{L}}_{n,s}(\tilde{e}_\theta)(\bX_n(s))\,\d s,\;\; t \in [0,T],
   \end{equation}
are martingales on the original probability space.
\end{corollary}
Observe that 
\begin{align}\label{eqn-17001}
\widetilde{\mathscr{L}}_{n,s}(\tilde{e}_\theta)(\bX))
&= e^{i \theta \langle\bX,z \rangle} \left[i \theta\; \duality{\bF_n(\bX)}{z}{}{} - \frac{\theta^2}{2} \Vert \duality{\newG_{0,n}(s,\bX)}{z}{}{} \Vert_{l^2}^2\right].
\end{align}
\begin{proposition}\label{prop-measurable} Assume that $z \in \nU$. If $0 \leq s \leq t \leq T$, we define the following maps
\begin{align}
\label{eqn-I_st-n}
I_{s,t}^{n,z} &: \bZ_T \ni \bX \mapsto \int_{s}^{t} e^{i \langle \bX(r), z \rangle} \varphi_n^{z}(r,\bX(r)) \, \d r \in \mathbb{C},
\\
\label{eqn-I_st}
I_{s,t}^{z} &: \bZ_T \ni \bX \mapsto \int_{s}^{t} e^{i \langle \bX(r), z \rangle} \varphi^{z}(r,\bX(r)) \, \d r \in \mathbb{C},
\end{align}
where
\begin{align}
\label{eqn-varphi-n}
\varphi_n^{z}(r,\bX(r))&= i \; \duality{\bF_n(\bX(r))}{z}{}{} 
- \frac{1}{2} \Vert \duality{\newG_{0,n}(r,\bX(r))}{z}{}{} \Vert_{\ell^2}^2,
\\
\label{eqn-varphi}
\varphi^{z}(r,\bX(r))&= i \; \duality{\bF(\bX(r))}{z}{}{} 
- \frac{1}{2} \Vert \duality{\newG_0(r,\bX(r))}{z}{}{} \Vert_{\ell^2}^2,
\end{align}
with
   \begin{align}
     \newG_0: [0,T] \times \mathbb{V} \ni (t,\bX=(\bu,\phi))\mapsto (\bSi_0(t)\bu + \bG_0(t,\bu), 0)_{k=1}^\infty\coloneqq (\newG_{0}^k(t,\bu))_{k=1}^\infty \in 
     \ell^2(\nHB).
  \end{align}
Then the maps $I_{s,t}^{n,z}$  and $I_{s,t}^{z}$ are measurable.  
\end{proposition}
\begin{proof}[Proof of Proposition \ref{prop-measurable}]
The proof can be done in a similar way to the measurability assertion about (5.22) in \cite{BKMR_2025}.
\end{proof}
\begin{definition}\label{def-M^nz_t}
Assume that  $z \in \nU$.
If $0 \leq  t \leq T$, we define the following maps
    \begin{align}
        \label{eqn-M^nz_t}
         M_{t}^{n,z} &: \bZ_T \ni \bX \mapsto e^{i \langle \bX(t), z \rangle} - I_{0,t}^{n,z}(\bX) \in \mathbb{C},
         \\
         M_{t}^{z} &: \bZ_T \ni \bX \mapsto e^{i \langle \bX(t), z \rangle} - I_{0,t}^{z}(\bX) \in \mathbb{C}.
                \label{eqn-M^z_t}
    \end{align}
\end{definition}
\begin{remark}\label{rem-I_st}
It follows from \eqref{eqn-17001} that almost surely on the original probability space, 
    \begin{align}\label{eqn-17002}
       M_t^{n,z}(\bX_n)=\tilde{e}_\theta(\duality{\bX_n(t)}{z}{}{}) -\int_0^t \widetilde{\mathscr{L}}_{n,s}(\tilde{e}_\theta)(\bX_n(s))\,\d s
    \end{align}
with $\theta=1$. Note that the LHS of \eqref{eqn-17002} makes sense  as we can view the strong solution $\bX_n$ as a measurable map from $\Omega$ to $\bZ_T$.
\newline
Moreover, it follows from identity \eqref{eqn-17002} and Proposition \ref{prop-martingale-1} that the process 
$M_t^{n,z}(\bX_n)$, defined by the LHS of equality \eqref{eqn-17002}, is a martingale on the original probability space.
\end{remark}
The following lemma concerns the solution of the martingale problem $(\bX_{0,n},\bF_{n}(\bX),\newG_{0,n}(\bX))$. We refer the reader to Appendix \ref{app-Solution-Martingale-problem} for further details on the solution of a martingale problem.
Let $T>0$ be fixed. We must bear in mind the Definition \ref{def-M^nz_t}  of the process $M_t^{n,z}$  defined on the path space $\bZ_T$.
\begin{lemma}\label{Lem-Pn-martingale-solution}
Let us define the following space 
\begin{equation}\label{eqn-space-V-system}
\mathscr{V}= \mathcal{V} \times \mathcal{C}_0^\infty(\domO,\mathbb{R}),    
\end{equation}
where $\mathcal{V}$ was defined in \eqref{eqn-spaces-NSEs}.
For each $n \in \mathbb{N}$, $\mathbb{P}_n$ is a Borel probability measure on $\bZ_T$ such that for each test function $z\in \nU$, the process $(M_t^{n,z}: t\in [0,T])$  
is a  continuous local martingale on the probability space $(\bZ_T,\mathbb{D},\mathbb{P}_n)$. 
\end{lemma}
\begin{proof}[Proof of Lemma \ref{Lem-Pn-martingale-solution}]
Let $s,t \in [0,T]$ be fixed such that $s \leq t$. We aim to show that
     \begin{align}\label{eqn-mart-13001}
       \mathbb{E}^{\mathbb{P}_n} \left(M_t^{n,z}-M_s^{n,z}\vert \mathscr{D}_s\right)=0.
     \end{align}
What we have proved is that 
\begin{align}\label{eqn-mart-13003}
    \mathbb{E} \left(e^{i \langle \bX_n(t), z \rangle}-e^{i \langle \bX_n(s), z \rangle} - I_{s,t}^{n,z}(\bX_n) \vert \mathscr{F}_s\right)=0,
\end{align}
where $I_{s,t}^{n,z}: \bZ_T \to \mathbb{C}$. \\ 
Recall that the law of the random object $\bX_n : \Omega \to \bZ_T$ is equal to $\mathbb{P}_n$.
Hence, equality \eqref{eqn-mart-13001} follows from \eqref{eqn-mart-13003} in the same way as the assertion between $(5.16)$ and $(5.17)$ in \cite{Brz+Motyl_2013}; 
see also an unproved assertion at the beginning of the proof of \cite[Corollary 2.8]{Mik+Roz_2005}. 
For a more detailed argument, we refer the reader to the proof of part $(i)$ in \cite[Lemma 5.8]{BKMR_2025}.
\end{proof}
The following auxiliary result is essential for the proof of Theorem \ref{thm-5.3}.
\begin{proposition}\label{Prop-equicontinuity}
Suppose the assumptions stated in Section \ref{Ass-Abstract formulation} hold and that $\bX_n\coloneqq(\bu_n,\phi_n) \to (\bu,\phi)\coloneqq \bX$ in $\bZ_T$. Then, the sequence $M_\cdot^{n,z}(\bX_n)$ is equicontinuous with respect to $t$. Moreover, for all $t \in[0,T]$ and $z \in \mathscr{V}$, it holds that
 \begin{equation}\label{eq-convergence-result-5.12}
     \begin{aligned}
       \lim_{n \to \infty} M_t^{n,z}(\bX_n)= M_t^{z}(\bX).
     \end{aligned}
   \end{equation}
\end{proposition}
\begin{proof}[Proof of Proposition \ref{Prop-equicontinuity}]
Throughout this proof, $C$ denotes a generic positive constant that does not depend on $n$, but may depend on $\domO,\, \newK,\, C_3,\, C_4,\, \Vert z_1 \Vert_{\rU_0}$, $\Vert z_2 \Vert_{\newone{V}}$. 
In some cases, we use the same symbol to denote different constants within the same argument to simplify the notation.
\newline
Let us choose and fix $\bX_n \to \bX$ in $\bZ_T$ and $z \in L^{\beta^\prime}(0,T;\mathbb{V}^\prime)$ with $\beta^\prime$ being the dual of $\beta$. Definition \eqref{eqn-Z_T} implies that
\begin{equation}\label{eqn-lim=0}
    \begin{aligned}
     \lim_{n \to \infty} \sup_{t \in [0,T]} \Vert \bu_n(t)-\bu(t) \Vert_{\rU_0^\prime}=0,
       \\
     \lim_{n \to \infty} \sup_{t \in [0,T]} \Vert \phi_n(t)-\phi(t) \Vert_{\newonep{V}}=0.
    \end{aligned}
\end{equation}
The above imply that the sequence $(\bX_n)_{n \in \mathbb{N}}$ is equicontinuous in $C([0,T],\nU^\prime)$.
Moreover, 
\begin{equation}\label{eq-5.9}
\begin{aligned}
\sup_{n\geq 1} \left( \int_0^T (\lvert \phi_n(s) \rvert_{\newone{H}}^4 +  \Vert  \phi_n(s) \Vert_{\newone{V}}^{\beta} + \lvert \phi(s) \rvert_{\newone{H}}^4 +  \Vert  \phi(s) \Vert_{\newone{V}}^{\beta}) \, \d s \right)< \infty,
\\
\sup_{n\geq 1} \left(\sup_{s \in [0,T]} \Vert \bX_n(s) \Vert_{\nHB} + \sup_{s \in [0,T]} \Vert \bX(s) \Vert_{\nHB} + \int_0^T (\Vert \bX_n(s) \Vert_{\StokesV \times \zero{H}{2}}^2 + \Vert \bX(s) \Vert_{\StokesV \times \zero{H}{2}}^2) \, \d s \right)< \infty.
\end{aligned}
\end{equation}
Furthermore, by the definition of the space $\bZ_T$ and Corollary \ref{Coro-convergence-phi_n-in-H_2}, we infer that
\begin{align}
\label{eq-5.10}
\Vert \bX_n - \bX \Vert_{L^2(0,T;\nHB)}^2 \underset{n \to \infty}{\to}  0,
\\
\label{eq-5.11}
\int_0^T \duality{\bX_n(s) - \bX(s)}{z(s)}{\mathbb{V}}{\mathbb{V}^\prime} \, \d s \underset{n \to \infty}{\to}  0,
\\
\label{eq-5.12}
\Vert \phi_n - \phi \Vert_{L^2(0,T;\zero{H}{2}(\domO))} \underset{n \to \infty}{\to}  0.
\end{align}
Next, fix $z=(z_1,z_2) \in \mathscr{V}$. 
From the equicontinuity properties established above, the sequence of complex-valued functions 
$e^{i \langle \bX_n, z \rangle}$ is equicontinuous.
The remainder of the proof of Proposition \ref{Prop-equicontinuity} is divided into two parts.
\subsection*{Part 1}
We show that the sequence $(M_{\cdot}^{n,z}(\bX_n))_{n \in \mathbb{N}}$ is equicontinuous. 
First, notice that for almost $s \in [0,T]$,
\begin{align*}
&\lvert \varphi_n^{z}(s,\bX_n(s)) \rvert
\leq \lvert \duality{\bF_n(\bX_n(s))}{z}{}{} \rvert 
+ \frac{1}{2} \Vert \duality{\newG_{0,n}(s,\bX_n(s))}{z}{}{} \Vert_{\ell^2}^2
\\
&= \lvert \varphi_n^{z_1}(s,\bX_n(s)) + \varphi_n^{z_2}(s,\bX_n(s)) \rvert
+ \frac{1}{2} \Vert \duality{\newG_{0,n}(s,\bX_n(s))}{z}{}{} \Vert_{\ell^2}^2,
\end{align*}
where
    \begin{equation*}
     \varphi_n^{z_1}(s,\bX_n(s))
       = \nu \duality{\Stokes_{,n} \bu_n(s)}{z_1}{\rU_0}{\rU_0^\prime} + \duality{\bB_{0,n}(\bu_n(s),\bu_n(s))}{z_1}{\rU_0}{\rU_0^\prime}
        + \newK  \duality{\bar{R}_{0,n}(\phi_n(s),\phi_n(s))}{z_1}{\rU_0}{\rU_0^\prime},
   \end{equation*}
and
\begin{align*}
\varphi_n^{z_2}(s,\bX_n(s))= \duality{- B_{1,n}(\bu_n(s),\phi_n(s))}{z_2}{\newone{V}}{\newonep{V}} + (\pi_{1,n} \cp_n(s), -\Delta z_2)_{\newone{H}},
\end{align*}
where $\cp_n(t)=  \Atwo \phi_n(t) + \pi_{1,n} (\psi^\prime(\phi_n(t)) - \avg{\psi^\prime(\phi_n(t))}),\;\; t \in[0,T]$.
We now turn to the estimates for $\varphi_n^{z_1}(s,\bX_n(s))$ and $\varphi_n^{z_2}(s,\bX_n(s))$.
Since $\Vert \Stokes \bu_n(s) \Vert_{\StokesVp} \leq C\Vert \bu_n(s) \Vert_{\StokesV}$, for every $n \in \mathbb{N}$, 
and $\StokesVp \embed \rU_0^\prime$, we deduce that for all $s \in [0,T]$,
\begin{align*}
\nu \lvert \duality{ \Stokes_{,n} \bu_n(s)}{z_1}{\rU_0}{\rU_0^\prime}\rvert
\leq \nu \Vert \pi_{0,n} \Vert_{\mathcal{L}(\rU_0^\prime)} \Vert \Stokes \bu_n(s) \Vert_{\rU_0^\prime} \Vert z_1 \Vert_{\rU_0}
\leq C \Vert \Stokes \bu_n(s) \Vert_{\StokesVp} \Vert z_1 \Vert_{\rU_0}
\leq C \Vert \bu_n(s) \Vert_{\StokesV}. 
\end{align*}
By the estimate \eqref{eq-4.1} in Corollary \ref{2nd-corollary}, we infer that for all $s \in [0,T]$,
\begin{align*}
&\lvert \duality{\bB_{0,n}(\bu_n(s),\bu_n(s))}{z_1}{\rU_0}{\rU_0^\prime} \rvert + \newK \lvert \duality{\bar{R}_{0,n}(\phi_n(s),\phi_n(s))}{z_1}{\rU_0}{\rU_0^\prime} \rvert
\\
&\leq (\Vert \bB_{0,n}(\bu_n(s),\bu_n(s)) \Vert_{\rU_0^\prime} + \newK \Vert \bar{R}_{0,n}(\phi_n(s),\phi_n(s)) \Vert_{\rU_0^\prime}) \Vert z_1 \Vert_{\rU_0} 
  \\
&\leq C(\lvert \bu_n(s) \rvert_{\StokesH}^2 + \lvert \cp_n(s) \rvert_{\newone{H}} \lvert \phi_n(s) \rvert_{\newone{H}}).
\end{align*}
It then follows from the two latter estimates that for every $s \in [0,T]$,
\[
\lvert \varphi_n^{z_1}(s,\bX_n(s)) \rvert
\leq C(\Vert \bu_n(s) \Vert_{\StokesV} + \lvert \bu_n(s) \rvert_{\StokesH}^2 + \lvert \cp_n(s) \rvert_{\newone{H}} \lvert \phi_n(s) \rvert_{\newone{H}}).
\]
Regarding the term $\varphi_n^{z_2}(s,\bX_n(s))$, we have for all $s \in [0,T]$,
\begin{align*}
&\lvert \varphi_n^{z_2}(s,\bX_n(s)) \rvert
= \lvert \duality{- B_{1,n}(\bu_n(s),\phi_n(s))}{z_2}{\newone{V}}{\newonep{V}} + (\pi_{1,n} \cp_n(s), -\Delta z_2)_{\newone{H}} \rvert 
\\
&\leq \Vert B_{1,n}(\bu_n(s),\phi_n(s)) \Vert_{\newonep{V}} \Vert z_2 \Vert_{\newone{V}} + \lvert \pi_{1,n}\cp_n(s) \rvert_{\newone{H}} \lvert -\Delta z_2 \rvert_{\newone{H}}
  \\
&\leq \Vert B_{1,n}(\bu_n(s),\phi_n(s)) \Vert_{\newonep{V}} \Vert z_2 \Vert_{\newone{V}} + \Vert \pi_{1,n} \Vert_{\mathcal{L}(\newone{H})} \lvert \cp_n(s)\rvert_{\newone{H}} \lvert -\Delta z_2 \rvert_{\newone{H}}
    \\
&\leq C \Vert \bu_n(s) \Vert_{\StokesV} \lvert \phi_n(s) \rvert_{\newone{H}} \Vert z_2 \Vert_{\newone{V}} + C \lvert \cp_n(s) \rvert_{\newone{H}} \Vert z_2 \Vert_{\newone{V}}.
\end{align*}
Consequently for all $s \in [0,T]$,
\begin{equation*}      
\lvert \varphi_n^{z_2}(s,\bX_n(s)) \rvert
\leq C(\lvert \cp_n(s) \rvert_{\newone{H}} + \Vert \bu_n(s) \Vert_{\StokesV} \lvert \phi_n(s) \rvert_{\newone{H}}).
\end{equation*}
Next, observe that
\begin{align*}
\Vert \duality{\newG_{0,n}(s,\bX_n(s))}{z}{}{} \Vert_{\ell^2}^2
=  \sum_{k=1}^\infty \dualitybig{[\newG_{0,n}(s,\bX_n(s))](\tilde{e}_k)}{z}{}{}{}{}{^2}
\leq \Vert \newG_{0,n}(s,\bX_n(s)) \Vert_{\ell^2(\nU^\prime)}^2 \Vert z \Vert_{\nU}^2, 
\end{align*}
and by \eqref{eq-4.1} we infer that for all $s \in [0,T]$,
\begin{equation}\label{eq-5.14b}
\Vert \newG_{0,n}(s,\bX_n(s)) \Vert_{\ell^2(\nU^\prime)}^2
\leq C (\lvert \bu_n(s) \rvert_{\StokesH}^2 + \lvert \tilde{h}(s) \rvert^2 ).
\end{equation}
Therefore, for all $s \in [0,T]$ and $n \in \mathbb{N}$, we have
\begin{equation}\label{eq-5.14a}
\begin{aligned}
&\lvert \varphi_n^{z}(s,\bX_n(s)) \rvert
\leq \lvert \varphi_n^{z_1}(s,\bX_n(s)) \rvert + \lvert \varphi_n^{z_2}(s,\bX_n(s)) \rvert
+ \frac{1}{2} \Vert \duality{\newG_{0,n}(s,\bX_n(s))}{z}{}{} \Vert_{\ell^2}^2
\\
&\leq C (\Vert \bu_n(s) \Vert_{\StokesV} + \lvert \bu_n(s) \rvert_{\StokesH}^2 + \lvert \cp_n(s) \rvert_{\newone{H}} \lvert \phi_n(s) \rvert_{\newone{H}} + \lvert \tilde{h}(s) \rvert^2
  \\
&\hspace{1cm} + \lvert \cp_n(s) \rvert_{\newone{H}} + \Vert \bu_n(s) \Vert_{\StokesV} \lvert \phi_n(s) \rvert_{\newone{H}}).
\end{aligned}
\end{equation}
We now examine the terms on the RHS of \eqref{eq-5.14a}. First, in view of \eqref{eq-5.9}, we see that
   \begin{equation*}
     \sup_{n} \int_0^T \Vert \bu_n(s) \Vert_{\StokesV}^2\, \d s< \infty.
   \end{equation*}
From \eqref{eq-5.9}, we deduce that 
   \begin{equation*}
     \int_0^T \lvert \bu_n(s) \rvert_{\StokesH}^4 \,\d s 
     \leq C \left[\sup_{s \in [0,T]} \lvert \bu_n(s) \rvert_{\StokesH}^2 \int_0^T \Vert \bu_n(s) \Vert_{\StokesV}^2 \,\d s \right]< \infty.
   \end{equation*}
Next, by virtue of estimates \eqref{Eq-nabla-mu} and \eqref{eq-5.9}, we have  
\begin{equation*}
\sup_n \int_0^T [\lvert \cp_n(s) \rvert_{\newone{H}} \lvert \phi_n(s) \rvert_{\newone{H}}]^{4/3}\, \d s
\leq \sup_n \left[ \sup_{s \in [0,T]} \lvert \phi_n(s) \rvert^{4/3}_{\newone{H}}  \int_0^T \lvert \cp_n(s) \rvert_{\newone{H}}^{4/3}\, \d s \right]
< \infty.
\end{equation*}
On the other hand, thanks to \eqref{Eq-nabla-mu} and \eqref{eq-5.9}, we deduce that 
        \begin{equation*}
          \sup_{n} \int_0^T \lvert \cp_n(s) \rvert_{\newone{H}}^{4/3}\,\d s< \infty
        \end{equation*}
and 
     \begin{equation*}
        \sup_{n} \int_{0}^{T} \Vert \bu_n(s) \Vert_{\StokesV}^2 \lvert \phi_n(s) \rvert_{\newone{H}}^2\,\d s 
        \leq  \sup_{n} \left[\sup_{s \in[0,T]} \lvert \phi_n(s) \rvert_{\newone{H}}^2  \int_{0}^{T} \Vert \bu_n(s) \Vert_{\StokesV}^2\,\d s \right]< \infty.
     \end{equation*}
We therefore deduce that the sequences of functions on the RHS of \eqref{eq-5.14a} are uniformly bounded in $L^{4/3}(0,T)$.
Furthermore, since the function $\tilde{h} \in L^2(0,T)$, the map
   \begin{align*}
     f: [0,T] \ni s \mapsto f(s)= \lvert \tilde{h}(s) \rvert^2
   \end{align*}
is an element of space $L^1(0,T)$. Thus, $f$ is absolutely continuous wrt Leb-measure. 
Consequently, by Corollary \ref{cor-uniformly integrable linear combination}, the sequence $(M_{\cdot}^{n,z}(\bX_n))_{n \in \mathbb{N}}$ is equicontinuous.
\subsection*{Part 2}
Now, let us turn to the proof of \eqref{eq-convergence-result-5.12}. We have
\begin{align*}
M_t^{n,z}(\bX_n) - M_t^{z}(\bX)
&= e^{i \langle \bX_n(t), z \rangle} - e^{i \langle \bX(t), z \rangle}  
- \int_0^t \left[e^{i \langle \bX_n(s), z \rangle} -e^{i \langle \bX(s), z \rangle}\right]\varphi_n^{z}(s,\bX_n(s))\, \d s \\
&\qquad - \int_0^t e^{i \langle \bX(s), z \rangle} \left[\varphi_n^{z}(s,\bX_n(s)) - \varphi^{z}(s,\bX(s))\right] \d s. 
\end{align*}
Subsequently, we have
\begin{align*}
\left|e^{i \langle \bX_n(t), z \rangle} - e^{i \langle \bX(t), z \rangle} \right|
&= \left \lvert 2 i \sin \left(\frac{\langle \bX_n(t) - \bX(t), z \rangle}{2} \right) e^{i \frac{1}{2} \langle \bX_n(t) + \bX(t), z \rangle} \right \rvert
\leq \lvert \langle \bX_n(t) - \bX(t), z \rangle \rvert,
\end{align*}
and then thanks to the assumption \eqref{eqn-lim=0}, we deduce that
\begin{align}\label{eq-5.21}
\sup_{t \in [0,T]} \left \lvert e^{i \langle \bX_n(t), z \rangle} - e^{i \langle \bX(t), z \rangle} \right \rvert \underset{n \to \infty}{\to}  0.
\end{align}
By \eqref{eq-5.21} and since the sequence $(\varphi_n^{z}(\cdot,\bX_n))_n$ is uniformly bounded, we infer that
   \begin{equation}\label{eq-5.25}
     \lim_{n \to \infty} \int_0^t \left \lvert \left[e^{i \langle \bX_n(s), z \rangle} -e^{i \langle \bX(s), z \rangle}\right] \varphi_n^{z}(s,\bX_n(s)) \right \rvert \,\d s= 0.
   \end{equation}
To complete the proof of Part 2, it remains to show that
    \begin{equation}\label{Eqn-convergence-varphi_n^{z}}
        \lim_{n \to \infty} \int_0^t e^{i \langle \bX(s), z \rangle} \left[\varphi_n^{z}(s,\bX_n(s)) - \varphi^{z}(s,\bX(s))\right]\d s= 0. 
    \end{equation}
Given the process $\phi= (\phi(t): t \in [0,T])$ from Proposition \ref{Prop-equicontinuity}, we introduce the auxiliary process $\cp= (\cp(t): t \in [0,T])$:
\[
\cp(t)\coloneqq \Atwo \phi(t) + \psi^\prime(\phi(t)) - \avg{\psi^\prime(\phi(t))}, \; \; t \in [0,T],
\]
and we note that for all $s \in [0,T]$,

\begin{align*}
&\varphi_n^{z}(s,\bX_n(s)) - \varphi^{z}(s,\bX(s))
\\
&= i \; \duality{\bF_n(\bX_n(s)) - \bF(\bX(s))}{z}{\nU}{\nU^\prime} 
- \frac12\left[\Vert \duality{\newG_{0,n}(s,\bX_n(s))}{z}{}{} \Vert_{\ell^2}^2 - \Vert \duality{\newG_0(s,\bX(s))}{z}{}{} \Vert_{\ell^2}^2 \right]
   \\
&= - i \nu \duality{\Stokes_{,n} \bu_n(s) - \Stokes \bu(s)}{z_1}{\rU_0}{\rU_0^\prime}  
   -  i \duality{\bB_{0,n}(\bu_n(s),\bu_n(s)) - \bB_0(\bu(s),\bu(s))}{z_1}{\rU_0}{\rU_0^\prime} 
      \\
& + i \duality{ \newK \bar{R}_{0,n}(\phi_n(s),\phi_n(s)) - \newK \bR_0(\phi(s),\phi(s))}{z_1}{\rU_0}{\rU_0^\prime}
       - i \duality{B_{1,n}(\bu_n(s),\phi_n(s)) - B_{1}(\bu(s),\phi(s))}{z_2}{\newone{V}}{\newonep{V}}
\\
& + i (\pi_{1,n} \cp_n(s) - \cp(s), -\Delta z_2)_{\newone{H}} - \frac12 \left[ \Vert \duality{\newG_{0,n}(s,\bX_n(s))}{z}{}{} \Vert_{\ell^2}^2 - \Vert \duality{\newG_0(s,\bX(s))}{z}{}{} \Vert_{\ell^2}^2 \right].
\end{align*}
Next, notice that since $\pi_{0,n} \circ \P= \pi_{0,n}$,
\begin{align*}
&\duality{\Stokes_{,n}\bu_n}{z_1}{\rU_0}{\rU_0^\prime}
=\duality{\Stokes(\pi_{0,n} \bu_n)}{z_1}{\rU_0}{\rU_0^\prime}
=\int_{\domO} \nabla (\pi_{0,n} \bu_n) \cdot \nabla z_1\,\d x
= \int_{\domO} \pi_{0,n} \bu_n \cdot \Stokes z_1\,\d x\\
&=(\pi_{0,n} \bu_n, \Stokes z_1)
=(\pi_{0,n}(\bu_n-\bu), \Stokes z_1) + (\pi_{0,n}\bu, \Stokes z_1),
\end{align*}
from which we infer
\begin{align*}
\duality{\Stokes_{,n}\bu_n - \Stokes \bu}{z_1}{\rU_0}{\rU_0^\prime}
= (\pi_{0,n}(\bu_n-\bu), \Stokes z_1) + (\pi_{0,n}\bu - \bu, \Stokes z_1).
\end{align*}
From this last equality, we deduce that for every $n \in \mathbb{N}$,
\begin{align*}
&\lvert \duality{\Stokes_{,n} \bu_n(s) - \Stokes \bu(s)}{z_1}{\rU_0}{\rU_0^\prime} \rvert
\leq \lvert (\pi_{0,n}(\bu_n-\bu), \Stokes z_1) \rvert + \lvert (\pi_{0,n}\bu - \bu, \Stokes z_1) \rvert
\\
&\leq \lvert \pi_{0,n}(\bu_n-\bu) \rvert \lvert \Stokes z_1 \rvert + \lvert \pi_{0,n}\bu - \bu \rvert \lvert \Stokes z_1 \rvert
\leq \Vert \pi_{0,n} \Vert_{\mathcal{L}(\mathbb{L}^2)} \lvert \bu_n-\bu \rvert \lvert \Stokes z_1 \rvert + \lvert \pi_{0,n}\bu - \bu \rvert \lvert \Stokes z_1 \rvert
\\
&\leq C (\lvert \bu_n-\bu \rvert \Vert z_1 \Vert_{\rU_0} + \lvert \pi_{0,n}\bu - \bu \rvert \Vert z_1 \Vert_{\rU_0}).
\end{align*}
Therefore, in view of of the assumption \eqref{eq-5.10}, we infer that
\begin{align*}
\lim_{n \to \infty} \int_0^t \lvert \duality{\Stokes_{,n} \bu_n(s) - \Stokes \bu(s)}{z_1}{\rU_0}{\rU_0^\prime} \rvert\, \d s=0,
\end{align*}
which, in turn, implies that
    \begin{equation*}
       \lim_{n \to \infty} \int_0^t e^{i \duality{\bX(s)}{z}{\nU}{\nU^\prime}} \duality{\Stokes_{,n} \bu_n(s) - \Stokes \bu(s)}{z_1}{\rU_0}{\rU_0^\prime}\,\d s=0.
     \end{equation*} 
By \eqref{eq-5.9}, \eqref{eq-5.10}, and Lemma \ref{B_on-Convergence}, we obtain
\begin{equation}
\lim_{n \to \infty} \int_0^t e^{i \duality{\bX(s)}{z}{\nU}{\nU^\prime}} \duality{\bB_{0,n}(\bu_n(s),\bu_n(s)) - \bB_0(\bu(s),\bu(s))}{z_1}{\rU_0}{\rU_0^\prime}\,\d s=0.
\end{equation}
Combining \eqref{eq-5.9}, \eqref{eq-5.10}, and Lemma \ref{R_on-Convergence}, we infer that
\begin{equation}
\lim_{n \to \infty} \int_0^t e^{i \duality{\bX(s)}{z}{\nU}{\nU^\prime}} \duality{\bar{R}_{0,n}(\phi_n(s),\phi_n(s)) - \bR_0(\phi(s),\phi(s))}{z_1}{\rU_0}{\rU_0^\prime}\, \d s=0. 
\end{equation}
Using \eqref{eq-5.9}, \eqref{eq-5.10}, and Lemma \ref{B_1n-Convergence}, we obtain
   \begin{equation}\label{eq-5.29a}
     \lim_{n \to \infty} \int_0^t e^{i \duality{\bX(s)}{z}{\nU}{\nU^\prime}} \duality{B_{1,n}(\bu_n(s),\phi_n(s))- B_1(\bu(s),\phi(s))}{z_2}{\newone{V}}{\newonep{V}}\,\d s= 0.
   \end{equation}
Our second step is to establish that 
  \begin{equation}\label{eq-A1-mu-n-convergence}
    \lim_{n \to \infty} \int_0^t e^{i \duality{\bX(s)}{z}{\nU}{\nU^\prime}} (\pi_{1,n} \cp_n(s) - \cp(s), -\Delta z_2)_{\newone{H}}\,\d s=0.
  \end{equation}
Regarding the proof of \eqref{eq-A1-mu-n-convergence}, we first state and prove the following auxiliary claims.
\begin{claim}\label{eqn-first-claim-for-bar-mu_n}
Assume that 
\[
r \coloneqq \sup_{n \in \mathbb{N}} \Vert \phi_n \Vert_{L^{3/2}(0,T;\newone{V})} + \Vert \phi \Vert_{L^{3/2}(0,T;\newone{V})}<\infty.
\]
and 
\[
\phi_n \to \phi \in L^2(0,T;\zero{H}{2}(\domO)).
\]
Then the  following convergences hold.
\begin{trivlist}
\item[(i)] $(\pi_{1,n} -I) \Atwo (\phi_n - \phi) \rightharpoonup 0 \; \mbox{ weakly in } L^{4/3}(0,T;\newone{H})$ as $n \to \infty$.
\item[(ii)] $\Atwo (\phi_n - \phi) \rightharpoonup  0 \; \mbox{ weakly in } L^{4/3}(0,T;\newone{H})$ as $n \to \infty$.
\end{trivlist}
\end{claim}
\begin{proof}[Proof of Claim \ref{eqn-first-claim-for-bar-mu_n}]
Let us choose and fix a test function $\varphi \in L^4(0,T;\newone{V})$. Note that for every $n \in \mathbb{N}$,
\begin{align}\label{((pi_{1,n} -I) Atwo (phi_n - phi), varphi)_{newone{H}}}
& \int_0^T \left \lvert ((\pi_{1,n} -I) \Atwo (\phi_n - \phi), \varphi)_{\newone{H}} \right \rvert\,\d s 
=  \int_0^T \left \lvert (\Atwo (\phi_n - \phi), (\tilde{\pi}_{1,n} -I) \varphi)_{\newone{H}} \right \rvert\,\d s \notag 
\\
&\leq \int_0^T \lvert \Atwo (\phi_n - \phi) \rvert_{\newone{H}} \lvert (\tilde{\pi}_{1,n} -I) \varphi \rvert_{\newone{H}}\,\d s
\leq C \int_0^T \Vert \phi_n - \phi \Vert_{\newone{V}} \lvert (\tilde{\pi}_{1,n} -I) \varphi \rvert_{\newone{H}}\,\d s
\\
&\leq C \Vert \phi_n - \phi \Vert_{L^{4/3}(0,T;\newone{V})} \Vert (\tilde{\pi}_{1,n} -I) \varphi \Vert_{L^4(0,T;\newone{H})}
\leq C \Vert \phi_n - \phi \Vert_{L^{4/3}(0,T;\newone{V})} \Vert (\tilde{\pi}_{1,n} -I) \varphi \Vert_{L^4(0,T;\newone{V})}. \notag
\end{align}
Moreover, by the item $(iv)$ in Properties \ref{eqn-properties-pi_{1,n}}, we infer that 
\[
\lim_{n \to \infty} \Vert (\tilde{\pi}_{1,n} -I) \varphi(t) \Vert_{\newone{V}}=0, \mbox{  for almost every } t \in [0,T].
\]
Furthermore, since $\Vert \tilde{\pi}_{1,n} \Vert_{\mathcal{L}(\newone{V})}\leq 1$, we deduce that the sequence $(\tilde{\pi}_{1,n}\varphi)_{n \in \mathbb{N}}$ is uniformly bounded in $L^4(0,T;\newone{V})$. Hence, by DCT, we infer that
\[
\lim_{n \to \infty} \Vert (\tilde{\pi}_{1,n} -I) \varphi \Vert_{L^4(0,T;\newone{V})}=0.
\]
In addition, by assumption, the sequence $(\phi_n)_{n \in \mathbb{N}}$ is uniformly bounded in the topological space $L^{3/2}(0,T;\newone{V}) \subset L^{4/3}(0,T;\newone{V})$ and 
its limit object $\phi \in L^{3/2}(0,T;\newone{V})$. Therefore, we can pass to the limit in \eqref{((pi_{1,n} -I) Atwo (phi_n - phi), varphi)_{newone{H}}} and deduce that
\[
\lim_{n \to \infty}  \int_0^T \left\lvert ((\pi_{1,n} -I) \Atwo (\phi_n(s) - \phi(s)), \varphi(s))_{\newone{H}}\right \rvert\,\d s=0.
\]
Let us choose and fix a test function $\varphi \in L^4(0,T;\newone{H})$ and  $\epsilon>0$. 
Since the space $L^4(0,T;\newone{V})$ is dense in $L^4(0,T;\newone{H})$, there exists $\varphi_\eps  \in L^4(0,T;\newone{V})$ s.t. $\Vert \varphi - \varphi_\eps  \Vert_{L^4(0,T;\newone{H})} \leq \eps $. 
We have
\begin{equation}\label{((pi_{1,n} -I) Atwo (phi_n - phi), varphi)_{newone{H}}-1}
\begin{aligned}
& \lvert ((\pi_{1,n} -I) \Atwo (\phi_n - \phi), \varphi)_{\newone{H}} \rvert
\\
&=  \lvert ((\pi_{1,n} -I) \Atwo (\phi_n - \phi), \varphi - \varphi_\eps )_{\newone{H}} +  ((\pi_{1,n} -I) \Atwo (\phi_n - \phi), \varphi_\eps )_{\newone{H}} \rvert
\\
&=  \lvert ((\pi_{1,n} -I) \Atwo (\phi_n - \phi), \varphi - \varphi_\eps )_{\newone{H}} +  (\Atwo (\phi_n - \phi), (\tilde{\pi}_{1,n} -I) \varphi_{\varepsilon})_{\newone{H}} \rvert
\\
&\leq  \lvert (\pi_{1,n} -I) \Atwo (\phi_n - \phi) \rvert_{\newone{H}} \lvert \varphi - \varphi_\eps  \rvert_{\newone{H}} +  \lvert (\Atwo (\phi_n - \phi), (\tilde{\pi}_{1,n} -I) \varphi_{\varepsilon})_{\newone{H}} \rvert 
\\
&\leq \vert \pi_{1,n} -I \vert_{\mathcal{L}(\newone{H})}  \lvert \Atwo (\phi_n - \phi) \rvert_{\newone{H}} \lvert \varphi - \varphi_\eps  \rvert_{\newone{H}} +  \lvert (\Atwo (\phi_n - \phi), (\tilde{\pi}_{1,n} -I) \varphi_{\varepsilon})_{\newone{H}} \rvert
\\
&\leq 2  \lvert \Atwo (\phi_n - \phi) \rvert_{\newone{H}} \lvert \varphi - \varphi_\eps  \rvert_{\newone{H}} +  \lvert (\Atwo (\phi_n - \phi), (\tilde{\pi}_{1,n} -I) \varphi_{\varepsilon})_{\newone{H}} \rvert.
\end{aligned}
\end{equation}
Thanks to the H\"older inequality, we infer that for every $n \in \mathbb{N}$,
\begin{align*}
&\int_0^T \lvert ((\pi_{1,n} -I) \Atwo (\phi_n - \phi), \varphi)_{\newone{H}} \rvert\,\d s
\\
& \leq C \Vert \phi_n - \phi \Vert_{L^{3/2}(0,T;\newone{V})} \Vert \varphi - \varphi_\eps  \Vert_{L^4(0,T;\newone{H})} + \int_0^T \lvert (\Atwo (\phi_n - \phi), (\tilde{\pi}_{1,n} -I) \varphi_{\varepsilon})_{\newone{H}} \rvert\,\d s
\\
&\leq C r \eps  + \int_0^T \lvert (\Atwo (\phi_n - \phi), (\tilde{\pi}_{1,n} -I) \varphi_{\varepsilon})_{\newone{H}} \rvert\,\d s.
\end{align*}
Therefore, passing to the upper limit as $n \to \infty$, we infer that
\[
\limsup_{n \to \infty} \int_0^T \lvert ((\pi_{1,n} -I) \Atwo (\phi_n(s) - \phi(s)), \varphi(s))_{\newone{H}} \rvert\,\d s
\leq C r \eps .
\]
So, by the arbitrariness of $\eps $, 
\[
\lim_{n \to \infty} \int_0^T \lvert ((\pi_{1,n} -I) \Atwo (\phi_n(s) - \phi(s)), \varphi(s))_{\newone{H}} \rvert\,\d s=0.
\]
Consequently,
\[
\lim_{n \to \infty} \int_0^T ((\pi_{1,n} -I) \Atwo (\phi_n(s) - \phi(s)), \varphi(s))_{\newone{H}}\,\d s=0.
\]
This completes the proof of the first part of Claim \ref{eqn-first-claim-for-bar-mu_n}.

\noindent
In order to prove the second  part of Claim \ref{eqn-first-claim-for-bar-mu_n} we choose and fix a test function $\varphi \in L^4(0,T;\newone{H})$ and  $\epsilon>0$. 
Using and integration by parts, we deduce that there exists a constant $C$ such that for all $n \in \mathbb{N}$, 
\begin{align*}
&\lvert (\Atwo (\phi_n - \phi),\varphi_{\varepsilon})_{\newone{H}} \rvert
= \lvert (\Atwo (\phi_n - \phi), \Atwo\varphi_{\varepsilon})_{L^2} \rvert
\\
&\leq \lvert \Atwo (\phi_n - \phi) \rvert_{L^2} \lvert \Atwo\varphi_{\varepsilon} \rvert_{L^2}
\leq C \Vert \phi_n - \phi \Vert_{\zero{H}{2}} \Vert \varphi_{\varepsilon} \Vert_{\newone{V}}.
\end{align*}
 Since by assumption $\phi_n \to \phi$ in $L^2(0,T;\zero{H}{2}(\domO))$ as $n \to \infty$, we infer that
\[
\lim_{n \to \infty} \int_0^T \lvert (\Atwo (\phi_n(s) - \phi(s)),\varphi_{\varepsilon}(s))_{\newone{H}} \rvert\,\d s=0.
\]
Observe also that for every $n \in \mathbb{N}$,
\begin{align*}
&\lvert (\Atwo (\phi_n - \phi),\varphi)_{\newone{H}} \rvert
= \lvert (\Atwo (\phi_n - \phi),\varphi - \varphi_\eps )_{\newone{H}} + (\Atwo (\phi_n - \phi),\varphi_\eps )_{\newone{H}} \rvert
\\
&\leq \lvert (\Atwo (\phi_n - \phi),\varphi - \varphi_\eps )_{\newone{H}} \rvert + \lvert (\Atwo (\phi_n - \phi),\varphi_\eps )_{\newone{H}} \rvert
\\
&\leq \lvert \Atwo (\phi_n - \phi) \rvert_{\newone{H}} \lvert \varphi - \varphi_\eps  \rvert_{\newone{H}} + \lvert (\Atwo (\phi_n - \phi),\varphi_\eps )_{\newone{H}} \rvert
\\
&\leq \Vert \phi_n - \phi \Vert_{\newone{V}} \lvert \varphi - \varphi_\eps  \rvert_{\newone{H}} + \lvert (\Atwo (\phi_n - \phi),\varphi_\eps )_{\newone{H}} \rvert,
\end{align*}
from which we deduce that
\begin{align*}
\int_0^T \lvert (\Atwo (\phi_n(s) - \phi(s)),\varphi(s))_{\newone{H}} \rvert\,\d s 
\leq C r \eps  + \int_0^T \lvert (\Atwo (\phi_n(s) - \phi(s)),\varphi_\eps )_{\newone{H}} \rvert\,\d s.
\end{align*}
Finally, by arguing similarly as in the proof of first part of Claim \ref{eqn-first-claim-for-bar-mu_n}, we infer that
\[
\lim_{n \to \infty} \int_0^T (\Atwo (\phi_n(s) - \phi(s)),\varphi(s))_{\newone{H}}\,\d s=0. 
\]
Thus the proof of the second part of Claim \ref{eqn-first-claim-for-bar-mu_n} is also complete.
\end{proof}
\begin{claim}\label{eqn-2nd-claim-for-bar-mu_n}
 $(\pi_{1,n} - I) \Atwo \phi \to 0$ strongly in $L^{3/2}(0,T;\newone{H})$ as $n \to \infty$.
\end{claim}
\begin{proof}[Proof of Claim \ref{eqn-2nd-claim-for-bar-mu_n}]
Observe that $\Atwo \phi \in L^{3/2}(0,T;\newone{H})$, because $\phi \in L^{3/2}(0,T;\newone{V})$ by assumption. The proof of the claim then follows by an application of the Lebesgue DCT.
\end{proof}
\begin{claim}\label{eqn-third-claim-for-bar-mu_n}
 $\pi_{1,n} [\psi^\prime(\phi_n) - \psi^\prime(\phi) + \avg{\psi^\prime(\phi)} - \avg{\psi^\prime(\phi_n)}]\to 0$ strongly in $L^{4/3}(0,T;\newone{H})$ as $n \to \infty$.
\end{claim}
\begin{proof}[Proof of Claim \ref{eqn-third-claim-for-bar-mu_n}]
Thanks to the equality \eqref{eqn-Psi'} and the H\"older inequality, 
\begin{align*}
&\lvert \pi_{1,n} [\psi^\prime(\phi_n) - \psi^\prime(\phi) + \avg{\psi^\prime(\phi)} - \avg{\psi^\prime(\phi_n)}] \rvert_{\newone{H}}
\leq \Vert \pi_{1,n} \Vert_{\mathcal{L}(\newone{H})} \lvert \psi^\prime(\phi_n) - \psi^\prime(\phi) + \avg{\psi^\prime(\phi)} - \avg{\psi^\prime(\phi_n)} \rvert_{\newone{H}}
\\
&\leq \lvert \psi^\prime(\phi_n) - \psi^\prime(\phi) + \avg{\psi^\prime(\phi)} - \avg{\psi^\prime(\phi_n)} \rvert_{\newone{H}}
= \lvert \nabla \psi^\prime(\phi_n) - \nabla \psi^\prime(\phi) \rvert_{\mathbb{L}^2}
\\
& = \lvert \psi^{\bis}(\phi_n) \nabla \phi_n - \psi^{\bis}(\phi) \nabla \phi\rvert_{\mathbb{L}^2}
=\lvert 3 \phi_n^2 \nabla(\phi_n - \phi) - \nabla(\phi_n - \phi) + 3 (\phi_n - \phi) (\phi_n + \phi) \nabla \phi  \rvert_{\mathbb{L}^2}
\\
&\leq \lvert 3 \phi_n^2 \nabla(\phi_n - \phi) - \nabla(\phi_n - \phi) \rvert_{\mathbb{L}^2} + 3 \lvert (\phi_n - \phi) (\phi_n + \phi) \nabla \phi  \rvert_{\mathbb{L}^2}.
\end{align*}
Arguing as in \eqref{eq-Psi-second-gradient-phi-n-1}, using also the Agmon inequality \eqref{eq-Agmon's-inequalities}, we infer that for every $n \in \mathbb{N}$,
\begin{equation}\label{Eqn-8.41}
\begin{aligned}
&\lvert 3 \phi_n^2 \nabla(\phi_n - \phi) - \nabla(\phi_n - \phi) \rvert_{\mathbb{L}^2}
\leq \lvert \phi_n - \phi \rvert_{\newone{H}} + C \Vert \phi_n \Vert_{L^\infty}^2 \lvert \phi_n - \phi \rvert_{\newone{H}}
\\
&\leq \lvert \phi_n - \phi \rvert_{\newone{H}} + C \lvert \phi_n \rvert_{\newone{H}} \Vert \phi_n \Vert_{\zero{H}{2}}  \lvert \phi_n - \phi \rvert_{\newone{H}}
\end{aligned}
\end{equation}
and
\begin{equation}\label{Eqn-8.42}
\begin{aligned}
&\lvert (\phi_n - \phi) (\phi_n + \phi) \nabla \phi  \rvert_{\mathbb{L}^2}
\leq \Vert \phi_n - \phi \Vert_{L^\infty} \Vert \phi_n + \phi \Vert_{L^\infty} \lvert \phi \rvert_{\newone{H}}
\\
&\leq C \lvert \phi_n - \phi \rvert_{\newone{H}}^{1/2} \Vert \phi_n - \phi \Vert_{\zero{H}{2}}^{1/2}  
\lvert \phi_n + \phi \rvert_{\newone{H}}^{1/2} \Vert \phi_n + \phi \Vert_{\zero{H}{2}}^{1/2}  
\lvert \phi \rvert_{\newone{H}}.
\end{aligned}
\end{equation}
Using the estimate \eqref{Eqn-8.41}, we infer that for every $n \in \mathbb{N}$,
\begin{align*}
&\int_0^T \lvert 3 \phi_n^2(s) \nabla(\phi_n(s) - \phi(s)) - \nabla(\phi_n(s) - \phi(s)) \rvert_{\mathbb{L}^2}^{\frac43}\,\d s
\\
&\leq C\Vert \phi_n - \phi \Vert_{L^{\frac43}(0,T;\newone{H})}^{\frac43} + C \int_0^T \lvert \phi_n(s) \rvert_{\newone{H}}^{\frac43} \Vert \phi_n(s) \Vert_{\zero{H}{2}}^{\frac43} \lvert \phi_n(s) - \phi(s) \rvert_{\newone{H}}^{\frac43}\,\d s
\\
&\leq C\Vert \phi_n - \phi \Vert_{L^{\frac43}(0,T;\newone{H})}^{\frac43} + C \sup_{s \in [0,T]} \lvert \phi_n(s) \rvert_{\newone{H}}^{\frac43} 
 \sup_{s \in [0,T]} \lvert \phi_n(s) - \phi(s) \rvert_{\newone{H}}^{\frac56}
\Vert \phi_n \Vert_{L^2(0,T;\zero{H}{2})}^{\frac43} \Vert \phi_n - \phi \Vert_{L^{\frac32}(0,T;\newone{H})}^{\frac12}
\\
&\leq  C \sup_{n \in \mathbb{N}} \sup_{s \in [0,T]} \lvert \phi_n(s) \rvert_{\newone{H}}^{\frac43} 
 \left[ \sup_{n \in \mathbb{N}} \sup_{s \in [0,T]} \lvert \phi_n(s)\rvert_{\newone{H}}^{\frac56} + \sup_{s \in [0,T]} \lvert \phi(s) \rvert_{\newone{H}}^{\frac56} \right]
\sup_{n \in \mathbb{N}} \Vert \phi_n \Vert_{L^2(0,T;\zero{H}{2})}^{\frac43} \Vert \phi_n - \phi \Vert_{L^{\frac32}(0,T;\newone{H})}^{\frac12}
\\
&\quad + C \Vert \phi_n - \phi \Vert_{L^{\frac43}(0,T;\newone{H})}^{\frac43}.
\end{align*}
This, jointly with the estimates in \eqref{eq-5.9} and the fact that $\phi_n \to \phi$ in $L^2(0,T;\zero{H}{2}(\domO))$, yields
\[
\lim_{n \to \infty} \int_0^T \lvert 3 \phi_n^2(s) \nabla(\phi_n(s) - \phi(s)) - \nabla(\phi_n(s) - \phi(s)) \rvert_{\mathbb{L}^2}^{\frac43}\,\d s=0.
\]
From the estimate \eqref{Eqn-8.42}, we infer that for every $n \in \mathbb{N}$,
\begin{align*}
&\int_0^T \lvert (\phi_n - \phi) (\phi_n + \phi) \nabla \phi  \rvert_{\mathbb{L}^2}^{\frac43}\,\d s
\leq C \int_0^T  \lvert \phi \rvert_{\newone{H}}^{\frac43} \lvert \phi_n - \phi \rvert_{\newone{H}}^{\frac23} \Vert \phi_n - \phi \Vert_{\zero{H}{2}}^{\frac23}  
\lvert \phi_n + \phi \rvert_{\newone{H}}^{\frac23} \Vert \phi_n + \phi \Vert_{\zero{H}{2}}^{\frac23}\,\d s  
\\
&\leq C \sup_{s \in [0,T]} \lvert \phi(s) \rvert_{\newone{H}}^{\frac43} \sup_{s \in [0,T]} \lvert \phi_n(s) + \phi(s) \rvert_{\newone{H}}^{\frac23} 
\Vert \phi_n - \phi \Vert_{L^2(0,T;\newone{H})}^{\frac13} 
\Vert \phi_n - \phi \Vert_{L^2(0,T;\zero{H}{2})}^{\frac13}  
 \Vert \phi_n + \phi \Vert_{L^2(0,T;\zero{H}{2})}^{\frac13}. 
\end{align*}
This, jointly with the estimates in \eqref{eq-5.9} and the convergence \eqref{eq-5.12}  yields
\[
\lim_{n \to \infty} \int_0^T  \lvert (\phi_n(s) - \phi(s)) (\phi_n(s) + \phi(s)) \nabla \phi(s)) \rvert_{\mathbb{L}^2}^{\frac43}\,\d s=0.
\]
As a direct consequence of the above two convergences results,
\[
\lim_{n \to \infty} \int_0^T \lvert \pi_{1,n} [\psi^\prime(\phi_n(s)) - \psi^\prime(\phi(s)) + \avg{\psi^\prime(\phi(s))} - \avg{\psi^\prime(\phi_n(s))}] \rvert_{\newone{H}}^{\frac43}\,\d s=0.
\]
The proof of Claim \ref{eqn-third-claim-for-bar-mu_n} is now complete.
\end{proof}
\begin{claim}\label{eqn-4-claim-for-bar-mu_n}
 $(\pi_{1,n} - I) [\psi^\prime(\phi) - \avg{\psi^\prime(\phi)}]\to 0$ strongly in $L^{4/3}(0,T;\newone{H})$ as $n \to \infty$.
\end{claim}
\begin{proof}[Proof of Claim \ref{eqn-4-claim-for-bar-mu_n}]
Arguing as in the proof of the inequality \eqref{eq-Psi-second-gradient-phi-n}, we infer that for every $n \in \mathbb{N}$,
\begin{equation}
\begin{aligned}
&\lvert \psi^\prime(\phi) - \avg{\psi^\prime(\phi)} \rvert_{\newone{H}}
=\lvert \nabla \psi^\prime(\phi) \rvert_{\mathbb{L}^2}
= \lvert \psi^{\bis}(\phi) \nabla \phi \rvert_{\mathbb{L}^2}
\\
&\leq \lvert \phi \rvert_{\newone{H}}  +  C \lvert \phi \rvert_{L^2}^{7/8} \lvert \phi \rvert_{\newone{H}}^{1/4} \Vert \phi \Vert_{\zero{H}{2}}^{15/8} 
\leq \lvert \phi \rvert_{\newone{H}}  +  C \lvert \phi \rvert_{L^2}^{7/8} \lvert \phi \rvert_{\newone{H}}^{19/16} \Vert \phi \Vert_{\newone{V}}^{15/16}. 
\end{aligned}
\end{equation}  
Combining the above inequalities together with the properties of $\phi$, cf. \eqref{eq-5.9}, we deduce that for every $n \in \mathbb{N}$,
\begin{align*}
&\int_0^T \lvert \psi^\prime(\phi(s)) - \avg{\psi^\prime(\phi(s))} \rvert_{\newone{H}}^{\frac43}\,\d s 
\leq C \int_0^T \lvert \phi(s) \rvert_{\newone{H}}^{\frac43}  +  C \int_0^T \lvert \phi(s) \rvert_{L^2}^{7/6} \lvert \phi(s) \rvert_{\newone{H}}^{19/12} \Vert \phi(s) \Vert_{\newone{V}}^{5/4}\,\d s
\\
&\leq C \Vert \phi \Vert_{L^{4/3}(0,T;\newone{H})}^{4/3} + C \sup_{s \in [0,T]} \lvert \phi(s) \rvert_{L^2}^{7/6} \sup_{s \in [0,T]} \lvert \phi(s) \rvert_{\newone{H}}^{19/16} \Vert \phi \Vert_{L^{3/2}(0,T;\newone{V})}^{9/5}<\infty.
\end{align*}  
Therefore, by the DCT, we infer that
\[
\lim_{n \to \infty} \int_0^T \Vert (\pi_{1,n} - I) [\psi^\prime(\phi(s)) - \avg{\psi^\prime(\phi(s))}] \Vert_{\newone{H}}^{4/3}\,\d s=0.
\]
The claim \ref{eqn-4-claim-for-bar-mu_n} is now verified.
\end{proof}
Recall that $\cp_n(t)=  \pi_{1,n}\Atwo \phi_n(t) + \pi_{1,n} (\psi^\prime(\phi_n(t)) - \avg{\psi^\prime(\phi_n(t))})$  
and  $\cp(t)=  \Atwo \phi(t) + \psi^\prime(\phi(t)) - \avg{\psi^\prime(\phi(t))},\; \; t \in [0,T]$. 
From the estimates \eqref{Eq-nabla-mu} and the assumption \eqref{eq-5.9}, we infer that
        \begin{equation*}
         \sup_{n} \int_0^T \lvert \cp_n(s) \rvert_{\newone{H}}^{4/3}\, \d s < \infty.
        \end{equation*}
Furthermore, thanks to the estimate \eqref{Eq-Psi-second-gradient-phi-b} and the assumption \eqref{eq-5.9}, we deduce that 
$\cp \in L^{4/3}(0,T;\newone{H})$.
\begin{corollary}\label{eqn-cor-weak-convergence-bar-mu_n-to-bar-mu}
$\cp_n \rightharpoonup \cp$ weakly in $L^{4/3}(0,T;\newone{H})$ as $n \to \infty$.    
\end{corollary}
\begin{proof}
Let $n \in \mathbb{N}$ be fixed. Notice that
\begin{align*}
\cp_n - \cp
=& (\pi_{1,n} -I) \Atwo (\phi_n - \phi) + \Atwo (\phi_n - \phi) + (\pi_{1,n} - I) \Atwo \phi\\
& + \pi_{1,n} [\psi^\prime(\phi_n) - \psi^\prime(\phi) + \avg{\psi^\prime(\phi)} - \avg{\psi^\prime(\phi_n)}]
+ (\pi_{1,n} - I) [\psi^\prime(\phi) - \avg{\psi^\prime(\phi)}].
\end{align*}
Since the strong convergence in a reflexive Banach space implies weak convergence and 
the limit  of a sum  of weakly sequences is equal to the sum of weak limits,  Corollary \ref{eqn-cor-weak-convergence-bar-mu_n-to-bar-mu} is  a direct consequence of 
the Claims \ref{eqn-first-claim-for-bar-mu_n}-\ref{eqn-4-claim-for-bar-mu_n}.
\end{proof}
Let us now move to the proof of the assertion \eqref{eq-A1-mu-n-convergence}. We observe that by \eqref{eqn-Galerkin-Modified-stochastic-CHNSEs-n}, 
 $\pi_{1,n}\cp_n =\cp_n$, and therefore 
\begin{align*}
(\pi_{1,n} \cp_n, -\Delta z_2)_{\newone{H}} - (\cp, -\Delta z_2)_{\newone{H}}
= (\pi_{1,n} \cp_n - \cp, -\Delta z_2)_{\newone{H}} 
=(\cp_n - \cp, -\Delta z_2)_{\newone{H}}.
\end{align*}
Therefore, thanks to the Corollary \ref{eqn-cor-weak-convergence-bar-mu_n-to-bar-mu}, we deduce that
\begin{align*}
&\lim_{n \to \infty} \int_0^t e^{i \duality{\bX(s)}{z}{\nU}{\nU^\prime}} (\pi_{1,n} \cp_n(s) - \cp(s), -\Delta z_2)_{\newone{H}}\,\d s
\\
&= \lim_{n \to \infty} \int_0^t e^{i \duality{\bX(s)}{z}{\nU}{\nU^\prime}} (\cp_n(s) - \cp(s), -\Delta z_2)_{\newone{H}}\,\d s
=0,
\end{align*}
which, in turn, implies \eqref{eq-A1-mu-n-convergence}. \newline
Now, let us prove that
   \begin{equation}\label{eq-5.29}
     \lim_{n \to \infty} \int_0^t \Vert \duality{\newG_{0,n}(s,\bX_n(s))}{z}{}{} \Vert_{\ell^2}^2\,\d s
     = \int_0^t \Vert \duality{\newG_0(s,\bX(s))}{z}{}{} \Vert_{\ell^2}^2\,\d s.
   \end{equation}
Notice that
\begin{align*}
& \int_0^t \left[\Vert \duality{\newG_{0,n}(s,\bX_n(s))}{z}{}{} \Vert_{\ell^2}^2 - \Vert \duality{\newG_0(s,\bX(s))}{z}{}{} \Vert_{\ell^2}^2 \right]\d s  
\\
&\leq \int_0^t [\Vert \duality{\newG_{0,n}(s,\bX_n(s)) - \newG_0(s,\bX(s))}{z}{}{} \Vert_{\ell^2} 
      \Vert \duality{\newG_{0,n}(s,\bX_n(s)) + \newG_0(s,\bX(s))}{z}{}{} \Vert_{\ell^2}]\,\d s
  \\
&\leq \left(\int_0^T \Vert \duality{\newG_{0,n}(s,\bX_n(s)) - \newG_0(s,\bX(s))}{z}{}{} \Vert_{\ell^2}^2 \right)^\frac{1}{2} \cdot \\
&\qquad \left(\int_0^T \Vert \duality{\newG_{0,n}(s,\bX_n(s)) + \newG_0(s,\bX(s))}{z}{}{} \Vert_{\ell^2}^2  \right)^\frac{1}{2} \coloneqq I_{4,1}^n I_{4,2}^n.
\end{align*}
We now turn to the estimate of $I_{4,2}^n$. Similarly to \eqref{eq-5.14b}, we infer that
\begin{align*}
I_{4,2}^n
&\leq C \left[\int_0^T \left(\Vert \duality{\newG_{0,n}(s,\bX_n(s)) }{z}{}{} \Vert_{\ell^2}^2 + \Vert \duality{\newG_0(s,\bX(s))}{z}{}{} \Vert_{\ell^2}^2 \right)\d s \right]^{1/2} 
 \\
&\leq C \left[\int_0^T \left(\lvert \bu_n(s) \rvert_{\StokesH}^2 + \lvert \bu(s) \rvert_{\StokesH}^2 + \lvert \tilde{h}(s) \rvert^2\right)\, \d s \cdot \Vert z_1 \Vert_{\rU_0}^2\right]^{1/2}.
\end{align*}
The latter inequality and \eqref{eq-5.9} imply that $I_{4,2}^n$ is uniformly bounded w.r.t $n$. 
\newline
Regarding the term $I_{4,1}^n$, we have
\begin{align*}
& \int_0^t \Vert \duality{\newG_{0,n}(s,\bX_n(s)) - \newG_0(s,\bX(s))}{z}{}{} \Vert_{\ell^2}^2\d s 
=\int_0^t \sum_{k=1}^\infty \left[\dualitybig{[\newG_{0,n}(s,\bX_n(s)) - \newG_0(s,\bX(s))](\tilde{e}_k)}{z}{}{}{}{}{^2}\right]\d s
\\
&= \int_0^t \sum_{k=1}^\infty \dualitybig{[\bG_{0,n}(s,\bu_n(s)) -  \bG_0(s,\bu(s)) + \bSi_0^n(s)\bu_n(s) - \bSi_0(s)\bu(s)](\tilde{e}_k)}{z_1}{}{}{}{}{^2}\d s
\\
&\leq 2 \int_0^T \Vert \duality{\bG_{0,n}(s,\bu_n(s)) -  \bG_0(s,\bu(s))}{z_1}{}{}\Vert_{\ell^2}^2\,\d s
      + 2 \int_0^T \Vert \duality{\bSi_0^n(s)\bu_n(s) - \bSi_0(s)\bu(s)}{z_1}{}{} \Vert_{\ell^2}^2\,\d s,
\end{align*}
and
\begin{align*}
&\bSi_0^n(s)\bu_n(s) - \bSi_0(s)\bu(s)
= \bSi_0^n(s)(\bu_n(s) - \bu(s)) + \bSi_0^n(s)\bu(s) - \bSi_0(s)\bu(s)
\\
&= \pi_{0,n} \circ  \bSi_0(s)(\bu_n(s) - \bu(s)) + (\pi_{0,n} - I) \circ  \bSi_0(s)\bu(s).
\end{align*}
We therefore deduce that
\begin{align*}
&\int_0^T \Vert \duality{\bSi_0^n(s)\bu_n(s) - \bSi_0(s)\bu(s)}{z_1}{}{} \Vert_{\ell^2}^2\,\d s
\\
&= \int_0^T \Vert \duality{\pi_{0,n} \circ  \bSi_0(s)(\bu_n(s) - \bu(s)) + (\pi_{0,n} - I) \circ  \bSi_0(s)\bu(s)}{z_1}{}{} \Vert_{\ell^2}^2\,\d s
 \\
&=\int_0^T \sum_{k=1}^\infty \duality{\pi_{0,n} \circ  \bSi_{0}^k(s)(\bu_n(s) - \bu(s)) + (\pi_{0,n} - I) \circ  \bSi_{0}^k(s)\bu(s)}{z_1}{}{}{}{^2}\,\d s
\\
&\leq 2 \int_0^T \sum_{k=1}^\infty \duality{\pi_{0,n} \circ  \bSi_{0}^k(s)(\bu_n(s) - \bu(s))}{z_1}{}{}{}{^2}\,\d s + 2 \int_0^T \sum_{k=1}^\infty \duality{(\pi_{0,n} - I) \circ  \bSi_{0}^k(s)\bu(s)}{z_1}{}{}{}{^2}\,\d s
\\
&\leq  2 \int_0^T \sum_{k=1}^\infty \duality{\bSi_{0}^k(s)(\bu_n(s) - \bu(s))}{\tilde{\pi}_{0,n} z_1}{}{}{}{^2}\,\d s + 2 \int_0^T \sum_{k=1}^\infty \duality{\bSi_{0}^k(s)\bu(s)}{(\tilde{\pi}_{0,n} - I) z_1}{}{}{}{^2}\,\d s.
\end{align*}
For the second term on the RHS of the above inequality, we infer from \eqref{Eqn-coercivity-3} that
\begin{align*}
&\sum_{k=1}^\infty \lvert \duality{\bSi_{0}^k(s)\bu(s)}{\tilde{\pi}_{0,n} z_1 - z_1}{}{} \rvert^2  
\leq \sum_{k=1}^\infty \lvert \bSi_{0}^k(s)\bu(s) \rvert_{\mathbb{L}^2}^2 \lvert \tilde{\pi}_{0,n} z_1 - z_1 \rvert_{\mathbb{L}^2}^2
\\
&= \Vert \bSi_0(s)\bu(s) \Vert_{\ell^2(\StokesH)}^2 \lvert \tilde{\pi}_{0,n} z_1 - z_1 \rvert_{\mathbb{L}^2}^2 
\leq \tilde{C}_1 \Vert \bu(s) \Vert_{\StokesV}^2 \cdot \lvert \tilde{\pi}_{0,n} z_1 - z_1 \rvert_{\mathbb{L}^2}^2,
\end{align*}
and then by \eqref{eq-5.9} and since $\lvert \tilde{\pi}_{0,n} z_1 - z_1 \rvert_{\mathbb{L}^2} \to 0$, we thus obtain
\begin{equation*}
\lim_{n \to \infty} \int_0^T \sum_{k=1}^\infty \lvert \duality{\bSi_{0}^k(s)\bu(s)}{\tilde{\pi}_{0,n} z_1 - z_1}{}{}\rvert^2\,\d s= 0.
\end{equation*}
In view of \eqref{eqn-duality-identity} and \eqref{eq-3.20aa} we infer that for every $n \in \mathbb{N}$,
\begin{align*}
&\int_0^T \sum_{k=1}^\infty \lvert \duality{\bSi_{0}^k(s)(\bu_n(s) - \bu(s))}{\tilde{\pi}_{0,n} z_1}{\rU_0}{\rU_0^\prime}\rvert^2\,\d s
\\
&=\int_0^T \sum_{k=1}^\infty \lvert \duality{\bSi_{0}^k(s)(\bu_n(s) - \bu(s))}{\tilde{\pi}_{0,n} z_1}{\StokesV}{\StokesVp} \rvert^2\,\d s
\\
&\leq \int_0^T \Vert \bSi_0(s)(\bu_n(s) - \bu(s)) \Vert_{\ell^2(\StokesVp)}^2\,\d s \cdot \Vert \tilde{\pi}_{0,n} z_1 \Vert_{\StokesV}^2
\\
&
\leq C_4 \Vert \bu_n - \bu \Vert_{L^2(0,T;\StokesH)}^2 \Vert \tilde{\pi}_{0,n} z_1 \Vert_{\StokesV}^2
\leq C\Vert \bu_n - \bu \Vert_{L^2(0,T;\StokesH)}^2 \Vert z_1 \Vert_{\StokesV}^2.
\end{align*}
From this latter estimate and since $\bu_n \to \bu$ in the topology of $L^2(0,T;\StokesH)$, we obtain
   \begin{equation*}
     \lim_{n \to \infty} \int_0^T \sum_{k=1}^\infty \lvert \duality{\bSi_{0}^k(s)(\bu_n(s) - \bu(s))}{\tilde{\pi}_{0,n} z_1}{\rU_0}{\rU_0^\prime} \rvert^2\, \d s=0.
   \end{equation*}
Hence, 
$\int_0^T \Vert \duality{\bSi_0^n(s)\bu_n(s) - \bSi_0(s)\bu(s)}{z_1}{}{}\Vert_{\ell^2}^2\,\d s \underset{n \to \infty}{\to}  0$.
Next, we have
\begin{align*}
&\Vert \duality{\bG_{0,n}(s,\bu_n(s)) - \bG_0(s,\bu(s))}{z_1}{}{} \Vert_{\ell^2}^2
= \sum_{k=1}^\infty \lvert \duality{[\bG_{0,n}(s,\bu_n(s))](\tilde{e}_k) - \bG_{0}^k(s,\bu(s))}{z_1}{}{}\rvert^2
  \\
&\leq \lvert z_1 \rvert_{\mathbb{L}^2}^2 \Vert \bG_{0,n}(s,\bu_n(s)) - \bG_0(s,\bu(s)) \Vert_{\ell^2(\StokesH)}^2.
\end{align*}
Now, thanks to the assumption \eqref{eq-5.10} and the Lemma \ref{Lem-approximation-c}, we infer that
   \begin{align*}
    &\int_0^T \Vert \duality{\bG_{0,n}(s,\bu_n(s)) - \bG_0(s,\bu(s))}{z_1}{}{} \Vert_{\ell^2}^2\,\d s
      \\
    &\leq \lvert z_1 \rvert_{\mathbb{L}^2}^2 \int_0^T \Vert \bG_{0,n}(s,\bu_n(s)) - \bG_0(s,\bu(s)) \Vert_{\ell^2(\StokesH)}^2\,\d s \underset{n \to \infty}{\to}  0.
\end{align*}
Therefore, we deduce $[I_{4,1}^n]^2 \to   0$ as $n \to \infty$,  and the convergence \eqref{eq-5.29} holds.
\newline
Combining the previous convergences, we obtain \eqref{Eqn-convergence-varphi_n^{z}}.
This, together with \eqref{eq-5.21} and \eqref{eq-5.25}, implies that
     \begin{equation*}
        \lim_{n \to \infty }\lvert M_t^{n,z}(\bX_n) - M_t^{z}(\bX) \rvert= 0.
     \end{equation*}
This completes the proof of Part 2.
The proof of Proposition \ref{Prop-equicontinuity} is now complete.
\end{proof}

\begin{definition}\label{def-martingale-solution}
A probability measure $\mathbb{P}_T$ on space $\bZ_T$ is said to be a solution to the martingale problem $(\bX_0,\bF,\newG_0)$, where  $\bX_0\coloneqq  (\bu_{0},\phi_{0}) \in\mathbb{H}$,  i.e. the problem \eqref{eqn-compact-modified-stochastic-CHNSEs-3}, if and only if 
\begin{trivlist}
\item[(i)] for almost every $\bX$ belonging to the support of the measure $\mathbb{P}_T$, the function $\cp$ defined by
              \begin{equation}\label{Eq-identition-mu}
                \cp(t)=  \Atwo \phi(t) +  \psi^\prime(\phi(t)) - \avg{\psi^\prime(\phi(t))}, \;\; t\in [0,T]    
              \end{equation}
belongs to $L^{2}(0,T;\newone{H})$ and 
\item[(ii)] 
for every $z \in \mathscr{V}$, the process $M_t^{z}$ defined on the canonical sample space $\bZ_T$ in \eqref{eqn-M^z_t} 
is a continuous $\mathbb{D}$ local martingale. 
\end{trivlist}
\end{definition}
\begin{remark}\label{rem-def-martingale-solution} 
The condition $(ii)$ in Definition \ref{def-martingale-solution} is equivalent to the following one.
\begin{trivlist}
\item[(ii')] for every $z \in \mathscr{V}$ and every $\theta \in \mathbb{R}$, the process $M_t^{\theta z}$ defined on the canonical sample space $\bZ_T$ in \eqref{eqn-M^z_t} is a continuous local martingale, i.e., its satisfies
         \begin{align*}
            M_t^{\theta z} \in \mathcal{M}_{\loc}^c(\mathbb{D},\mathbb{P}_T).
         \end{align*}
\end{trivlist}
It follows from the proof of Theorem \ref{thm-5.3} that the localizing sequence of stopping times depends only on $z$ but not on $\theta \in \mathbb{R}$.
\end{remark}
With the above result, i.e., Proposition \ref{Prop-equicontinuity}, in hand, and the above definition in mind, we are now ready to state and prove the following result--the central step in the proof of Theorem \ref{First-main-result}. Here, $\mathbb{E}^{\mathbb{P}_T} \fb$  denotes the integral of a measurable function $\fb$ with respect to the measure $\mathbb{P}_T$.
\begin{theorem}\label{thm-5.3}
Assume that the assumptions in Section \ref{Ass-Abstract formulation} hold and that $\bX_0 \in\mathbb{H}$. Then the following assertions hold. 
\begin{trivlist}
\item[(i)]   
There exists a measure $\mathbb{P}_T$ on $\bZ_T$ solving the  martingale problem $(\bX_0,\bF,\newG_0)$ in the sense of Definition \ref{def-martingale-solution}.
\item[(ii)]
The above measure $\mathbb{P}_T$ satisfies
     \begin{equation}\label{X-estimate}
       \mathbb{E}^{\mathbb{P}_T} \sup_{s \in [0,T]} \Vert \bX(s) \Vert_{\nHB}^2  + \mathbb{E}^{\mathbb{P}_T} \Vert \bX \Vert_{L^2(0,T;\StokesV \times \zero{H}{2})}^2< \infty,
     \end{equation}
and
\begin{align}
\label{phi-H_1-power-4-estimate}
\mathbb{E}^{\mathbb{P}_T} \Vert \phi \Vert_{L^4(0,T;\newone{H})}^4<\infty,
\\
\label{phi-H3-estimate}
\mathbb{E}^{\mathbb{P}_T} \Vert \phi \Vert_{L^\beta(0,T;\newone{V})}^\beta<\infty.
\end{align}
\item[(iii)]
If $d=2$, we have $\mathbb{P}_T$-a.s.
     \begin{equation*}
       \Vert \bF(\bX) \Vert_{L^2(0,T;\mathbb{V}^\prime)}^2< \infty.
     \end{equation*}
 \end{trivlist}
\end{theorem}
Before we embark  on  the proof of Theorem \ref{thm-5.3}, 
let us formulate three important corollaries of the previous results, i.e. of Proposition \ref{Prop-equicontinuity}. 
We recall that the notation $M_t^{z}$ has been introduced in \eqref{eqn-M^z_t}. 
\begin{corollary}\label{cor-eq-5.37}
Assume that $\bX_n \to \bX $ in $\bZ_T$. Suppose $z \in \mathscr{V}$. Then       
   \begin{equation}\label{eq-5.37}
      \sup_{s \in [0,T]} \lvert  M_s^{n,z}(\bX_n) - M_s^z(\bX) \rvert \underset{n \to \infty}{\to}  0.
    \end{equation}
\end{corollary}
\begin{proof} 
This Corollary is a direct consequence of Propositions \ref{Prop-equicontinuity} and \ref{prop-AZ-Theorem}.
\end{proof}
\begin{corollary}\label{cor-eq-5.40}
Assume that $\bX_n \to \bX $ in $\bZ_T$ and $z \in \mathscr{V}$. Then       \begin{equation}\label{eq-5.40}
     \sup_{s \in [0,T]} \lvert  M_s^{z}(\bX_n) - M_s^z(\bX) \rvert \underset{n \to \infty}{\to}  0.
   \end{equation}
\end{corollary}
\begin{proof}
By proceeding as in the proof of \eqref{eq-5.37} and invoking \eqref{eq-5.9} along with the properties of the projection operators $\pi_{0,n}$ and $\pi_{1,n}$, we can establish that
   \begin{equation}\label{eq-5.39}
      \sup_{s \in [0,T]} \lvert  M_s^{n,z}(\bX_n) - M_s^z(\bX_n) \rvert \underset{n \to \infty}{\to}  0.
    \end{equation}
Combining \eqref{eq-5.37} and \eqref{eq-5.39} yields \eqref{eq-5.40}.
\end{proof}
The next result strengthens, in some sense, the previous one.
\begin{corollary}\label{cor-MR-below 2.31}
For every compact set $K \subseteq \bZ_T$, 
 \begin{equation}\label{Eqt-convergence-on-compact-set}
    \lim_{n \to \infty} \sup_{s \in [0,T],\, \bX \in K} \lvert  M_s^{n,z}(\bX) - M_s^z(\bX) \rvert = 0.
  \end{equation}
Moreover, for every $\eta>0$, we have
   \begin{equation}\label{eq-probability-convergence}
      \lim_{n \to \infty} \mathbb{P}_n \left( \left\{ \bX \in \bZ_T: 
      \sup_{s \in [0,T]} \lvert  M_s^{n,z}(\bX) - M_s^z(\bX) \rvert> \eta  \right\}\right)= 0.
     \end{equation}
\end{corollary}
 
\begin{proof}[Proof of Corollary \ref{cor-MR-below 2.31}]
Taking into account compactness of $K$, the first result follows by a simple contradiction argument from Corollary \ref{cor-eq-5.37}.

Now let us consider the proof of \eqref{eq-probability-convergence}. 
Let us choose and fix $z \in \mathscr{V}$ and let $K$ be a compact subset of $\bZ_T$.
Let us choose and fix $\eta>0$ and  $\eps>0$.
Since by Lemma \ref{Lem-compactness-measure}, the set of probability measures $\{\mathbb{P}_n,\; n\geq 1\}$ is tight, there exists a compact subset $\mathbb{K}_\eps \subset \bZ_T$ such that 
  \begin{equation*}
    \mathbb{P}_n(\bZ_T \backslash \mathbb{K}_\eps) < \eps, \;\; n \in \mathbb{N}.
  \end{equation*}
Moreover, by \eqref{Eqt-convergence-on-compact-set}, there exists $n_0 \in \mathbb{N}$ such that for all  $n \geq n_0$,
    \begin{equation*}
      \mbox{ if }  \bX \in \mathbb{K}_\eps, \mbox{ then }\sup_{s \in [0,T]} \lvert  M_s^{n,z}(\bX) - M_s^z(\bX) \rvert \leq \eta.
   \end{equation*}
Therefore, if $n \geq n_0$, then 
\begin{equation*}
 \mathbb{P}_n \left(\left\{ \bX \in \bZ_T:\sup_{s \in [0,T]} \lvert  M_s^{n,z}(\bX) - M_s^z(\bX) \rvert> \eta  \right\}\right)  \leq \mathbb{P}_n(\bZ_T \backslash \mathbb{K}_\eps) < \eps.
\end{equation*}
This proves \eqref{eq-probability-convergence}, which completes the proof of Corollary \ref{cor-MR-below 2.31}.
\end{proof}
Once more, before we embark on the proof of Theorem \ref{thm-5.3}, we recall the following important result from \cite[Lemma 11.1.2]{Stroock+Varadhan_2006}.
\begin{lemma}\label{Lemma 11.1.2-SV2006}
Let us consider, for $R>0$, a function $\widetilde{\tau}_R$ defined by 
\begin{align}\label{eqn-tilde-tau_R}
       \widetilde{\tau}_R: C([0,T];\mathbb{R}^2) \ni \omega \mapsto \inf\{ t\in [0,T]: \lvert \omega(t)\rvert \geq R\}\in [0,T].
     \end{align}
     Then, the following assertions hold. 
\begin{trivlist}
\item[(i)] For every $R>0$, the function $\widetilde{\tau}_R$ is lower-semicontinuous.
\item[(ii)] If additionally $\mathbb{P}$ is a Borel probability measure on $C([0,T];\mathbb{R}^2)$, then  there exists a set $\mathscr{R} \subset (0,\infty)$, whose complement is  at most countable and a full $\mathbb{P}$-measure set $\Omega^\prime \subset C([0,T];\mathbb{R}^2)$ such that  for every $R \in \mathscr{R}$ 
the function $\widetilde{\tau}_R$ is continuous  at every  $\omega \in \Omega^\prime$. 
\end{trivlist}
\end{lemma}
We have the following immediate corollary. 
\begin{corollary}\label{cor-tau_R is continuous}
Let $\bZ_T$ be the topological space defined in \eqref{eqn-Z_T} and $\mathbb{P}_T$ be the Borel probability measure on $\bZ_T$ introduced in Proposition \ref{prop-Prohorov-general}. 
Suppose that a map $
      \Phi_T:\bZ_T \to  C([0,T],\mathbb{R}^2)$
 is  sequentially continuous. 
Then,   there exists a set $\mathscr{R} \subset (0,\infty)$, whose complement is  at most countable and a full $\mathbb{P}_T$-measure set $\bZ_T^\prime \subset \bZ_T$ such that  for every $R \in \mathscr{R}$ the function  $\widetilde{\tau}_R \circ \Phi_T$  is sequentially continuous  at every  $\bX \in \bZ_T^\prime$. 
\end{corollary}
\begin{proof} Corollary \ref{cor-tau_R is continuous} follows from Lemma 
\ref{Lemma 11.1.2-SV2006} applied to the Borel probability measure $\mathbb{P} \coloneqq  \Phi_T(\mathbb{P}_T)$ on $C([0,T],\mathbb{R}^2)$.
\end{proof}
After all these preparations, we now proceed to the promised proof of Theorem \ref{thm-5.3}.
\begin{proof}[Proof of part (i) of Theorem \ref{thm-5.3}]
Let us recall that we assume Assumptions \ref{ass-G+Sigma}-\ref{ass-bg+sigma} and we choose and fix $T>0$, $\bX_0 \in\mathbb{H}$ and $z \in \mathscr{V}$. We will show that the constructed earlier in Section \ref{Limit-meaures}  the probability  measure $\mathbb{P}_T$ on $\bZ_T$ is 
a solution to the   martingale problem $(\bX_0,\bF,\newG_0)$ in the sense of Definition \ref{def-martingale-solution}.\\
Define the following map    \begin{equation}\label{eqn-Phi_T}
      \Phi_T:\bZ_T \ni \bX \mapsto M_{\cdot}^{z}(\bX) \in C([0,T],\mathbb{R}^2),
  \end{equation}
where $M_{t}^{z}(\bX)$, $t \in [0,T]$,  has been defined in formula \eqref{eqn-M^z_t}. \\
Let us observe that by Corollary \ref{cor-eq-5.40}, the map $\Phi_T$ is sequentially continuous. Therefore, we can apply Corollary \ref{cor-tau_R is continuous}.\\
Let $\mathscr{R} \subset (0,\infty)$ and  $\bZ_T^\prime \subset \bZ_T$ 
be as in the assertion of that Corollary. Suppose first that  $R \in \mathscr{R}$.  We define the following stopping time $\tau_R$ by 
\[
\tau_R: \bZ_T \ni \bX \mapsto  \inf\{ t>0: \lvert M_t^z(\bX) \rvert \geq  R \} \in [0,T],
\]
and observe that 
\begin{equation}\label{eqn-def-stopping-times}
    \tau_R= \widetilde{\tau}_R \circ \Phi_T.
\end{equation}
Hence, by Corollary \ref{cor-tau_R is continuous}, we infer that 
the map $\tau_R$ is  sequentially continuous at every $\bX \in \bZ_T^\prime$.\\
Consider now a function $f : \bZ_T \to \mathbb{R}$  which is  continuous, bounded and $\mathscr{D}_s$-measurable. \\
The following equality will be proved in  three steps.
\begin{equation}\label{eq-5.46-ZB}
\mathbb{E}^{\mathbb{P}_T} \left[ f \times\left( M_{t \wedge \tau_R}^z - M_{s \wedge \tau_R}^z \right) \right]= 0.
\end{equation} 

\begin{proof}[\textbf{Step 1}]
Let us recall that we defined process $M_t^{n,z}$, $t \in [0,T]$,  in formula \eqref{eqn-M^nz_t}. Next, we put 
\begin{align*}
 \tau^n &= \inf\{t>0: \lvert  M_t^{n,z} - M_t^z \rvert> 1\},
 \\
 \tau_R^n&= \tau_R \wedge \tau^n.
\end{align*}
From \eqref{eq-probability-convergence}, we infer that $\lim_{n \to \infty}\mathbb{P}_n \left(\sup_{s \in [0,T]} \lvert  M_s^{n,z} - M_s^z \rvert> 1 \right)=0$. 
Moreover,  
\[\mathbb{P}_n(\tau^n< t)= 
\mathbb{P}_n \left(\sup_{s \in [0,t]} \lvert  M_s^{n,z} - M_s^z \rvert> 1 \right) 
\leq \mathbb{P}_n \left(\sup_{s \in [0,T]} \lvert  M_s^{n,z} - M_s^z \rvert> 1 \right) 
\] 
and hence we deduce that 
    \begin{equation}\label{eq-tau-n-convergence}
       \lim_{n \to \infty} \mathbb{P}_n(\tau^n< t) =0.
     \end{equation}
In addition, it is rather easy to see that the process $M_{s}^{n,z}$ stopped at $\tau_R^n$ satisfies
   \begin{equation}\label{eq-5.41}
     \sup_{s \in [0,T],\, n \in \mathbb{N}} \lvert M_{s \wedge \tau_R^n}^{n,z} \rvert \leq R + 1.
   \end{equation}
   \end{proof}
\begin{proof}[\textbf{Step 2}]
By the choice of $R$ we infer that  $\mathbb{P}_T(\tau_R= \tau_{R-})= 1$, where $\tau_{R-}= \lim_{ \delta \downarrow 0} \tau_{R-\delta}$. Now we prove that if $\bX \in \bZ_T^\prime $ and $(\bX_n)$ is $ \bX$-valued sequence such that 
$\bX_n \to \bX$ in $\bZ_T$ and  $\tau_R(\bX)= \tau_{R-}(\bX)$, then 
       \begin{equation}\label{eq-5.44}
         \sup_{s \in [0,T]} \lvert M^z_{s \wedge \tau_R(\bX_n)}(\bX_n) - M^z_{s \wedge \tau_R(\bX)}(\bX) \rvert \underset{n \to \infty}{\to}  0.
      \end{equation}
Considering the proof of \eqref{eq-5.44}, we first observe that for every  $n \in \mathbb{N}$, 
\begin{align*}
\sup_{s \in [0,T]} \lvert M^z_{s \wedge \tau_R(\bX_n)}(\bX_n) - M^z_{s \wedge \tau_R(\bX)}(\bX) \rvert 
\leq& 2 \sup_{s \in [0,T]} \lvert M^z_{s \wedge \tau_R(\bX_n)}(\bX_n) - M^z_{s \wedge \tau_R(\bX_n)}(\bX) \rvert \\
& + 2 \sup_{s \in [0,T]} \lvert M^z_{s \wedge \tau_R(\bX_n)}(\bX) - M^z_{s \wedge \tau_R(\bX)}(\bX) \rvert. 
\end{align*}
Next, thanks to \eqref{eq-5.40}, we deduce that 
    \begin{equation*}
      \sup_{s \in [0,T]} \lvert M^z_{s \wedge \tau_R(\bX_n)}(\bX_n) - M^z_{s \wedge \tau_R(\bX_n)}(\bX) \rvert  \underset{n \to \infty}{\to}  0.
    \end{equation*}
We claim that the map $\varphi_s$ defined by 
    \begin{equation*}
       \varphi_s: [0,T] \ni a \mapsto s \wedge a= \varphi_s(a) \in C([0,T])
    \end{equation*}
is Lipschitz continuous, and thus uniformly continuous. 
\newline
Indeed, let $a,\, b \in [0,T]$ and assume without loss of generality that $a<b$. Obviously, we have
\begin{equation*}
\lvert \varphi_s(a) - \varphi_s(b) \rvert \leq 
\begin{cases}
0 &\mbox{ if } s \in [0,a], 
\\
b-a &\mbox{ if } s \in [a,b],
\\
b-a &\mbox{ if } s \in [b,T],
\end{cases}
\end{equation*}
and so
     \begin{equation*}
       \sup_{s \in [0,T]} \lvert \varphi_s(a) - \varphi_s(b) \rvert \leq \lvert a - b \rvert, \quad a,\, b \in [0,T]. 
    \end{equation*}
Since $\tau_R: \bZ_T \to [0,T]$ is continuous by Corollary \ref{cor-tau_R is continuous} and $\bX_n \to \bX$ in $\bZ_T$, we infer that 
    \begin{equation*}
       \sup_{s \in [0,T]} \lvert s \wedge \tau_R(\bX_n) - s \wedge \tau_R(\bX)  \rvert= \sup_{s \in [0,T]} \lvert \varphi_s(\tau_R(\bX_n)) - \varphi_s(\tau_R(\bX)) \rvert \underset{n \to \infty}{\to}  0,
    \end{equation*}
from which we hence deduce
\begin{equation*}
\sup_{s \in [0,T]} \lvert M^z_{s \wedge \tau_R(\bX_n)}(\bX) - M^z_{s \wedge \tau_R(\bX)}(\bX) \rvert \underset{n \to \infty}{\to}  0.
\end{equation*}
This yields \eqref{eq-5.44}.  
\end{proof}
\begin{proof}[\textbf{Step 3}]
Notice that
\begin{align}
\mathbb{E}^{\mathbb{P}_n} \left[f \times\left(M_{t \wedge \tau_R^n}^{n,z} - M_{s \wedge \tau_R^n}^{n,z} \right) \right] 
&= \mathbb{E}^{\mathbb{P}_n} \left[f \times\left(M_{t \wedge \tau_R^n}^{n,z} - M_{t \wedge \tau_R^n}^z \right)\right] -  \mathbb{E}^{\mathbb{P}_n} \left[f \times\left(M_{s \wedge \tau_R^n}^{n,z} - M_{s \wedge \tau_R^n}^z \right) \right] 
\nonumber \\
&\quad + \mathbb{E}^{\mathbb{P}_n} \left[f \times\left(M_{t \wedge \tau_R^n}^z - M_{t \wedge \tau_R}^z \right)\right] - \mathbb{E}^{\mathbb{P}_n} \left[f \times\left( M_{s \wedge \tau_R^n}^z - M_{s \wedge \tau_R}^z\right)\right] 
\nonumber    \\
&\quad + \mathbb{E}^{\mathbb{P}_n} \left[f \times\left( M_{t \wedge \tau_R}^z - M_{s \wedge \tau_R}^z \right)\right].
\label{eq-5.45}
\end{align}
We will prove that the first  term on the RHS of \eqref{eq-5.45} converge to zero as $n \to \infty$. For this aim we put 
\[
a_n= \mathbb{E}^{\mathbb{P}_n} \left[f \times\left(M_{t \wedge \tau_R^n}^{n,z} - M_{t \wedge \tau_R^n}^z \right)\right].
\]
Let us choose and fix $\eps>0$. Since by assumptions $f$ is a bounded function, i.e.  $\lvert f \rvert \leq C$, by using inequality  \eqref{eq-5.41}, we deduce 
\begin{equation}
\begin{aligned}
\lvert a_n \rvert
&\leq \mathbb{E}^{\mathbb{P}_n} \left[
\lvert f \rvert\, \lvert M_{t \wedge \tau_R^n}^{n,z} - M_{t \wedge \tau_R^n}^z \rvert\,\1_{\lvert M_{t \wedge \tau_R^n}^{n,z} - M_{t \wedge \tau_R^n}^{z} \rvert> \eps}\right]
\\
&\quad
+ \mathbb{E}^{\mathbb{P}_n} \left[\lvert f \rvert\, \lvert M_{t \wedge \tau_R^n}^{n,z} - M_{t \wedge \tau_R^n}^z \rvert \1_{\lvert M_{t \wedge \tau_R^n}^{n,z} - M_{t \wedge \tau_R^n}^z \rvert \leq \eps} \right]
\\
&\leq  2C (R + 1)\,\mathbb{P}_n \left(\lvert M_{t \wedge \tau_R^n}^{n,z} - M_{t \wedge \tau_R^n}^z \rvert> \eps \right) + C \eps.
\end{aligned}
\end{equation} 
Hence, by applying  the convergence \eqref{eq-probability-convergence} in Corollary \ref{cor-MR-below 2.31}, we infer that
\begin{equation}
\lim \sup_n \lvert a_n \rvert
\leq C \eps.
\end{equation}
By the arbitrariness of $\eps$, we deduce that
\begin{equation}
\lim  \lvert a_n \rvert=0.
\end{equation}
The same proof implies that the second  term on the RHS of \eqref{eq-5.45} converges to zero as $n \to \infty$.\newline
We now prove that the third  term converges to zero as $n \to \infty$. We put now
\begin{align*}
    a_n \coloneqq \mathbb{E}^{\mathbb{P}_n} \left[f \times \left(M_{t \wedge \tau_R^n}^z - M_{t \wedge \tau_R}^z \right)\right]. 
\end{align*}
Then we have 
\begin{equation}\label{eqn-a_n-2}
    \begin{aligned}
   \lvert a_n \rvert
&\leq \mathbb{E}^{\mathbb{P}_n} \left[\lvert f \rvert  \lvert M_{t \wedge \tau_R^n}^z - M_{t \wedge \tau_R}^z \rvert  
\mathds{1}_{\{ \tau^n<t\}} \right] + \mathbb{E}^{\mathbb{P}_n} \left[\lvert f \rvert  \lvert M_{t \wedge \tau_R^n}^z - M_{t \wedge \tau_R}^z \rvert  
\mathds{1}_{ \{\tau^n \geq t \}} \right]
\\
&\leq  2C  R \, \mathbb{P}_n \left( \{ \tau^n<t\}\right)  +  0,
\end{aligned}
\end{equation}
because, in view of the definitions of $\tau^n$ and $\tau_R^n$, we have,  almost surely 
\[
 \lvert M_{t \wedge \tau_R^n}^z - M_{t \wedge \tau_R}^z \rvert 
\mathds{1}_{ \{\tau^n \geq t \}}  =0.
\]
Inequality \eqref{eqn-a_n-2} combined with the convergence \eqref{eq-tau-n-convergence} implies that $\lim_{n \to \infty} a_n=0$ as claimed.\newline
The  proof that the fourth  term on the RHS of \eqref{eq-5.45} converges to zero as $n \to \infty$ follows  the proof above about the third term. 
\noindent
Finally, let us deal with the fifth term on the RHS of \eqref{eq-5.45}. Observe that in light of \eqref{eq-5.41}, \eqref{eq-5.44}, and the properties of the function $f:\bZ_T \to \mathbb{R} $, the  map 
$f \times\left(M_{t \wedge \tau_R}^z - M_{s \wedge \tau_R}^z \right): \bZ_T \to \mathbb{R} $ is bounded, and as above,  continuous function on $\bZ_T^\prime$ with $\mathbb{P}_T(\bZ_T^\prime)=1$,  by Proposition \ref{prop-Prohorov-general}, we infer that 
\begin{equation}
\mathbb{E}^{\mathbb{P}_n} \left[f \times\left( M_{t \wedge \tau_R}^z - M_{s \wedge \tau_R}^z \right)\right] \to \mathbb{E}^{\mathbb{P}_T} \left[f \times\left( M_{t \wedge \tau_R}^z - M_{s \wedge \tau_R}^z \right)\right].
\end{equation}
Consequently, since by Lemma \ref{Lem-Pn-martingale-solution},
\[
\mathbb{E}^{\mathbb{P}_n} \left[f \times\left(M_{t \wedge \tau_R^n}^{n,z} - M_{s \wedge \tau_R^n}^{n,z} \right) \right]= 0, 
\]
we infer that condition \eqref{eq-5.46-ZB} holds true. 
\end{proof} 
This, in view of  Definition \ref{def-martingale-solution},  concludes Step 1 of the proof of Theorem \ref{thm-5.3}, i.e. Part $(i)$. The latter part of this section is deferred to the proof of assertion \eqref{Eq-identition-mu} below.
\end{proof}
Arguing as at the end of the proof of \cite[Lemma 11.1.3]{Stroock+Varadhan_2006}, we infer that Identity \eqref{eq-5.46-ZB}  holds true for every $R>0$. We deduce that 
the process $M_{t}^z$ is a local martingale, see Definition II.1.7 in \cite{Ikeda+Watanabe_1989}. Indeed, it is   adapted and that there exists a localizing sequence of stopping times since every stopped process $ M_{t \wedge \tau_R}^z$ is a martingale.

Now we move to our main objective of this Section, i.e. 
\begin{proof}[Proof of part (ii) Theorem \ref{thm-5.3}]
Let, for  $n\in \mathbb{N}$,  $\bX_n=(\bu_n,\phi_n)$, be the solution to \eqref{eqn-Compact-Galerkin-Modified-stochastic-CHNSEs-n}. 
According to Lemma \ref{Lem-compactness-measure}, the set of measures $\{\mathbb{P}_n,\; n\geq 1\}$ is tight on $(\bZ_T,\mathscr{Z})$, and the law  $\mathbb{P}_n$ of the process $\bX_n$ converges weakly (and even in the stronger sense) to a measure $\mathbb{P}_T$ on $(\bZ_T,\mathscr{Z})$, cf. \eqref{eqn-P-measure}. 
\newline
By the Skorokhod-Jakubowski Theorem in the framework of non-metric spaces, see \cite[Theorem 2]{Jakubowski_1998}, there exists a subsequence $(n_k)_{k \in \mathbb{N}}$, a new probability space $(\tilde{\Omega},\tilde{\mathscr{F}},\tilde{\mathbb{P}})$, and on this space, $\bZ_T$-valued random variables $\tilde{\bX}$, $\tilde{\bX}_k$, $k \in \mathbb{N}$ such that
\begin{equation}
\label{eqn-6.60} 
\begin{aligned}
&\mathscr{L}(\tilde{\bX}_k)= \mathscr{L}(\bX_{n_k}) \mbox{ for all } k \in \mathbb{N},
\\
&\lim_{k \to \infty} \tilde{\bX}_k= \tilde{\bX} \mbox{ in }\bZ_T, \;\; \tilde{\mathbb{P}}\mbox{-a.s.,}
\\
&\mbox{ and the law of the r.v.  } \tilde{\bX}\coloneqq (\tilde{\bu},\tilde{\phi}) \mbox{ is equal to  }\mathbb{P}_T.
\end{aligned}
\end{equation}
For the sake of brevity, we continue to denote these sequences by $(\tilde{\bX}_n)_{n \in \mathbb{N}}$ and $(\bX_n)_{n \in \mathbb{N}}$. \newline
Let us point out that as $C([0,T];\rU_0^\prime) \cap L^2(0,T;\StokesH)$ is a Polish space, by the Kuratowski Theorem, see \cite[Theorem I.3.9]{Parthasarathy_1967}, $C([0,T];H_{0,n})$ is a Borel subset of $C([0,T];\rU_0^\prime) \cap L^2(0,T;\StokesH)$. Therefore, by Proposition \ref{Propo-Borel-subset}, $C([0,T];H_{0,n}) \cap \bZ_{T,1}$ is a Borel subset of $C([0,T];\rU_0^\prime) \cap L^2(0,T;\StokesH) \cap \bZ_{T,1}= \bZ_{T,1}$.
\newline
Similarly, we deduce that $ C([0,T];\newonep{V}) \cap \bZ_{T,2}$ is a Borel subset of $\bZ_{T,2}$. \newline
Since the laws of $\bX_n$ and $\tilde{\bX}_n$ coincide, and given that $C([0,T];H_{0,n}) \cap \bZ_{T,1}$ and $ C([0,T];\newonep{V}) \cap \bZ_{T,2}$ are Borel subsets of  $\bZ_{T,1}$ and  $\bZ_{T,2}$, resp., the sequence $(\tilde{\bX}_n)_n$ inherits the same regularity properties as $\bX_n$ proved in Theorem \ref{thm-main-Galerkin}, Corollary \ref{cor-Propositions 5.1 and 5.3}, and Proposition \ref{prop-phi-n-estimates}. In particular, we have 
\begin{equation}\label{Eqn-subsequence-X_n}
\begin{aligned}
&\sup_{n \in \mathbb{N}} \tilde{\mathbb{E}} \sup_{s \in [0,T]} \Vert \tilde{\bX}_n(s) \Vert_{\nHB}^2< \infty, 
\\
&\sup_{n \in \mathbb{N}} \tilde{\mathbb{E}} \Vert \tilde{\bX}_n \Vert_{L^2(0,T;\StokesV \times \zero{H}{2})}^2< \infty, 
  \\
&\sup_{n \in \mathbb{N}} \tilde{\mathbb{E}} \Vert \tilde{\phi}_n \Vert_{L^\beta(0,T;\newone{V})}^\beta< \infty, 
\\
&\sup_{n \in \mathbb{N}} \tilde{\mathbb{E}} \Vert \tilde{\phi}_n \Vert_{L^4(0,T;\newone{H})}^4< \infty, 
\end{aligned}
\end{equation}
and for $\tilde{\mathbb{P}}$-a.s. $\tilde{\omega}$,
\begin{align}\label{Eqn-subsequence-X_n-1}
&\sup_{n \in \mathbb{N}} \lvert \tilde{\phi}_n(t,\tilde{\omega}) \rvert_{L^2}^2
\leq C_T \lvert \phi_0 \rvert_{L^2}^2, \;\; t\in [0,T],
\\
\label{eq-estimate-for-phi-d-a-1}
&\sup_{n \in \mathbb{N}} \Vert \tilde{\phi}_n(\cdot,\tilde{\omega}) \Vert_{L^4(0,T;\newone{H})}^4\leq C_T \lvert \phi_0 \rvert_{L^2}^4,
\\
\label{Eqn-subsequence-X_n-2}
&\sup_{n \in \mathbb{N}} \Vert \tilde{\phi}_n(\cdot,\tilde{\omega}) \Vert_{L^2(0,T;\zero{H}{2})}^2\leq C_T \lvert \phi_0 \rvert_{L^2}^2,
\end{align}
with the constant $C_T$ being independent of $\tilde{\omega}$. \newline
Now, we observe that the map
\begin{equation}\label{eqn-measurability-of the map-in bu}
  \bZ_T \ni \bX \mapsto \sup_{s \in [0,T]} \Vert \bX(s) \Vert_{\nHB}^2   
\end{equation}
is measurable. Indeed, as in \cite[Appendix E]{Brz+Ferr+Zan_2024}, it suffices to show that the map defined in \eqref{eqn-measurability-of the map-in bu} is lower-semicontinuous. This is indeed the case, since for every $R>0$, the set
\[
\left\{\bX\in C([0,T];\mathbb{H}_w): \sup_{s \in [0,T]} \Vert \bX(s) \Vert_{\mathbb{H}} \leq R\right\},
\]
is closed in $C([0,T];\mathbb{H}_w)$, cf. Proposition 1.6 (v) in \cite{Jakubowski_1986}. Recall that $\tilde{\bX}_n \to \tilde{\bX}$ in $C([0,T];\nHB_{w})$, $\tilde{\mathbb{P}}$-a.s. by definition of the space $\bZ_T$, cf. \eqref{eqn-Z_T}. Therefore, using  the Corollary E.3 and Proposition E.1 from \cite{Brz+Ferr+Zan_2024} in conjunction with the Fatou Lemma, we infer that 
\begin{equation}
\begin{aligned}
&\mathbb{E}^{\tilde{\mathbb{P}}} 
\sup_{s \in [0,T]} \Vert \tilde{\bX}(s) \Vert_{\nHB}^2
=\int_{\tilde{\Omega}} \sup_{s \in [0,T]} \Vert \tilde{\bX}(s,\tilde{\omega}) \Vert_{\nHB}^2 \; \d\tilde{\mathbb{P}}(\tilde{\omega})
\\
&\leq  \int_{\tilde{\Omega}} \liminf_{n \to \infty} \sup_{s \in [0,T]} \Vert \tilde{\bX}_n(s,\tilde{\omega}) \Vert_{\nHB}^2 \, \d \tilde{\mathbb{P}}(\tilde{\omega}) 
\leq \liminf_{n \to \infty} \int_{\tilde{\Omega}}  \sup_{s \in [0,T]} \Vert \tilde{\bX}_n(s,\tilde{\omega}) \Vert_{\nHB}^2 \, \d \tilde{\mathbb{P}}(\tilde{\omega}).
\end{aligned}
\end{equation}
This, together with the first part of \eqref{Eqn-subsequence-X_n},  proves the left part of condition \eqref{X-estimate}. 
\newline
Let us move to the proof of the other part. First, notice that the map
\begin{equation}
\bZ_T \ni \bX \mapsto \int_0^T \Vert \bX(s) \Vert_{\mathbb{V}}^2\,\d s
\end{equation}
is also Borel measurable since it is $\mathscr{B}\left(L_w^2(0,T;\mathbb{V}) \right)/\mathscr{B}(\mathbb{R})$-measurable.
Indeed, this latter map is lower-semicontinuous owing to Corollaries to Theorem 3.12 in \cite{Rudin_1991}. \newline
Next, from the convergence \eqref{eqn-6.60}, the definition of the space $\bZ_{T,2}$, and the Corollary \ref{Coro-convergence-phi_n-in-H_2},
we infer that $\tilde{\phi}_n \to \tilde{\phi}$ in $L^2(0,T;\zero{H}{2}(\domO))$, $\tilde{\mathbb{P}}$-a.s.
Therefore, by the second estimate in \eqref{Eqn-subsequence-X_n-2} and the fact that $H^2$ is an algebra, we get for $\tilde{\mathbb{P}}$-a.s. $\tilde{\omega}$,
\begin{equation}\label{eqn-H-2-estimate-tilde-phi}
\Vert \tilde{\phi}(\cdot,\tilde{\omega}) \Vert_{L^2(0,T;\zero{H}{2})}^2\leq C_T \lvert \phi_0 \rvert_{L^2}^2.
\end{equation}
Furthermore, by the weak lower-semicontinuity, i.e. the lower-semicontinuity of 
the $L^2(0,T;\StokesV \times \zero{H}{2})$-norm on the space  $L^2_w(0,T;\StokesV \times \zero{H}{2})$ and the Fatou Lemma, we deduce
\begin{equation}
\begin{aligned}
&\mathbb{E}^{\tilde{\mathbb{P}}} \Vert \tilde{\bX} \Vert_{L^2(0,T;\StokesV \times \zero{H}{2})}^2
= \int_{\tilde{\Omega}}  \Vert \tilde{\bX}(\tilde{\omega}) \Vert_{L^2(0,T;\StokesV \times \zero{H}{2})}^2\, \d \tilde{\mathbb{P}}(\tilde{\omega}) 
\\
&\leq  \int_{\tilde{\Omega}} \liminf_{n \to \infty} \Vert \tilde{\bX}_n(\tilde{\omega}) \Vert_{L^2(0,T;\StokesV \times \zero{H}{2})}^2\, \d \tilde{\mathbb{P}}(\tilde{\omega}) 
\leq \liminf_{n \to \infty} \int_{\tilde{\Omega}}  \Vert \tilde{\bX}_n(\tilde{\omega}) \Vert_{L^2(0,T;\StokesV \times \zero{H}{2})}^2\, \d \tilde{\mathbb{P}}(\tilde{\omega}).
\end{aligned}
\end{equation}
This, together with the second  part of \eqref{Eqn-subsequence-X_n},  proves the right  part of condition \eqref{X-estimate}. \newline
Similarly, we can prove that conditions \eqref{phi-H_1-power-4-estimate} and \eqref{phi-H3-estimate} are satisfied. \newline
We conclude the proof of the second part of Theorem \ref{thm-5.3} by applying assertion \eqref{eqn-6.60}
that the law on $\bZ_T$ of the process $\tilde{\bX}$ is equal to the measure $\mathbb{P}_T$.
\end{proof}
\begin{proof}[Proof of assertion \eqref{Eq-identition-mu}]
The sequence of process $\left(\tilde{\cp}_n\right)$, defined by 
\begin{equation}\label{eqn-tilde mu_n}
\tilde{\cp}_n(t)\coloneqq \pi_{1,n} (\Atwo \tilde{\phi}_n(t) + \psi^\prime(\tilde{\phi}_n(t)) - \avg{\psi^\prime(\tilde{\phi}_n(t))}), \;\; t \in [0,T],     
\end{equation}
shares the  properties of the sequence of process $(\cp_n)_{n \in \mathbb{N}}$. 
In particular, it is bounded in the Hilbert space $L^2(\tilde{\Omega};L^2(0,T;\newone{H}))$, see the estimate \eqref{eqn-improved-estimates} in Proposition \ref{prop-2nd proposition}. Furthermore, since the mean of $\tilde{\cp}_n$ is zero, we also deduce that the sequence $(\tilde{\cp}_n)_{n \in \mathbb{N}}$ is uniformly bounded in $L^2(\tilde{\Omega};L^2(0,T;\zero{L}{2}(\domO)))$.
Hence, by the Banach Alaoglu Theorem, one has, up to a subsequence,
   \begin{equation}\label{Eqn-tilde-mu-convergence}
       \tilde{\cp}_n \to \tilde{\cp} \; \mbox{ weakly in } 
        L^2(\tilde{\Omega};L^2(0,T;\zero{L}{2}(\domO))).
    \end{equation}
By the estimates \eqref{Eqn-subsequence-X_n} and \eqref{eq-estimate-for-phi-d-a-1} and the fact that $\newone{H} \embed L^6(\domO)$, we get for every $n \in \mathbb{N}$,
\begin{align*}
&\tilde{\mathbb{E}} \Vert \tilde{\phi}_n \Vert_{L^5(0,T;L^6(\domO))}^5
\leq C\, \tilde{\mathbb{E}} \Vert \tilde{\phi}_n \Vert_{L^5(0,T;\newone{H})}^5
\leq C \tilde{\mathbb{E}} \left[\sup_{s \in [0,T]} \lvert \tilde{\phi}_n(s) \rvert_{\newone{H}} \Vert \tilde{\phi}_n \Vert_{L^4(0,T;\newone{H})}^4
\right]
\\
&\leq C(T) \lvert \phi_0 \rvert_{L^2}^4 \tilde{\mathbb{E}} \left[\sup_{s \in [0,T]} \lvert \tilde{\phi}_n(s) \rvert_{\newone{H}}^2\right]^{1/2}<\infty.
\end{align*}
Combining this with \eqref{eqn-Psi'-L^6}, we ensure that the sequence $(\psi^\prime(\tilde{\phi}_n))_{n \in \mathbb{N}}$ is  bounded in the topological space $L^{5/3}(\tilde{\Omega}; L^{5/3}(0,T;L^{2}(\domO)))$, i.e. 
   \begin{equation}\label{L2-bound-for-comp-maps-Psi-prime-and-titlde-phi_n}
    \sup_n \Vert \psi^\prime(\tilde{\phi}_n) \Vert_{L^{5/3}(\tilde{\Omega}; L^{5/3}(0,T;L^{2}(\domO)))}< \infty.
   \end{equation}
In light of the pathwise estimate  \eqref{Eqn-subsequence-X_n}, we infer that for all $p \geq 2$,
\[
\sup_n \left[ \tilde{\mathbb{E}} \sup_{s \in [0,T]} \lvert \tilde{\phi}_n(s) \rvert_{L^2}^p \right]< \infty.
\]
Moreover, from \eqref{eqn-6.60}, we infer that $\tilde{\phi}_n \to \tilde{\phi}$ in $L^2([0,T] \times \domO)$, $\tilde{\mathbb{P}}$-a.s. Therefore, the application of the Vitali Theorem entails 
   \begin{equation}
     \tilde{\phi}_n \to \tilde{\phi} \mbox{ in } L^2([0,T] \times \domO \times \tilde{\Omega}).
   \end{equation}
Thus, we deduce that, up to a subsequence, 
  \begin{equation}\label{Pointwise-convergence-titlde-phi_n}
    \tilde{\phi}_n \to \tilde{\phi} \quad \d t \otimes \d x \otimes \tilde{\mathbb{P}}\mbox{-a.e.}
  \end{equation}
From \eqref{Pointwise-convergence-titlde-phi_n} and the fact that $\psi^\prime: \mathbb{R} \to \mathbb{R}$ is continuous, we infer that
\begin{align}\label{eqn-a.s}
\psi^\prime(\tilde{\phi}_n) \to \psi^\prime(\tilde{\phi}) \quad \d t \otimes \d x \otimes \tilde{\mathbb{P}}\mbox{-a.e.}
\end{align}
Combining \eqref{eqn-a.s} and \eqref{L2-bound-for-comp-maps-Psi-prime-and-titlde-phi_n}, we deduce the following weak convergence, at least sub-sequential,
   \begin{equation}\label{weak-convergence-for-Psi-titlde-phi-n}
     \psi^\prime(\tilde{\phi}_n) \rightharpoonup \psi^\prime(\tilde{\phi})  \mbox{ weakly in } L^{5/3}(\tilde{\Omega}; L^{5/3}(0,T;L^{2}(\domO))).
   \end{equation}
Now from \eqref{eqn-tilde mu_n} for all $v \in \bigcup_{n=1}^\infty \newone{H}_{,n} \subseteq \newone{H}$, $\varphi \in L^\infty([0,T] \times \tilde{\Omega})$ we obtain
\begin{equation}\label{Eqn-tilde-mu_n}
\begin{aligned}
&\int_{\tilde{\Omega}} \int_0^T (\tilde{\cp}_n(t,\tilde{\omega}), \varphi(t,\tilde{\omega}) v)_{L^2}\,\d t \, \d \tilde{\mathbb{P}}(\tilde{\omega})
\\
&= \int_{\tilde{\Omega}} \int_0^T (\pi_{1,n} \Atwo \tilde{\phi}_n(t,\tilde{\omega}) + \pi_{1,n} [\psi^\prime(\tilde{\phi}_n(t,\tilde{\omega})) -\avg{\psi^\prime(\tilde{\phi}_n(t,\tilde{\omega}))}],\varphi(t,\tilde{\omega}) v)_{L^2} \,\d t \, \d \tilde{\mathbb{P}}(\tilde{\omega})
  \\
&= \int_{\tilde{\Omega}} \int_0^T (\Atwo \tilde{\phi}_n(t,\tilde{\omega}) + \psi^\prime(\tilde{\phi}_n(t,\tilde{\omega})) - \avg{\psi^\prime(\tilde{\phi}_n(t,\tilde{\omega}))}, \varphi(t,\tilde{\omega}) \pi_{1,n} v)_{L^2}  \,\d t \, \d \tilde{\mathbb{P}}(\tilde{\omega}) 
\\
&= \int_{\tilde{\Omega}} \int_0^T (\Atwo \tilde{\phi}_n(t,\tilde{\omega}) + \psi^\prime(\tilde{\phi}_n(t,\tilde{\omega})) - \avg{\psi^\prime(\tilde{\phi}_n(t,\tilde{\omega}))}, \varphi(t,\tilde{\omega}) v)_{L^2}  \,\d t \, \d \tilde{\mathbb{P}}(\tilde{\omega})
\\
&= \int_{\tilde{\Omega}} \int_0^T \left[(\Atwo \tilde{\phi}_n(t,\tilde{\omega}), \varphi(t,\tilde{\omega}) v)_{L^2} + \left(\psi^\prime(\tilde{\phi}_n(t,\tilde{\omega})), \varphi(t,\tilde{\omega}) v\right)_{L^2} \right]\d t\,\d \tilde{\mathbb{P}}(\tilde{\omega}),
\end{aligned}
\end{equation}
where we used the fact that the mean of $v$ is zero.
From the weak convergence \eqref{Eqn-tilde-mu-convergence}, we infer as $n \to \infty$,
\[
\int_{\tilde{\Omega}} \int_0^T (\tilde{\cp}_n(t,\tilde{\omega}), \varphi(t,\tilde{\omega}) v)_{L^2} \,\d t \, \d \tilde{\mathbb{P}}(\tilde{\omega}) \to \int_{\tilde{\Omega}} \int_0^T (\tilde{\cp}(t,\tilde{\omega}), \varphi(t,\tilde{\omega}) v)_{L^2} \, \d t \, \d \tilde{\mathbb{P}}(\tilde{\omega}).
\]
Using the weak convergence \eqref{weak-convergence-for-Psi-titlde-phi-n}, we deduce that as $n \to \infty$,
\[
\int_{\tilde{\Omega}} \int_0^T  \left(\psi^\prime(\tilde{\phi}_n(t,\tilde{\omega})), \varphi(t,\tilde{\omega}) v \right)_{L^2} \d t \, \d \tilde{\mathbb{P}}(\tilde{\omega}) \to \int_{\tilde{\Omega}} \int_0^T \left(\psi^\prime(\tilde{\phi}(t,\tilde{\omega})), \varphi(t,\tilde{\omega}) v \right)_{L^2} \d t \, \d \tilde{\mathbb{P}}(\tilde{\omega}).
\]
Thanks to \eqref{Eqn-subsequence-X_n} and $\tilde{\phi}_n \to \tilde{\phi}$ in $L^2(0,T;\zero{H}{2}(\domO))$, $\tilde{\mathbb{P}}\mbox{-a.s.}$, 
we infer up to a subsequence that
\[
 \Atwo \tilde{\phi}_n \rightharpoonup \Atwo \tilde{\phi}  \mbox{ weakly in } L^2(\tilde{\Omega}; L^2(0,T;L^{2}(\domO))).
\]
Therefore, we obtain
\[
\int_{\tilde{\Omega}} \int_0^T  (\Atwo \tilde{\phi}_n(t,\tilde{\omega}), \varphi(t,\tilde{\omega}) v)_{L^2} \,\d t \, \d \tilde{\mathbb{P}}(\tilde{\omega}) \to \int_{\tilde{\Omega}} \int_0^T  (\Atwo \tilde{\phi}(t,\tilde{\omega}), \varphi(t,\tilde{\omega}) v)_{L^2} \,\d t \, \d \tilde{\mathbb{P}}(\tilde{\omega}).
\]
Collecting now the previous convergence results and passing to the limit as $n \to \infty$ in \eqref{Eqn-tilde-mu_n}, we deduce that for every $v \in \bigcup_{n=1}^\infty H_{1,n} \subseteq H_1$, $\varphi \in L^\infty([0,T] \times \tilde{\Omega})$,
\begin{align*}
&\int_{\tilde{\Omega}} \int_0^T (\tilde{\cp}(t,\tilde{\omega}), \varphi(t,\tilde{\omega}) v) \, \d t \, \d \tilde{\mathbb{P}}(\tilde{\omega})
=\int_{\tilde{\Omega}} \int_0^T  (\Atwo \tilde{\phi}(t,\tilde{\omega}) + \psi^\prime(\tilde{\phi}(t,\tilde{\omega})), \varphi(t,\tilde{\omega}) v) \,\d t \, \d \tilde{\mathbb{P}}(\tilde{\omega})
\\
&= \int_{\tilde{\Omega}} \int_0^T  (\Atwo \tilde{\phi}(t,\tilde{\omega}) + \psi^\prime(\tilde{\phi}(t,\tilde{\omega})) - \avg{\psi^\prime(\tilde{\phi}(t,\tilde{\omega}))}, \varphi(t,\tilde{\omega}) v)\,\d t \, \d \tilde{\mathbb{P}}(\tilde{\omega}).
\end{align*}
It follows that
\[
\tilde{\cp}= \Atwo \tilde{\phi} + \psi^\prime(\tilde{\phi}) - \avg{\psi^\prime(\tilde{\phi})}\; \; \d t \otimes \d x \otimes \tilde{\mathbb{P}}\mbox{-a.e.}
\]
Since the law on $\bZ_T$ of the process $\tilde{\bX}$ is equal to the measure $\mathbb{P}_T$, see \eqref{eqn-6.60}, we easily complete the proof of the assertion \eqref{Eq-identition-mu}.
\end{proof}
\begin{proof}[Proof of the paths continuity part of Theorem \ref{First-main-result-uniqueness} and of part (iii) of Theorem \ref{thm-5.3}]
Arguing as in \eqref{Ladyzhenskaya inequality-a}, we find that for $d=2$,
    \begin{align*}
       \Vert \nabla \phi \Vert_{\mathbb{L}^4}
       \leq C \lvert \phi \rvert_{\newone{H}}^{1/2} \Vert \phi \Vert _{\zero{H}{2}}^{1/2}, \; \; \phi \in \zero{H}{2}(\domO).
    \end{align*}
Thanks to the estimate \eqref{X-estimate}, we infer that $\mathbb{P}_T$-a.s.,
\begin{equation}
\begin{aligned}
&\int_0^T \Vert \bR_0(\phi(s),\phi(s)) \Vert_{\StokesVp}^2 \, \d s
\leq \int_0^T \Vert \nabla \phi(s) \Vert_{\mathbb{L}^4}^4\, \d s \\
&\leq C \sup_{s \in [0,T]}  \lvert \phi(s) \rvert_{\newone{H}}^2 \int_0^T \Vert \phi(s) \Vert_{\zero{H}{2}}^2\, \d s < \infty.
\end{aligned}
\end{equation}
Combining \eqref{eqn-b_0-trilinear-estimate}, \eqref{properties-B_0}, and \eqref{X-estimate}, we deduce that $\mathbb{P}_T$-a.s.,
\begin{align*}
&\int_0^T \Vert \bB_0(\bu(s),\bu(s)) \Vert_{\StokesVp}^2\, \d s
\leq C \int_0^T \lvert \bu(s) \rvert_{\StokesH}^2 \Vert \bu(s) \Vert_{\StokesV}^2\, \d s \\
&\leq C \sup_{s \in[0,T]} \lvert \bu(s) \rvert_{\StokesH}^2 \int_0^T \Vert \bu(s) \Vert_{\StokesV}^2\, \d s< \infty.
\end{align*}
By \eqref{eq-B1} and \eqref{X-estimate} together with the Poincar\'e inequality, we infer that $\mathbb{P}_T$-a.s.,
\begin{align*}
&\int_0^T \Vert B_1(\bu(s), \phi(s)) \Vert_{\newonep{V}}^2 \, \d s
\leq C \int_0^T \lvert \bu(s) \rvert_{\StokesH}  \Vert \bu(s) \Vert_{\StokesV} \lvert  \phi(s) \rvert_{\newone{H}}^2\, \d s \\
&\leq C \int_0^T \lvert \nabla \bu(s) \rvert_{\mathbb{L}^2}^2 \lvert  \phi(s) \rvert_{\newone{H}}^2\, \d s 
\leq C \sup_{s \in [0,T]} \lvert  \phi(s) \rvert_{\newone{H}}^2 \int_0^T \Vert \bu(s) \Vert_{\StokesV}^2\, \d s< \infty.
\end{align*}
Consequently,
    \begin{equation*}
      \Vert \bF(\bX) \Vert_{L^2(0,T;\mathbb{V}^\prime)}^2< \infty,\; \mathbb{P}_T\text{-a.s.,}
    \end{equation*}
and in turn, by the Krylov-Rozovskii Theorem, see \cite[Theorem 2.3.3]{Prevot+Rockner_2007}, the process $\bX=(\bX(t): t \in [0,T])$ has an $\nHB$-valued (strongly) continuous modification, denoted by the same symbol, and the It\^o formula holds for the square of its $\nHB$-norm $\mathbb{P}_T$-a.s, and thus the proof of Part 3 of Theorem \ref{thm-5.3} is complete. In particular, the proof of the paths continuity part of Theorem \ref{First-main-result-uniqueness} is complete as well.
\end{proof}
Before proving Theorem \ref{First-main-result}, let us state the following modified version of the Stone-Weierstrass theorem,
which will be useful in the subsequent section.
\begin{lemma}\label{eqn-Thm-Stone-Weierstrass} Let us choose and fix  $R>0$ and let 
\[
\mathbb{S}= \left\{[-R,R] \ni x \mapsto \sum_{j=1}^n c_j e^{i \theta_j x}, \; \; n \in \mathbb{N}, \;\; c_j \in \mathbb{C},\, \theta_j \in \mathbb{R} \right\}.
\]
 Then $\mathbb{S}$ is dense in $C^2([-R,R],\mathbb{C})$. The latter space is endowed with the following classical norm.
     \begin{equation}\label{eqn-norm-C^2}
        \vert u \vert_{C^2}\coloneqq \sum_{j=0}^2 \sup_{x \in [-R,R]} \vert D^j u(x) \vert.
    \end{equation}
 \end{lemma}
\begin{proof}[Proof of Lemma \ref{eqn-Thm-Stone-Weierstrass}]
Fix $R>0$.
The result follows from the Stone-Weierstrass theorem because the set $\mathbb{S}$ is a separating subset of $C([-R,R],\mathbb{C})$.
\end{proof}

The lemma \ref{eqn-Thm-Stone-Weierstrass} implies the following.
\begin{corollary}\label{eqn-Cor-uniform-convergence}
For every $g\in C_0^\infty(\mathbb{R},\mathbb{C})$ there exists an $\mathbb{S}$-valued sequence $(f_n)$ which converges to $g$
 locally uniformly.
\end{corollary}

\begin{proof}[Proof of Corollary \ref{eqn-Cor-uniform-convergence}]
Let us choose and fix  $g \in C_0^\infty(\mathbb{R},\mathbb{C})$.  Then there exists a closed ball $K_1$ of $\mathbb{R}$ such that
$\supp  g \subset K_1$. Thanks to Lemma \ref{eqn-Thm-Stone-Weierstrass}, we infer that there exists $f_1 \in \mathbb{S}$ such that
\[
\Vert g - f_1 \Vert_{C^2(K_1,\mathbb{C})}
<\frac12.
\]
Now, let $(K_n)_n$ be a sequence of closed balls of $\mathbb{R}$ such that
$K_1 \subset K_2 \subset \ldots $ and $\bigcup_j \stackrel{\circ}{K_j}= \mathbb{R}$.
Once more, by Lemma \ref{eqn-Thm-Stone-Weierstrass}, there exists $f_2 \in \mathbb{S}$ such that
\[
\Vert g - f_2 \Vert_{C^2(K_2,\mathbb{C})}
< \frac{1}{4}.
\]
Similarly, for each $K_n$, $n \in \mathbb{N}$, there exists $f_n \in \mathbb{S}$ such that
\[
\Vert g - f_n \Vert_{C^2(K_n,\mathbb{C})}
< \frac{1}{2^n}.
\]
Next, fix a compact subset $K$ of $\mathbb{R}$ and $\eps>0$. There exists $n \in \mathbb{N}$ such that $K \subset K_n$ and
\[
\Vert g - f_n \Vert_{C^2(K,\mathbb{C})}
\leq \Vert g - f_n \Vert_{C^2(K_n,\mathbb{C})}
< \frac{1}{2^n},
\]
and then for every $K \subset K_m$ with $m \geq n$, we see that
\[
\Vert g - f_m \Vert_{C^2(K,\mathbb{C})}
\leq \Vert g - f_m \Vert_{C^2(K_m,\mathbb{C})}
< \frac{1}{2^m}
\leq \frac{1}{2^n}
< \eps.
\]
Therefore, we constructed an $\mathbb{S}$-valued  sequence  $(f_m) $ s.t. for every compact $K\subset \mathbb{R}$,
\[
\Vert g - f_m \Vert_{C^2(K,\mathbb{C})}< \eps.
\]
\end{proof}
Let us now move to the proof of Theorem \ref{First-main-result}, our main result.

\section{Proof of Theorem \ref{First-main-result}}\label{sec-proof-main-result}
The proof of Theorem \ref{First-main-result} is divided into four parts. \\
\textbf{Part 1} extends the local martingale property of the process $M^{\theta z}=(M_t^{\theta z}: t \in [0,T])$, $\theta \in \mathbb{R}$, from the specific maps $\mathbb{R} \ni x \mapsto e^{i \theta x} \in \mathbb{C}$ to all functions of $C^\infty$-class with compact support on $\mathbb{R}$. 
Furthermore, we show that for every $z \in \mathscr{V}$, the process $\langle \bX(t), z \rangle$, $t \in [0,T]$, is a continuous semimartingale. 

\noindent
In \textbf{Part 2}, we prove that there exists an $(\bZ_T,\mathbb{D},\mathbb{P}_T)$ local martingale $L_t^z$, $t \in [0,T]$, with
    \begin{equation}\label{Eqn-quadratic-variation-of-L_t^z}
      \langle L^z, L^z \rangle_t
       = \int_0^t \sum_{k=1}^\infty \dualitybig{\newG^{k}_0(s,\bX(s))}{z}{\nU}{\nU^\prime}{2}\, \d s,\;\; t \in [0,T].
    \end{equation}
Let us point out that a key step in the proof of \eqref{Eqn-quadratic-variation-of-L_t^z} relies on the integration by parts formula and the fact that the process $\langle \bX(t), z \rangle$, $t \in [0,T]$,  is a continuous semimartingale.

\noindent
In \textbf{Part 3} we will  prove that there exists an
$\nHB$-valued continuous local square integrable $\mathbb{D}$-martingale $M=(M_t: t \in [0,T])$, on $(\bZ_T,\mathbb{P}_T)$ such that  
     \begin{equation*}
       M_t= \sum_{j=1}^\infty L_t^{\tilde{\be}_j} \tilde{\be}_j, \;\; t \in [0,T],
    \end{equation*}
where $(\tilde{\be}_j)_j$ is an ONB basis of $\nHB$. 

\noindent
Finally, \textbf{Part 4} is devoted to the application of Lemma \ref{Lem-B.4}.
\subsection{Proof of Theorem \ref{First-main-result}, Part 1}
Fix $T>0$. Let us choose and fix $z \in \mathscr{V}= \mathcal{V} \times \mathcal{C}_0^\infty(\domO,\mathbb{R})$. 
In Theorem \ref{thm-5.3}, we proved that there exists a probability measure $\mathbb{P}_T$ on $\bZ_T$ such that \eqref{X-estimate}-\eqref{phi-H3-estimate} hold. Moreover, we proved, see also Remark \ref{rem-def-martingale-solution},  that the processes $M^{\theta z}=(M_t^{\theta z}: t \in [0,T])$, $\theta \in \mathbb{R}$, defined,  on the canonical sample space 
$\bZ_{T}$, by 
    \begin{equation}\label{eqn-Def-M-t-z-1}
       M_t^{\theta z}\coloneqq M_t^{\theta z}(\bX)
       = e^{i \theta\langle \bX(t), z \rangle}  - I_{0,t}^{ \theta z}(\bX),\;\; t\in [0,T],
    \end{equation}
are  continuous local martingales with the same localizing sequence of stopping times.
Furthermore, since $I_{0,t}^{ \theta z}(\bX)$ is a continuous process with finite variation, it follows from definition of a semimartingale, \cite[Definition 4.21]{Jacod+Shiryaev_2002}, that the process $e^{i \theta \langle \bX(t), z \rangle}$ is a complex semimartingale.
\newline
Next, following \cite{Jacod+Shiryaev_2002}, we will show  the following result. 
\begin{lemma}\label{lem-semimartingale}
The process $\langle \bX(t), z \rangle$,  $t \in [0,T]$,  is a continuous $\mathscr{D}_t$-semimartingale.     
\end{lemma}
\begin{proof}[Proof of Lemma \ref{lem-semimartingale}] 
Note that for every  $\lambda \in \mathbb{R}$, $\sin(\lambda \langle \bX(t), z \rangle)$, $t \in [0,T]$, is a semimartingale.
\newline
Let us observe that there exists $h \in \mathcal{C}^2(\mathbb{R})$ such that $h(\sin y)= y$ whenever $\lvert y \rvert \leq \frac12$. 
\newline
Therefore, on the stochastic interval $[0,T_k)$, $k \in \mathbb{N}$, where
\[
T_k= \inf\{t: \lvert \langle \bX(t), z \rangle \rvert> k/2\},
\]
with $T_k \to T$ a.s., the process $\langle \bX(t), z \rangle$, $t \in [0,T]$, coincides with the semimartingale $k h(\sin(\langle \bX(t), z \rangle/k))$.
Therefore, by \cite[Proposition 4.25]{Jacod+Shiryaev_2002}, $\langle \bX(t), z \rangle$ is a semimartingale.
\end{proof}
Next, let us introduce the following $C^\infty$-class  function
\[\tilde{e}_\theta: \mathbb{R} \ni x \mapsto e^{ i \theta x  } \in \mathbb{C},
\]
with $\theta \in \mathbb{R}$, and we note that
   \begin{equation*}
       D \tilde{e}_\theta= i  \theta \tilde{e}_\theta \quad \mbox{ and } \quad D^2 \tilde{e}_\theta= - \, \theta^2 \tilde{e}_\theta.
    \end{equation*}
Analogously to \eqref{eqn-L_n-2} or \eqref{eqn-L_n-3}, we introduce the following linear operator, which transforms functions of $C^2_b$-class on $\mathbb{R}$ into $C_b$-functions on $\nU^\prime$, where $\nU=\rU_0 \times \Vtwo$:
    \begin{equation}\label{eqn-L_n-4}
      \widetilde{\mathscr{L}}_{s}(f)(y)
      = f^\prime(\duality{y}{z}{}{}) \duality{\bF(y)}{z}{}{} 
         + \frac12 f^{\prime\prime} (\duality{y}{z}{}{}) \sum_{k=1}^\infty \duality{\newG_{0}^k(s,y)}{z}{}{}{}{^2},\;\; y \in \nU^\prime.
    \end{equation}
In particular, one has
\begin{equation*}
\widetilde{\mathscr{L}}_s(\tilde{e}_\theta)(y)
= i\theta \tilde{e}_\theta (\duality{y}{z}{}{}) \duality{\bF(y)}{z}{}{} 
- \frac12 \theta^2  \tilde{e}_\theta (\duality{y}{z}{}{}) \sum_{k=1}^\infty 
    \duality{\newG_{0}^k(s,y)}{z}{}{}{}{^2},\;\; y \in \nU^\prime,
\end{equation*}
and therefore, we have the following representation of the local continuous martingale $M_t^{\theta z}$: 
\begin{equation}\label{eqn-Def-M-t-z-a}
\begin{aligned}
M_t^{\theta z}
=  \tilde{e}_\theta(\duality{\bX(t)}{z}{}{})  - \int_0^t \widetilde{\mathscr{L}}_s(\tilde{e}_\theta)(\bX(s))\,\d s,\;\; t \in [0,T].
\end{aligned}
\end{equation}
Now, we take an $\mathscr{R}$-valued sequence  $(R_k)_{k=1}^\infty $, see Lemma \ref{Lemma 11.1.2-SV2006}, and  the corresponding stopping times $\tau_{R_k}$. 
Let us choose and fix $s,\,t \in [0,T]$ such that $t \geq s$. The local martingale property implies that 
   \begin{equation}\label{eqn-Def-M-t-z-b}
     \mathbb{E}^{\mathbb{P}_T} \left(M_{t \wedge\tau_{R_k}}^{\theta z} - M_{s \wedge\tau_{R_k}}^{\theta z} \vert \mathscr{D}_s \right)=0,
   \end{equation}
or, equivalently, 
    \begin{equation*}
     \int_{A} \left[\tilde{e}_\theta(\duality{\bX(t \wedge\tau_{R_k})}{z}{}{}) - \tilde{e}_\theta(\duality{\bX(s \wedge\tau_{R_k})}{z}{}{}) - \int_{s \wedge \tau_{R_k}}^{t \wedge \tau_{R_k}} \widetilde{\mathscr{L}}_s(\tilde{e}_\theta)(\bX(s))\,\d s\right]\d \mathbb{P}_T
     = 0,\; \forall A \in \mathscr{D}_s.
  \end{equation*}
It is important to observe that the same sequence of stopping times works for every $\theta \in \mathbb{R}$.
Therefore,  replacing $z$ by $\theta_j z$ in \eqref{eqn-Def-M-t-z-b}, multiplying the resulting equation by $c_j$ and then taking the summation over $j$, we obtain that for every $h \in \mathbb{S}$,
    \begin{equation}\label{eqn-Def-M-t-z-c}
         \mathbb{E}^{\mathbb{P}_T} \left(M_{t \wedge \tau_{R_k}}^{z,h} - M_{s \wedge \tau_{R_k}}^{z,h} \vert \mathscr{D}_s \right)=0,
    \end{equation}  
where
    \begin{equation*}
      M_t^{z,h}= h(\duality{\bX(t)}{z}{}{}) - \int_0^{t} \widetilde{\mathscr{L}}_s(h)(\bX(s))\,\d s.
    \end{equation*}
Next, let us choose and fix $g\in C_0^\infty(\mathbb{R},\mathbb{C})$ and an increasing  
 sequence $(K_m)$ of compact subsets of $\mathbb{R}$ such that $\bigcup_m \stackrel{\circ}{K_m}= \mathbb{R}$. 
By Corollary \ref{eqn-Cor-uniform-convergence}, for every $m \in\mathbb{N}$, we can  pick $h_m \in \mathbb{S}$ such that 
\[
\Vert g - h_m \Vert_{C^2(K_m,\mathbb{C})}
< \frac{1}{2^m}.
\]
Subsequently, since for each $m$,  $h_m \in \mathbb{S}$, by \eqref{eqn-Def-M-t-z-c}, we obtain
     \begin{equation}\label{eqn-Def-M-t-z-d}
         \mathbb{E}^{\mathbb{P}_T} \left(M_{t \wedge \tau_{R_k}}^{z,h_m} - M_{s \wedge \tau_{R_k}}^{z,h_m} \vert \mathscr{D}_s \right)=0.
    \end{equation}
Therefore, by means of the Lebesgue DCT for conditional expectation, see Section 9.7(g) in \cite{Williams_1991},   we can pass to the limit by letting $m \to \infty$ and deduce that 
    \begin{equation}
      \mathbb{E}^{\mathbb{P}_T} \left(M_{t \wedge \tau_{R_k}}^{z,g} - M_{s \wedge \tau_{R_k}}^{z,g} \vert \mathscr{D}_s \right)=0, 
    \end{equation}
where
     \begin{align}\label{Eqn-def-M_t^{z,g}}
        M_t^{z,g}=  g(\duality{\bX(t)}{z}{}{}) - \int_0^t \widetilde{\mathscr{L}}_s(g)(\bX(s))\,\d s.
     \end{align}
Hence, we infer that  the process $M_t^{z,g}$ is an $(\bZ_T,\mathbb{D},\mathbb{P}_T)$ continuous local martingale. This completes the proof in Part 1 of Theorem \ref{First-main-result}.
\subsection{Proof of Theorem \ref{First-main-result}, Part 2}
Hereafter, we choose and fix $s,\,t \in [0,T]$ such that $t \geq s$. Assume $z \in \mathscr{V}= \mathcal{V} \times \mathcal{C}_0^\infty(\domO,\mathbb{R})$. Next, we will prove the following result. 
\begin{lemma}\label{lem-L_t^z}
The  $\mathbb{R}$-valued process $L^z=(L_t^{z}: t \in [0,T])$ defined by the formula
   \begin{equation}\label{eq-5.56-1}
     L_t^{z}= \duality{\bX(t)}{z}{}{} - \duality{\bX(0)}{z}{}{} - \int_0^t \duality{\bF(\bX(s))}{z}{}{}\,\d s
   \end{equation}
is a continuous  local martingale on  $(\bZ_T,\mathbb{D},\mathbb{P}_T)$ and 
   \begin{equation}\label{eqn-L_t^z-square}
       \lvert L_t^{z} \rvert^2 - \int_0^t \sum_{k=1}^\infty \dualitybig{\newG^{k}_0(s,\bX(s))}{z}{\nU}{\nU^\prime}{2}\, \d s \in \mathcal{M}_{\loc}^{\rc}(\mathbb{D},\mathbb{P}_T).
   \end{equation}
   \end{lemma}
   \begin{proof}[Proof of Lemma \ref{lem-L_t^z}]
Let us introduce the following standard $C^\infty$-class function $\psi: \mathbb{R} \to [0,1]$ defined by 
\begin{equation*}
\psi(x)=
\begin{cases}
e^{x^2/(x^2-1)} &\mbox{ if } \lvert x \rvert<1,
\\
0 &\mbox{ if } \lvert x \rvert \geq 1,
\end{cases}
\end{equation*}
and we  choose and fix $p \in \mathbb{N}$. Let $g_p:\mathbb{R}\ni x \mapsto  x^p \in  \mathbb{R}$ and  
for any $n \in \mathbb{N}$, let 
\[
g_{n,p}:\mathbb{R}\ni x \mapsto  g_p(x) \psi(x/n) \in \mathbb{R}
\]
so that $g_{n,p} \in C_0^\infty(\mathbb{R})$. Then the following properties hold:
   \begin{align}\label{eqn-properties-g_n-limits}
    \lim_{n \to \infty} g_{n,p}(x)= g_p(x), \; \; \lim_{n \to \infty} g_{n,p}^\prime(x)= g^\prime_p(x), \mbox{ and } \; \; \lim_{n \to \infty} g_{n,p}^{\prime \prime}(x)= g_p^{\prime \prime}(x),
   \end{align}
locally uniformly w.r.t.  $x \in \mathbb{R}$. Moreover, there exists $C>0$  such that for all $n \in \mathbb{N},\, x \in \mathbb{R}$, 
  \begin{align}\label{eqn-properties-g_n-bounded}
   \lvert g_{n,p}(x) \rvert \leq e, \;\;\;  \lvert g_{n,p}^\prime(x) \rvert \leq 4e \mbox{ and } \lvert g_{n,p}^{\prime \prime} (x) \rvert \leq C.
   \end{align}
Now, it follows from  Part 1  that for every $n \in \mathbb{N}$, the process $M_t^{z,g_{n,1}}$ is a local continuous  martingale, which, in turn, yields that the process $L_t^{z,n}\coloneqq  M_t^{z,g_{n,1}} - g_{n,1}(\duality{\bX(0)}{z}{}{}) \in \mathcal{M}_{\loc}^{\rc}(\mathbb{D},\mathbb{P}_T)$,
and, in view of identity   \eqref{Eqn-def-M_t^{z,g}}, 
\begin{equation*}
 g_{n,1}(\duality{\bX(t)}{z}{}{}) -  g_{n,1}(\duality{\bX(0)}{z}{}{})
- \int_0^t \widetilde{\mathscr{L}}_s(g_{n,1})(\bX(s))\,\d s
= L_t^{z,n},
\end{equation*}
i.e. 
\begin{equation}\label{eqn-Def-M-t-z-f}
\begin{aligned}
0&= \mathbb{E}^{\mathbb{P}_T} \left(g_{n,1}(\duality{\bX(t \wedge\tau_{R_k})}{z}{}{}) - g_{n,1}(\duality{\bX(s \wedge\tau_{R_k})}{z}{}{}) - \int_{s \wedge\tau_{R_k}}^{t \wedge \tau_{R_k}} \widetilde{\mathscr{L}}_s(g_{n,1})(\bX(s))\,\d s \vert \mathscr{D}_s \right).
\end{aligned}
\end{equation}
Therefore, by \eqref{eqn-L_n-4}, the properties of the sequence $(g_{n,1})_{n \in \mathbb{N}}$ expressed in \eqref{eqn-properties-g_n-limits} (with $p=1$) and \eqref{eqn-properties-g_n-bounded}, and the Lebesgue DCT, it follows from
\eqref{eqn-Def-M-t-z-f} that
\begin{equation}\label{eqn-Def-M-t-z-e}
\begin{aligned}
0&= \mathbb{E}^{\mathbb{P}_T} \left(g_1(\duality{\bX(t \wedge\tau_{R_k})}{z}{}{}) - g_1(\duality{\bX(s \wedge\tau_{R_k})}{z}{}{}) - \int_{s \wedge\tau_{R_k}}^{t \wedge \tau_{R_k}} \widetilde{\mathscr{L}}_s(g_1)(\bX(s))\,\d s\vert \mathscr{D}_s \right)
\\
&=\mathbb{E}^{\mathbb{P}_T} \left(\duality{\bX(t \wedge\tau_{R_k})}{z}{}{} -\duality{\bX(s \wedge\tau_{R_k})}{z}{}{} - \int_{s \wedge\tau_{R_k}}^{t \wedge \tau_{R_k}} \duality{\bF(\bX(s))}{z}{}{}\,\d s
\vert \mathscr{D}_s \right)
\\
&\coloneqq
\mathbb{E}^{\mathbb{P}_T} \left(L_{t \wedge\tau_{R_k}}^{z} - L_{s \wedge\tau_{R_k}}^{z} \vert \mathscr{D}_s \right).
\end{aligned}
\end{equation}
Consequently, the process $L_t^z= \duality{\bX(t)}{z}{}{} - \duality{\bX(0)}{z}{}{} - \int_0^t \duality{\bF(\bX(s))}{z}{}{}\,\d s$ is a  continuous local martingale on $(\bZ_T,\mathbb{D},\mathbb{P}_T)$.

\noindent
We now move to the proof of the second part of our claim, i.e. \eqref{eqn-L_t^z-square}.
\newline
Recall that we have already shown that the process $M^{z,g_{n,2}} \in \mathcal{M}_{\loc}^{\rc}(\mathbb{D},\mathbb{P}_T)$.
Hence, the process $N_t^{z,n} \coloneqq M_t^{z,g_{n,2}} - g_{n,2}(\duality{\bX(0)}{z}{}{})$ also belongs to $\mathcal{M}_{\loc}^{\rc}(\mathbb{D},\mathbb{P}_T)$. 
Thus, by arguing as in the steps leading to \eqref{eqn-Def-M-t-z-e}, and applying \eqref{eqn-properties-g_n-limits} with $p=2$, we infer that the process
\begin{align*}
&N_t^z\coloneqq g_2(\duality{\bX(t)}{z}{}{}) - g_2(\duality{\bX(0)}{z}{}{}) - \int_0^t \widetilde{\mathscr{L}}_s(g_2)(\bX(s))\,\d s
\\
&= \duality{\bX(t)}{z}{}{}{}{^2} - \duality{\bX(0)}{z}{}{}{}{^2}  - 2 \int_0^t \duality{\bX(s)}{z}{}{} \duality{\bF(\bX(s))}{z}{}{}\,\d s - \int_0^t \sum_{k=1}^\infty \duality{\newG^{k}_0(s,\bX(s))}{z}{}{}{}{^2}\, \d s
\end{align*}
is a continuous $\mathbb{D}$ local martingale. 
\newline
For the sake of brevity, we put
\[
H_t^z= \int_0^t \duality{\bF(\bX(s))}{z}{}{}\,\d s \mbox{ and } G_t^z= \int_0^t \sum_{k=1}^\infty \duality{\newG^{k}_0(s,\bX(s))}{z}{}{}{}{^2}\, \d s,
\]
and we observe that
\begin{equation}\label{eqn-(L_t)^2-G_t^z-equality}
\begin{aligned}
&\lvert L_t^z \rvert^2 - G_t^z=\lvert L_t^z +\duality{\bX(0)}{z}{}{} - \duality{\bX(0)}{z}{}{} \rvert^2 - (G_t^z+N_t^z)+ N_t^z
\\
&= - 2 \duality{\bX(0)}{z}{}{} L_t^z + N_t^z + \lvert L_t^z + \duality{\bX(0)}{z}{}{} \rvert^2 - \duality{\bX(0)}{z}{}{}{}{^2} -G_t^z - N_t^z
 \\
&= - 2 \duality{\bX(0)}{z}{}{} L_t^z + N_t^z + \lvert \duality{\bX(t)}{z}{}{} - H_t^z \rvert^2 - \duality{\bX(t)}{z}{}{}{}{^2} + 2 \int_0^t \duality{\bX(s)}{z}{}{}\d H_s^z
\\
&= - 2 \duality{\bX(0)}{z}{}{} L_t^z + N_t^z - 2 \duality{\bX(t)}{z}{}{} H_t^z + \lvert H_t^z \rvert^2 + 2 \int_0^t \duality{\bX(s)}{z}{}{}\d H_s^z.
\end{aligned}
\end{equation}
From \eqref{eq-5.56-1}, one can see that $\duality{\bX(t)}{z}{}{}= \duality{\bX(0)}{z}{}{} + H_t^z + L_t^z$. Since by Lemma \ref{lem-semimartingale} the process $\langle \bX(t), z \rangle$,  $t \in [0,T]$,  is a continuous $\mathscr{D}_t$-semimartingale, an application of \cite[Proposition 3.1]{Revuz+Yor_1999} entails that
\begin{align*}
&\duality{\bX(t)}{z}{}{} H_t^z
= \int_0^t \duality{\bX(s)}{z}{}{}\d H_s^z + \int_0^t H_s^z \d \duality{\bX(s)}{z}{}{}
\\
&= \int_0^t \duality{\bX(s)}{z}{}{}\d H_s^z + \int_0^t H_s^z \d H_s^z + \int_0^t H_s^z \d L_s^z
= \int_0^t \duality{\bX(s)}{z}{}{}\d H_s^z + \frac12 (H_t^z)^2 + \int_0^t H_s^z \d L_s^z.
\end{align*}
Plugging this into the RHS of \eqref{eqn-(L_t)^2-G_t^z-equality}, we obtain
   \begin{equation}\label{eqn-(L_t)^2-G_t^z-equality-1}
     \lvert L_t^z \rvert^2 - G_t^z
     = - 2 \duality{\bX(0)}{z}{}{} L_t^z + N_t^z - 2 \int_0^t H_s^z \d L_s^z.
  \end{equation}
Since the terms on the RHS of \eqref{eqn-(L_t)^2-G_t^z-equality-1} are local martingales, we obtain \eqref{eqn-L_t^z-square}. 
Thus the proof  of Lemma \ref{lem-L_t^z} is complete.
\end{proof}
Obviously Lemma  \ref{lem-L_t^z}  allows us to  conclude the proof of Part 2 of Theorem \ref{First-main-result}.
\subsection{Proof of Theorem \ref{First-main-result}, Part 3}
Assume that   a $\mathbb{V}$-valued  sequence  $\left( \tilde{\be}_k: k \in \mathbb{N}\right)$  is an ONB of $\nHB$, see Remark 
\ref{rem-def-compact-modified-stochastic-CHNSEs-3},  that is  orthogonal in $\nU$.  
\begin{claim}\label{claim-0}
Let us  define the stopping time $\tau_{R_0}$ for any $R_0>0$:
\begin{equation}\label{eqn-tau_R_0}
\tau_{R_0} \coloneqq \inf\left\{t \in [0,T]: \int_0^t  \Vert \newG_0(s,\bX(s)) \Vert_{\ell^2(\nHB)}^2\,\d s\geq R_0^2 \right\} \wedge T.
\end{equation}
Then 
\begin{equation}\label{eqn-tau_R_0-b}
\tau_{R_0}(\bX)\nearrow T \mbox{ as } R_0 \to \infty, \;\; \mathbb{P}_T\mbox{-a.e. }\bX \in \bZ_T.
\end{equation}
\end{claim}
\begin{proof}[Proof of Claim \ref{claim-0}]
Let us choose and fix  $R_0>0$ and let  $\tau_{R_0}$ be  the stopping time  defined by  \eqref{eqn-tau_R_0}. 
By Assumptions $(iv)$ and $(vi)$ in Section \ref{Ass-Abstract formulation}, we observe that for all $s \in [0,T]$,
\begin{equation}\label{Eqn-Hilbert-norm-Sigma+G-1}
\begin{aligned}
&\Vert \newG_0(s,\bX(s)) \Vert_{\ell^2(\nHB)}^2
=\Vert \bSi_0(s)\bu(s) + \bG_0(s,\bu(s))\Vert_{\ell^2(\StokesH)}^2
\\
&\leq 2 \Vert \bSi_0(s)\bu(s) \Vert_{\ell^2(\StokesH)}^2 + 2 \Vert \bG_0(s,\bu(s))\Vert_{\ell^2(\StokesH)}^2
\leq C [\Vert \bu(s) \Vert_{\StokesV}^2 + \vert \tilde{h}(s)\vert^2 + \vert \bu(s) \vert_{\StokesH}^2].
\end{aligned}
\end{equation}
Then it follows from the estimate \eqref{X-estimate} in Theorem \ref{thm-5.3} that assertion \eqref{eqn-tau_R_0-b} is satisfied. 
\end{proof}
Now, let us choose and fix $R_0>0$. We will  prove the following assertion.
\begin{claim}\label{claim-2}
The sequence of processes $M^{R_0,n}= (M_t^{R_0,n}: \, t \in [0,T])$ defined by
    \begin{equation}\label{eqn-serie_M_t-0}
      M_t^{R_0,n} \coloneqq \sum_{j=1}^n L_{t \wedge \tau_{R_0}}^{\tilde{\be}_j} \tilde{\be}_j, \;\; t \in [0,T],
    \end{equation}
is   Cauchy in   $L^2(\bZ_T;\mathbb{P}_T, C([0,T];\nHB))$.
\end{claim}
\begin{proof}[Proof of Claim \ref{claim-2}]
Let us choose and fix $n,\,m \in \mathbb{N}$.
Then from \eqref{eqn-L_t^z-square}, by applying \cite[Exercise 5.19]{Karatzas+Shreve_1991}, we deduce that
\begin{equation}
\begin{aligned}
&\mathbb{E}^{\mathbb{P}_T} \Vert M_t^{R_0,n} \Vert_{\nHB}^2
=  \mathbb{E}^{\mathbb{P}_T} \left[\sum_{j=1}^n \vert  L_{t \wedge \tau_{R_0}}^{\tilde{\be}_j} \vert^2 \right]
\leq  \mathbb{E}^{\mathbb{P}_T} \left[\sum_{j=1}^n \int_0^{t \wedge \tau_{R_0}}  \sum_{k=1}^\infty \duality{\newG^{k}_0(s,\bX(s))}{\tilde{\be}_j}{}{}{}{^2}\,\d s\right]
\\
&\leq \mathbb{E}^{\mathbb{P}_T} \left[\int_0^{t \wedge \tau_{R_0}}  \sum_{k=1}^\infty \Vert \newG^{k}_0(s,\bX(s)) \Vert_{\nHB}^2\,\d s \right]
= \mathbb{E}^{\mathbb{P}_T} \left[\int_0^{t \wedge \tau_{R_0}} \Vert \newG_0(s,\bX(s)) \Vert_{\ell^2(\nHB)}^2 \,\d s \right].
\end{aligned}
\end{equation}
Combining the above inequality with the definition of $\tau_{R_0}$, we infer that there exists a constant $C>0$ such that,
for every $t \in [0,T]$ and every $n \in \mathbb{N}$,
\begin{align*}
\mathbb{E}^{\mathbb{P}_T} \Vert M_t^{R_0,n} \Vert_{\mathbb{H}}^2
\leq C.
\end{align*}
Thus, the process $M^{R_0,n}$ is an $\nHB$-valued continuous square integrable martingale. \newline
Moreover, by the polarization identity
\[
\langle L^{\tilde{\be}_l}, L^{\tilde{\be}_j} \rangle
=\frac14 \left(\langle L^{\tilde{\be}_l} + L^{\tilde{\be}_j} , L^{\tilde{\be}_l} + L^{\tilde{\be}_j} \rangle - \langle L^{\tilde{\be}_l} -L^{\tilde{\be}_j} , L^{\tilde{\be}_l} - L^{\tilde{\be}_j}\rangle \right),
\]
and the fact that the quadratic variation of the process $L^z=(L_t^z: t \in [0,T])$ in $\mathbb{R}$ is $\langle L^z, L^z\rangle_t= G_t^z$, 
it then follows that the quadratic variation of $M^{R_0,n}$ in $\nHB$ is given by
\begin{equation}\label{eqn-serie_<M^{R_0,n}, M^{R_0,n}>_{t}}
\begin{aligned}
&\langle \langle M^{R_0,n}, M^{R_0,n} \rangle \rangle_{t}
= \sum_{l,j=1}^{n} \langle L^{\tilde{\be}_l}, L^{\tilde{\be}_j}\rangle_{t \wedge \tau_{R_0}}\, \tilde{\be}_l \otimes \tilde{\be}_j
\\
&= \sum_{l,j=1}^{n} \int_0^{t \wedge \tau_{R_0}} \sum_{k=1}^\infty \duality{\newG^{k}_0(s,\bX(s))}{\tilde{\be}_l}{}{}\duality{\newG^{k}_0(s,\bX(s))}{\tilde{\be}_j}{}{}\,\tilde{\be}_l \otimes \tilde{\be}_j\,\d s.
\end{aligned}
\end{equation}
Furthermore, we have
\begin{equation*}
\begin{aligned}
&\mathbb{E}^{\mathbb{P}_T} \left[\sum_{k,j=1}^{+ \infty} \int_{0}^{T \wedge \tau_{R_0}} \duality{\newG^{k}_0(s,\bX(s))}{\tilde{\be}_j}{}{}{}{^2}\,\d s \right]
= \mathbb{E}^{\mathbb{P}_T} \left[\int_0^{T \wedge \tau_{R_0}} \sum_{k=1}^\infty \Vert \newG^{k}_0(s,\bX(s)) \Vert_{\nHB}^2\,\d s \right]
\\
&= \mathbb{E}^{\mathbb{P}_T} \left[\int_0^{T \wedge \tau_{R_0}} \Vert \newG_0(s,\bX(s)) \Vert_{\ell^2(\nHB)}^2 \,\d s \right]<\infty.
\end{aligned}
\end{equation*}
This, jointly with the Burkholder inequality, yields that
\begin{align*}
&\mathbb{E}^{\mathbb{P}_T} \sup_{t \in [0,T]} \Vert M_t^{R_0,m} - M_t^{R_0,n} \Vert_{\nHB}^2
\leq C \mathbb{E}^{\mathbb{P}_T} \Vert M_T^{R_0,m} - M_T^{R_0,n} \Vert_{\nHB}^2
\\
&\leq C \mathbb{E}^{\mathbb{P}_T} \left[\sum_{j= n + 1}^m \int_0^{T \wedge \tau_{R_0}} \sum_{k=1}^\infty \duality{\newG^{k}_0(s,\bX(s))}{\tilde{\be}_j}{}{}{}{^2}\,\d s\right]
\underset{n,m \to \infty}{\to}  0.
\end{align*}
Therefore, Claim \ref{claim-2} follows.
\end{proof}

\begin{claim}\label{claim-1}
The series 
    \begin{equation}\label{eqn-serie_M_t}
       M_t \coloneqq  \sum_{j=1}^\infty L_t^{\tilde{\be}_j} \tilde{\be}_j, \;\; t \in [0,T].
\end{equation}
is convergent in  $L^2(\bZ_T,\mathbb{P}_T;C([0,T];\nHB))$ and the process $M$ defined as the sum of that series is an
$\nHB$-valued continuous local 
$\mathbb{D}$-martingale, see \eqref{eqn-mb-D},   on $(\bZ_T,\mathbb{P}_T)$.
\end{claim}

\begin{proof}[Proof of Claim \ref{claim-1}] In view of Lemma E.2 from \cite{BKMR_2025}, it is sufficient to prove the martingale property with respect to the original filtration 
$\mathscr{D}_t= \sigma(\bX_s,\,s \leq t), t\in [0,T]$. \\
By  Claim \ref{claim-2}, we deduce that the sequence $(M^{R_0,n}: n\in \mathbb{N})$ has a limit  in   $L^2(\bZ_T,\mathbb{P}_T;C([0,T];\nHB))$, which we denote by 
$M^{R_0}$ and 
\begin{align*}
&\lim_{n\to \infty} \mathbb{E}^{\mathbb{P}_T} \sup_{t \in [0,T]} \Vert M_t^{R_0,n} - M_t^{R_0} \Vert_{\nHB}^2=0 .
\end{align*}
Hence, since $1+\delta>1$, $\delta \in (0,1)$, by the Vitali Theorem, see \cite[Theorem C.4]{Oksendal_2003}, we infer that 
 $\mathbb{P}_T$-a.s.
\[
\lim_{n\to \infty} \sup_{t \in [0,T]} \Vert M_t^{R_0,n} - M_t^{R_0} \Vert_{\nHB}^{1+\delta}=0. 
\]
This implies that $\mathbb{P}_T$-a.s., for every $t \in [0,T]$, the sequence $ M_t^{R_0,n}$ is convergent in $\nHB$ as $n \to \infty $, i.e. 
\[
\sum_{j=1}^{n} L_{t \wedge \tau_{R_0}}^{\tilde{\be}_j} \tilde{\be}_j \mbox{ is convergent in }\nHB\mbox{ as } n \to \infty.
\]
This means that $\mathbb{P}_T$-a.s., for every $t \in [0,T]$, the series $\sum_{j=1}^\infty L_{t \wedge \tau_{R_0}}^{\tilde{\be}_j} \tilde{\be}_j$ is convergent in $\nHB$ and 
by \eqref{eqn-serie_M_t-0} that 
      \begin{equation}\label{eqn-serie_M_t-a-2}
        M_t^{R_0} = \sum_{j=1}^\infty L_{t \wedge \tau_{R_0}}^{\tilde{\be}_j} \tilde{\be}_j.
      \end{equation}
Observe that $\mathbb{P}_T$-a.s.,  $\tau_{R_0} \nearrow T$ as $R_0 \to \infty$. Therefore, again   $\mathbb{P}_T$-a.s., if $t<T$, then 
\[
\sum_{j=1}^\infty L_{t \wedge \tau_{R_0}}^{\tilde{\be}_j} \tilde{\be}_j=
\sum_{j=1}^\infty L_{t}^{\tilde{\be}_j} \tilde{\be}_j
= M_t.
\]
Therefore, Claim \ref{claim-1} follows.
\end{proof}
This completes the proof of Part 3 of Theorem \ref{First-main-result}. \newline
We also have the following two auxiliary results.
\begin{claim}\label{claim-4}
The stopped process $M^{R_0}= (M_{t \wedge \tau_{R_0}}: \, t \in [0,T])$ is $\nHB$-valued continuous square integrable martingale with quadratic variation,  see \cite[Section 21.5]{Metivier_1982},  \cite[Section 10.2]{Metivier+Pellaumail_1980} and \cite[Section 2]{Metivier_1988},
\begin{align*}
&\langle \langle M^{R_0}, M^{R_0} \rangle \rangle_{t}
= \sum_{l,j=1}^{+\infty} \langle L^{\tilde{\be}_l}, L^{\tilde{\be}_j}\rangle_{t \wedge \tau_{R_0}}\, \tilde{\be}_l \otimes \tilde{\be}_j
\\
&= \sum_{l,j=1}^{+\infty} \int_0^{t \wedge \tau_{R_0}} \sum_{k=1}^\infty \duality{\newG^{k}_0(s,\bX(s))}{\tilde{\be}_l}{}{}  \duality{\newG^{k}_0(s,\bX(s))}{\tilde{\be}_j}{}{}\,\tilde{\be}_l \otimes \tilde{\be}_j\,\d s,\;\; t \in [0,T].
\end{align*}
\end{claim}
\begin{proof}[Proof of Claim \ref{claim-4}]
We recall that the real process $L^z=(L_t^{z}: t \in [0,T])$, $z \in \mathscr{V}$ is a continuous  local martingale on  $(\bZ_T,\mathbb{D},\mathbb{P}_T)$, cf. Lemma \ref{lem-L_t^z}.
Now, by the above pointwise convergence of the partial sum $M^{R_0,n}$ and the equality \eqref{eqn-serie_<M^{R_0,n}, M^{R_0,n}>_{t}}, we deduce the Claim \ref{claim-4}.
\end{proof}
\begin{corollary}\label{cor-L=M} We have, for every $t \in [0,T]$,  
\begin{equation}\label{eqn-M_tv=L_t^v}
\langle M_t, \bv \rangle=L_t^{\bv}, \;\; \forall \bv \in \nU.
  \end{equation}
\end{corollary}
\begin{proof}[Proof of Corollary \ref{cor-L=M}]
We prove that the series $M_t=\sum_j L_t^{\tilde{\be}_j}\tilde{\be}_j$ is convergent a.s. in $C([0,T];\nHB)$, see Claim \ref{claim-1}. Having done this, it will  follow that  a.s.  for every $k \in \mathbb{N}$,
\[
\langle M_t, \tilde{\be}_k\rangle=L_t^{\tilde{\be}_k}, \;\; \mbox{ for every } t\in [0,T],
\]
and, by linearity, we deduce  \eqref{eqn-M_tv=L_t^v}.
\end{proof}
\subsection{Proof of Theorem \ref{First-main-result}, Part 4}
In this section, our aim is to show that Lemma \ref{Lem-B.4} is applicable.  Let us  define the following operator valued function 
\[\widetilde{Q}: [0,T] \times \bZ_T \to \mathscr{L}^+(\mathbb{H},\mathbb{H}),\]
by, for $s\in [0,T]$, $\bX=(\bu,\phi)  \in \bZ_T$   and  $\bv, \tilde{\bv} \in \nU$,
\begin{equation}
\begin{aligned}
\ilsc{\widetilde{Q}(s,\bX)\bv}{\tilde{\bv}}{\mathbb{H}}
& \coloneqq \sum_{k=1}^{\infty} (\newG^{k}_0(s,\bX(s)),\bv)_{\mathbb{H}} \cdot (\newG^{k}_0(s,\bX(s)), \tilde{\bv})_{\mathbb{H}}. 
\end{aligned}
\end{equation}
In view of assumptions \eqref{eqn-linear growth} and \eqref{Eqn-coercivity-3}  and definition of the space $\ell^2(\StokesH)$,  the series on the RHS above is convergent and moreover there exists a constant $C>0$ such that for all $\bv, \tilde{\bv} \in \nU$,
\begin{equation}
\vert \ilsc{\widetilde{Q}(s,\bX)\bv}{\tilde{\bv}}{\mathbb{H}} \vert \leq C \vert \bv\vert_{\mathbb{H}}    \vert \tilde{\bv}\vert_{\mathbb{H}}.
\end{equation}
This implies that the map $\widetilde{Q}(s,\bX)$ extends uniquely to a bounded linear map, also denoted by $\widetilde{Q}(s,\bX)$, from $\mathbb{H}$ to $\mathbb{H}$, which moreover is obviously positive.
Next, using the notation introduced in Example \ref{example:HS}, we define  the following function 
\begin{equation}\label{space-E}
    \begin{aligned}
        \widetilde{\sigma}: &[0,T] \times \bZ_T \times E \ni (s,\bX,j) \mapsto \widetilde{\sigma}_s(j;\bX)\coloneqq  \newG^{j}_0(s,\bX(s)) \in  \mathbb{H}.
    \end{aligned}
\end{equation}
Recall that  $L^2(E,\mathcal{E},\tilde{\kappa})=\ell^2 $ and observe that for all $s \in [0,T]$, $\bX=(\bu,\phi) \in \bZ_T$ and $\bv \in \nU$,
\begin{equation}\label{eqn-sigma on E}
\ilsc{\widetilde{Q}(s,\bX)\bv}{\bv}{\mathbb{H}}
= \sum_{k=1}^{\infty} ( \newG^{k}_0(s,\bX) , \bv)_{\mathbb{H}}^2
=\int_{E} (\newG^{j}_0(s,\bX(s)) , \bv)_{\mathbb{H}}^2 \d \tilde{\kappa}(j)
=\int_{E} (\widetilde{\sigma}_s(j;\bX), \bv)_{\mathbb{H}}^2 \d \tilde{\kappa}(j).
\end{equation}
Since by  \eqref{eqn-Gelfand-full}, $ \nU \embed \nHB\cong \mathbb{H}^\prime \embed \nU^\prime$
is a Gelfand triple, we can  naturally define functions 
\begin{equation}\label{eqn-def-Q}
  Q: [0,T] \times \bZ_T \to \mathscr{L}^+(\nU,\nU^\prime)
\end{equation}
and 
\begin{equation}\label{eqn-sigma}
    \begin{aligned}
        \sigma: &[0,T] \times \bZ_T \times E \ni (s,\bX,j) \mapsto \sigma_s(j;\bX)\coloneqq  \newG^{j}_0(s,\bX(s)) \in  \nU^\prime,
    \end{aligned}
\end{equation}
such that the following equality holds 
\begin{equation}\label{eqn-Q=sigma}
\duality{Q(s,\bX)\bv}{\bv}{\nU}{\nU^\prime} 
=\int_{E} \dualitybig{\sigma_s(j;\bX)}{\bv}{\nU}{\nU^\prime}{2}    \d \tilde{\kappa}(j).
\end{equation}
Thus, by Corollary \ref{cor-L=M}, equality \eqref{eqn-Q=sigma} and Lemma \ref{lem-L_t^z}, we see that for all $\bv \in \nU$,
\begin{equation}\label{eqn-1007}
\begin{aligned}
& \dualitybig{M_t}{\bv}{\nU}{\nU^\prime}{2} - \int_0^t  \int_E \dualitybig{\sigma_s(x)}{\bv}{\nU}{\nU^\prime}{2} \, \tilde{\kappa} (\d x) \, \d s 
\\ 
&=\dualitybig{M_t}{\bv}{\nU}{\nU^\prime}{2} - \int_0^t  \duality{ Q_s \bv}{\bv}{\nU}{\nU^\prime} \, \d s
     \\
&= \vert L_t^{\bv}\vert^{2} - \int_0^t  \sum_{k=1}^\infty \dualitybig{\newG^{k}_0(s,\bX(s))}{\bv}{\nU}{\nU^\prime}{2}\,  \, \d s \in \mathcal{M}_{\loc}^c(\mathbb{D},\mathbb{P}).
\end{aligned}   
 \end{equation}
Here is the moment when we need to work with spaces $\nU$ and $\nU^\prime$ because we need to use Lemma \ref{Lem-B.4}.
Therefore, it follows from Lemma \ref{Lem-B.4}   that there exists   an $\ell^2$-cylindrical Wiener process $W$ defined on 
 the probability space $(\bZ_T, \mathbb{D}, \mathbb{P}_T)$ such that  for every $t\in [0,T]$, $\mathbb{P}_T$-almost surely, 
         \begin{equation}\label{eqn-M-1}
            M_t=\int_0^t  \sigma_s \, \d W(s) \mbox{ in } \nU^\prime \mbox{ for all } t \in [0,T].
        \end{equation}
Using the above definition of the function $\sigma$, we deduce that 
       \begin{equation}\label{eqn-M-2}
            M_t=\int_0^t [ \newG_0(s,\bX(s))]\,\d W(s) 
            \mbox{ in } \nU^\prime \mbox{ for all } t \in [0,T].
       \end{equation}
Hence,  in view of Corollary \ref{cor-L=M} and the definition \eqref{eq-5.56-1} of the processes $L_t^z$, we infer that      
\begin{equation}\label{eqn-main-2}
\bX(t)=\bX_0+ \int_0^t  \bF(\bX(s))\,\d s + \int_0^t [ \newG_0(s,\bX(s))]\,\d W(s) \mbox{ in } \nU^\prime, \mbox{ for every } t \in [0,T],
\end{equation}            
This completes the proof of Part 4 of Theorem \ref{First-main-result}.
\medskip

\begin{remark}\label{Rmk-general-result}
The proof of Theorem \ref{First-main-result} can be generalized or adapted to the case where the $C^2$-class potential $\psi:\mathbb{R} \to \mathbb{R}$ satisfies 
the conditions listed in Remark \ref{rem-potential general}.
\end{remark}


\section{Applications of the abstract results from section \ref{Ass-Abstract formulation} to the stochastic NSCHEs}
\label{Sect-approximation of g} 
In this section, we verify that the maps $\bG_{0,n}$, defined in Assumption \ref{assumption-abstract-3} of the primary assumptions in Section \ref{Ass-Abstract formulation}, satisfy the specific requirements listed therein. \newline
We introduce a $C^\infty$-class function 
$\varphi: \mathbb{R}^d \to [0,\infty)$ with support in the unit ball and such that $\int_{\mathbb{R}^d} \varphi(x)\, \d x=1$.
For any $\delta>0$, let 
    \begin{equation}
       \varphi_\delta(x)= \delta^{-d} \varphi(\delta^{-1} x), \; x \in \mathbb{R}^d.
    \end{equation}
Let $\chi: \mathbb{R}^d \to [0,1]$ be a smooth function with compact support on $\mathbb{R}^d$ such that
\begin{align*}
   \begin{cases}
      \chi(x)= 1 \mbox{ if } \lvert x \rvert \leq 1, \\
      \chi(x)= 0 \mbox{ if }  \lvert x \rvert > 2,
    \end{cases}
   \end{align*}
and for $n \in \mathbb{N}$ set $\chi_n(\cdot)= \chi(\frac{\cdot}{n})$ and define 
   \begin{equation*}
     \widetilde{\bG}^{\delta,n}(t,x,z)= \mathbb{1}_{\{\lvert \tilde{H}(t,x) \rvert \leq n\}} [\bg(t,x,\cdot) \ast \varphi_{\delta}](z) \mbox{ and } \bar{\bg}^{\delta_n}_n(t,x,z)= \chi_n(z) \widetilde{\bG}^{\delta_n,n}(t,x,z).
   \end{equation*}
Here $(\delta_n)_{n\in \mathbb{N}}\subset (0,\infty)$. We put for $n \in \mathbb{N}$ (the approximation index)
       \begin{align*}
        \bg_{n}(t,x,z) \coloneqq \bar{\bg}^{\delta_n}_n(t,x,z).
       \end{align*}  
Note that 
  \begin{align*}
     \bg_{n}: [0,T]\times \domO \times \mathbb{R}^d \ni  (t,x,z) 
     \mapsto \bg_{n}(t,x,z)= (\bg_{n,k}(t,x,z))_{k=1}^\infty \in \mathrm{Y}.
 \end{align*}
Moreover, for every $n \in \mathbb{N}$, the function $\bg_{n}$ satisfies Assumption \ref{ass-A3} and is bounded, i.e. 
there exists $M_n>0$ such that 
    \begin{align}\label{eqn-boundedness-bg_n}
       \Vert \bg_{n}(t,x,z) \Vert_\mathrm{Y} \leq M_n, \mbox{ for all } (t,x,z) \in [0,T]\times \domO \times \mathbb{R}^d.
     \end{align}
In addition, for every $n \in \mathbb{N}$, the functions $\bg_{n}(t,x,\cdot):\mathbb{R}^d \to \mathrm{Y}$,  are uniformly  globally Lipschitz, i.e.  there exists $L_n>0$ such that 
    \begin{align}\label{eqn-globally Lipschitz-bg_n}
      \Vert \bg_{n}(t,x,z_2) - \bg_{n}(t,x,z_1) \Vert_\mathrm{Y} 
       \leq L_n \vert z_2-z_1 \vert,  \mbox{ for all } (t,x,z_1,z_2) \in [0,T]\times \domO \times \mathbb{R}^d \times \mathbb{R}^d.
    \end{align}
Let us state and prove the following two important lemmas.
\begin{lemma}\label{Lem-approximation-a}
Let $(\delta_n)_{n\in \mathbb{N}}\subset (0,1)$ such that $\delta_n \to 0$ as $n \to \infty$. Assume Assumption \ref{ass-A3} holds. Suppose that $\bu_n \to \bu$ in $L^2([0,T] \times \domO)$. Then, we have
\begin{equation}
\lim_{n \to \infty} \int_0^T \int_{\domO} \Vert \bg_{n}(t,x;\bu_n(t,x)) -  \widetilde{\bG}^{\delta_n,n}(t,x;\bu_n(t,x)) \Vert_\mathrm{Y}^2\, \d x\,\d t
=0.
\end{equation}
\end{lemma}
\begin{proof}[Proof of Lemma \ref{Lem-approximation-a}]
We set $\domO_T\coloneqq [0,T] \times \domO$. Let us write, by abuse of notation, $\omega$ instead of $(t,x)$ below.
Using the Jensen inequality, see \cite[Proposition II.2.20]{Boyer+Fabrie_2012} together with the fact that $\Vert \varphi \Vert_{L^1(B(0,1))}= \Vert \varphi \Vert_{L^1(\mathbb{R}^d)}= 1$, we infer that
\begin{align*}
&\Vert \bg_{n}(\omega,\bu_n(\omega)) - \widetilde{\bG}^{\delta_n,n}(\omega,\bu_n(\omega)) \Vert_\mathrm{Y}^2
= \Vert (\chi_n(\bu_n(\omega)) - 1) \mathbb{1}_{\{\lvert \tilde{H}(\omega) \rvert \leq n\}} [\bg \ast \varphi_{\delta_n}](\bu_n(\omega)) \Vert_\mathrm{Y}^2
\\
&= \lvert (\chi_n(\bu_n(\omega)) - 1) \rvert^2 \mathbb{1}_{\{\lvert \tilde{H}(\omega) \rvert \leq n\}} \left\Vert  \int_{B(0,1)} \bg(\omega,\bu_n(\omega) - \delta_n \bv) \varphi(\bv)\, \d \bv \right\Vert_\mathrm{Y}^2
\\
&\leq \lvert (\chi_n(\bu_n(\omega)) - 1) \rvert^2 \int_{B(0,1)} \Vert \bg(\omega,\bu_n(\omega) - \delta_n \bv)\Vert_\mathrm{Y}^2 \varphi(\bv)\, \d \bv. 
\end{align*}
Since by assumption $\bu_n \to \bu$ in $L^2(\domO_T)$, then by \cite[Theorem 4.9]{Brezis_2011}, there exists a subsequence still denoted by $(\bu_n)_n$ and a function $h \in L^2(\domO_T)$ such that
\begin{itemize}
\item[(i)] $\bu_n(\omega) \to \bu(\omega)$, a.e. $\omega \in \domO_T$,
\item[(ii)] $\lvert \bu_n(\omega) \rvert \leq h(\omega)$, for every $n$, a.e. $\omega \in \domO_T$.
\end{itemize}
Hence, $\frac{\bu_n}{n} \to 0$ a.e. on $\domO_T$, and so $\chi_n(\bu_n) \to \chi(0)=1$ a.e. on $\domO_T$. \newline
Consequently, $\Vert \bg_{n}(\omega,\bu_n(\omega)) - \widetilde{\bG}^{\delta_n,n}(\omega,\bu_n(\omega)) \Vert_\mathrm{Y}^2 \to 0$ for a.e. $ \omega \in \domO_T$ as $n \to \infty$.
\newline
Moreover, thanks to the assumption \ref{ass-A3} and item $(ii)$, we infer that for every $n \in \mathbb{N}$ and for a.e. $\omega \in \domO_T$,
\begin{align*}
&\int_{B(0,1)} \Vert \bg(\omega,\bu_n(\omega) - \delta_n \bv)\Vert_\mathrm{Y}^2 \varphi(\bv)\, \d \bv
\leq 2 C_g^2 \int_{B(0,1)} \lvert \bu_n(\omega) - \delta_n \bv \rvert^2 \varphi(\bv)\, \d \bv + 2 \lvert \tilde{H}(\omega) \rvert^2
\\
&\leq 4 C_g^2 \lvert \bu_n(\omega) \rvert^2 + 4 \delta_n^2 C_g^2 \int_{B(0,1)} \lvert \bv \rvert^2 \varphi(\bv)\, \d \bv + 2 \lvert \tilde{H}(\omega) \rvert^2
 \\
&\leq 4 C_g^2 h^2(\omega) + 4 \delta_n^2 C_g^2 \int_{B(0,1)} \lvert \bv \rvert^2 \varphi(\bv)\, \d \bv + 2 \lvert \tilde{H}(\omega) \rvert^2,
\end{align*}
which, in turn, implies that for every $n \in \mathbb{N}$ and for a.e. $\omega \in \domO_T$,
\begin{align*}
&\Vert \bg_{n}(\omega,\bu_n(\omega)) - \widetilde{\bG}^{\delta_n,n}(\omega,\bu_n(\omega)) \Vert_\mathrm{Y}^2
\leq \lvert (\chi_n(\bu_n(\omega)) - 1) \rvert^2 \int_{B(0,1)} \Vert \bg(\omega,\bu_n(\omega) - \delta_n \bv)\Vert_\mathrm{Y}^2 \varphi(\bv)\, \d \bv
\\
&\leq \int_{B(0,1)} \Vert \bg(\omega,\bu_n(\omega) - \delta_n \bv)\Vert_\mathrm{Y}^2 \varphi(\bv)\, \d \bv
\leq 4 C_g^2 h^2(\omega) + 4 C_g^2 \int_{B(0,1)} \lvert \bv \rvert^2 \varphi(\bv)\, \d \bv + 2 \lvert \tilde{H}(\omega) \rvert^2.
\end{align*}
Furthermore, because the map $\domO_T \ni \omega \mapsto 4 C_g^2 h^2(\omega) + 4 C_g^2 \int_{B(0,1)} \lvert \bv \rvert^2 \varphi(\bv)\, \d \bv + 2 \lvert \tilde{H}(\omega) \rvert^2$
belongs to $L^1(\domO_T)$, we can apply the Lebesgue Dominated Convergence Theorem (LDCT), see \cite[Theorem 4.2]{Brezis_2011} and deduce that
\[
\lim_{n \to \infty} \int_{\domO_T}  \Vert \bg_{n}(\omega,\bu_n(\omega)) - \widetilde{\bG}^{\delta_n,n}(\omega,\bu_n(\omega)) \Vert_\mathrm{Y}^2 \,\d\omega=0.
\]
\end{proof}
\begin{lemma}\label{Lem-approximation-b}
Let $(\delta_n)_{n\in \mathbb{N}}\subset (0,1)$ such that $\delta_n \to 0$ as $n \to \infty$. Assume that the assumption \ref{ass-A3} holds. Suppose that $\bu_n \to \bu$ in $L^2([0,T] \times \domO)$. Then, we have
\begin{equation}
\lim_{n \to \infty} \Vert \widetilde{\bG}^{\delta_n,n}(\cdot,\cdot;\bu_n) - \bg(\cdot,\cdot;\bu) \Vert_{L^2([0,T] \times \domO;\mathrm{Y})}^2
=0.
\end{equation}
\end{lemma}
\begin{proof}[Proof of Lemma \ref{Lem-approximation-b}]
We set $\domO_T\coloneqq [0,T] \times \domO$. Let us write, by abuse of notation, $\omega$ instead of $(t,x)$ below. Notice that
\begin{align*}
\Vert \widetilde{\bG}^{\delta_n,n}(\cdot,\cdot;\bu_n) - \bg(\cdot,\cdot;\bu) \Vert_{L^2(\domO_T;\mathrm{Y})}^2
&\leq 2 \Vert \widetilde{\bG}^{\delta_n,n}(\cdot,\cdot;\bu_n) - \widetilde{\bG}^{\delta_n,n}(\cdot,\cdot;\bu) \Vert_{L^2(\domO_T;\mathrm{Y})}^2 \\
&\quad + 2 \Vert \widetilde{\bG}^{\delta_n,n}(\cdot,\cdot;\bu) - \bg(\cdot,\cdot;\bu) \Vert_{L^2(\domO_T;\mathrm{Y})}^2.
\end{align*}
Let us consider the first term on the RHS of the above inequality. By the Jensen inequality, we obtain for all $\omega \in \domO_T$,
\begin{align*}
&\Vert \widetilde{\bG}^{\delta_n,n}(\omega;\bu_n(\omega)) - \widetilde{\bG}^{\delta_n,n}(\omega;\bu(\omega)) \Vert_{\mathrm{Y}}^2 
\leq \int_{B(0,1)} \Vert \bg(\omega,\bu_n(\omega) - \delta_n \bv) - \bg(\omega,\bu(\omega) - \delta_n \bv)\Vert_\mathrm{Y}^2 \varphi(\bv)\, \d \bv
\\
&\leq \sup_{\bv \in B(0,1)} \Vert \bg(\omega,\bu_n(\omega) - \delta_n \bv) - \bg(\omega,\bu(\omega) - \delta_n \bv)\Vert_\mathrm{Y}^2
\\
&\leq \sup_{\bv \in B(0,1)} \Vert \bg(\omega,\bu_n(\omega) - \delta_n \bv) - \bg(\omega,\bu(\omega))\Vert_\mathrm{Y}^2 + \sup_{\bv \in B(0,1)} \Vert \bg(\omega,\bu(\omega) - \delta_n \bv) - \bg(\omega,\bu(\omega))\Vert_\mathrm{Y}^2.
\end{align*}
Since $\bu_n \to \bu$ in $L^2(\domO_T)$ by assumption, we infer, up to a subsequence, that
$\Vert \widetilde{\bG}^{\delta_n,n}(\omega;\bu_n(\omega)) - \widetilde{\bG}^{\delta_n,n}(\omega;\bu(\omega)) \Vert_{\mathrm{Y}}^2 \to 0$ for a.e. $\omega \in \domO_T$.
Moreover, arguing as in the proof of Lemma \ref{Lem-approximation-a}, we infer that for every $n \in \mathbb{N}$ and a.e. $\omega \in \domO_T$,
\begin{align*}
&\Vert \widetilde{\bG}^{\delta_n,n}(\omega;\bu_n(\omega)) - \widetilde{\bG}^{\delta_n,n}(\omega;\bu(\omega)) \Vert_{\mathrm{Y}}^2 
\\
&\leq 2 \sup_{\bv \in B(0,1)} \Vert \bg(\omega,\bu_n(\omega) - \delta_n \bv) \Vert_\mathrm{Y}^2 + 2 \sup_{\bv \in B(0,1)} \Vert \bg(\omega,\bu(\omega) - \delta_n \bv)\Vert_\mathrm{Y}^2
\\
&\leq 8 C_g^2 [h^2(\omega) + \lvert \bu(\omega) \rvert^2] + 16 \delta_n^2 C_g^2 + 8 \lvert \tilde{H}(\omega) \rvert^2
\leq (8 C_g^2 [h^2(\omega) + \lvert \bu(\omega) \rvert^2] + 16 C_g^2 + 8 \lvert \tilde{H}(\omega)\rvert^2).
\end{align*}
Hence, by the LDCT,
\[
\lim_{n \to \infty} \int_{\domO_T} \Vert \widetilde{\bG}^{\delta_n,n}(\omega;\bu_n(\omega)) - \widetilde{\bG}^{\delta_n,n}(\omega;\bu(\omega)) \Vert_{\mathrm{Y}}^2\,\d \omega
=0. 
\]
We now consider the term $\Vert \widetilde{\bG}^{\delta_n,n}(\cdot,\cdot;\bu) - \bg(\cdot,\cdot;\bu) \Vert_{L^2(\domO_T;\mathrm{Y})}^2$. 
Since $\bg(\omega,\cdot): \mathbb{R}^d \to \mathrm{Y}$ is continuous, cf. Assumption \ref{ass-A3}, then $\bg(\omega,\cdot) \ast \varphi_{\delta_n} \to \bg(\omega,\cdot)$ uniformly on compact sets of $\mathbb{R}^d$ as $n \to \infty$. Furthermore,
\[\Vert \widetilde{\bG}^{\delta_n,n}(\omega;z) - \bg(\omega;z) \Vert_{\mathrm{Y}}^2 \to 0, \;\; z \in \mathbb{R}^d.\]
Indeed, 
\begin{align*}
&\widetilde{\bG}^{\delta_n,n}(\omega,z) - \bg(\omega,z)
=
\mathbb{1}_{\{\lvert \tilde{H}(t,x) \rvert \leq n\}} [\bg(t,x,\cdot) \ast \varphi_{\delta_n}](z) -\bg(\omega,z)
\\
&=
(\mathbb{1}_{\{\lvert \tilde{H}(t,x) \rvert \leq n\}} - 1) [\bg(t,x,\cdot) \ast \varphi_{\delta_n}](z) + [\bg(t,x,\cdot) \ast \varphi_{\delta_n}](z) - \bg(\omega,z)
\\
&=  [\bg(t,x,\cdot) \ast \varphi_{\delta_n}](z)-\bg(\omega,z), \mbox{ for $n\geq n_0(t,x)$.}
\end{align*}
This implies that for a.e. $\omega \in \domO_T$,
\[
\lim_{n \to \infty} \Vert \widetilde{\bG}^{\delta_n,n}(\omega,\bu(\omega)) - \bg(\omega,\bu(\omega)) \Vert_{\mathrm{Y}}^2=0.
\]
On the other hand,
\begin{align*}
&\Vert \widetilde{\bG}^{\delta_n,n}(\omega;\bu(\omega)) - \bg(\omega;\bu(\omega)) \Vert_{\mathrm{Y}}^2 
\leq 2 \sup_{\bv \in B(0,1)} \Vert \bg(\omega,\bu(\omega) - \delta_n \bv) \Vert_\mathrm{Y}^2 + 2 \sup_{\bv \in B(0,1)} \Vert \bg(\omega,\bu(\omega))\Vert_\mathrm{Y}^2
\\
&\leq 12 C_g^2 \lvert \bu(\omega) \rvert^2 + 8 \delta_n^2 C_g^2 + 8 \lvert \tilde{H}(\omega) \rvert^2
\leq [12 C_g^2 \lvert \bu(\omega) \rvert^2 + 8 C_g^2 + 8 \lvert \tilde{H}(\omega)\rvert^2] \in L^1(\domO_T).
\end{align*}
Hence, by LDCT we infer that
\[
\lim_{n \to \infty} \int_{\domO_T} \Vert \widetilde{\bG}^{\delta_n,n}(\omega,\bu(\omega)) - \bg(\omega,\bu(\omega)) \Vert_{\mathrm{Y}}^2\,\d \omega=0,
\]
i.e. $\lim_{n \to \infty} \Vert \widetilde{\bG}^{\delta_n,n}(\cdot,\cdot;\bu) - \bg(\cdot,\cdot;\bu) \Vert_{L^2(\domO_T;\mathrm{Y})}^2= 0$. Combining the two limits, we conclude that
\begin{align*}
\lim_{n \to \infty} \Vert \widetilde{\bG}^{\delta_n,n}(\cdot,\cdot;\bu_n) - \bg(\cdot,\cdot;\bu) \Vert_{L^2(\domO_T;\mathrm{Y})}^2
=0.
\end{align*}
This completes the proof of Lemma \ref{Lem-approximation-b}.
\end{proof}
\begin{corollary}\label{cor-approximation-a}
Let $(\delta_n)_{n\in \mathbb{N}}\subset (0,1)$ such that $\delta_n \to 0$ as $n \to \infty$. Assume Assumption \ref{ass-A3} holds. Suppose that $\bu_n \to \bu$ in $L^2([0,T] \times \domO)$. Then, we have
\begin{equation}
\lim_{n \to \infty} \int_0^T \int_{\domO} \Vert \bg_{n}(t,x;\bu_n(t,x)) -  \bg(t,x;\bu(t,x)) \Vert_\mathrm{Y}^2\, \d x\,\d t
=0.
\end{equation}
\end{corollary}

\begin{proof}[Proof of Corollary \ref{cor-approximation-a}]
The proof follows from Lemmas \ref{Lem-approximation-a} and \ref{Lem-approximation-b}.
\end{proof}
Now, taking $\bG_{0,n}$ built from $\bg_n$ as above, Corollary \ref{cor-approximation-a} establishes the convergence property (v)-b), and the boundedness/Lipschitz properties stated in \eqref{eqn-boundedness-bg_n} and \eqref{eqn-globally Lipschitz-bg_n} give (v)-a); hence Assumption \ref{assumption-abstract-3} in Section \ref{Ass-Abstract formulation} holds for the map $\bG$ of \eqref{eqn-G}.

\section{Pathwise uniqueness of weak solutions in dimension 2}\label{eqn-uniqueness-weak-solution-2-dim}
In this section, we will prove Theorem \ref{First-main-result-uniqueness}. In other words, we will prove the pathwise uniqueness of the solution to
\eqref{eqn-compact-modified-stochastic-CHNSEs-3} in $2D$ case. Given an $\nHB$-valued non random initial variable $\boldsymbol{\upxi}=(\upxi,\chi)$, we denote by $(\bu(\cdot,\upxi),\phi(\cdot,\chi))$ the solution to \eqref{eqn-compact-modified-stochastic-CHNSEs-3} with initial data $\boldsymbol{\upxi}$.
\begin{theorem}\label{Thm-uniqueness-solution}
Assume the assumptions stated in Section \ref{Ass-Abstract formulation}. 
Suppose that for all $s \in [0,T]$ and $\bv,\,\bw \in \StokesH$, there exists $\tilde{c}_g>0$ such that
   \begin{equation}\label{g-Lipschitz-condition}
       \Vert \bG_0(s,\bv) - \bG_0(s,\bw)\Vert_{\ell^2(\StokesH)}^2
      \leq \tilde{c}_g \lvert \bv - \bw \rvert_{\StokesH}^2.
    \end{equation}
Let $(\upxi_1,\chi_1)$ and $(\upxi_2,\chi_2)$ be two $\mathbb{H}$-valued initial variables. Assume that on some probability space $(\Omega,\mathscr{F},\mathbb{P})$, with a right-continuous filtration of $\sigma$-fields $\mathbb{F}=(\mathscr{F}_t)_{t \in [0,T]}$ and $\mathrm{Y}$-valued cylindrical Wiener process $W$, we have two solutions 
$$(\bu_1(\cdot,\upxi_1),\phi_1(\cdot,\chi_1)) \mbox{ and } (\bu_2(\cdot,\upxi_2),\phi_2(\cdot,\chi_2))$$ 
to the problem \eqref{eqn-compact-modified-stochastic-CHNSEs-3} corresponding to the initial conditions $(\upxi_1,\chi_1)$ and $(\upxi_2,\chi_2)$, respectively, and such that $\mathbb{P}$-a.s.
\begin{align}
\sup_{s \in [0,T]} \Vert (\bu_i(s,\upxi_i),\phi_i(s,\chi_i)) \Vert_{\nHB}^2 + \int_0^T \Vert (\bu_i(s,\upxi_i),\phi_i(s,\chi_i)) \Vert_{\mathbb{V}}^2\,\d s< \infty,  \quad i=1,2.
\end{align}
Then, for every $t \in [0,T]$, there exists a constant $C>0$ such that
   \begin{equation}\label{eqn-process-y_1}
     \mathbb{E }\left[e^{-\int_0^t \mathcal{Y}_2(s)\, \d s} \Vert (\bu_1(t,\upxi_1),\phi_1(t,\upxi_1)) - (\bu_2(t,\upxi_2),\phi_2(t,\upxi_2)) \Vert_{\nHB}^2\right]
     \leq C \Vert (\upxi_1,\chi_1) - (\upxi_2,\chi_2) \Vert_{\nHB}^2,
   \end{equation}
where
\begin{equation}\label{eqn-process-y_2}
\begin{aligned}
\mathcal{Y}_2(t)
\coloneqq & C_0 \bigl(1 + \Vert \bu_1(t) \Vert_{\StokesV}^2 + \Vert \phi_1(t) \Vert_{\zero{H}{2}}^2 + \Vert \phi_2(t) \Vert_{\zero{H}{2}}^2 + \lvert \phi_1(t) \rvert_{\newone{H}}^2  \Vert \phi_1(t) \Vert_{\zero{H}{2}}^2 
 \\
&\qquad + \lvert \phi_2(t) \rvert_{\newone{H}} ^2 \Vert \phi_2(t) \Vert_{\zero{H}{2}}^2  + \lvert \bu_2(t) \rvert_{\StokesH}^2 \Vert \bu_2(t)\Vert_{\StokesV}^2  
   \\
&\qquad + \lvert \phi_2(t) \rvert_{\newone{H}}  \Vert \phi_2(t) \Vert_{\zero{H}{2}} \lvert \phi_1(t) + \phi_2(t) \rvert_{L^2} \Vert \phi_1(t) + \phi_2(t) \Vert_{\zero{H}{2}}\Bigr).
\end{aligned}
\end{equation}
Here $C_0$ is a positive constant depending on $\domO$ and the parameters of the problem, but is independent of time. \newline
Moreover, if $(\upxi_1,\chi_1)= (\upxi_2,\chi_2)$, then for every $t \in [0,T]$ we have $\mathbb{P}$-a.s.
\begin{equation}\label{eqn-u_=u_2}
(\bu_1(t,\upxi_1),\phi_1(t,\chi_1))= (\bu_2(t,\upxi_2),\phi_2(t,\chi_2)).
\end{equation}
\end{theorem}
\begin{proof}[Proof of the uniqueness part of Theorem \ref{First-main-result-uniqueness}]
This is precisely the uniqueness part of  Theorem \ref{Thm-uniqueness-solution}.
\end{proof}

\begin{proof}[Proof of Theorem \ref{Thm-uniqueness-solution}]
 Let us point out that the uniqueness part, i.e.  \eqref{eqn-u_=u_2} follows directly from the continuous dependence on data part \eqref{eqn-process-y_1}. 
 In what follows, we will occupy ourselves with the proof of the latter. For this purpose, let us choose and fix  two solutions  $(\bu_1,\phi_1,\mu_1)$ and $(\bu_2,\phi_2,\mu_2)$  of Problem \eqref{eqn-compact-modified-stochastic-CHNSEs-3} defined on the same stochastic basis $(\Omega,\mathscr{F},\mathbb{P})$ 
and corresponding not necessarily to the same initial data $(\upxi_1,\chi_1)$ and $(\upxi_2,\chi_2)$.
Let us put  $(\bu,\phi,\mu)= (\bu_1,\phi_1,\mu_1) - (\bu_2,\phi_2,\mu_2)$, $\cp= \mu - \bar\mu$. Hence, 
    \begin{equation*}
      \mu(t)= \mu_1(t) - \mu_2(t)= \Atwo \phi(t) + \psi^\prime(\phi_1(t)) -  \psi^\prime(\phi_2(t)), \; t\in [0,T]
    \end{equation*}
and  $\langle \cp(t) \rangle= 0,\; t\in [0,T]$. 
For every $R>0$, let us define the following $\mathbb{F}$-stopping time
\[
\tau_R \coloneqq \inf\{t \in [0,T]: \lvert \bX_1(t) \rvert_{\mathbb{H}} + \lvert \bX_2(t) \rvert_{\mathbb{H}} \geq R\}\wedge T.
\]
Indeed, since the process $\bX_i=((\bu_i(t),\phi_i(t)): t \in [0,T])$, $i=1,2$, is $\mathbb{F}$-adapted and continuous, see Theorems \ref{First-main-result} and \ref{First-main-result-uniqueness}, hence so right-continuous, and the filtration $\mathbb{F}$ is right-continuous, then by \cite[Problem 2.1]{Karatzas+Shreve_1991}, we infer that $\tau_R$ is an $\mathbb{F}$-stopping time.\newline
Since the trajectories $[0,T] \ni t \mapsto \bX(t)$ are bounded, $\mathbb{P}$-a.s., we deduce that  $\tau_R \toup  T$, $\mathbb{P}$-a.s., as $ R \to \infty$.
\newline
Observe that the process $(\bX,\cp)= (((\bu(t),\phi(t)),\cp(t)):\,t \in [0,T])$ satisfies
\begin{equation}\label{difference-of-solutions}
\begin{cases}
\d \bu + \nu \Stokes \bu  = - [  \bB_0(\bu,\bu_1) + \bB_0(\bu_2,\bu) + \newK \bR_0(\phi,\phi_1) + \newK \bR_0(\phi_2,\phi)] \,\d t \\
\hspace{1.5cm} +  [\bG_0(t,\bu_1) - \bG_0(t,\bu_2) + \bSi_0(t)\bu] \, \d W,
\\
\d \phi= - [B_1(\bu,\phi_1) + B_1(\bu_2,\phi)] \, \d t - \Athree \cp \, \d t,
   \\
\cp(t)= \Atwo \phi(t) + \psi^\prime(\phi_1(t)) - \psi^\prime(\phi_2(t)) - \bar\mu,\;\; t \in [0,T].
\end{cases}
\end{equation}
Now, in order to apply  \cite[Theorem 3.2]{Pardoux_1979}, let us put
\begin{align*}
v(t)&\coloneqq - [  \bB_0(\bu(t),\bu_1(t)) + \bB_0(\bu_2(t),\bu(t)) + \newK \bR_0(\phi(t),\phi_1(t)) + \newK \bR_0(\phi_2(t),\phi(t))],\;\; t \in [0,T],
\\
M_t&\coloneqq \int_0^t [\bG_0(s,\bu_1(s)) - \bG_0(s,\bu_2(s)) + \bSi_0(s)\bu(s)]\, \d W(s),\;\; t \in [0,T].
\end{align*}

From the estimate \eqref{eqn-b_0-trilinear-estimate} and the equality \eqref{properties-B_0}, we infer that
\begin{align*}
\Vert \bB_0(\bu,\bu_1) \Vert_{\StokesVp} + \Vert \bB_0(\bu_2,\bu) \Vert_{\StokesVp}
\leq C [\lvert \bu_1 \rvert_{\StokesH}^{\frac12} \lvert \nabla \bu_1 \rvert_{\mathbb{L}^2}^{\frac12} + \lvert \bu_2 \rvert_{\StokesH}^{\frac12} \lvert \nabla \bu_2 \rvert_{\mathbb{L}^2}^{\frac12}] \lvert \bu \rvert_{\StokesH}^{\frac12} \lvert \nabla \bu \rvert_{\mathbb{L}^2}^{\frac12},
\end{align*}
which, in turn, yields that for all $t$,
\begin{align*}
&\int_0^{t \wedge \tau_R} (\Vert \bB_0(\bu(s),\bu_1(s)) \Vert_{\StokesVp}^2 + \Vert \bB_0(\bu_2(s),\bu(s)) \Vert_{\StokesVp}^2)\,\d s
\\
&\leq C \int_0^{t \wedge \tau_R} [\lvert \bu_1(s) \rvert_{\StokesH} \lvert \nabla \bu_1(s) \rvert_{\mathbb{L}^2} + \lvert \bu_2(s) \rvert_{\StokesH} \lvert \nabla \bu_2(s) \rvert_{\mathbb{L}^2}] \lvert \bu(s) \rvert_{\StokesH} \lvert \nabla \bu(s) \rvert_{\mathbb{L}^2}\,\d s
\\
&\leq C(R) \int_0^{t \wedge \tau_R} [\lvert \nabla \bu_1(s) \rvert_{\mathbb{L}^2} + \lvert \nabla \bu_2(s) \rvert_{\mathbb{L}^2}] \lvert \nabla \bu(s) \rvert_{\mathbb{L}^2}\,\d s.
\end{align*}
Therefore, by the H\"older inequality and the first estimate in Theorem \ref{First-main-result}, we infer that
\[
\mathbb{E} \int_0^{t \wedge \tau_R} (\Vert \bB_0(\bu(s),\bu_1(s)) \Vert_{\StokesVp}^2 + \Vert \bB_0(\bu_2(s),\bu(s)) \Vert_{\StokesVp}^2)\,\d s<\infty.
\]
By the estimate \eqref{r-estimate} and the equality \eqref{eq-2.12}, we deduce that
\begin{align*}
\Vert \bR_0(\phi,\phi_1) \Vert_{\StokesVp} + \Vert \bR_0(\phi_2,\phi) \Vert_{\StokesVp}
\leq C [\lvert  \phi_1 \rvert_{\newone{H}}^{\frac12} \Vert \phi_1 \Vert_{\zero{H}{2}}^{\frac12} + \lvert  \phi_2 \rvert_{\newone{H}}^{\frac12} \Vert \phi_2 \Vert_{\zero{H}{2}}^{\frac12}] \lvert \phi \rvert_{\newone{H}}^{\frac12} \Vert \phi \Vert_{\zero{H}{2}}^{\frac12},
\end{align*}
which implies that for all $t$,
\begin{align*}
&\int_0^{t \wedge \tau_R} (\Vert \bR_0(\phi(s),\phi_1(s)) \Vert_{\StokesVp}^2 + \Vert \bR_0(\phi_2(s),\phi(s)) \Vert_{\StokesVp}^2)\,\d s
\\
&\leq C(R) \int_0^{t \wedge \tau_R} [\Vert \phi_1(s) \Vert_{\zero{H}{2}} + \Vert \phi_2(s) \Vert_{\zero{H}{2}}]  \Vert \phi(s) \Vert_{\zero{H}{2}}\,\d s.
\end{align*}
Now using the first estimate in Theorem \ref{First-main-result}, we infer that
\begin{equation*}
\mathbb{E} \int_0^{t \wedge \tau_R} (\Vert \bR_0(\phi(s),\phi_1(s)) \Vert_{\StokesVp}^2 + \Vert \bR_0(\phi_2(s),\phi(s)) \Vert_{\StokesVp}^2)\,\d s<\infty.    
\end{equation*}
Hence we deduce that 
\[
\mathbb{E} \int_0^{t \wedge \tau_R} \Vert v(s) \Vert_{\StokesVp}^2\,\d s<\infty.
\]
Observe also that $\mathbb{E} \int_0^{t \wedge \tau_R} \Vert \Stokes \bu(s) \Vert_{\StokesVp}^2\,\d s
\leq \mathbb{E} \int_0^{t \wedge \tau_R} \lvert \nabla \bu(s) \rvert_{\mathbb{L}^2}^2\,\d s<\infty$.
By the assumptions \eqref{eqn-linear growth} and \eqref{Eqn-coercivity-3} in Abstract formulation \ref{Ass-Abstract formulation} and Example \ref{example:HS}, and the first estimate in Theorem \ref{First-main-result}, we infer that for all $t$, 
\begin{align*}
&\mathbb{E} \int_0^{t \wedge \tau_R} \Vert \bG_0(s,\bu_1(s)) - \bG_0(s,\bu_2(s)) + \bSi_0(s)\bu(s) \Vert_{\ell^2(\StokesH)}^2\,\d s
\\
&\leq 2 \mathbb{E} \int_0^{t \wedge \tau_R} [\Vert \bG_0(s,\bu_1(s)) - \bG_0(s,\bu_2(s))\Vert_{\ell^2(\StokesH)}^2 + \Vert \bSi_0(s)\bu(s) \Vert_{\ell^2(\StokesH)}^2]\,\d s
\\
&\leq C \mathbb{E} \int_0^{t \wedge \tau_R} (\vert \tilde{h}(s)\vert^2 + \vert \bu(s) \vert_{\StokesH}^2 + \Vert \bu(s) \Vert_{\StokesV}^2)\,\d s<\infty. 
\end{align*}
Therefore, by applying the "stopped" It\^o formula, see \cite[Theorem 3.2 or 1.4]{Pardoux_1979}, to the function
\[
\psi_1(x)\coloneqq \lvert x \rvert_{\StokesH}^2,\;\; x \in \StokesH,
\]
we obtain for all $t \in [0,T]$,
\begin{equation}\label{u-difference-1}
\begin{aligned}
&\frac12 \lvert \bu(t \wedge \tau_R) \rvert_{\StokesH}^2 
+ \nu \int_0^{t \wedge \tau_R} \Vert \bu(s) \Vert_{\StokesV}^2 \, \d s + \int_0^{t \wedge \tau_R}  \bb_0(\bu(s),\bu_1(s),\bu(s))\,\d s  
\\
&= \frac{1}{2} \lvert \upxi_1 - \upxi_2 \rvert_{\StokesH}^2  - \newK \int_0^{t \wedge \tau_R} [r_0(\phi(s),\phi_1(s),\bu(s))+ r_0(\phi_2(s),\phi(s),\bu(s))]
\,\d s 
\\
&\qquad + \frac12 \int_0^{t \wedge \tau_R} \Vert \bG_0(s,\bu_1(s)) - \bG_0(s,\bu_2(s)) + \bSi_0(s)\bu(s) \Vert_{\ell^2(\StokesH)}^2\,\d s
\\ 
&\qquad + \int_0^{t \wedge \tau_R} (\bu(s),[\bG_0(s,\bu_1(s)) - \bG_0(s,\bu_2(s)) + \bSi_0(s)\bu(s)]\d W(s)).
\end{aligned}
\end{equation}
In order to apply Lemma 1.3 from \cite[Chapter 3]{Temam_2001}, we rewrite the equation for $\phi$ in the following form: 
   \begin{equation}\label{difference-of-solutions-phi}
      \d \phi= - [B_1(\bu,\phi_1) + B_1(\bu_2,\phi_1) - B_1(\bu_2,\phi_2)] \, \d t  - \Athree [\Atwo \phi + \psi^\prime(\phi_1) - \psi^\prime(\phi_2) - \bar\mu]  \, \d t.
   \end{equation}
This equation is like equation (4.84)-(4.85) on page 253 in \cite{Lions+Magenes_1972_vol-1}. Since $\Athree\langle\mu\rangle=0$, the external force $f$ is 
\begin{align*}
&f(t)\coloneqq -\Athree [ \psi^\prime(\phi_1(t)) - \psi^\prime(\phi_2(t)) - \bar\mu] - [B_1(\bu(t),\phi_1(t)) + B_1(\bu_2(t),\phi_1(t)) - B_1(\bu_2(t),\phi_2(t))]
\\
&= -\Athree [ \psi^\prime(\phi_1(t)) - \psi^\prime(\phi_2(t))] - [B_1(\bu(t),\phi_1(t)) + B_1(\bu_2(t),\phi_1(t)) - B_1(\bu_2(t),\phi_2(t))],\; t \in [0,T].
\end{align*}
Equation \eqref{difference-of-solutions-phi} can be written in the variational form, i.e., the Gelfand triple form (4.77) therein.\\
We claim 
\begin{equation}\label{eqn-Vert f Vert_{V_1-prime}}
\mathbb{E} \int_0^{t \wedge \tau_R} \Vert f(s) \Vert_{\newonep{V}}^2\,\d s<\infty.
\end{equation}
Indeed, from the estimate \eqref{eq-B1} and the Gagliardo-Nirenberg inequality \eqref{Gagliardo-Nirenberg-inequality}, we deduce that
\begin{align*}
\Vert B_1(\bu,\phi_1)\Vert_{\newonep{V}}^2 
\leq C \vert \bu \vert_{L^4}^2 \vert \phi_1 \vert_{L^4}^2
\leq C \vert \bu \vert_{L^2} \Vert \bu \Vert_{\StokesV} \vert \phi_1 \vert_{L^2} \vert \phi_1 \rvert_{\newone{H}}.
\end{align*}
Now, we have
   \begin{equation*}
      \mathbb{E} \int_0^{t \wedge\tau_R} \Vert B_1(\bu(s),\phi_1(s))\Vert_{\newonep{V}}^2\,\d s
     \leq C(R) \left(\mathbb{E} \int_0^T \Vert \bu(s) \Vert_{\StokesV}^2\,ds \right)^{\frac12} \left(\mathbb{E} \int_0^T \Vert \phi_1(s) \Vert_{H^1}^2\, ds\right)^{\frac12}<\infty.
  \end{equation*}
Here, we used the properties of the processes $\bu=(\bu(t): t \in[0,T])$ and $\phi_1=(\phi_1(t): t \in[0,T])$, see Theorem \ref{First-main-result}.\\
In a similar way, we have
\begin{align*}
\mathbb{E} \int_0^{t \wedge\tau_R} (\Vert B_1(\bu_2(s),\phi_1(s)) \Vert_{\newonep{V}}^2 + \Vert B_1(\bu_2(s),\phi_2(s))\Vert_{\newonep{V}}^2)\,\d s<\infty.
\end{align*}
Let us choose and fix $i=1,2$. Notice that since $\phi_i \in H^3(\domO)$ and $\frac{\partial \phi_i}{\partial \bn}=0$ on $\partial \domO$, we have $\Delta (\psi^\prime(\phi_i) - \avg{\psi^\prime(\phi_i)}) \in \newone{H}$. Thus, by the Gelfand triple \eqref{eqn-Gelfand triple-abstract-2}, we deduce that for every $v \in \newone{V}$,
\begin{align}\label{duality-Delta psi-prime(phi_i)}
\duality{\Athree (\psi^\prime(\phi_i) - \avg{\psi^\prime(\phi_i)})}{v}{\newone{V}}{\newonep{V}}
=(\Athree(\psi^\prime(\phi_i) - \avg{\psi^\prime(\phi_i)}),v)_{\newone{H}}
= -(\nabla \psi^\prime(\phi_i), \nabla (-\Delta v)),
\end{align}
which implies that
\[
\Vert \Athree(\psi^\prime(\phi_i) - \avg{\psi^\prime(\phi_i)}) \Vert_{\newonep{V}}
\leq \lvert \nabla \psi^\prime(\phi_i) \rvert_{\mathbb{L}^2}.
\]
By the equality \eqref{eqn-Psi'} in Lemma \ref{eqn-Lemma-Psi'}, the H\"older and Agmon inequality, we infer that
\begin{align*}
&\lvert \nabla \psi^\prime(\phi_i) \rvert_{\mathbb{L}^2}^2
= \lvert \psi^{\prime \prime}(\phi_i) \nabla \phi_i \rvert_{\mathbb{L}^2}^2
= \lvert 3 \phi_i^2 \nabla \phi_i  - \nabla \phi_i \rvert_{\mathbb{L}^2}^2 
\\
&\leq 6 \lvert \phi_i^2 \nabla \phi_i \rvert_{\mathbb{L}^2}^2 + 2 \lvert \phi_i \rvert_{\newone{H}} ^2 
\leq 6 \lvert \phi_i \rvert_{L^\infty}^4 \lvert \phi_i \rvert_{\newone{H}} ^2 + 2\lvert \phi_i \rvert_{\newone{H}} ^2 
\\
&\leq C \lvert \phi_i \rvert_{L^2}^2 \Vert \phi_i \Vert_{\zero{H}{2}}^2 \lvert \phi_i \rvert_{\newone{H}} ^2 + 2 \lvert \phi_i \rvert_{\newone{H}} ^2.
\end{align*}    
This, jointly with the estimates in Theorem \ref{First-main-result} implies that for all $t \in [0,T]$,
\begin{align*}
&\mathbb{E} \int_0^{t \wedge\tau_R}
\Vert \Athree(\psi^\prime(\phi_i(s)) - \avg{\psi^\prime(\phi_i(s))}) \Vert_{\newonep{V}}^2\,\d s
\leq \mathbb{E} \int_0^{t \wedge\tau_R} \lvert \nabla \psi^\prime(\phi_i(s)) \rvert_{\mathbb{L}^2}^2\,\d s
\\
&\leq C(R) \mathbb{E} \int_0^{t \wedge\tau_R} \Vert \phi_i(s) \Vert_{\zero{H}{2}}^2\,\d s  + 2 \mathbb{E} \int_0^{T}\lvert \phi_i(s) \rvert_{\newone{H}} ^2\,\d s<\infty.
\end{align*}
From the above estimates, we deduce that the claim \eqref{eqn-Vert f Vert_{V_1-prime}} holds true. \newline
Consequently, according to Propositions 2.1 \& 4.5 and Theorem 3.1 in \cite{Lions+Magenes_1972_vol-1}, and Lemma 1.3 in \cite{Temam_2001}, we infer that a.e. in $[0,T]$ and $\mathbb{P}$-a.s.,
     \begin{equation}\label{eqn-bar phi-H_1-norm}
         \frac12 \d \lvert \phi(t) \rvert_{\newone{H}}^2 + \newzero{a}(\phi(t),\phi(t))\,\d t
           =\duality{f(t)}{\phi(t)}{\newone{V}}{\newonep{V}}\,\d t.
      \end{equation}
In the sequel, we omit the explicit dependence on $(t)$ for simplicity.\\
From the third equality in \eqref{difference-of-solutions} and the definition of $f$, we have
\begin{align*}
&\newzero{a}(\phi,\phi)= (\nabla (-\Delta \phi), \nabla(-\Delta \phi))
= (\nabla (\cp - \psi^\prime(\phi_1) + \psi^\prime(\phi_2)), \nabla(-\Delta \phi))
\\
&= (\nabla \cp, \nabla(-\Delta \phi)) + (\nabla (- \psi^\prime(\phi_1) + \psi^\prime(\phi_2)), \nabla(-\Delta \phi)),
\end{align*}
and
\begin{align*}
&\duality{f}{\phi}{\newone{V}}{\newonep{V}}
= \left[- \duality{B_1(\bu,\phi_1)}{\phi}{\newone{V}}{\newonep{V}} - \duality{B_1(\bu_2,\phi_1)}{\phi}{\newone{V}}{\newonep{V}} \right.
\\
&\hspace{2.5cm} \left. + \duality{B_1(\bu_2,\phi_2)}{\phi}{\newone{V}}{\newonep{V}} - \duality{\Athree[\psi^\prime(\phi_1) - \avg{\psi^\prime(\phi_1)} - \psi^\prime(\phi_2) + \avg{\psi^\prime(\phi_2)}]}{\phi}{\newone{V}}{\newonep{V}} \right]
\\
&= - b_1(\bu,\phi_1,\phi) - b_1(\bu_2,\phi_1,\phi) + b_1(\bu_2,\phi_2,\phi) - \duality{\Athree[\psi^\prime(\phi_1) - \avg{\psi^\prime(\phi_1)} - \psi^\prime(\phi_2) + \avg{\psi^\prime(\phi_2)}]}{\phi}{\newone{V}}{\newonep{V}}
\\
&= - b_1(\bu,\phi_1,\phi) - b_1(\bu_2,\phi,\phi) -(\nabla (\psi^\prime(\phi_1) - \psi^\prime(\phi_2)), \nabla (-\Delta \phi)),
\end{align*}
where we have also used the fact that the map $b_1$ is trilinear together with, see \eqref{duality-Delta psi-prime(phi_i)},
\[
\duality{\Athree [\psi^\prime(\phi_1) - \avg{\psi^\prime(\phi_1)} - \psi^\prime(\phi_2) + \avg{\psi^\prime(\phi_2)}]}{\phi}{\newone{V}}{\newonep{V}}
= (\nabla (\psi^\prime(\phi_1) - \psi^\prime(\phi_2)), \nabla (-\Delta \phi)).
\]
It then follows from the above observations and \eqref{eqn-bar phi-H_1-norm} that for all $t \in [0,T]$,
\begin{equation}\label{eqn-bar phi-H_1-norm-1}
\begin{aligned}
&\frac{\newK}{2} \lvert \phi(t \wedge\tau_R) \rvert_{\newone{H}}^2 + \newK \int_0^{t \wedge\tau_R} (\nabla \cp(s), \nabla(-\Delta \phi(s)))\,\d s
 \\
 &= \frac{\newK}{2} \lvert \chi_1 - \chi_2 \rvert_{\newone{H}}^2 - \newK \int_0^{t \wedge\tau_R} [b_1(\bu(s),\phi_1(s),\phi(s)) + b_1(\bu_2(s),\phi(s),\phi(s))]\,\d s.
\end{aligned}
\end{equation}
Recall that $\cp= \Atwo \phi + \psi^\prime(\phi_1) - \psi^\prime(\phi_2) - \bar\mu$. Thus,
\begin{align*}
\nabla \cp= \nabla \Atwo \phi + \nabla (\psi^\prime(\phi_1) - \psi^\prime(\phi_2)),
\end{align*}
from which we deduce a.e. in $[0,T]$,
\begin{equation}\label{eqn-lvert nabla bar mu rvert-L^2-norm}
\lvert \cp \rvert_{\newone{H}}^2
= (\nabla \cp, \nabla \Atwo \phi) + (\nabla \cp,\nabla (\psi^\prime(\phi_1) - \psi^\prime(\phi_2))).
\end{equation}
Furthermore, a.e. in $[0,T]$, we have
\begin{equation}\label{eqn-(-Delta barphi,barmu)}
\begin{aligned}
&(\Atwo \phi,\cp)
= (\Atwo \phi,\Atwo \phi + \psi^\prime(\phi_1) - \psi^\prime(\phi_2) - \bar\mu)
\\
&= \lvert \Atwo \phi \rvert_{L^2}^2 + (\psi^\prime(\phi_1) - \psi^\prime(\phi_2), \Atwo \phi).
\end{aligned}
\end{equation}
Now, multiplying the equality \eqref{eqn-lvert nabla bar mu rvert-L^2-norm} by $\newK$ and the equality \eqref{eqn-(-Delta barphi,barmu)} by $- \newK \kappa_0$, where $\kappa_0>0$ is a sufficiently small constant to be selected in the sequel,
and then adding the corresponding equalities, we deduce that for all $t \in [0,T]$,
\begin{equation}\label{eqn-(lvert nabla bar mu rvert-L^2-norm + kappa_0 lvert Delta barphi rvert_L^2-norm)}
\begin{aligned}
&\newK \int_0^{t \wedge\tau_R} [\lvert \cp(s) \rvert_{\newone{H}}^2 + \kappa_0 \lvert \Atwo \phi(s) \rvert_{L^2}^2]\,\d s 
= \newK \int_0^{t \wedge\tau_R} [(\nabla \cp(s), \nabla \Atwo \phi(s)) + \kappa_0 (\Atwo \phi(s),\cp(s))]\,\d s 
\\
& + \newK \int_0^{t \wedge\tau_R} [(\nabla \cp(s),\nabla (\psi^\prime(\phi_1(s)) - \psi^\prime(\phi_2(s))))
- \kappa_0 (\psi^\prime(\phi_1(s)) - \psi^\prime(\phi_2(s)), \Atwo \phi(s))]\,\d s.
\end{aligned}
\end{equation}
Adding up the equalities \eqref{eqn-bar phi-H_1-norm-1} and \eqref{eqn-(lvert nabla bar mu rvert-L^2-norm + kappa_0 lvert Delta barphi rvert_L^2-norm)} side by side, we obtain for all $t \in [0,T]$,
\begin{align}\label{eq-difference-for-phi-1}
&\frac{\newK}{2} \lvert \phi(t \wedge\tau_R) \rvert_{\newone{H}}^2 
+ \newK \int_0^{t \wedge\tau_R} [\lvert \cp(s) \rvert_{\newone{H}}^2 + \kappa_0 \lvert \Atwo \phi(s) \rvert_{L^2}^2]\,\d s 
\\
\nonumber
&= \frac{\newK}{2} \lvert \chi_1 - \chi_2 \rvert_{\newone{H}}^2 - \newK \int_0^{t \wedge\tau_R} [b_1(\bu(s),\phi_1(s),\phi(s)) + b_1(\bu_2(s),\phi(s),\phi(s))
   - \kappa_0 (\Atwo \phi(s),\cp(s))]\,\d s 
\\
\nonumber
&\quad + \newK \int_0^{t \wedge\tau_R} [(\nabla \cp(s),\nabla (\psi^\prime(\phi_1(s)) - \psi^\prime(\phi_2(s)))) - \kappa_0 (\psi^\prime(\phi_1(s)) - \psi^\prime(\phi_2(s)), \Atwo \phi(s))]\,\d s.
\end{align}

Next, put
      \begin{equation*}
        U(t) \coloneqq \frac{1}{2} \lvert \bu(t) \rvert_{\StokesH}^2 + \frac{\newK}{2} \lvert \phi(t) \rvert_{\newone{H}}^2,\;\; t \in [0,T].
     \end{equation*}
Then, by adding up \eqref{u-difference-1} and \eqref{eq-difference-for-phi-1}, we deduce that $\mathbb{P}$-a.s. and for all $t \in [0,T]$,
\begin{equation}\label{U-equality}
\begin{aligned}
& U(t \wedge \tau_R)
+ \nu \int_0^{t \wedge \tau_R} \Vert \bu(s) \Vert_{\StokesV}^2 \, \d s + \newK \int_0^{t \wedge\tau_R} (\lvert \cp(s) \rvert_{\newone{H}}^2 + \kappa_0 \lvert \Atwo \phi(s) \rvert_{L^2}^2)\,\d s 
\\
&= \frac{1}{2} \lvert \upxi_1 - \upxi_2 \rvert_{\StokesH}^2 + \frac{\newK}{2} \lvert \chi_1 - \chi_2 \rvert_{\newone{H}}^2
  - \int_0^{t \wedge \tau_R}  \bb_0(\bu(s),\bu_1(s),\bu(s))\,\d s   
\\
&  - \newK \int_0^{t \wedge \tau_R} r_0(\phi(s),\phi_1(s),\bu(s))\,\d s  
   - \newK \int_0^{t \wedge \tau_R} r_0(\phi_2(s),\phi(s),\bu(s))\,\d s
\\
& - \newK \int_0^{t \wedge\tau_R} [b_1(\bu(s),\phi_1(s),\phi(s)) + b_1(\bu_2(s),\phi(s),\phi(s))]\,\d s
\\
& + \newK \kappa_0 \int_0^{t \wedge\tau_R} (\Atwo \phi(s),\cp(s))\,\d s 
  + \newK \int_0^{t \wedge\tau_R} (\nabla \cp(s),\nabla (\psi^\prime(\phi_1(s)) - \psi^\prime(\phi_2(s))))\,\d s 
\\
& - \newK \kappa_0 \int_0^{t \wedge\tau_R} (\psi^\prime(\phi_1(s)) - \psi^\prime(\phi_2(s)), \Atwo \phi(s))\,\d s 
  \\
& + \frac12 \int_0^{t \wedge \tau_R} \Vert \bG_0(s,\bu_1(s)) - \bG_0(s,\bu_2(s)) + \bSi_0(s)\bu(s) \Vert_{\ell^2(\StokesH)}^2\,\d s
   \\ 
& + \int_0^{t \wedge \tau_R} (\bu(s),[\bG_0(s,\bu_1(s)) - \bG_0(s,\bu_2(s)) + \bSi_0(s)\bu(s)]\d W(s)).
\end{aligned}
\end{equation}
Let us proceed with estimating all the terms on the RHS of \eqref{U-equality}.\\
Hereafter, $C$ denotes a generic positive constant that may depend on $\domO,\,\delta_0,\,\newK,\,\tilde{c}_g$, and $\kappa_0$.\\
From the properties \eqref{eqn-b_0-trilinear-estimate} and \eqref{properties-B_0}, and the Young inequality, we infer that
   \begin{align*}
      \lvert\bb_0(\bu,\bu_1,\bu)   \rvert
       \leq C \lvert \bu \rvert_{\StokesH}^{\frac12} \Vert \bu \Vert_{\StokesV}^{\frac12} \Vert \bu_1 \Vert_{\StokesV} \lvert \bu \rvert_{\StokesH}^{\frac12} \Vert \bu \Vert_{\StokesV}^{\frac12} 
      \leq \frac{\delta_0}{16} \Vert \bu \Vert_{\StokesV}^2 + C(\delta_0) \Vert \bu_1 \Vert_{\StokesV}^2 \lvert \bu \rvert_{\StokesH}^2.
   \end{align*}
By integration by parts, using also the summation convention on repeated indices, we obtain
\begin{align*}
r_0(\phi,\phi_1,\bu) 
&= \int_{\domO} \partial_i \phi \,\partial_j \phi_1 \, \partial_j \bu^i\, \d x
= -\int_{\domO} (\bu^i \, \partial_i \phi \,\partial_j^2 \phi_1 + \bu^i \, \partial_j \phi_1 \, \partial_i \partial_j \phi)\, \d x,
\end{align*}
and
\begin{align*}
r_0(\phi_2,\phi,\bu)  
=  -\int_{\domO} (\bu^i \, \partial_i \phi_2 \,\partial_j^2 \phi +  \bu^i \, \partial_j \phi \, \partial_i \partial_j \phi_2)\, \d x.
\end{align*}
By the H\"older inequality, the Gagliardo-Nirenberg inequality \eqref{Gagliardo-Nirenberg-inequality}, 
the inequality \eqref{eqn-phi_n-H-2-norm}, \eqref{eqn-H^1-norm-special}, and \eqref{eqn-H^2-norm-special}, we deduce that
\begin{align*}
&\lvert \newK r_0(\phi,\phi_1,\bu)\rvert
\leq C \lvert \bu \rvert_{\StokesH}^{1/2} \Vert \bu \vert_{\StokesV}^{1/2} \left[\lvert \phi \rvert_{\newone{H}}^{1/2} \Vert \phi \Vert_{\zero{H}{2}}^{1/2} \Vert \phi_1 \Vert_{\zero{H}{2}} +  \vert \phi_1 \rvert_{\newone{H}}^{1/2} \Vert \phi_1 \Vert_{\zero{H}{2}}^{1/2} \Vert \phi \Vert_{\zero{H}{2}} \right] 
\\
&\leq C \lvert \bu \rvert_{\StokesH}^{1/2} \Vert \bu \Vert_{\StokesV}^{1/2} \left[\lvert \phi \rvert_{\newone{H}}^{1/2} \lvert \Atwo \phi \rvert_{L^2}^{1/2} \Vert \phi_1 \Vert_{\zero{H}{2}} +  \vert \phi_1 \rvert_{\newone{H}}^{1/2} \Vert \phi_1 \Vert_{\zero{H}{2}}^{1/2} \lvert \Atwo \phi \rvert_{L^2} \right].
\end{align*}
Hence, by the Young inequalities, we further deduce that,
\begin{align*}
&\lvert \newK r_0(\phi,\phi_1,\bu)\rvert   
\leq \frac{\newK \kappa_0}{10} \lvert \Atwo \phi \rvert_{L^2}^2 + C \lvert \bu \vert_{\StokesH}^{2/3} \Vert \bu \Vert_{\StokesV}^{2/3} \lvert \phi \rvert_{\newone{H}}^{2/3}  \Vert \phi_1 \Vert_{\zero{H}{2}}^{4/3} + C \lvert \bu \rvert_{\StokesH} \Vert \bu \Vert_{\StokesV} \vert \phi_1 \rvert_{\newone{H}} \Vert \phi_1 \Vert_{\zero{H}{2}} 
    \\
&\leq \frac{\newK \kappa_0}{10} \lvert \Atwo \phi \rvert_{L^2}^2 + \frac{\delta_0}{16} \Vert \bu \Vert_{\StokesV}^2 + C \lvert \bu \rvert_{\StokesH}  \lvert \phi \rvert_{\newone{H}}  \Vert \phi_1 \Vert_{\zero{H}{2}}^2 + C \lvert \bu \rvert_{\StokesH}^2 \vert \phi_1 \rvert_{\newone{H}}^2 \Vert \phi_1 \Vert_{\zero{H}{2}}^2 
      \\
&\leq \frac{\newK \kappa_0}{10} \lvert \Atwo \phi \rvert_{L^2}^2 + \frac{\delta_0}{16} \Vert \bu \Vert_{\StokesV}^2 + C \Vert \phi_1 \Vert_{\zero{H}{2}}^2 (\lvert \bu \rvert_{\StokesH}^2 + \newK \lvert \phi \rvert_{\newone{H}}^2)  + C \vert \phi_1 \rvert_{\newone{H}}^2 \Vert \phi_1 \Vert_{\zero{H}{2}}^2 \lvert \bu \rvert_{\StokesH}^2.
\end{align*}
Similarly, we can show that
\begin{align*}
&\lvert \newK r_0(\phi_2,\phi,\bu) \rvert
\leq C \lvert \bu \rvert_{\StokesH}^{1/2} \Vert \bu \Vert_{\StokesV}^{1/2} \left[\vert \phi_2 \rvert_{\newone{H}}^{1/2} \Vert \phi_2 \Vert_{\zero{H}{2}}^{1/2} \Vert \phi \Vert_{\zero{H}{2}} + \lvert \phi \rvert_{\newone{H}}^{1/2} \Vert \phi \Vert_{\zero{H}{2}}^{1/2} \Vert \phi_2 \Vert_{\zero{H}{2}} \right] 
 \\
&\leq C \lvert \bu \rvert_{\StokesH}^{1/2} \Vert \bu \Vert_{\StokesV}^{1/2} \left[\vert \phi_2 \rvert_{\newone{H}}^{1/2} \Vert \phi_2 \Vert_{\zero{H}{2}}^{1/2} \lvert \Atwo \phi \rvert_{L^2} + \lvert \phi \rvert_{\newone{H}}^{1/2} \lvert \Atwo \phi \rvert_{L^2}^{1/2} \Vert \phi_2 \Vert_{\zero{H}{2}} \right] 
   \\
&\leq \frac{\newK \kappa_0}{10} \lvert \Atwo \phi \rvert_{L^2}^2 + \frac{\delta_0}{16} \Vert \bu \Vert_{\StokesV}^2 + C \Vert \phi_2 \Vert_{\zero{H}{2}}^2 (\lvert \bu \rvert_{\StokesH}^2 +  \newK \lvert \phi \rvert_{\newone{H}}^2) + C \vert \phi_2 \rvert_{\newone{H}}^2 \Vert \phi_2 \Vert_{\zero{H}{2}}^2 \lvert \bu \rvert_{\StokesH}^2.
\end{align*}
By the definition of the operator $b_1$, see \eqref{eqn-b_1-trilinear-form-2}, the H\"older inequality, and the estimate \eqref{Ladyzhenskaya inequality-a}, we infer that
\begin{align*}
&\lvert  \newK b_1(\bu,\phi_1,\phi)\rvert
\leq  C \lvert \bu \rvert_{\StokesH}^{1/2} \Vert \bu \Vert_{\StokesV}^{1/2} \vert \phi_1 \rvert_{\newone{H}}^{1/2} \Vert \phi_1 \Vert_{\zero{H}{2}}^{1/2} \lvert \Atwo \phi \rvert_{L^2} 
\\
&\leq \frac{\newK \kappa_0}{10} \lvert \Atwo \phi \rvert_{L^2}^2 + C \lvert \bu \rvert_{\StokesH} \Vert \bu \Vert_{\StokesV} \vert \phi_1 \rvert_{\newone{H}} \Vert \phi_1 \Vert_{\zero{H}{2}} 
 \\
&\leq \frac{\newK \kappa_0}{10} \lvert \Atwo \phi \rvert_{L^2}^2 + \frac{\delta_0}{16} \Vert \bu \Vert_{\StokesV}^2 + C \vert \phi_1 \rvert_{\newone{H}}^2 \Vert \phi_1 \Vert_{\zero{H}{2}}^2 \lvert \bu \rvert_{\StokesH}^2.
\end{align*}
Arguing as above, using also \eqref{eqn-H^1-norm-special}, and \eqref{eqn-H^2-norm-special}, we deduce that
\begin{align*}
&\lvert \newK b_1(\bu_2,\phi,\phi)\rvert
\leq C \lvert \bu_2 \rvert_{\StokesH}^{1/2} \Vert \bu_2 \Vert_{\StokesV}^{1/2} \lvert \phi \rvert_{\newone{H}}^{1/2} \Vert \phi \Vert_{\zero{H}{2}}^{1/2} \lvert \Atwo \phi \rvert_{L^2}
\\
&\leq  C \lvert \bu_2 \rvert_{\StokesH}^{1/2} \Vert \bu_2 \Vert_{\StokesV}^{1/2} \lvert  \phi \rvert_{\newone{H}}^{1/2} \lvert \Atwo \phi \rvert_{L^2}^{3/2} 
\leq \frac{\newK \kappa_0}{10} \lvert \Atwo \phi \rvert_{L^2}^2 + C \lvert \bu_2 \rvert_{\StokesH}^2 \Vert \bu_2 \Vert_{\StokesV}^2 \lvert \phi \rvert_{\newone{H}}^2.
\end{align*}
By the H\"older, Young and Poincar\'e-Wirtinger inequalities, we deduce that
\begin{align*}
\newK \kappa_0 \lvert (\Atwo \phi,\cp) \rvert
\leq \newK \kappa_0 \lvert \Atwo \phi \rvert_{L^2} \lvert \cp \rvert_{L^2}
\leq \newK \kappa_0 C_{\domO} \lvert \Atwo \phi \rvert_{L^2} \lvert \cp \rvert_{\newone{H}} 
\leq \frac{\newK \kappa_0}{10} \lvert \Atwo \phi \rvert_{L^2}^2 + \frac{5 \newK \kappa_0 C_{\domO}^2 }{2} \lvert \cp \rvert_{\newone{H}} ^2.
\end{align*}
Here, $C_{\domO}$ is the Poincar\'e-Wirtinger constant which depends on $\domO$.\\  
Next, by the H\"older and Young inequalities, we infer that
\begin{align*}
&\newK \lvert (\nabla (\psi^\prime(\phi_1) - \psi^\prime(\phi_2)), \nabla \cp) \rvert 
\leq \newK \lvert \nabla (\psi^\prime(\phi_1) - \psi^\prime(\phi_2)) \rvert_{\mathbb{L}^2} \lvert \nabla \cp \rvert_{\mathbb{L}^2} 
  \\
&= \newK \lvert \psi^{\bis}(\phi_1) \nabla \phi_1 - \psi^{\bis}(\phi_2) \nabla \phi_2 \rvert_{\mathbb{L}^2} \lvert \cp \rvert_{\newone{H}}  
= \newK \lvert 3 \phi_1^2 \nabla \phi + 3 \phi (\phi_1 + \phi_2) \nabla \phi_2 - \nabla \phi \rvert_{\mathbb{L}^2} \lvert \cp \rvert_{\newone{H}} 
\\
&\leq \frac{\newK}{2} \lvert \cp \rvert_{\newone{H}} ^2 + \frac{\newK}{2} \lvert 3 \phi_1^2 \nabla \phi + 3 \phi (\phi_1 + \phi_2) \nabla \phi_2 - \nabla \phi \rvert_{\mathbb{L}^2}^2 
\\
&\leq \frac{\newK}{2} \lvert \cp \rvert_{\newone{H}} ^2 + C\lvert \phi_1^2 \nabla \phi \rvert_{\mathbb{L}^2}^2 + C \lvert \phi (\phi_1 + \phi_2) \nabla \phi_2 \rvert_{\mathbb{L}^2}^2 + C \lvert \phi \rvert_{\newone{H}}^2.
\end{align*}
On the other hand, using the Agmon inequality \eqref{eq-Agmon's-inequalities}, we have
   \begin{align*}
     \lvert \phi_1^2 \nabla \phi \rvert_{\mathbb{L}^2}^2
      \leq \Vert \phi_1 \Vert_{L^\infty}^4 \lvert \phi \rvert_{\newone{H}}^2
     \leq C \lvert \phi_1 \rvert_{L^2}^2 \Vert \phi_1 \Vert_{\zero{H}{2}}^2  \lvert \phi \rvert_{\newone{H}}^2.
  \end{align*}
Using the H\"older and Agmon inequality together with the estimate \eqref{Ladyzhenskaya inequality-a}, we deduce that
\begin{align*}
&\lvert \phi (\phi_1 + \phi_2) \nabla \phi_2 \rvert_{\mathbb{L}^2}^2
\leq \Vert \phi \Vert_{L^4}^2 \Vert \nabla \phi_2 \Vert_{\mathbb{L}^4}^2 \Vert \phi_1 + \phi_2 \Vert_{L^\infty}^2
\\
&\leq C \vert \phi_2 \rvert_{\newone{H}} \Vert \phi_2 \Vert_{\zero{H}{2}} \lvert \phi_1 + \phi_2 \rvert_{L^2} \Vert \phi_1 + \phi_2 \Vert_{\zero{H}{2}} \lvert \phi \rvert_{\newone{H}}^2.
\end{align*}
From the above estimates, we infer that
\begin{align*}
&\newK \lvert (\nabla (\psi^\prime(\phi_1) - \psi^\prime(\phi_2)), \nabla \cp) \rvert
\\
&\leq \frac{\newK}{2} \lvert \cp \rvert_{\newone{H}} ^2 + C (1 + \lvert \phi_1 \rvert_{L^2}^2 \Vert \phi_1 \Vert_{\zero{H}{2}}^2 + \vert \phi_2 \rvert_{\newone{H}} \Vert \phi_2 \Vert_{\zero{H}{2}} \lvert \phi_1 + \phi_2 \rvert_{L^2} \Vert \phi_1 + \phi_2 \Vert_{\zero{H}{2}}) \lvert \phi \rvert_{\newone{H}}^2.
\end{align*}
Analogously, we prove that 
\begin{align*}
&\newK \kappa_0 \lvert (\psi^\prime(\phi_1) - \psi^\prime(\phi_2), \Atwo \phi) \rvert
= \newK \kappa_0 \lvert (\nabla(\psi^\prime(\phi_1) - \psi^\prime(\phi_2)), \nabla \phi) \rvert
\\
&\leq C (1 + \lvert \phi_1 \rvert_{L^2}^2 \Vert \phi_1 \Vert_{\zero{H}{2}}^2 + \vert \phi_2 \rvert_{\newone{H}} \Vert \phi_2 \Vert_{\zero{H}{2}} \lvert \phi_1 + \phi_2 \rvert_{L^2} \Vert \phi_1 + \phi_2 \Vert_{\zero{H}{2}}) \lvert \phi \rvert_{\newone{H}}^2.
\end{align*}
Next, we use the assumptions \eqref{Eqn-coercivity-3} and \eqref{g-Lipschitz-condition}, and we choose and fix $\eta$ small enough so that 
$\tilde{C}_1 \eta \leq \frac{\delta_0}{2}$. Therefore, one has for every $s \in [0,T]$,
\begin{align*}
&\frac12 \Vert \bSi_0(s)\bu(s) + \bG_0(s,\bu_1(s)) - \bG_0(s,\bu_2(s))\Vert_{\ell^2(\StokesH)}^2
\\
& \leq \frac12 \left(1 + \eta \right) \Vert \bSi_0(s)\bu(s) \Vert_{\ell^2(\StokesH)}^2 + C(\eta) \Vert \bG_0(s,\bu_1(s)) - \bG_0(s,\bu_2(s)) \Vert_{\ell^2(\StokesH)}^2
  \\
& \leq  \frac12 \Vert \bSi_0(s)\bu(s) \Vert_{\ell^2(\StokesH)}^2  + \tilde{C}_1 \eta \Vert \bu(s) \Vert_{\StokesV}^2  + C(\eta) \tilde{c}_g^2\lvert \bu(s) \rvert_{\StokesH}^2
   \\
& \leq  \frac12 \Vert \bSi_0(s)\bu(s) \Vert_{\ell^2(\StokesH)}^2  + \frac{\delta_0}{2} \Vert \bu(s) \Vert_{\StokesV}^2  + C(\eta) \tilde{c}_g^2 \lvert \bu(s) \rvert_{\StokesH}^2.
\end{align*}
From the above inequality and the  assumption \eqref{Eqn-coercivity-2} in Lemma \ref{Lem1}, we deduce that 
\begin{align*}
&\nu \Vert \bu(s) \Vert_{\StokesV}^2 - \frac12 \Vert \bSi_0(s)\bu(s) + \bG_0(s,\bu_1(s)) - \bG_0(s,\bu_2(s))\Vert_{\ell^2(\StokesH)}^2
\\
&\geq \nu \Vert \bu(s) \Vert_{\StokesV}^2 - \frac12 \Vert \bSi_0(s)\bu(s) \Vert_{\ell^2(\StokesH)}^2 - \frac{\delta_0}{2} \Vert \bu(s) \Vert_{\StokesV}^2 - C(\eta,\tilde{c}_g) \lvert \bu(s) \rvert_{\StokesH}^2
  \\
&\geq \nu \Vert \bu(s) \Vert_{\StokesV}^2 - (\nu - \delta_0) \Vert \bu(s) \Vert_{\StokesV}^2 - \frac{\delta_0}{2} \Vert \bu(s) \Vert_{\StokesV}^2 - C(\eta,\tilde{c}_g) \lvert \bu(s) \rvert_{\StokesH}^2
\\
&= \frac{\delta_0}{2} \Vert \bu(s) \Vert_{\StokesV}^2 - C(\eta,\tilde{c}_g) \lvert \bu(s) \rvert_{\StokesH}^2.
\end{align*}
Hereafter, we choose and fix for the remainder of the proof $\kappa_0$ small enough so that
    \begin{equation*}
      1 - 5 \kappa_0 C_{\domO}^2>0 .
   \end{equation*}
We finish the proof by using the so called Schmalfuss trick. For this purpose,  let us introduce the following auxilliary  real-valued process 
$$
\mathcal{Y}_1(t)= e^{- \int_0^{t} \mathcal{Y}_2(s)\, \d s}, \; t \in [0,T],
$$
with the process $\mathcal{Y}_2= (\mathcal{Y}_2(t), \; t \in [0,T])$ being defined in \eqref{eqn-process-y_2}. Note that
   \begin{equation*}
     \int_0^T \mathcal{Y}_2(s)\, \d s< \infty, \quad \mathbb{P}\mbox{-a.s. }
   \end{equation*}
Then, applying the It\^o formula to the process $\mathcal{Y}_1 U$, using the above estimates, we arrive at
\begin{equation*}
\begin{aligned}
& \mathcal{Y}_1(t \wedge \tau_R) U(t \wedge \tau_R) + \frac{\delta_0}{4} \int_0^{t \wedge \tau_R} \mathcal{Y}_1(s) \Vert \bu(s) \Vert_{\StokesV}^2 \, \d s + \frac{(1 - 5 \kappa_0 C_{\domO}^2) \newK}{2}  \int_0^{t \wedge \tau_R} \mathcal{Y}_1(s) \lvert \cp(s) \rvert_{\newone{H}} ^2\, \d s \\
& + \frac{\newK \kappa_0}{2} \int_0^{t \wedge \tau_R} \mathcal{Y}_1(s) \lvert \Atwo \phi(s) \rvert_{L^2}^2\, \d s
\leq \frac{1}{2} \lvert \upxi_1 - \upxi_2 \rvert_{\StokesH}^2 + \frac{\newK}{2} \lvert \chi_1 - \chi_2 \rvert_{\newone{H}}^2 + C(\eta,\tilde{c}_g) \int_0^{t \wedge \tau_R} \mathcal{Y}_1(s) U(s)\, \d s \\
&\hspace{5cm} + \int_0^{t \wedge \tau_R} \mathcal{Y}_1(s) (\bu(s),[\bG_0(s,\bu_1(s)) - \bG_0(s,\bu_2(s)) + \bSi_0(s)\bu(s)]\d W(s)).
\end{aligned}
\end{equation*}
Since $0< \mathcal{Y}_1(t) \leq 1$, taking the expectation to both sides of the above inequality yields
\begin{align*}
& \mathbb{E} [\mathcal{Y}_1(t \wedge \tau_R) U(t \wedge \tau_R)] 
+ \frac{\delta_0}{4} \mathbb{E} \int_0^{t \wedge \tau_R} \mathcal{Y}_1(s) \Vert \bu(s) \Vert_{\StokesV}^2 \, \d s 
+ \frac{(1 - 5 \kappa_0 C_{\domO}^2) \newK}{2} \mathbb{E} \int_0^{t \wedge \tau_R} \mathcal{Y}_1(s) \lvert \cp(s) \rvert_{\newone{H}} ^2\, \d s 
\\
& + \frac{\newK \kappa_0}{2} \mathbb{E} \int_0^{t \wedge \tau_R} \mathcal{Y}_1(s) \lvert \Atwo \phi(s) \rvert_{L^2}^2\, \d s
\leq \frac{1}{2} \lvert \upxi_1 - \upxi_2 \rvert_{\StokesH}^2 + \frac{\newK}{2} \lvert \chi_1 - \chi_2 \rvert_{\newone{H}}^2 + C(\eta) \tilde{c}_g^2 \mathbb{E} \int_0^{t \wedge \tau_R} \mathcal{Y}_1(s) U(s)\, \d s.
\end{align*}
Therefore, by the Gronwall inequality applied to $t \mapsto \mathbb{E} [\mathcal{Y}_1(t \wedge \tau_R) U(t \wedge \tau_R)]$, we deduce that there exists a constant $C>0$ independent of $R$, such that for all $t \in [0,T]$,
  \begin{equation}
    \mathbb{E} \left[\mathcal{Y}_1(t \wedge \tau_R) U(t \wedge \tau_R) \right] \leq C \Vert (\upxi_1,\chi_1) - (\upxi_2,\chi_2) \Vert_{\nHB}^2.
  \end{equation}
We recall that $\tau_R \toup  T$, $\mathbb{P}$-a.s., as $ R \to \infty$, and then by continuity of the process 
$\mathcal{Y}_1 U= (\mathcal{Y}_1(t) U(t), \; t \in [0,T])$, we have 
$\mathcal{Y}_1(t \wedge \tau_R) U(t \wedge \tau_R) \to \mathcal{Y}_1(t) U(t)$  $\mathbb{P}$-a.s., as $ R \to \infty$. Furthermore, since $\mathcal{Y}_1(t \wedge \tau_R) U(t \wedge \tau_R)\geq 0$, by the Fatou Lemma, we infer that
   \begin{equation}
    \mathbb{E} \left[\mathcal{Y}_1(t) U(t) \right] \leq C \Vert (\upxi_1,\chi_1) - (\upxi_2,\chi_2) \Vert_{\nHB}^2,\;\; t\in [0,T].
   \end{equation}
Hence, the proof  of the theorem then follows from the preceding inequality. 
\end{proof}
\begin{remark}\label{rem-uniquness}
In principle, our proof of the Theorem is similar to the proof of \cite[Lemma 7.3]{Brz+Motyl_2013}, but it is technically and computationally more involved.
\end{remark}

\section{Concluding Remarks}\label{sec-remarks}
We conclude this paper with a brief comparison between the method employed here, i.e. an  extension and elaboration of the method used by Mikulevicius and  Rozovskii in \cite{Mik+Roz_2005} and the method based on the Jakubowski-Skorokhod theorem, which was used in the context of stochastic fluid dynamics by the first author and E. Motyl in \cite{Brz+Motyl_2013}. The first advantage of the Mikulevicius and  Rozovskii approach is that it does not require a change of probability space, provided it is rich enough. 
The second advantage of our present approach over that one from \cite{Brz+Motyl_2013},  is that here we can assume that the external force $f$ satisfies the following assumption. 
\[
\int_0^T \Vert f(t) \Vert_{\StokesVp}^2\,\d t< \infty,
\]
while in the other papers, see also Theorem 4.8 in \cite{Brz+Mot+Ondr_2017}, we assume that  
\[
\int_0^T \Vert f(t)\Vert_{\StokesVp}^p\,\d t< \infty, \mbox{ for at least for one $p>2$}.
\]
Let us point out that the methods developed in the present paper allow one to study similar problems in the following cases.
\begin{trivlist}
    \item[(i)] The domain $\domO$ is unbounded, 
    \item[(ii)] The Cahn-Hilliard Equation depends on the noise. 
    \item[(iii)] The potential is singular, for example, the Flory-Huggins potential.
\end{trivlist}
\section{Acknowledgments}
The second author acknowledges financial support from the European Union's Horizon Europe Marie Sk{\l}odowska-Curie Action Postdoctoral Fellowship  No. 101151937-SNSCHEs (via the UKRI guarantee framework).
Both authors would like to thank the Institute of Mathematics, Jagiellonian University (Krak{\'o}w Poland) for hospitality in October 2025. We would like also to thank Boris Rozovskii and  Remigijus Mikulevicius for an important discussion about the martingale property of the stopped process on the first part of Theorem \ref{thm-5.3}; and to Tachim Medjo and Paul Razafimandimby for an important discussion about Section \ref{Sect-approximation of g}.

    \section*{Declarations:} 
	
	\noindent 	\textbf{Ethical Approval:}   Not applicable 
	
	
	\noindent  \textbf{Conflict of interest: } On behalf of all authors, the corresponding author states that there is no conflict of interest.
	
	\noindent 	\textbf{Authors' contributions:} All authors have contributed equally.

	\noindent 	\textbf{Availability of data and materials:} Not applicable.

\appendix

\section{Limit of nonlinear operators}
\begin{lemma}\label{B_on-Convergence}
Assume that $\bu \in L^2(0,T;\StokesH)$ and $(\bu_n)_{n \in \mathbb{N}}$ is a bounded sequence in $L^2(0,T;\StokesH)$ such that
$\bu_n \to \bu$ in $L^2(0,T;\StokesH)$. Then, for every $z_1 \in \rU_0$,
  \begin{equation}\label{Convergence-for-Bo-n}
    \lim_{n \to \infty} \int_0^T \lvert \duality{\bB_{0,n}(\bu_n(s),\bu_n(s)) - \bB_0(\bu(s),\bu(s))}{z_1}{\rU_0}{\rU_0^\prime} \rvert\,\d s=0.
  \end{equation}
\end{lemma}
\begin{proof}[Proof of Lemma \ref{B_on-Convergence}]
In the sequel, the duality pairing $\duality{\cdot}{\cdot}{\rU_0}{\rU_0^\prime}$ will be denoted by $\duality{\cdot}{\cdot}{}{}$.
Let us choose and fix $z_1 \in \rU_0$. We recall that, see \eqref{eq-B_0-k_0},
   \begin{equation}\label{eq-B_0-k_0-1}
     \lvert \duality{\bB_0(\bv_1,\bv_2)}{z_1}{}{} \rvert
     \leq C \lvert \bv_1 \rvert_{\StokesH} \lvert \bv_2 \rvert_{\StokesH} \Vert z_1 \Vert_{\rU_0},\;\;\forall \bv_1,\,\bv_2 \in \rU_0.  
   \end{equation}
Observe also that for every $n \in \mathbb{N}$,
\begin{align*}
\duality{\bB_{0,n}(\bu_n,\bu_n) - \bB_0(\bu,\bu)}{z_1}{}{} 
&= \duality{\bB_0(\bu_n - \bu,\bu_n)}{\tilde{\pi}_{0,n} z_1}{}{} + \duality{\bB_0(\bu,\bu_n - \bu)}{\tilde{\pi}_{0,n} z_1}{}{}
\\
&\qquad + \duality{\bB_0(\bu,\bu)}{\tilde{\pi}_{0,n} z_1  - z_1 }{}{}.
\end{align*}
Using the estimate \eqref{eq-B_0-k_0-1} and the H\"older inequality, we infer that there exists a generic constant $C>0$ such that
for every $n \in \mathbb{N}$,
\begin{align*}
&\int_0^T (\lvert \duality{\bB_0(\bu_n(s) - \bu(s),\bu_n(s))}{\tilde{\pi}_{0,n} z_1}{}{} \rvert + \lvert \duality{\bB_0(\bu(s),\bu_n(s) - \bu(s))}{\tilde{\pi}_{0,n} z_1}{}{} \rvert)\,\d s
\\
&\leq C \int_0^T (\lvert \bu_n(s) \rvert_{\StokesH} + \lvert \bu(s) \rvert_{\StokesH}) \lvert \bu_n(s) - \bu(s) \rvert_{\StokesH} \,\d s \cdot  \Vert \tilde{\pi}_{0,n} z_1 \Vert_{\rU_0}
\\
&\leq C \int_0^T (\lvert \bu_n(s) \rvert_{\StokesH} + \lvert \bu(s) \rvert_{\StokesH}) \lvert \bu_n(s) - \bu(s) \rvert_{\StokesH}\,\d s \cdot
\Vert \tilde{\pi}_{0,n} \Vert_{\mathcal{L}(\rU_0)} \Vert z_1 \Vert_{\rU_0}
\\
&\leq C (\vert \bu_n \Vert_{L^2(0,T;\StokesH)} + \Vert \bu \Vert_{L^2(0,T;\StokesH)}) \Vert \bu_n - \bu \Vert_{L^2(0,T;\StokesH)} \Vert z_1 \Vert_{\rU_0}.
\end{align*}
Now, since $\bu_n \to \bu$ in $L^2(0,T;\StokesH)$, we deduce that
\[
\lim_{n \to \infty} \int_0^T (\lvert \duality{\bB_0(\bu_n(s) - \bu(s),\bu_n(s))}{\tilde{\pi}_{0,n} z_1}{}{} \rvert + \lvert \duality{\bB_0(\bu(s),\bu_n(s) - \bu(s))}{\tilde{\pi}_{0,n} z_1}{}{} \rvert)\,\d s=0.
\]
Using the estimate \eqref{eq-B_0-k_0-1}, we infer that there exists $C>0$ such that
for every $n \in \mathbb{N}$,
   \begin{equation*}
     \int_0^T \lvert \duality{\bB_0(\bu(s),\bu(s))}{\tilde{\pi}_{0,n} z_1  - z_1}{}{} \rvert\,\d s
      \leq \Vert \bu \Vert_{L^2(0,T;\StokesH)}^2 \Vert \tilde{\pi}_{0,n} z_1  - z_1\Vert_{\rU_0}. 
   \end{equation*}
Since $\Vert \tilde{\pi}_{0,n} z_1  - z_1\Vert_{\rU_0} \to 0$, see the item $(iii)$ in Properties \ref{eqn-properties-pi_{0,n}}, we infer that
   \begin{equation*}
     \lim_{n \to \infty} \int_0^T \lvert \duality{\bB_0(\bu(s),\bu(s))}{\tilde{\pi}_{0,n} z_1  - z_1}{}{} \rvert\,\d s=0.
   \end{equation*}
From the above convergence results, we can easily complete the proof of Lemma \ref{B_on-Convergence}.
\end{proof}
\begin{lemma}\label{R_on-Convergence}
Assume that $\phi \in L^\infty(0,T;\newone{H}) \cap L^2(0,T;\zero{H}{2}(\domO))$ and $(\phi_n)_{n \in \mathbb{N}}$ is a bounded sequence in $L^\infty(0,T;\newone{H}) \cap L^2(0,T;\zero{H}{2}(\domO))$ i.e. 
\[
\sup_{n \in \mathbb{N}} \left\{\int_0^T \Vert \phi_n(s) \Vert_{\zero{H}{2}}^{2}\, \d s + \esssup_{s \in [0,T]} \lvert \phi_n(s) \rvert_{\newone{H}} 
\right\}< \infty,
\]
such that $\phi_n \to \phi$ in $L^2(0,T;\newone{H})$. Then, for every $z_1 \in \rU_0$,
\begin{equation}\label{Convergence-for-bar{R}_{0,n}}
   \lim_{n \to \infty} \int_0^T \duality{\bar{R}_{0,n}(\phi_n(s),\phi_n(s)) - \bR_0(\phi(s),\phi(s))}{ z_1}{\rU_0}{\rU_0^\prime}\,\d s= 0.
\end{equation}
\end{lemma}
\begin{proof}
In the sequel, the duality pairing $\duality{\cdot}{\cdot}{\rU_0}{\rU_0^\prime}$ will be denoted by $\duality{\cdot}{\cdot}{}{}$.
Let us choose and fix $z_1 \in \rU_0$ and notice that 
\begin{align*}
&\duality{\bar{R}_{0,n}(\phi_n,\phi_n)}{z_1}{}{}
=\duality{\pi_{0,n} \left(\Atwo \phi_n  \nabla \phi_n \right)}{z_1}{}{} + \duality{\pi_{0,n} ([\pi_{1,n} (\psi^\prime(\phi_n) - \avg{\psi^\prime(\phi_n)})] \nabla \phi_n)}{z_1}{}{}
\\
&=\duality{\pi_{0,n} \left(\Atwo \phi_n  \nabla \phi_n \right)}{z_1}{}{} + \duality{[\pi_{0,n}(\pi_{1,n} (\psi^\prime(\phi_n) - \avg{\psi^\prime(\phi_n)})] \nabla \phi_n)}{z_1}{}{}
\\
&=\duality{\P \left(\Atwo \phi_n  \nabla \phi_n \right)}{\tilde{\pi}_{0,n} z_1}{}{} + \duality{\pi_{0,n}(\pi_{1,n} (\psi^\prime(\phi_n) - \avg{\psi^\prime(\phi_n)}) \nabla \phi_n)}{z_1}{}{}
\\
&= \duality{\bR_0(\phi_n,\phi_n)}{\tilde{\pi}_{0,n} z_1}{}{} + \duality{\pi_{0,n}(\pi_{1,n} (\psi^\prime(\phi_n) - \avg{\psi^\prime(\phi_n)}) \nabla \phi_n)}{z_1}{}{}.
\end{align*}
This, jointly with the fact that
$\pi_{0,n}([(\psi^\prime(\phi) - \avg{\psi^\prime(\phi)})] \nabla \phi) =\pi_{0,n} \circ \P ([(\psi^\prime(\phi) - \avg{\psi^\prime(\phi)})] \nabla \phi)=0$ yields
\begin{align*}
&\duality{\bar{R}_{0,n}(\phi_n,\phi_n)}{z_1}{}{} - \duality{\bR_0(\phi,\phi)}{z_1}{}{}
\\
&=  \duality{\bR_0(\phi_n,\phi_n) - \bR_0(\phi,\phi)}{\tilde{\pi}_{0,n} z_1}{}{} + \duality{\bR_0(\phi,\phi)}{\tilde{\pi}_{0,n} z_1 - z_1}{}{} 
\\
&\quad + \duality{\pi_{0,n}(\pi_{1,n} (\psi^\prime(\phi_n) - \avg{\psi^\prime(\phi_n)}) \nabla \phi_n) - \pi_{0,n}([(\psi^\prime(\phi) - \avg{\psi^\prime(\phi)})] \nabla \phi)}{z_1}{}{}.
\end{align*}
Moreover, since $\bR_0(\phi_n,\phi_n) - \bR_0(\phi,\phi)= \bR_0(\phi_n - \phi,\phi_n) + \bR_0(\phi,\phi_n - \phi)$, we find that
\begin{align*}
&\duality{\bar{R}_{0,n}(\phi_n,\phi_n)}{z_1}{}{} - \duality{\bR_0(\phi,\phi)}{z_1}{}{}
\\
&= \duality{\bR_0(\phi_n - \phi,\phi_n)}{\tilde{\pi}_{0,n} z_1}{}{}  
+ \duality{\bR_0(\phi,\phi_n - \phi)}{\tilde{\pi}_{0,n} z_1}{}{} 
+ \duality{\bR_0(\phi,\phi)}{\tilde{\pi}_{0,n} z_1 - z_1}{}{} 
\\
&\quad + \duality{\pi_{0,n}(\pi_{1,n} (\psi^\prime(\phi_n) - \avg{\psi^\prime(\phi_n)}) \nabla \phi_n) - \pi_{0,n}([(\psi^\prime(\phi) - \avg{\psi^\prime(\phi)})] \nabla \phi)}{z_1}{}{}.
\end{align*}
Let us deal with the first two terms on the RHS of the above equalities.
Using the H\"older inequality and since $\Vert \tilde{\pi}_{0,n} \Vert_{\mathcal{L}(\rU_0)}\leq 1$, we infer that there exists a generic constant $C>0$ such that for every $n \in \mathbb{N}$,
\begin{align*}
&\int_0^T  \lvert \duality{\bR_0(\phi_n(s) - \phi(s),\phi_n(s)) + \bR_0(\phi(s),\phi_n(s) - \phi(s))}{\tilde{\pi}_{0,n} z_1 }{}{} \rvert\, \d s  
\\
&\leq C \int_0^T (\lvert \phi_n(s) \rvert_{\newone{H}} + \lvert \phi(s) \rvert_{\newone{H}}) \lvert \phi_n(s) - \phi(s) \rvert_{\newone{H}}\, \d s \cdot
\Vert \tilde{\pi}_{0,n} z_1 \Vert_{\rU_0}
\\
&\leq C \int_0^T (\lvert \phi_n(s) \rvert_{\newone{H}} + \lvert \phi(s) \rvert_{\newone{H}}) \lvert \phi_n(s) - \phi(s) \rvert_{\newone{H}}\, \d s \cdot \Vert \tilde{\pi}_{0,n} \Vert_{\mathcal{L}(\rU_0)}\Vert z_1 \Vert_{\rU_0}
\\
&\leq C (\Vert \phi_n \Vert_{L^2(0,T;\newone{H})} + \Vert \phi \Vert_{L^2(0,T;\newone{H})}) \Vert \phi_n - \phi \Vert_{L^2(0,T;\newone{H})} \Vert z_1 \Vert_{\rU_0}.
\end{align*}
Recall that by assumption $\Vert \phi_n - \phi \Vert_{L^2(0,T;\newone{H})} \to 0$ as $n \to \infty$. Thus,
\[
\lim_{n \to \infty} \int_0^T \lvert \duality{\bR_0(\phi_n(s) - \phi(s),\phi_n(s)) + \bR_0(\phi(s),\phi_n(s) - \phi(s))}{\tilde{\pi}_{0,n} z_1}{}{} \rvert \d s=0.
\]
Similarly, we have for every $n \in \mathbb{N}$,
\begin{equation*}
\int_0^T \lvert \duality{\bR_0(\phi(s),\phi(s))}{\tilde{\pi}_{0,n} z_1 - z_1}{}{} \rvert\, \d s
\leq C(\domO) \Vert \phi \Vert_{L^2(0,T;\newone{H})}^2 \Vert \tilde{\pi}_{0,n} z_1 - z_1 \Vert_{\rU_0}.
\end{equation*}
Since $\Vert \tilde{\pi}_{0,n} z_1  - z_1\Vert_{\rU_0} \to 0$ as $n \to \infty$, we deduce that
\[
\lim_{n \to \infty} \int_0^T \lvert \duality{\bR_0(\phi(s),\phi(s))}{\tilde{\pi}_{0,n} z_1 - z_1}{}{} \rvert\, \d s=0.
\]
Let us now consider the term
\[
I_n(t) \coloneqq \duality{\pi_{0,n}(\pi_{1,n} (\psi^\prime(\phi_n(t)) - \avg{\psi^\prime(\phi_n(t))}) \nabla \phi_n(t)) - \pi_{0,n}([(\psi^\prime(\phi(t)) - \avg{\psi^\prime(\phi(t))})] \nabla \phi(t))}{z_1}{}{}
\]
with $t \in[0,T],\;\; n \in \mathbb{N}$.
We first split the term $I_n$ as follows:
\begin{align*}
I_n=& \duality{\pi_{0,n} (\pi_{1,n} (\psi^\prime(\phi_n) - \avg{\psi^\prime(\phi_n)}) \nabla(\phi_n - \phi))}{z_1}{}{}
\\
& + \duality{\pi_{0,n}(\pi_{1,n} (\psi^\prime(\phi_n) - \avg{\psi^\prime(\phi_n)} - [\psi^\prime(\phi) - \avg{\psi^\prime(\phi)}]) \nabla \phi)}{z_1}{}{}
\\
& + \duality{\pi_{0,n}[(\pi_{1,n} [\psi^\prime(\phi) - \avg{\psi^\prime(\phi)}] - [\psi^\prime(\phi) - \avg{\psi^\prime(\phi)}])\nabla\phi]}{z_1}{}{} \coloneqq I_{1,n} + I_{2,n} + I_{3,n},
\end{align*}
where we omitted $(t)$. \newline
For the term $I_{1,n}$, we use the H\"older inequality and the fact that $\Vert \pi_{0,n} \Vert_{\mathcal{L}(\rU_0^\prime)}\leq 1$. Hence, there exists a generic constant $C>0$ such that for every $n \in \mathbb{N}$,
\begin{align*}
&\int_0^T \lvert I_{1,n}(s)\rvert \d s
=\int_0^T \lvert \duality{\pi_{0,n} (\pi_{1,n} (\psi^\prime(\phi_n(s)) - \avg{\psi^\prime(\phi_n(s))}) \nabla(\phi_n(s) - \phi(s)))}{z_1}{}{}\rvert\,\d s
\\
&\leq \int_0^T \Vert \pi_{0,n} (\pi_{1,n} (\psi^\prime(\phi_n(s)) - \avg{\psi^\prime(\phi_n(s))}) \nabla(\phi_n(s) - \phi(s)))\Vert_{\rU_0^\prime}\,\d s \cdot \Vert z_1 \Vert_{\rU_0}
  \\
&= \int_0^T \Vert \pi_{0,n} \circ \P (\pi_{1,n} (\psi^\prime(\phi_n(s)) - \avg{\psi^\prime(\phi_n(s))}) \nabla(\phi_n(s) - \phi(s)))\Vert_{\rU_0^\prime}\,\d s \cdot \Vert z_1 \Vert_{\rU_0}
\\
&\leq \int_0^T \Vert \pi_{0,n} \Vert_{\mathcal{L}(\rU_0^\prime)} \Vert \P(\pi_{1,n} (\psi^\prime(\phi_n(s)) - \avg{\psi^\prime(\phi_n(s))}) \nabla(\phi_n(s) - \phi(s)))\Vert_{\rU_0^\prime}\,\d s \cdot \Vert z_1 \Vert_{\rU_0}
  \\
&\leq \int_0^T \Vert \P(\pi_{1,n} (\psi^\prime(\phi_n(s)) - \avg{\psi^\prime(\phi_n(s))}) \nabla(\phi_n(s) - \phi(s)))\Vert_{\rU_0^\prime}\,\d s \cdot \Vert z_1 \Vert_{\rU_0}.
\end{align*}
Next, let $z \in \rU_0$ be arbitrarily but fixed. By the Sobolev embedding $\rU_0 \embed \mathbb{L}^\infty(\domO)$ and the fact that $\Vert \pi_{1,n} \Vert_{\mathscr{L}(\newone{H})} \leq 1$, one has for every $n \in \mathbb{N}$,
\begin{align*}
&\lvert \duality{\P(\pi_{1,n} (\psi^\prime(\phi_n) - \avg{\psi^\prime(\phi_n)}) \nabla(\phi_n - \phi))}{z}{}{} \rvert
= \lvert (\P(\pi_{1,n} (\psi^\prime(\phi_n) - \avg{\psi^\prime(\phi_n)}) \nabla(\phi_n - \phi)),z) \rvert
\\
&= \lvert (\pi_{1,n} (\psi^\prime(\phi_n) - \avg{\psi^\prime(\phi_n)}) \nabla(\phi_n - \phi),\P z) \rvert
= \lvert (\pi_{1,n} (\psi^\prime(\phi_n) - \avg{\psi^\prime(\phi_n)}) \nabla(\phi_n - \phi),z) \rvert
\\
&\leq \lvert \pi_{1,n} (\psi^\prime(\phi_n) - \avg{\psi^\prime(\phi_n)}) \rvert_{L^2} \lvert \phi_n - \phi \rvert_{\newone{H}} \Vert z \Vert_{\mathbb{L}^\infty}
\leq C \lvert \pi_{1,n} (\psi^\prime(\phi_n) - \avg{\psi^\prime(\phi_n)}) \rvert_{\newone{H}} \lvert \phi_n - \phi \rvert_{\newone{H}} \Vert z \Vert_{\rU_0}
\\
&\leq C \Vert \pi_{1,n} \Vert_{\mathcal{L}(\newone{H})} \lvert (\psi^\prime(\phi_n) - \avg{\psi^\prime(\phi_n)}) \rvert_{\newone{H}} \lvert \phi_n - \phi \rvert_{\newone{H}} \Vert z \Vert_{\rU_0}
\leq C \lvert (\psi^\prime(\phi_n) - \avg{\psi^\prime(\phi_n)}) \rvert_{\newone{H}} \lvert \phi_n - \phi \rvert_{\newone{H}} \Vert z \Vert_{\rU_0},
\end{align*}
which implies 
$\Vert \P(\pi_{1,n} (\psi^\prime(\phi_n) - \avg{\psi^\prime(\phi_n)}) \nabla(\phi_n - \phi))\Vert_{\rU_0^\prime}
\leq C \lvert (\psi^\prime(\phi_n) - \avg{\psi^\prime(\phi_n)}) \rvert_{\newone{H}} \lvert \phi_n - \phi \rvert_{\newone{H}}.
$
Consequently,
\begin{align*}
&\int_0^T \lvert I_{1,n}(s)\rvert \d s
\leq C \int_0^T \lvert (\psi^\prime(\phi_n(s)) - \avg{\psi^\prime(\phi_n(s))}) \rvert_{\newone{H}} \lvert \phi_n(s) - \phi(s) \rvert_{\newone{H}}\,\d s \cdot \Vert z_1 \Vert_{\rU_0}
\\
&\leq C \Vert \phi_n - \phi \Vert_{L^2(0,T;\newone{H})} \Vert \psi^\prime(\phi_n) - \avg{\psi^\prime(\phi_n)} \Vert_{L^2(0,T;\newone{H})} \Vert  z_1 \Vert_{\rU_0}.
\end{align*}
Moreover, arguing as in \eqref{eq-Psi-second-gradient-phi-n-1} and using the Agmon inequality \eqref{eq-Agmon's-inequalities}, we infer that for every $n \in \mathbb{N}$,
\begin{equation}
\lvert \psi^{\prime\prime}(\phi_n) \nabla \phi_n \rvert_{\mathbb{L}^2}
\leq \lvert \phi_n \rvert_{\newone{H}}  +  C \lvert \phi_n \rvert_{\newone{H}}^2  \Vert \phi_n \Vert_{\zero{H}{2}},
\end{equation}
from which we obtain
\begin{equation}\label{eqn- Psi^prime(phi_n) - avg{Psi^prime(phi_n)}}
\Vert \psi^\prime(\phi_n) - \avg{\psi^\prime(\phi_n)} \Vert_{L^2(0,T;\newone{H})}
\leq C (\Vert \phi_n \Vert_{L^\infty(0,T;\newone{H})} + \Vert \phi_n \Vert_{L^\infty(0,T;\newone{H})}^2 \Vert \phi_n \Vert_{L^2(0,T;\zero{H}{2}(\domO))}).
\end{equation}
Hence, there exists $C>0$ independent of $n$ such that
\begin{align*}
\int_0^T \lvert I_{1,n} (s) \rvert \d s
\leq C \Vert \phi_n - \phi \Vert_{L^2(0,T;\newone{H})} (\Vert \phi_n \Vert_{L^\infty(0,T;\newone{H})} + \Vert \phi_n \Vert_{L^\infty(0,T;\newone{H})}^2 \Vert \phi_n \Vert_{L^2(0,T;\zero{H}{2}(\domO))}) \Vert  z_1 \Vert_{\rU_0}.
\end{align*}
By the assumption of Lemma \ref{R_on-Convergence}, we deduce that
\[
\lim_{n \to \infty} \int_0^T \lvert I_{1,n}(s) \rvert\,\d s=0.
\]
Next, arguing as in the proof of the term $I_{1,n}$, also using the identity \eqref{eqn-Psi'}, we infer that there exists a generic constant $C>0$ independent of $n$ such that
\begin{align*}
&\lvert I_{2,n} \rvert
\leq C \lvert \pi_{1,n} (\psi^\prime(\phi_n) - \avg{\psi^\prime(\phi_n)} - [\psi^\prime(\phi) - \avg{\psi^\prime(\phi)}]) \rvert_{L^2}
\lvert \phi \rvert_{\newone{H}} \Vert z_1 \Vert_{\rU_0}
\\
&\leq C \lvert \psi^\prime(\phi_n) - \avg{\psi^\prime(\phi_n)} - [\psi^\prime(\phi) - \avg{\psi^\prime(\phi)}] \rvert_{\newone{H}}
\lvert \phi \rvert_{\newone{H}} \Vert z_1 \Vert_{\rU_0}
\leq C \lvert \nabla \psi^\prime(\phi_n)  - \nabla \psi^\prime(\phi) \rvert_{\mathbb{L}^2}
\lvert \phi \rvert_{\newone{H}} \Vert z_1 \Vert_{\rU_0} \\
&\leq C (\lvert 3 \phi_n^2 \nabla(\phi_n - \phi) - \nabla(\phi_n - \phi) \rvert_{\mathbb{L}^2} + 3 \lvert (\phi_n - \phi) (\phi_n + \phi) \nabla \phi  \rvert_{\mathbb{L}^2}) \lvert \phi \rvert_{\newone{H}} \Vert z_1 \Vert_{\rU_0}.
\end{align*}
So, arguing as in \eqref{Eqn-8.41} and \eqref{Eqn-8.42}, we deduce there exists $C>0$ independent of $n$ such that
\begin{align*}
\lvert I_{2,n} \rvert
\leq& C(\lvert \phi_n - \phi \rvert_{\newone{H}} + \lvert \phi_n \rvert_{\newone{H}} \Vert \phi_n \Vert_{\zero{H}{2}}  \lvert \phi_n - \phi \rvert_{\newone{H}})\lvert \phi \rvert_{\newone{H}} \Vert z_1 \Vert_{\rU_0}
\\
&\quad + C (\lvert \phi_n - \phi \rvert_{\newone{H}}^{1/2} \Vert \phi_n - \phi \Vert_{\zero{H}{2}}^{1/2}  
\lvert \phi_n + \phi \rvert_{\newone{H}}^{1/2} \Vert \phi_n + \phi \Vert_{\zero{H}{2}}^{1/2}  
\lvert \phi \rvert_{\newone{H}})\lvert \phi \rvert_{\newone{H}} \Vert z_1 \Vert_{\rU_0},
\end{align*}
which, in turn, implies that
\begin{align*}
&\int_0^T \lvert I_{2,n}(s) \rvert\,\d s
\leq C \Bigl[ \Vert \phi_n - \phi \Vert_{L^2(0,T;\newone{H})} + \Vert \phi_n \Vert_{L^\infty(0,T;\newone{H})} \Vert \phi \Vert_{L^\infty(0,T;\newone{H})}
\Vert \phi_n \Vert_{L^2(0,T;\zero{H}{2}(\domO))} \Vert \phi_n - \phi \Vert_{L^2(0,T;\newone{H})}\\
& + \Vert \phi_n - \phi \Vert_{L^2(0,T;\newone{H})}^{\frac12} \Vert \phi_n - \phi \Vert_{L^2(0,T;\zero{H}{2}(\domO))}^{\frac12} \Vert \phi_n + \phi \Vert_{L^2(0,T;\zero{H}{2}(\domO))}^{\frac12} \Vert \phi_n + \phi \Vert_{L^\infty(0,T;\newone{H})}^{\frac12} \Vert \phi \Vert_{L^\infty(0,T;\newone{H})}^2\Bigr]\Vert z_1 \Vert_{\rU_0}.
\end{align*}
By the uniform boundedness of the sequence $(\phi_n)_{n \in \mathbb{N}}$ in the spaces $L^\infty(0,T;\newone{H})$ and  $L^2(0,T;\zero{H}{2}(\domO))$,  
cf. Lemma \ref{R_on-Convergence}, together with the convergence $\phi_n \to \phi$ in $L^2(0,T;\newone{H})$, we deduce that
\[
\lim_{n \to \infty} \int_0^T \lvert I_{2,n}(s) \rvert\,\d s=0.
\]
Let us now consider the term $I_{3,n}$. Firstly, by arguing as in the proof of $I_{1,n}$, we infer that there exists a generic constant $C>0$ such that for every $n \in \mathbb{N}$,
\begin{align}\label{Eqn-S_n Psi-prime}
&\int_0^T \lvert I_{3,n}(s) \rvert\, \d s
\leq C \int_0^T \lvert \pi_{1,n} \psi^\prime(\phi(s)) - \psi^\prime(\phi(s)) \rvert_{L^2} \lvert \phi(s) \rvert_{\newone{H}} \, \d s \cdot \Vert \pi_{0,n} z_1 \Vert_{\rU_0} \notag
\\
&\leq C \int_0^T \lvert \pi_{1,n} [\psi^\prime(\phi(s)) - \avg{\psi^\prime(\phi(s))}] - [\psi^\prime(\phi(s)) - \avg{\psi^\prime(\phi(s))}] \rvert_{L^2} \lvert \phi(s) \rvert_{\newone{H}} \, \d s \cdot \Vert z_1 \Vert_{\rU_0} \notag
\\
&\leq C \int_0^T \lvert \pi_{1,n} [\psi^\prime(\phi(s)) - \avg{\psi^\prime(\phi(s))}] - [\psi^\prime(\phi(s)) - \avg{\psi^\prime(\phi(s))}] \rvert_{\newone{H}} \lvert \phi(s) \rvert_{\newone{H}}\, \d s \cdot \Vert z_1 \Vert_{\rU_0} 
\\
&\leq C \left(\int_0^T \lvert \pi_{1,n} [\psi^\prime(\phi(s)) - \avg{\psi^\prime(\phi(s))}] - [\psi^\prime(\phi(s)) - \avg{\psi^\prime(\phi(s))}] \rvert_{\newone{H}}^2\,\d s \right)^{1/2} \Vert \phi \Vert_{L^2(0,T;\newone{H})} \cdot \Vert z_1 \Vert_{\rU_0}. \notag
\end{align}
Next, arguing as in the proof of \eqref{eqn- Psi^prime(phi_n) - avg{Psi^prime(phi_n)}}, we infer that
\begin{equation*}
\Vert \psi^\prime(\phi) - \avg{\psi^\prime(\phi)} \Vert_{L^2(0,T;\newone{H})}
\leq C (\Vert \phi \Vert_{L^\infty(0,T;\newone{H})} + \Vert \phi \Vert_{L^\infty(0,T;\newone{H})}^2 \Vert \phi \Vert_{L^2(0,T;\zero{H}{2}(\domO))})<\infty.
\end{equation*}
Hence, by the DCT, we deduce that
   \begin{equation*}
     \lim_{n \to \infty} \int_0^T \lvert I_{3,n}(s) \rvert\,\d s=0.
   \end{equation*}
Consequently,
   \begin{equation*}
     \lim_{n \to \infty} \int_0^T \lvert I_n(s) \rvert\,\d s=0.
   \end{equation*}
This completes the proof of Lemma \ref{R_on-Convergence}.
\end{proof}
\begin{lemma}\label{B_1n-Convergence}
Suppose $\phi \in L^2(0,T;\zero{H}{2}(\domO)) \cap L^{\frac32}(0,T;\newone{V})$ and $\bu \in L^2(0,T;\StokesV)$. Assume also that 
\begin{trivlist} 
\item[(i)] the sequence $(\bu_n)_{n \in \mathbb{N}}$ is bounded in $L^2(0,T;\StokesV)$, i.e. 
\[
\sup_{n \in \mathbb{N}} \int_0^T \Vert \bu_n(s) \Vert_{\StokesV}^{2}\, \d s< \infty,
\]
\item[(ii)] the sequence $(\phi_n)_{n \in \mathbb{N}}$ is bounded in $L^2(0,T;\zero{H}{2}(\domO)) \cap L^{\frac32}(0,T;\newone{V})$, i.e.  
\[
\sup_{n \in \mathbb{N}} \left\{\int_0^T \Vert \phi_n(s) \Vert_{\zero{H}{2}}^{2}\, \d s + \int_0^T \Vert \phi_n(s) \Vert_{\newone{V}}^{3/2}\, \d s
\right\}< \infty
\]
and
\item[(iii)]  $\bu_n \to \bu$ in $L^2(0,T;\StokesH)$ and $\phi_n \to \phi$ in $L^2(0,T;\newone{H})$. 
\end{trivlist}
Then, for every $z_2 \in \newone{V}$,
    \begin{equation}
        \lim_{ n \to \infty} \int_0^T \lvert \duality{B_{1,n}(\bu_n(s),\phi_n(s)) - B_1(\bu(s),\phi(s))}{z_2}{\newone{V}}{\newonep{V}} \rvert\,\d s= 0.
     \end{equation}
\end{lemma}
\begin{proof}[Proof of Lemma \ref{B_1n-Convergence}]
In the sequel, the duality pairing $\duality{\cdot}{\cdot}{\newone{V}}{\newonep{V}}$ will be denoted by $\duality{\cdot}{\cdot}{}{}$.
Let us choose and fix $z_2 \in \newone{V}$. Notice that
\begin{align*}
&\duality{B_{1,n}(\bu_n,\phi_n)}{z_2}{}{}
=(B_{1,n}(\bu_n,\phi_n), z_2)_{\newone{H}}
=(B_1(\bu_n,\phi_n), \tilde{\pi}_{1,n} z_2)_{\newone{H}}
\\
&=b_1(\bu_n,\phi_n, \tilde{\pi}_{1,n} z_2)
= b_1(\bu_n - \bu,\phi_n, \tilde{\pi}_{1,n} z_2) + b_1(\bu,\phi_n - \phi, \tilde{\pi}_{1,n} z_2) + b_1(\bu,\phi, \tilde{\pi}_{1,n} z_2),
\end{align*}
and $\duality{B_1(\bu,\phi)}{z_2}{}{}= b_1(\bu,\phi,z_2)$. Hence,
\begin{align*}
&\duality{B_{1,n}(\bu_n,\phi_n)}{z_2}{}{} - \duality{B_1(\bu,\phi)}{z_2}{}{}
\\
&= b_1(\bu_n - \bu,\phi_n, \tilde{\pi}_{1,n} z_2) + b_1(\bu,\phi_n - \phi, \tilde{\pi}_{1,n} z_2) + b_1(\bu,\phi, \tilde{\pi}_{1,n} z_2 - z_2).
\end{align*}
Using the H\"older inequality, the fact that $\mathbb{H}^1(\domO) \embed \mathbb{L}^6(\domO)$ continuously and $\Vert \tilde{\pi}_{1,n} \Vert_{\mathcal{L}(\newone{V})}\leq 1$,
we infer that for every $n \in \mathbb{N}$,
\begin{align*}
&\lvert b_1(\bu_n - \bu,\phi_n, \tilde{\pi}_{1,n} z_2) \rvert
= \left \lvert \int_{\domO} [(\bu_n - \bu) \cdot \nabla \phi_n] \tilde{\pi}_{1,n} (-\Delta z_2)\,\d x \right\rvert
\\
&\leq \Vert \bu_n - \bu \Vert_{\mathbb{L}^3} \Vert \nabla \phi_n \Vert_{\mathbb{L}^6} \lvert \tilde{\pi}_{1,n} (-\Delta z_2) \rvert_{L^2}
\leq \Vert \bu_n - \bu \Vert_{\mathbb{L}^3} \Vert \nabla \phi_n \Vert_{\mathbb{L}^6} \lvert \tilde{\pi}_{1,n} \Delta z_2 \rvert_{\newone{H}}
\\
&\leq C \lvert \bu_n - \bu \rvert_{L^2}^{1/2} \Vert \bu_n - \bu \Vert_{\StokesV}^{1/2} \Vert \phi_n \Vert_{\zero{H}{2}} \Vert \tilde{\pi}_{1,n} z_2 \Vert_{\newone{V}}
\leq C \lvert \bu_n - \bu \rvert_{L^2}^{1/2} \Vert \bu_n - \bu \Vert_{\StokesV}^{1/2} \Vert \phi_n \Vert_{\zero{H}{2}} \Vert z_2 \Vert_{\newone{V}}.
\end{align*}
From the above inequalities, we deduce that for every $n \in \mathbb{N}$,
\begin{align*}
&\int_0^T \lvert b_1(\bu_n(s) - \bu(s),\phi_n(s), \tilde{\pi}_{1,n} z_2) \rvert\,\d s
\leq C \Vert \bu_n - \bu \Vert_{L^2(0,T;\StokesH)}^{1/2} \Vert \bu_n - \bu \Vert_{L^2(0,T;\StokesV)}^{1/2} \Vert \phi_n \Vert_{L^2(0,T;\zero{H}{2})} \Vert z_2 \Vert_{\newone{V}}
\\
&\leq C \Vert \bu_n - \bu \Vert_{L^2(0,T;\StokesH)}^{1/2} (\Vert \bu_n \Vert_{L^2(0,T;\StokesV)}^{1/2} + \Vert \bu \Vert_{L^2(0,T;\StokesV)}^{1/2} ) \Vert \phi_n \Vert_{L^2(0,T;\zero{H}{2})} \Vert z_2 \Vert_{\newone{V}}.
\end{align*}
Since $\bu_n \to \bu$ in $L^2(0,T;\StokesH)$ and the sequences $(\bu_n)_{n \in \mathbb{N}}$ and $(\phi_n)_{n \in \mathbb{N}}$ are uniformly bounded in $L^2(0,T;\StokesV)$ and $L^2(0,T;\zero{H}{2}(\domO))$, respectively, we infer that
\[
\lim_{n \to \infty} \int_0^T \lvert b_1(\bu_n(s) - \bu(s),\phi_n(s), \tilde{\pi}_{1,n} z_2) \rvert\,\d s=0.
\]
Let us now consider the term $b_1(\bu,\phi_n - \phi, \tilde{\pi}_{1,n} z_2)$. 
Using the H\"older and the Gagliardo-Nirenberg interpolation inequality, and since $\StokesV \subset \mathbb{L}^6$ by the Sobolev embedding Theorem,
we infer that there exists $C>0$ independent of $n$ such that, if $d=2$,
\begin{align*}
&\lvert b_1(\bu,\phi_n - \phi, \tilde{\pi}_{1,n} z_2) \rvert
\leq \Vert \bu \Vert_{\mathbb{L}^6} \Vert \nabla (\phi_n - \phi) \Vert_{\mathbb{L}^3} \lvert \tilde{\pi}_{1,n} (-\Delta z_2) \rvert_{L^2}
\\
&\leq \Vert \bu \Vert_{\mathbb{L}^6} \Vert \nabla (\phi_n - \phi) \Vert_{\mathbb{L}^3} \lvert \tilde{\pi}_{1,n} \Delta z_2 \rvert_{\newone{H}}
 \leq C \Vert \bu \Vert_{\StokesV} \lvert \phi_n - \phi \rvert_{\newone{H}}^{1/3} \Vert \phi_n - \phi \Vert_{\zero{H}{2}}^{2/3} \Vert z_2 \Vert_{\newone{V}}
 \\
&\leq C \Vert \bu \Vert_{\StokesV} \lvert \phi_n - \phi \rvert_{\newone{H}}^{1/3} (\Vert \phi_n \Vert_{\zero{H}{2}}^{2/3} + \Vert \phi \Vert_{\zero{H}{2}}^{2/3}) \Vert z_2 \Vert_{\newone{V}}. 
\end{align*}
Thus, in the case $d=2$, we deduce that for every $n \in \mathbb{N}$,
\begin{align*}
&\int_0^T \lvert b_1(\bu(s),\phi_n(s) - \phi(s),\tilde{\pi}_{1,n} z_2) \rvert\,\d s
\\
&\leq C \Vert \bu \Vert_{L^2(0,T;\StokesV)} \Vert \phi_n - \phi \Vert_{L^2(0,T;\newone{H})}^{1/3} (\Vert \phi_n \Vert_{L^2(0,T;\zero{H}{2})}^{2/3} + \Vert \phi \Vert_{L^2(0,T;\zero{H}{2})}^{2/3}) \Vert z_2 \Vert_{\newone{V}}.
\end{align*}
Arguing as previously, we obtain in the case $d=3$ that for every $n \in \mathbb{N}$,
\begin{align*}
&\int_0^T \lvert b_1(\bu(s),\phi_n(s) - \phi(s),\tilde{\pi}_{1,n} z_2) \rvert\,\d s\\
&\leq C \int_0^T \Vert \bu(s) \Vert_{\StokesV} \lvert \phi_n(s) - \phi(s) \rvert_{\newone{H}}^{\frac12} (\Vert \phi_n(s) \Vert_{\zero{H}{2}}^{\frac12} + \Vert \phi(s) \Vert_{\zero{H}{2}}^{\frac12}) \,\d s \cdot \Vert z_2 \Vert_{\newone{V}}
\\
&\leq C \Vert \bu \Vert_{L^2(0,T;\StokesV)} \Vert \phi_n - \phi \Vert_{L^2(0,T;\newone{H})}^{1/2} (\Vert \phi_n \Vert_{L^2(0,T;\zero{H}{2})}^{1/2} + \Vert \phi \Vert_{L^2(0,T;\zero{H}{2})}^{1/2}) \Vert z_2 \Vert_{\newone{V}}.
\end{align*}
Since $\phi_n \to \phi$ in $L^2(0,T;\newone{H})$ and $(\phi_n)_{n \in \mathbb{N}} \subset L^2(0,T;\zero{H}{2}(\domO))$ is uniformly bounded, we infer that
\[
\lim_{n \to \infty} \int_0^T \lvert b_1(\bu(s),\phi_n(s) - \phi(s),\tilde{\pi}_{1,n} z_2) \rvert\,\d s=0.
\]
By the H\"older inequality and since $\StokesV \subset \mathbb{L}^6$ and $\mathbb{H}^1(\domO) \subset \mathbb{L}^3(\domO)$ by the Sobolev embedding Theorem, we deduce that there exists $C>0$ independent of $n$ such that
\begin{align*}
&\int_0^T \lvert b_1(\bu(s),\phi(s), \tilde{\pi}_{1,n} z_2 - z_2) \rvert\,\d s
\leq \int_0^T \Vert \bu(s) \Vert_{\mathbb{L}^6} \Vert \nabla \phi(s) \Vert_{\mathbb{L}^3}\,\d s \cdot \lvert \tilde{\pi}_{1,n} (-\Delta z_2) + \Delta z_2\rvert_{L^2}
\\
&\leq C \int_0^T \Vert \bu(s) \Vert_{\StokesV} \Vert \phi(s) \Vert_{\zero{H}{2}}\,\d s \cdot \lvert \tilde{\pi}_{1,n} (-\Delta z_2) + \Delta z_2\rvert_{\newone{H}}
\leq C \int_0^T \Vert \bu(s) \Vert_{\StokesV} \Vert \phi(s) \Vert_{\zero{H}{2}}\,\d s \cdot \Vert \tilde{\pi}_{1,n} z_2 - z_2 \Vert_{\newone{V}}
\\
&\leq C \Vert \bu \Vert_{L^2(0,T;\StokesV)} \Vert \phi \Vert_{L^2(0,T;\zero{H}{2})} \Vert \tilde{\pi}_{1,n} z_2 - z_2 \Vert_{\newone{V}}.
\end{align*}
Since $\Vert \tilde{\pi}_{1,n} z_2 - z_2 \Vert_{\newone{V}} \to 0$ as $n \to \infty$, we can pass to the limit in the above inequalities and therefore deduce that
\[
\lim_{n \to \infty} \int_0^T \lvert b_1(\bu(s),\phi(s), \tilde{\pi}_{1,n} z_2 - z_2) \rvert\,\d s=0.
\]
It then follows that
\begin{equation*}
\lim_{n \to \infty} \int_0^T \lvert \duality{B_{1,n}(\bu_n(s),\phi_n(s)) - B_1(\bu(s),\phi(s))}{z_2}{}{} \rvert\,\d s= 0.
\end{equation*}
This completes the proof of Lemma \ref{B_1n-Convergence}.
\end{proof}


\section{About solution of the martingale problem}\label{app-Solution-Martingale-problem}
\begin{definition}\label{def-Q(E')}
Assume that  $\rU$ is a topological vector space and $\rU^\prime$ its topological dual.
We denote by $\mathscr{L}^+(\rU,\rU^\prime)$ the space of symmetric non-negative definite linear maps $Q:\rU \to \rU^\prime $, i.e. 
for all $ f,\, g \in \rU$  the following two conditions hold, 
\begin{equation*}
\begin{aligned}
    \duality{f}{Q g}{\rU^\prime}{\rU}&= \duality{g}{Q f}{\rU^\prime}{\rU} \mbox{ and } 
    \duality{f}{Q f}{\rU^\prime}{\rU} \geq 0.
\end{aligned}
\end{equation*}
\end{definition}
\begin{definition}\label{def-precictable}
Assume that  $(\Omega,\mathscr{F},\mathbb{P})$ is  a probability space with the right-continuous filtration of $\sigma$-fields $\mathbb{F}=(\mathscr{F}_t)_{t \in [0,T]}$. 
We denote by $\mathscr{P}(\mathbb{F})$ the $\sigma$-field of predictable subsets  of $[0,\infty) \times \Omega$,  with respect to $\mathbb{F}$, i.e. the  $\sigma$-field  generated by the following family of subsets of $[0,\infty) \times \Omega$:
\[
\bigcup_{0\leq s < t < \infty} \left\{ (s,t] \times A: A \in \mathscr{F}_s \right\},
\]
see \cite[Section 1.7]{Metivier+Pellaumail_1980}.\\
By $\mathcal{M}_{\loc}^{\rc}(\mathbb{F},\mathbb{P})$, we denote  the class of $\mathbb{R}$-valued local c\`adl\`ag martingales on the stochastic basis $(\Omega,\mathscr{F},\mathbb{F},\mathbb{P})$. 
A process $u:[0,\infty)\times \Omega \to \rU$ is called predictable, or more precisely $\mathbb{F}$-predictable, if and only if the map 
\[
u: [0,\infty)\times \Omega \to \rU
\]
is $\mathscr{P}(\mathbb{F})/\mathscr{B}(\rU)$ measurable. 
\end{definition}

\begin{definition}\label{def-strongly predictable}
Assume that  $\rU$ is a topological vector space and $\rU^\prime$ its topological dual.

A function
\[
Q: [0,\infty) \times \Omega \to \mathscr{L}^+(\rU,\rU^\prime) 
\]
is said to be  a strongly predictable process if and only if for every $f \in \rU$, the new process 
  \[
[0,\infty) \times \Omega \ni (t,\omega)  \mapsto  Q(t,\omega)[f] \in \rU^\prime 
\]  
is predictable in the sense of the previous Definition \ref{def-precictable}, i.e.  for all $f,\, g \in \rU$ the process  
    \begin{equation*}
      s \mapsto \duality{f}{Q_s g}{\rU^\prime}{\rU} \mbox{ is } \mathscr{P}(\mathbb{F})\mbox{-measurable}.
   \end{equation*}
\end{definition}
Next, let $\rU \subset V \subset H$ be separable Banach spaces with dense injections. Assume that $\rU$ and $H$ are Hilbert spaces.
By identifying $H$ with its dual $H^\prime$, we have the following embedding
\[
\rU \embed  V \embed   H \cong H^\prime \embed  V^\prime \embed  \rU^\prime
\]
with $\rU^\prime$ and $V^\prime$ being the duals of the space $\rU$ and $V$, respectively. \\
From now on, we will use the Gelfand triple $(\rU,H,\rU^\prime)$.

\begin{definition}\label{def-canonical-process}
Put
\begin{equation}\label{eqn-bZ}
\bZ= C([0,\infty),\rU^\prime),
\end{equation}
and let $X$ be the canonical process on $\bZ$ defined by
 \[
 X: [0,\infty) \times  \bZ \ni (t,\omega) \mapsto \omega(t) \in \rU^\prime.
 \]
As in \cite{Mik+Roz_1998}, we introduce the following objects 
 \begin{align}
     \mathscr{D}_t= \sigma(X_s,\,s \leq t),\;\;
     \mathscr{D}_{t +}= \bigcap_{s >t} \mathscr{D}_s,\;\;
     \mathbb{D}= (\mathscr{D}_{t +})_{t \in [0,T]}, \mbox{ and }
     \mathscr{D}= \sigma(X_s,\,s \geq 0),
      \end{align}
and we also refer the reader to the definition \eqref{eqn-D_t-D_t+} and \eqref{eqn-mb-D}.
\end{definition}

\begin{definition}[cf. \cite{Mik+Roz_1998}]\label{def-martingale problem}
Let $x_0 \in H$. Assume that functions 
\begin{align}\label{eqn-Lambda-1}
&\Lambda: [0,\infty) \times V \to \rU^\prime, 
\\
\label{eqn-Q}
&Q: [0,\infty) \times V \to \mathscr{L}^+(\rU,\rU^\prime),
\end{align}
are predictable, resp. strongly predictable, processes.

\noindent
A probability measure $\mathbb{P}$ on $\left(\bZ,\mathscr{D}\right)$ is said to be a solution of the martingale problem $(x_0,\Lambda,Q)$ if and only if the following conditions are satisfied 
\begin{trivlist}
\item[(i)]
$X_s \in V $,  $\Leb \otimes  \mathbb{P}$-almost surely for $(s,\omega) \in \mathbb{R}_+ \times \bZ$.
\item[(ii)] for every  $t>0$, 
$\int_0^t \Vert \Lambda(s,X_s) \Vert_{\rU^\prime}\,\d s< \infty$, $\mathbb{P}$-a.s.,
\item[(iii)]
and for every $u \in \rU$, a process $M_t(u)$ defined by 
      \begin{align*}
             M_t(u)\coloneqq  \duality{X_t}{u}{\rU}{\rU^\prime} - \duality{x_0}{u}{\rU}{\rU^\prime}  - \int_0^t \duality{\Lambda(s,X_s)}{u}{\rU}{\rU^\prime} \, \d s 
             \end{align*}
belongs to $\mathcal{M}_{\loc}^c(\mathbb{D},\mathbb{P})$ and for all  $u,\,v \in \rU$,
\begin{align*}
               M_t(u) M_t(v)- \int_0^t \duality{Q(s,X_s) u}{v}{\rU}{\rU^\prime} \, \d s \in \mathcal{M}_{\loc}^c(\mathbb{D},\mathbb{P}).
      \end{align*}
\end{trivlist}
The set of all solutions to the martingale problem $(x_0, \Lambda, Q)$ is denoted by $\mathbb{S}(x_0, \Lambda, Q)$.
\end{definition}
Now we report the following important results borrowed from \cite[Section Martingale Problems]{Mik+Roz_1998}, especially Lemma 3.2 on page 280. 
\begin{definition}\label{eqn-varphi^u}
    Let $u \in \rU$ be arbitrary but fixed. We define a function  $\varphi^u: [0,\infty) \times V \to \mathbb{C}$ by  
  \begin{equation*}
      \varphi^u(s,X)= i \duality{\Lambda(s,X)}{u}{\rU}{\rU^\prime} - \frac{1}{2} \,\duality{Q(s,X) u}{u}{\rU}{\rU^\prime}, \; \;s\geq 0, \;\; X \in V.
  \end{equation*}
\end{definition}

\begin{lemma}\label{Lema-A.4}
Let $\mathbb{P}$ be a probability measure on $\bZ$ such that conditions (i) and (ii) of Definition \ref{def-martingale problem} are satisfied. 
Then the following three conditions are equivalent. 
\begin{trivlist}
    \item[(i)] $\mathbb{P} \in \mathbb{S}(x_0, \Lambda, Q)$; 
    \item[(ii)] 
    there exists a  dense vector subspace $\rU_0 \subset \rU$ such that 
for each $u \in \rU_0 $ the process $M^u$ defined by the following formula 
\begin{equation}\label{eqn-M^u}
M_t^u= e^{i \langle X_t,u \rangle} - \int_0^t e^{i \langle X_s,u \rangle} \varphi^u(s,X_s) \, \d s , \;\; t \in [0,T]
\end{equation}
belongs to $ \mathcal{M}_{\loc}^c(\mathbb{D},\mathbb{P})$.
\item[(iii)] 
    there exists a   subset $\rU_{00} \subset \rU$ such that the linear span of   $\rU_{00}$ is a dense subset of $\rU$
and for each $u \in \rU_{00} $ the process  $M^u$ defined by formula  \eqref{eqn-M^u} belongs to $ \mathcal{M}_{\loc}^c(\mathbb{D},\mathbb{P})$.
\end{trivlist}
\end{lemma}

\begin{lemma}\label{Lem-B.4}
Assume that   $(\Omega, \mathscr{F}, \mathbb{F}, \mathbb{P})$ is a filtered probability space with a right continuous filtration   $\mathbb{F}=(\mathscr{F}_t)_{t \in [0,T]}$
such that on this probability space two i.i.d. copies of  $\ell^2$-cylindrical $\mathbb{F}$-adapted  Wiener processes are defined. 
Suppose also that $(\mO,\mathcal{E})$ is a measurable space and $\tilde{\kappa}$ is a non-negative measure on it such that the Hilbert space  $\rK\coloneqq L^2(\mO,\mathcal{E},\tilde{\kappa})$ is separable. 
 Assume that 
\[\sigma: [0,\infty) \times \Omega \times \mO \to \rU^\prime\] is  
 $\mathscr{P}(\mathbb{F}) \times \mathcal{E}/\mathscr{B}(\rU^\prime)$-measurable function.
Assume that 
$M=(M_t: t \in [0,T])$ is  a $\rU^\prime$-valued continuous local martingale  such that for each $u \in \rU$,
  \begin{equation}\label{eqn-<M_tu,u>^2}
    \dualitybig{M_t}{u}{\rU}{\rU^\prime}{2} - \int_0^t \int_\mO \dualitybig{\sigma_s(x)}{u}{\rU}{\rU^\prime}{2} \, \tilde{\kappa} (\d x)\, \d s \in \mathcal{M}_{\loc}^c(\mathbb{D},\mathbb{P}).
 \end{equation}
Then there exists a  $\rK$-cylindrical Wiener process $W$  such that
for every $t \in [0,T]$, 
  \begin{equation}\label{eqn-M_t=int_0^t dW}
    M_t= \int_0^t  \sigma_s \d W_s    \mbox{ in } \rU^\prime \mbox{ $\mathbb{P}$-a.s.}
  \end{equation}
where we  used the notation introduced in \eqref{eqn-M_t=int_0^t xi_s  d W_s}.
\end{lemma}

\begin{remark}\label{rem-Lem-B.4-1}
It follows implicitly from \eqref{eqn-<M_tu,u>^2} that for every $s \in [0,\infty)$, $\mathbb{P}$-a.s.,
 \begin{equation}
     \int_E \dualitybig{\sigma_s(x)}{u}{\rU}{\rU^\prime}{2} \, \tilde{\kappa} (\d x)<\infty, \;\; u \in \rU. 
     \end{equation}
The It\^o integral used above is discussed in subsection \ref{subsec-Ito integral}.
\end{remark}
Note that Lemma \ref{Lem-B.4} holds true when we identify $\rU^\prime$ with $\rU$ 
provided all dualities in the formulation of it are replaced by the inner product in $\rU$. 
\begin{remark}
As a consequence of Lemma \ref{Lem-B.4}, if $\mathbb{P}$ is a solution to the martingale problem, then $X_t$ is a weak solution of an evolution equation in $\rU^\prime$.

\end{remark}

\section{Some results on uniform integrability}\label{app-uniform integrability}
One major reason for the usefulness of uniform integrability is the following result related to the equicontinuity of functions.
\begin{proposition}\label{propo-equicontinuity-append}
Let $T>0$ be arbitrary but fixed. Let $\{f_n\}_{n \in \mathbb{N}}$ be a sequence of complex, Lebesgue measurable functions $f_n$ on $[0,T]$.
Suppose that $\{f_n\}_{n \in \mathbb{N}}$ is uniformly integrable, i.e.
   \begin{equation*}
      \lim_{M \to \infty} \left(\sup_n \left\{ \int_{\{ \lvert f_n \rvert > M \}} \lvert f_n \rvert \, \d x \right\} \right)= 0.
   \end{equation*}
Then the sequence  $\{g_n\}_{n \in \mathbb{N}}$ of complex functions $g_n$ on $[0,T]$ defined by  
   \begin{equation}\label{eqn-g_n from f_n}
g_n: [0,T] \ni t \mapsto \int_0^t f_n(s) \,\d s, \; \; t \in [0,T], 
\end{equation}
is equicontinuous.
\end{proposition}

\begin{corollary}\label{cor-equicontinuity-append}
Let $T>0$ be arbitrary but fixed. Let $\{f_n\}_{n \in \mathbb{N}}$ be a sequence of complex, Lebesgue measurable functions $f_n$ on $[0,T]$.
Suppose that $\{f_n\}_{n \in \mathbb{N}}$ is bounded in $L^p(0,T)$ for some $p>1$. 
Then the sequence  $\{g_n\}_{n \in \mathbb{N}}$ defined by   \eqref{eqn-g_n from f_n}, 
is equicontinuous.
\end{corollary}

\begin{remark}\label{rem-L^1-equicontinuous}
It is possible to find  a sequence $\{f_n\}_{n \in \mathbb{N}}$ that is bounded in $L^1(0,T)$ such that the corresponding  sequence $\{g_n\}_{n \in \mathbb{N}}$  is not  equicontinuous. 
\end{remark}
\begin{lemma}\label{lem-uniformly integrable linear combination}
If two sequences $\{f_n\}_{n \in \mathbb{N}}$ and $\{h_n\}_{n \in \mathbb{N}}$ are uniformly integrable, then their linear combination is also a uniformly integrable sequence. 
If two sequences $\{f_n\}_{n \in \mathbb{N}}$ and $\{h_n\}_{n \in \mathbb{N}}$ are
equicontinuous then their linear combination is also an equicontinuous sequence.
\end{lemma}

\begin{proof}[Proof of Proposition \ref{propo-equicontinuity-append}]
\newdela{NOTE: this proof establishes Proposition \ref{propo-equicontinuity-append} (it uses uniform integrability, the Proposition's hypothesis, and concludes ``the proof of our proposition follows''), not Corollary \ref{cor-equicontinuity-append}. The header has been corrected. Corollary \ref{cor-equicontinuity-append} then follows because an $L^p(0,T)$-bounded sequence with $p>1$ is uniformly integrable (de la Vall\'ee-Poussin); consider adding that one-line deduction.}
Let us choose and fix  $\eps>0$. By the uniform integrability assumption, there exists $M_0=M_0(\eps)>0$ such that for all $M \geq M_0$ and all $n \in \mathbb{N}$, 
   \begin{equation}\label{Eqn-uniform-integrability-a}
     \int_{\{ \lvert f_n \rvert > M \} } \lvert f_n \rvert \, \d x \leq \frac{\eps}{2}.
   \end{equation}
Next, we set $\delta= \frac{\eps}{2 M_0}$ and we take $s,\, t \in [0,T]$ such that $0\leq s < t \leq T$ and $\lvert t - s \rvert \leq \delta$. Obviously, we have 
\begin{align*}
&\lvert g_n(t) - g_n(s) \rvert 
\leq \int_s^t \lvert f_n(\tau) \rvert \, \d \tau
= \int_{ [s,t] \cap \{ \lvert f_n \rvert > M_0 \} }  \lvert f_n(\tau) \rvert \, \d \tau + \int_{ [s,t] \cap \{ \lvert f_n \rvert \leq M_0 \} }  \lvert f_n(\tau) \rvert \, \d \tau \\
&\leq \int_{\{ \lvert f_n \rvert > M_0 \} }  \lvert f_n(\tau) \rvert \, \d \tau + M_0 \lvert t - s \rvert
\leq \int_{\{ \lvert f_n \rvert > M_0 \} }  \lvert f_n(\tau) \rvert \, \d \tau + \delta M_0.
\end{align*}
This, jointly with \eqref{Eqn-uniform-integrability-a} yields that 
  \begin{equation*}
    \lvert g_n(t) - g_n(s) \rvert 
     \leq \frac{\eps}{2} +  \delta M_0 
     =  \frac{\eps}{2} +  \frac{\eps}{2 M_0} M_0
     = \frac{\eps}{2} + \frac{\eps}{2}
     = \eps,
   \end{equation*}
and the proof of our proposition follows.
\end{proof}
Since a constant sequence $h_n=f$, where $f\in L^1(0,T)$, is uniformly integrable, from the previous two results we deduce the following result which we use in the paper. 
\begin{corollary} \label{cor-uniformly integrable linear combination}
Assume that $T>0$. If a sequence  $\{f_n\}_{n \in \mathbb{N}}$ is bounded in $L^p(0,T)$ for some $p>1$ and  $g \in L^1(0,T)$,
then the sequence $\{\alpha f_n + \beta g\}_{n \in \mathbb{N}}$ with $\alpha,\beta \in \mathbb{C}$ is uniformly integrable on $(0,T)$.

\end{corollary}
\begin{corollary}
Let $\{f_n\}_{n \in \mathbb{N}}$ be a family of complex, Lebesgue measurable functions $f_n$ on $[0,T]$.
Suppose that $\{f_n\}_{n \in \mathbb{N}}$ is uniformly integrable. Then the family $\{g_n\}_{n \in \mathbb{N}}$ of complex functions $g_n$ on $[0,T]$  defined in \eqref{eqn-g_n from f_n} 
 is uniformly absolutely continuous, that is, for every $\eps>0$, there exists $\delta> 0$ such that for all $n \in \mathbb{N}$ and for all pairwise disjoint segments $\{[x_i,y_i]\}_{1 \leq i \leq N }$: if 
\[
\sum_{i=1}^N \lvert y_i - x_i \rvert \leq \delta,
\]
then 
     \begin{equation*}
        \sum_{i=1}^N \lvert g_n(y_i) -  g_n(x_i) \rvert \leq \eps.
      \end{equation*}
\end{corollary}
The following useful result is a consequence of the Arzel\`a-Ascoli Theorem, see \cite[p. 30]{Reed+Simon_1980_vI}.
\begin{proposition}\label{prop-AZ-Theorem}
If $u_n: [0,T] \to \mathbb{C}$ is a continuous function for every $n \in \mathbb{N}$ such that 
\[
u_n(t) \to u(t) \mbox{ for every } t \in [0,T],
\]
and the sequence $(u_n)$ is equicontinuous, then $u: [0,T] \to \mathbb{C}$ is a continuous function and 
\[
u_n \to u \mbox{ in } C([0,T]).
\]
\end{proposition}


\section{Compactness criterion}
We report the following compactness and embedding results, see \cite{Strauss_1966,Vishik+Fursikov_1988}.
\begin{lemma}\label{Lem C1} 
Let $X$ and $Y$ be two Banach spaces such that the embeddings $X \embed  Y$ and $Y^\prime \embed  X^\prime$ are dense and continuous. Then,
\[
L^\infty(0,T;X) \cap C([0,T];Y_w) \embed   C([0,T];X_w).
\]
\end{lemma}
Here, $C([0,T];Y_w)$ and $C([0,T];X_w)$ denote the spaces of weakly continuous functions $\bv: [0,T] \to Y$ and $\bv: [0,T] \to X$ endowed with the weak topology, respectively.
\section{Kuratowski Theorem}\label{sec-Kuratowski Theorem}

\begin{proposition}\label{Propo-Borel-subset}
Let $X$ be a Polish space and $(Y,\mathscr{Y})$ be a topological space such that $Y \subset X$ and the embedding $Y \embed X$ is continuous. If a set $A \subset X$ is a Borel subset of $X$, then $A \cap Y$ is a Borel subset of $Y$. 
\end{proposition}
\begin{proof}[Proof of Proposition \ref{Propo-Borel-subset}] 
By the continuity of the embedding $Y \embed X$ it follows  that the family  $\mathscr{F}$ defined by 
\[
\mathscr{F}= \{A \subset X: A \cap Y \in \mathcal{B}(Y)\}
\]
is a $\sigma$-field on $X$ and it contains the topology of $X$. Thus, $\mathcal{B}(X) \subset \mathscr{F}$ and the proof of our Proposition follows.
\end{proof}

\section{Approximation of functions of linear growth}

Assume that $g:\mathbb{R}^d \to Y=\mathbb{R}$ is continuous and of linear growth, i.e.
\[
\lvert g(z) \rvert\leq C_1 + C_2 \vert z \vert,\;\; z\in \mathbb{R}^d.
\]
Define 
\[
g \ast \varphi_\delta(z)\coloneqq \int_{\mathbb{R}^d} g(z-y) \varphi_\delta(y)\,\d y, \;\; z \in \mathbb{R}^d,\;\;\delta>0,
\]
where
$\varphi_\delta$ is the standard mollifier of Section \ref{Sect-approximation of g}, i.e. $\varphi_\delta(x)=\delta^{-d}\varphi(\delta^{-1}x)$ with $\varphi\ge 0$, $\supp\varphi\subseteq B(0,1)$ and $\int\varphi=1$ (so $\supp\varphi_\delta\subseteq B(0,\delta)$ and $\int\varphi_\delta=1$).
\begin{claim}
The function $g \ast \varphi_\delta$ is well defined and of linear growth.
\end{claim}
Indeed, let us choose and fix $\delta \in (0,1)$. Then we have
\begin{align*}
&\int_{\mathbb{R}^d} \vert g(z-y) \vert \varphi_\delta(y)\, \d y 
=\int_{B(0,\delta)} \vert g(z-y) \vert  \varphi_\delta(y)\,\d y
\\
&\leq \int_{B(0,\delta)} (C_1 + C_2 \vert z - y \vert) \varphi_\delta(y)\,\d y
\\
&\leq C_1 + C_2 \int_{B(0,\delta)} (\vert z \vert + \vert y \vert) \varphi_\delta(y)\, \d y
\leq C_1 + C_2 (\delta + \vert z\vert)
\\
&\leq(C_1+C_2 \delta) + C_2\vert z\vert
\leq C_1 + C_2(1 + \vert z\vert), \;\; z \in \mathbb{R}^d.
\end{align*}
Furthermore, by the LDCT, the function $g \ast \varphi_\delta$ is continuous and 
\[
g \ast \varphi_\delta \to g \mbox{ locally uniformly}.
\]
Indeed, 
\begin{align*}
\vert g \ast \varphi_\delta(z) - g(z) \vert    
\leq \int_{\mathbb{R}^d} \vert g(z - y) - g(z) \vert \varphi_\delta(y)\, \d y 
= \int_{B(0,\delta)} \vert g(z - y) - g(z) \vert \varphi_\delta(y)\, \d y.
\end{align*}
Since $g$ is continuous, it is uniformly continuous on compact sets and the result follows.

\end{document}